\documentclass{ims9x6}
\usepackage{amscd}
\usepackage{amsmath}
\usepackage{amsfonts}
\usepackage{mathtext}
\usepackage{amssymb}
\usepackage{euscript}
\usepackage{graphics}
\usepackage{times}
\usepackage[T1]{fontenc}
\usepackage{graphicx}

\makeindex

\begin{document}

\chapter{FULLERENES, POLYTOPES AND TORIC TOPOLOGY}

\markboth{V.M. Buchstaber, N.Yu. Erokhovets}{Fullerenes, Polytopes and Toric
Topology} 

\author{Victor M. Buchstaber}

\address{Steklov Mathematical Institute of Russian Academy of Sciences\\
Gubkina str. 8, 119991, Moscow, Russia\\
Department of Geometry and Topology\\
Faculty of Mechanics and Mathematics\\
Lomonosov Moscow State University\\
Leninskie Gory 1, 119991, Moscow, Russia\\
e-mail: buchstab@mi.ras.ru}

\author{Nikolay Yu. Erokhovets}

\address{Department of Geometry and Topology\\
Faculty of Mechanics and Mathematics\\
Lomonosov Moscow State University\\
Leninskie Gory 1, 119991, Moscow, Russia\\
e-mail: erochovetsn@hotmail.com}

\begin{abstract}
The lectures are devoted to a remarkable class of $3$-dimensional polytopes, which are mathematical models of the important object of quantum physics, quantum chemistry and nanotechnology -- fullerenes. The main goal is to show how results of toric topology help to build  combinatorial invariants of fullerenes. Main notions are introduced during the lectures. The lecture notes are addressed to a wide audience.
\end{abstract}

\vspace*{12pt}

\chaptercontents  

\section*{Introduction}
\addcontentsline{toc}{section}{Introduction}

These lecture notes are devoted to results on crossroads of the classical polytope theory, toric topology, and mathematical theory of fullerenes.  Toric topology  is a new area of mathematics that emerged at the end of the 1990th on the border of equivariant topology, algebraic and symplectic geometry, combinatorics, and commutative algebra. Mathematical theory of fullerenes is a new area of mathematics focused on problems formulated on the base of outstanding achievements of quantum physics, quantum chemistry and nanotechnology. 

The text is based on the lectures delivered by the first author on the Young Topologist Seminar during the program on Combinatorial and Toric Homotopy (1-31 August 2015) organized jointly by the Institute for Mathematical Sciences and the Department of Mathematics of National University of Singapore. 

The lectures are oriented to a wide auditorium. We give all necessary notions and constructions. For key results, including new results, we either give a full prove, or a sketch of a proof with an appropriate reference. These results are oriented for the applications to the combinatorial study and classification of fullerenes.  

\subsection*{Lecture guide}
\begin{itemlist}
\item One of the main objects of the toric topology is the moment-angle
    functor~$P\to\mathcal{Z}_P$.
\item It assigns to each simple $n$-polytope $P$ with $m$ facets an $(n+m)$-dimensional moment-angle complex  $\mathcal{Z}_P$ with an action of a compact torus $T^m$, whose orbit space $\mathcal{Z}_P/T^m$ can be identified with $P$.
\item The space $\mathcal{Z}_P$ has the structure of a smooth manifold with a smooth action of $T^m$.
\item A mathematical fullerene is a three dimensional convex simple
    polytope with all $2$-faces being pentagons and hexagons.
\item In this case the number $p_5$ of pentagons is $12$.
\item The number $p_6$ of hexagons can be arbitrary except for $1$.
\item Two combinatorially nonequivalent fullerenes with the same number
    $p_6$ are called combinatorial isomers. The number of
    combinatorial isomers of fullerenes grows fast as a function of
    $p_6$.
\item  At that moment the problem of classification of fullerenes is
    well-known and is vital due to the applications in chemistry,
    physics, biology and nanotechnology.
\item Our main goal is to apply methods of toric topology to build
    combinatorial invariants distinguishing isomers.
\item Thanks to the toric topology, we can assign to each fullerene $P$
    its moment-angle manifold $\mathcal{Z}_P$.
\item The cohomology ring $H^*(\mathcal{Z}_P)$ is a combinatorial
    invariant of the fullerene $P$.
\item We shall focus upon results on the rings $H^*(\mathcal{Z}_P)$ and
    their applications based on geometric interpretation of cohomology
    classes and their products.
\item The multigrading in the ring $H^*(\mathcal{Z}_P)$, coming from the
    construction of $\mathcal{Z}_P$, and the multigraded Poincare duality
    play an important role here.
\item There exist $7$ truncation operations on simple $3$-polytopes such that any fullerene is combinatorially equivalent to a polytope obtained from the dodecahedron by a sequence of these operations.    
\end{itemlist}

\newpage
\section{Lecture 1. Basic notions}
\subsection{Convex polytopes}
\begin{definition}
A {\em convex polytope}\index{polytope} $P$ is a {\em bounded} set of the form
$$
P=\{\boldsymbol{x}\in \mathbb R^n\colon \boldsymbol{a}_i\boldsymbol{x}+b_i\geqslant 0, i=1,\dots,m\}, 
$$
where $\boldsymbol{a}_i\in\mathbb R^n$, $b_i\in\mathbb R$, and $\boldsymbol{x}\boldsymbol{y}=x_1y_1+\dots+x_ny_n$ is the standard scalar product in $\mathbb R^n$. Let this representation be {\em irredundant}, that is a deletion
of any inequality changes the set. Then each hyperplane
$\mathcal{H}_i=\{\boldsymbol{x}\in\mathbb R^n\colon
\boldsymbol{a}_i\boldsymbol{x}+b_i=0\}$ defines a {\em facet} $F_i=P\cap
\mathcal{H}_i$. Denote by $\mathcal{F}_P=\{F_1,\dots,F_m\}$ the ordered set of facets
of $P$.  For a subset $S\subset \mathcal{F}_P$ denote $|S|=\bigcup_{i\in S}F_i$. We have $|\mathcal{F}_P|=\partial P$ is the boundary of $P$.

A {\em face}\index{face}\index{polytope!face} is a subset of a polytope that is an intersection of facets. Two
convex polytopes $P$ and $Q$ are {\em combinatorially equivalent} ($P\simeq Q$) if there is an
inclusion-preserving bijection between their sets of faces. A {\em combinatorial polytope}\index{combinatorial polytope}\index{polytope!combinatorial polytope} is an equivalence class of combinatorially equivalent convex polytopes. In most cases we  consider combinatorial polytopes and write $P=Q$ instead of $P\simeq Q$.
\end{definition}
\begin{example} An $n$-simplex\index{simplex} $\Delta^n$ in $\mathbb R^n$ is the convex hull of $n+1$ affinely
independent points. Let $\{\boldsymbol{e}_1,\dots,\boldsymbol{e}_n\}$ be the
standard basis in $\mathbb R^n$. The $n$-simplex ${\rm
conv}\{\mathbf{0},\boldsymbol{e}_1,\dots,\boldsymbol{e}_n\}$ is called {\em
standard}. It is defined in $\mathbb R^n$ by $n+1$ inequalities:
$$
x_i\geqslant 0\text{ for } i=1,\dots,n,\quad\text{ and }-x_1-\dots-x_n+1\geqslant 0.
$$
The {\em standard $n$-cube}\index{cube} $I^n$ is defined in $\mathbb R^n$ by $2n$
inequalities
$$
x_i\geqslant 0,\quad -x_i+1\geqslant 0,\quad\text{for }i=1,\dots,n.
$$
\end{example}
\begin{definition}
An {\em orientation}\index{polytope!orientation}  of a  combinatorial convex $3$-polytope is a choice of the cyclic order of vertices of each facet such that for any two facets with a common edge the orders of vertices induced from facets to this edge are opposite.
A combinatorial convex $3$-polytope with given orientation is called {\em oriented}\index{oriented polytope}. 
\end{definition}

\noindent {\bf Exercise:} 
\begin{itemize}
\item Any geometrical realization of a combinatorial $3$-polytope $P$ in $\mathbb R^3$ with standard orientation induces an orientation of $P$.
\item Any combinatorial $3$-polytope has exactly two orientations.
\item Define an oriented combinatorial convex $n$-polytope.
\end{itemize}

\begin{definition}
A polytope is called  {\em combinatorially chiral}\index{chiral polytope}\index{polytope!chiral polytope} if any it's combinatorial equivalence to itself preserves the orientation. 
\end{definition}
Simplex $\Delta^3$ and cube $I^3$ are not combinatorially chiral.

\noindent {\bf Exercise:} Give an example of a combinatorially chiral  $3$-polytope.

There is a classical notion of a {\em (geometrically) chiral} polytope (connected with the right-hand and the left-hand rules).
\begin{definition}
A convex $3$-polytope $P\subset\mathbb R^3 $ is called {\em  (geometrically) chiral}\index{polytope!geometrically chiral} if there is no orientation preserving isometry of $\mathbb R^3$ that maps $P$ to its mirror image.
\end{definition}
\begin{proposition}\label{Chiral}
A combinatorially chiral polytope is geometrically chiral, while a geometrically chiral polytope can be not combinatorially chiral. 
\end{proposition}
\begin{proof}
The orientation-preserving isometry of $\mathbb R^3$ that maps $P$ to its mirror image defines the combinatorial
 equivalence that changes the orientation. On the other hand, the simplex $\Delta^3$ realized with all angles of all facets different can not be mapped to itself by an isometry of $\mathbb R^3$ different from the identity. Hence it is chiral. The odd permutation of vertices defines the combinatorial equivalence that changes the orientation; hence $\Delta^3$ is not combinatorially chiral.
\end{proof} 
\subsection{Schlegel diagrams}
Schlegel diagrams were introduced by Victor Schlegel (1843 - 1905) in 1886.
\begin{definition} A Schlegel diagram\index{Schlegel diagram}\index{polytope!Schlegel diagram} of a convex polytope
$P$ in $\mathbb R^3$ is a {\em projection} of $P$ from $\mathbb R^3$ into $\mathbb R^2$
through a point beyond one of its facets.
\end{definition}

The resulting entity is a subdivision of the projection of this facet that is
{\em combinatorial invariant} of the original polytope. It is clear
that a Schlegel diagram depends on the choice of the facet.

\noindent{\bf{Exercise:}} Describe the Schlegel diagram of the cube and the octahedron.
\begin{figure}
\begin{center}
\begin{tabular}{cc}
\includegraphics[scale=0.25]{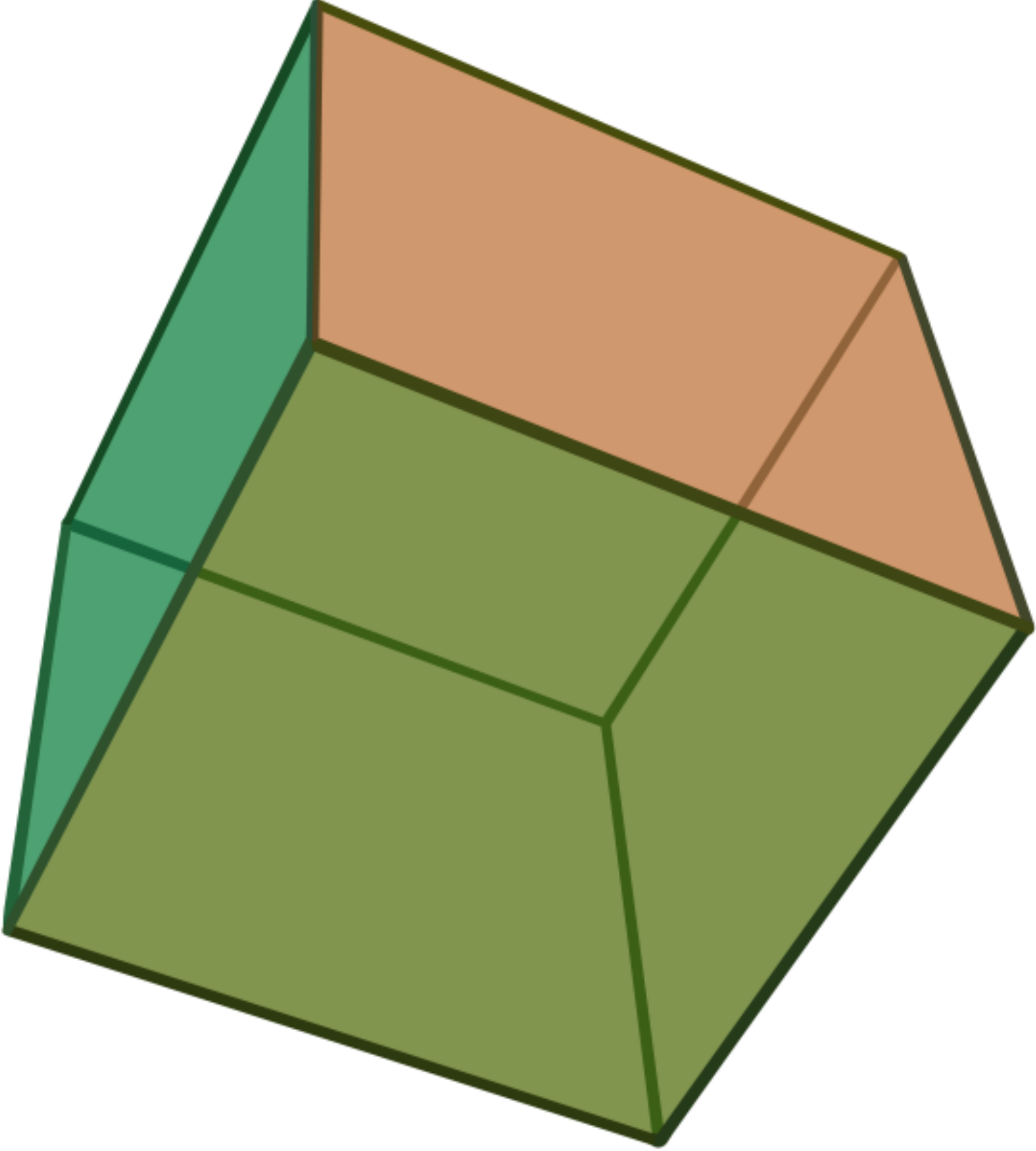}&\includegraphics[scale=0.25]{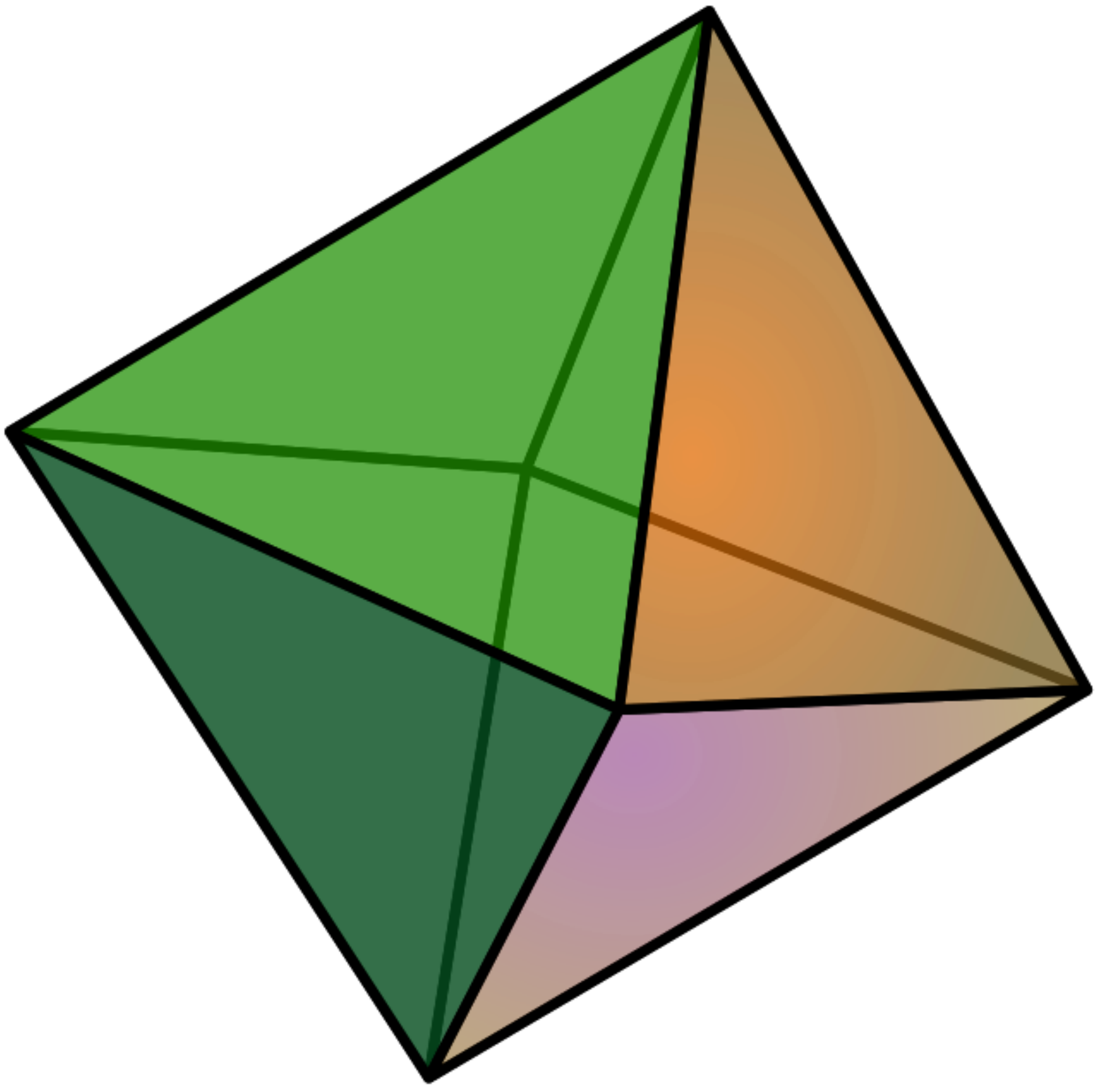}\\
\end{tabular}
\caption{Cube and octahedron (www.wikipedia.org)}
\end{center}
\end{figure}

\begin{example}
\begin{figure}
\begin{center}
\begin{tabular}{cc}
\includegraphics[height=3cm]{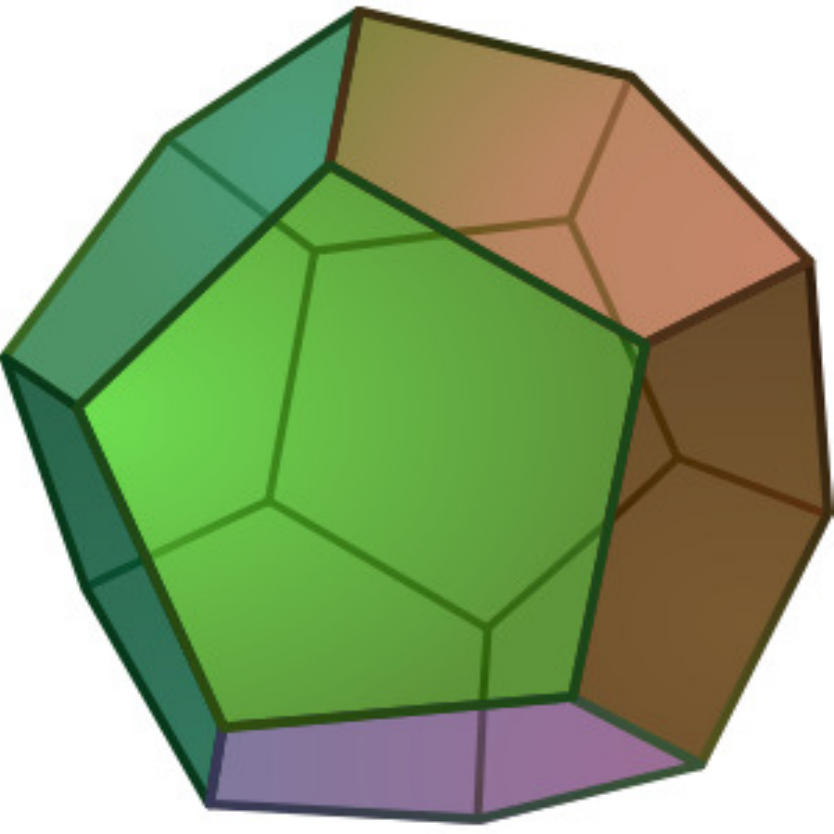}&\includegraphics[height=3cm]{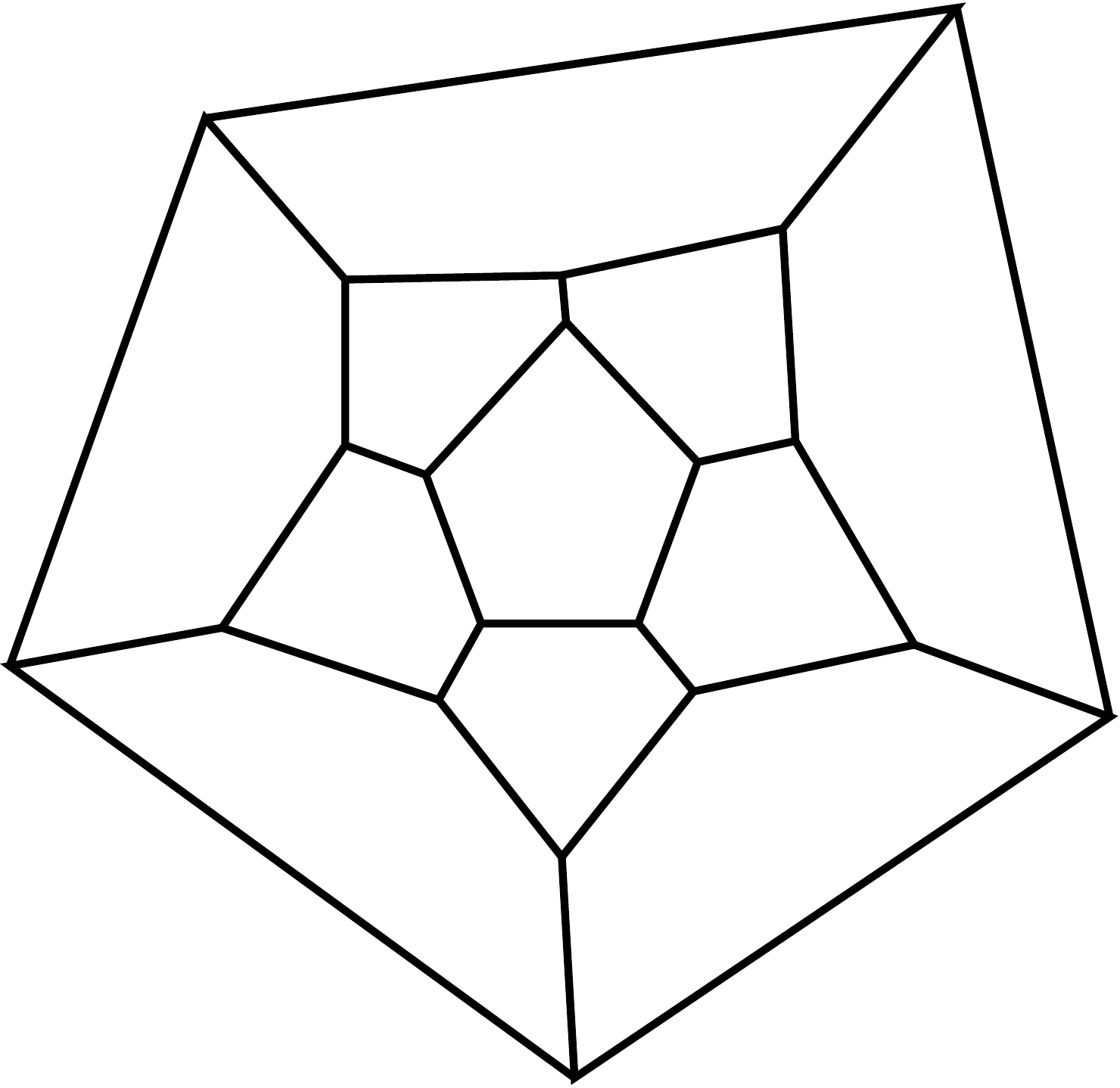}\\
\end{tabular}
\caption{Dodecahedron and its Schlegel digram  (www.wikipedia.org)}
\end{center}
\end{figure}
\end{example}
\subsection{Euler's formula}
Let $f_0$, $f_1$, and $f_2$ be numbers of vertices, edges, and $2$-faces of a
$3$-polytope. Leonard Euler (1707-1783) proved the following fundamental relation\index{Euler's formula}: 
{\bf\Huge
$$
f_0-f_1+f_2=2
$$
}

By a {\em fragment}\index{fragment on a polytope} we mean a subset $W\subset P$ that is a union of faces of
$P$. Define an {\em Euler characteristics} of $W$ by
$$
\chi(W)=f_0(W)-f_1(W)+f_2(W).
$$
If $W_1$ and $W_2$ are fragments, then $W_1\cup W_2$ and $W_1\cap W_2$ are
fragments.

\noindent{\bf Exercise:} Proof the {\em inclusion-exclusion formula}
$$
\chi(W_1\cup W_2)=\chi(W_1)+\chi(W_2)-\chi(W_1\cap W_2),
$$

\subsection{Platonic solids}

\begin{definition}
A {\em regular polytope (Platonic solid)} \cite{C73}\index{Platonic solid}\index{polytope!Platonic solid}\index{polytope!regular} is a convex $3$-polytope with all facets being congruent regular polygons that are assembled in the
same way around each vertex.
\end{definition}
There are only $5$ Platonic solids, see Fig. \ref{Pls}.
\begin{figure}
\begin{center}
\begin{tabular}{ccc}
\begin{tabular}{c}
\includegraphics[scale=0.12]{Cube.pdf}\\
Cube\\
$(8,12,6)$\\
\\
\includegraphics[scale=0.2]{Dod.pdf}\\
Dodecahedron\\
$(20,30,12)$
\end{tabular}
&
\begin{tabular}{c}\\
\includegraphics[scale=0.15]{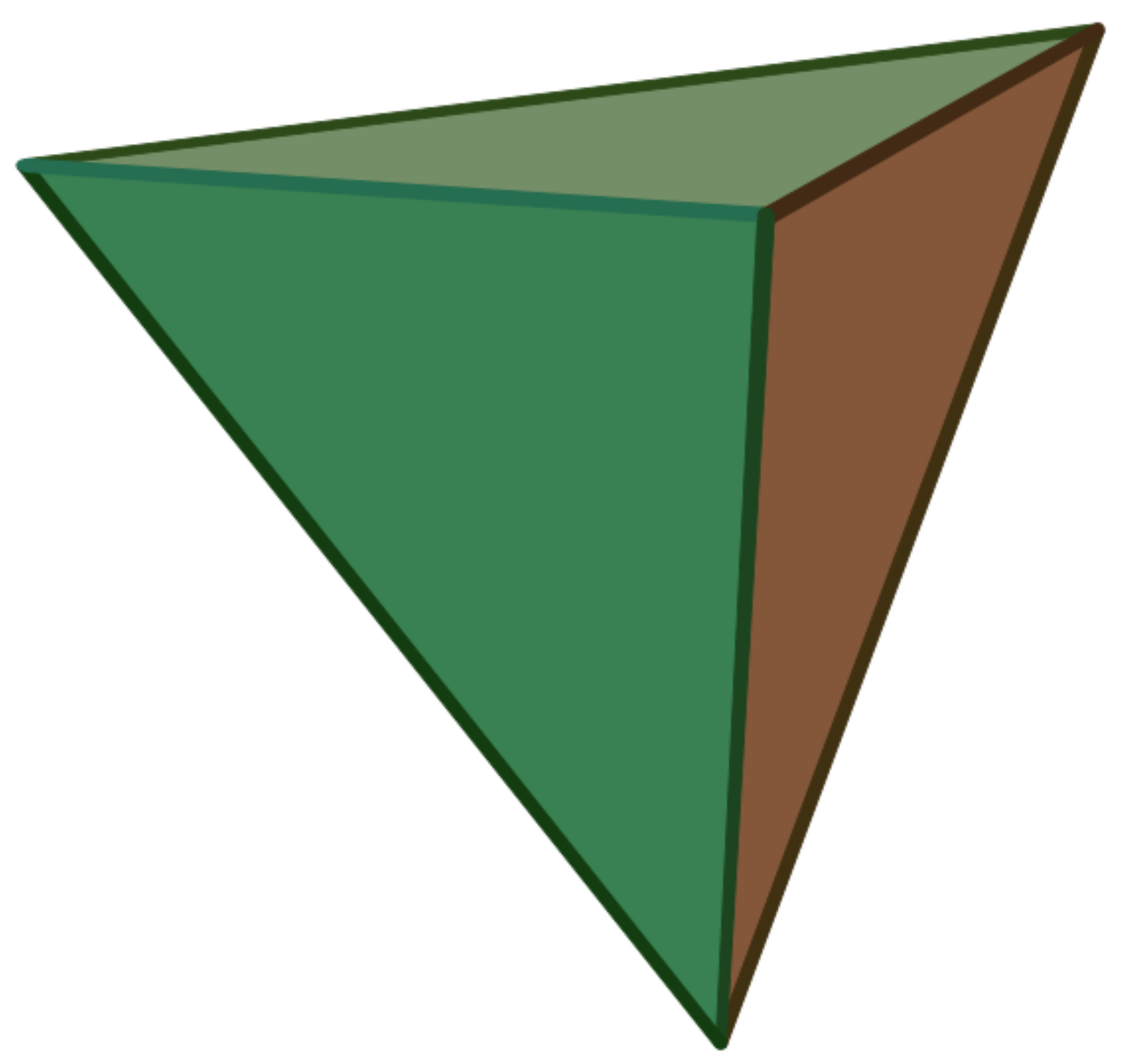}\\
Tetrahedron\\
$(4,6,4)$\\
\end{tabular}
&
\begin{tabular}{c}
\includegraphics[scale=0.12]{Oct.pdf}\\
Octahedron\\
$(6,12,8)$\\
\\
\includegraphics[scale=0.15]{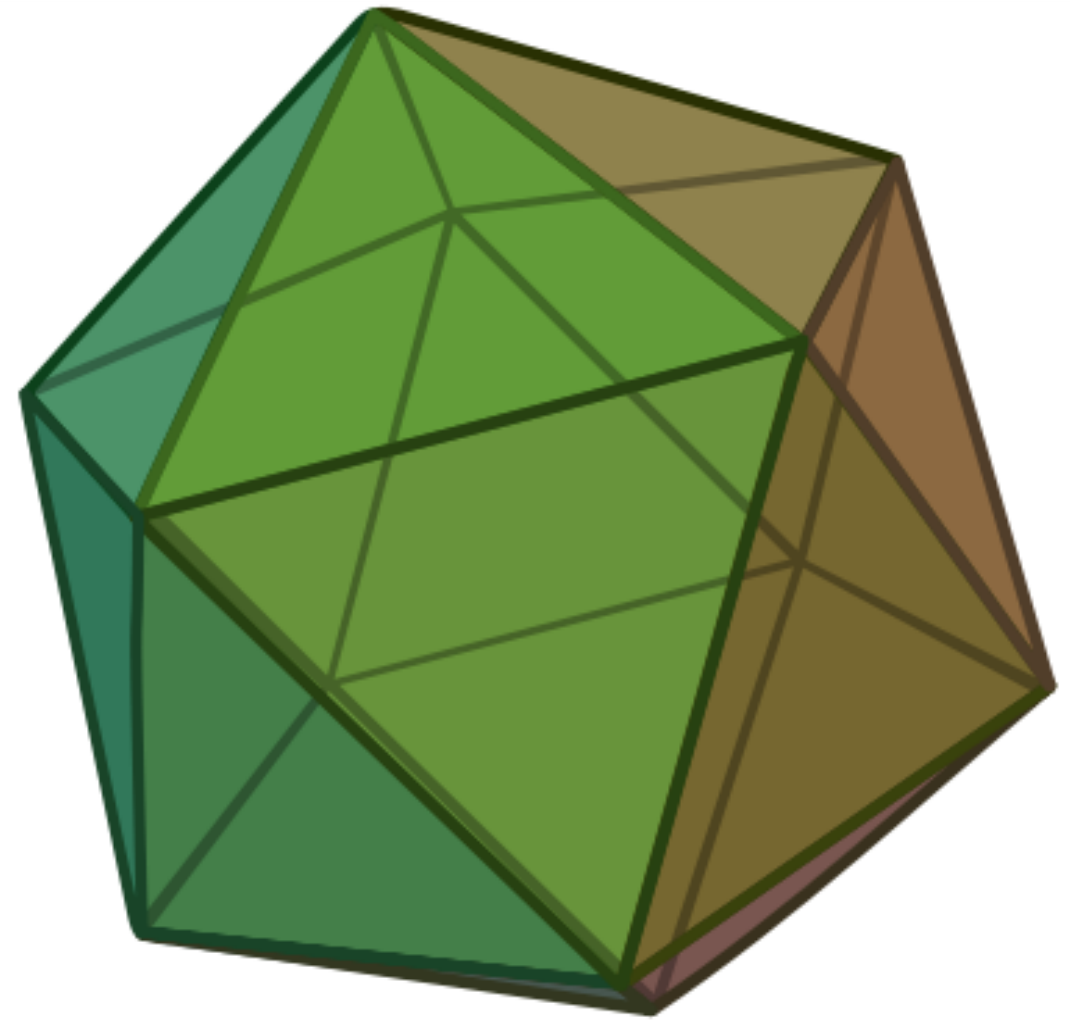}\\
Icosahedron\\
$(12,30,20)$\\
\end{tabular}
\end{tabular}
\end{center}
\caption{Platonic solids with $f$-vectors $(f_0,f_1,f_2)$ (www.wikipedia.org)}\label{Pls}
\end{figure}
\index{$f$-vector}\index{polytope!$f$-vector}All Platonic solids are vertex-, edge-, and facet-transitive. They are not combinatorially chiral.
\subsection{Archimedean solids}
\begin{definition}
An {\em Archimedean solid}\index{Archimedean solid}\index{polytope!Archimedean solid} \cite{C73} is a convex $3$-polytope with all facets -- regular polygons of {\em two or more types}, such that for any pair
of vertices there is a  symmetry of the polytope that moves one vertex to
another.
\end{definition}
There are only $13$ solids with this properties: $10$ with facets of two
types, and $3$ with facets of three types.  On the following figures we
present Archimedean solids. For any polytope we give vectors $(f_0,f_1,f_2)$
and $(k_1,\dots,k_p; q)$, where $q$ is the valency of any vertex and a tuple
$(k_1,\dots k_p)$ show which $k$-gons are present.
\begin{figure}
\begin{center}
\begin{tabular}{ccc}
\includegraphics[scale=0.5]{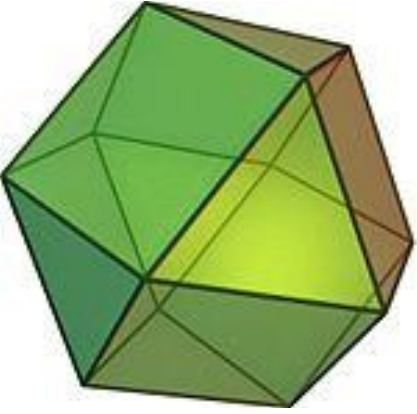}&\hspace{0.5cm}&
\includegraphics[scale=0.5]{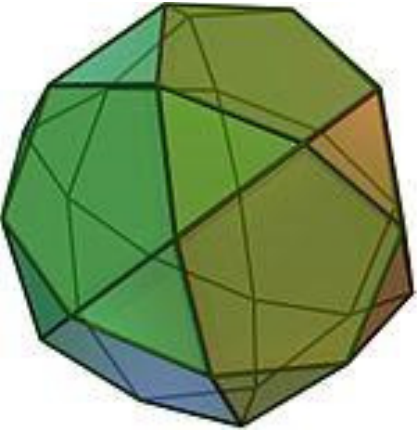}\\
Cuboctahedron&&Icosidodecahedron\\
$(12,24,14), (3,4;4)$&&$(30,60,32),(3,5;4)$
\end{tabular}
\end{center}
\hspace{-0.18cm}\begin{tabular}{ccc}
\includegraphics[scale=0.5]{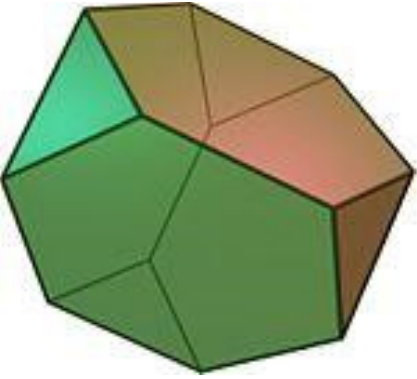}&
\includegraphics[scale=0.5]{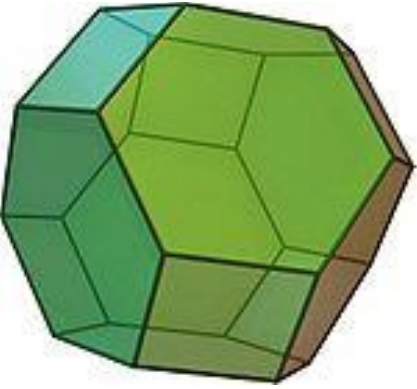}&
\includegraphics[scale=0.5]{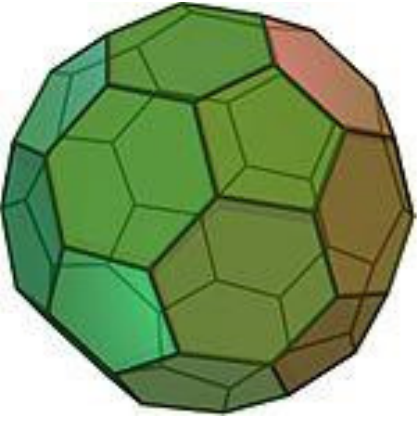}\\
\hspace{-1cm}Truncated tetrahedron&Truncated octahedron&{\bf Truncated icosahedron}\\
$(12,18,8),(3,6;3)$&$(24,36,14),(4,6;3)$&$(60,90,32),(5,6;3)$
\end{tabular}

\begin{center}
\begin{tabular}{ccc}
\includegraphics[scale=0.5]{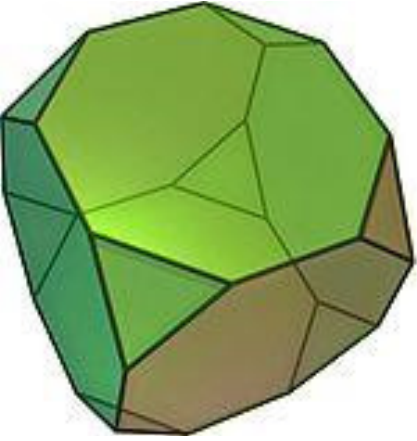}&\hspace{2cm}&
\includegraphics[scale=0.5]{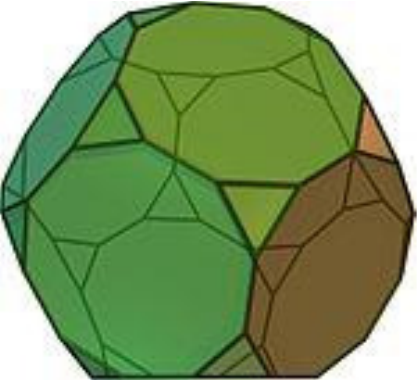}\\
Truncated cube&\hspace{2cm}&Truncated dodecahedron\\
$(24,36,14),(3,8;3)$&\hspace{2cm}&$(60,90,32),(3,10;3)$
\end{tabular}
\begin{tabular}{ccc}
\includegraphics[scale=0.5]{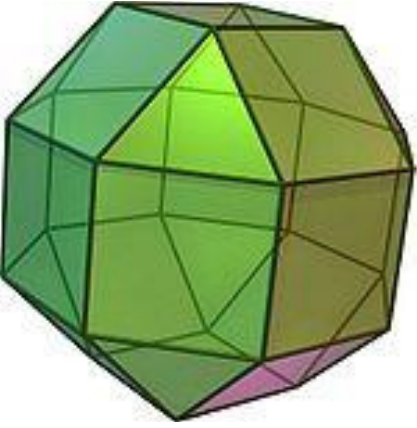}&\hspace{1cm}&
\includegraphics[scale=0.5]{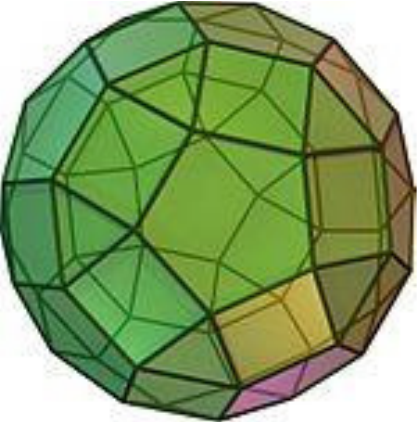}\\
Rhombicuboctahedron&&Rhombicosidodecahedron\\
$(24,48,26),(3,4;4)$&&$(60,120,62),(3,4,5;4)$
\end{tabular}
\end{center}

\begin{center}
\begin{tabular}{ccc}
\includegraphics[scale=0.5]{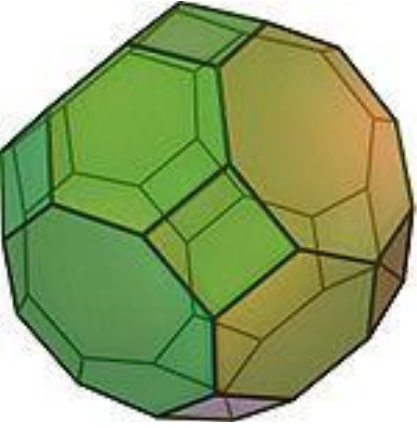}&\hspace{1cm}&
\includegraphics[scale=0.5]{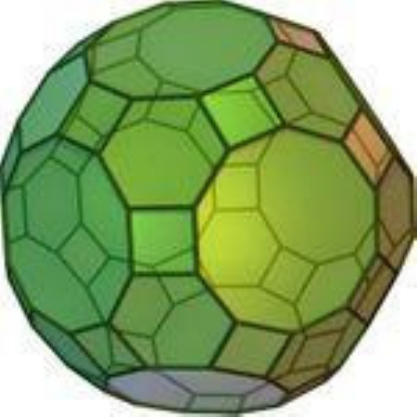}\\
Truncated&&Truncated\\
cuboctahedron&&icosidodecahedron\\
$(48,72,26),(4,6,8;3)$&&$(120,180,62),(4,6,10;3)$
\end{tabular}
\begin{tabular}{cc}
\includegraphics[scale=0.5]{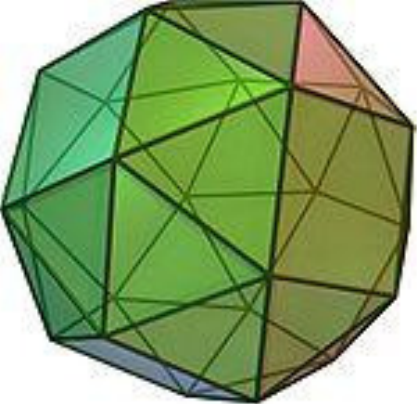}&
\includegraphics[scale=0.5]{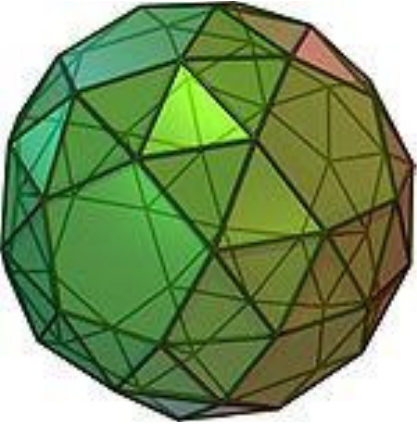}\\
Snub cube&Snub dodecahedron\\
$(24,60,38),(3,4;5)$&$(60,150,92),(3,5;5)$
\end{tabular}
\end{center}
\caption{Archimedean solids with $f$-vectors and facet-vertex types (www.wikipedia.org)}
\end{figure}
Snub cube and snub dodecahedron are combinatorially chiral, while other $11$ Archimedean solids are not combinatorially chiral.
\begin{figure}[h]
\begin{center}
\begin{tabular}{cc}
\includegraphics[height=3cm]{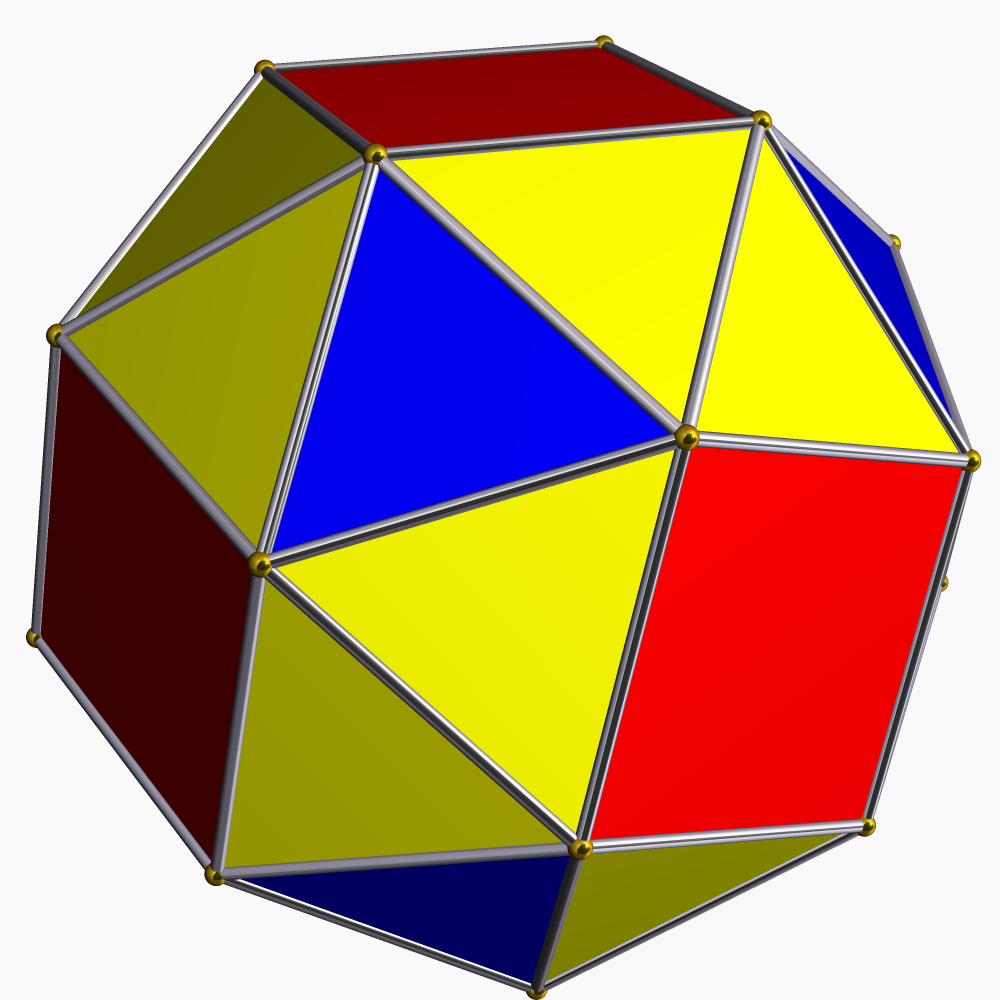}&
\includegraphics[height=3cm]{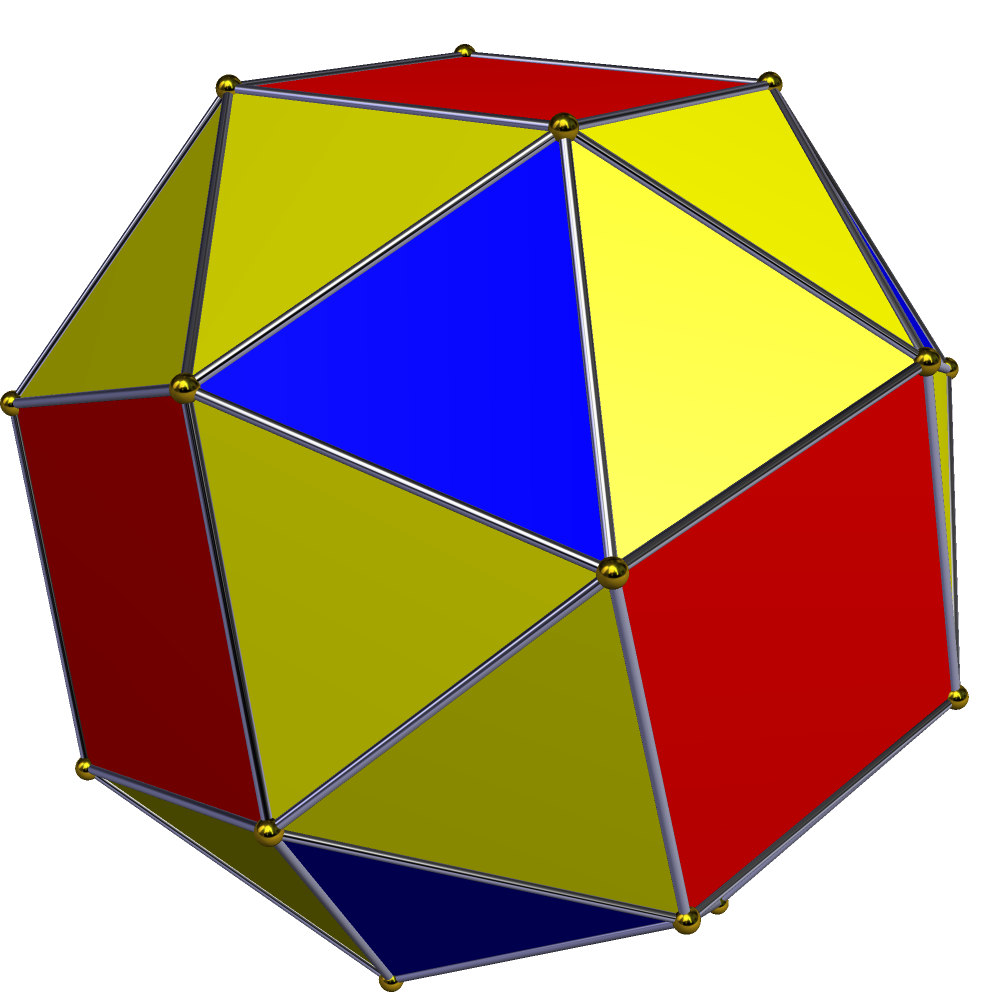}\\
\end{tabular}
\end{center}
\caption{
Left and right snub cube. Fix the orientations induced from ambient space. There is no combinatorial equivalence preserving this orientation. (www.wikipedia.org)
}
\end{figure}
\subsection{Simple polytopes}

An $n$-polytope is {\em simple}\index{simple polytope}\index{polytope!simple} if any its vertex is contained in exactly $n$
facets.
\begin{figure}[h]
\begin{center}
\includegraphics[height=3cm]{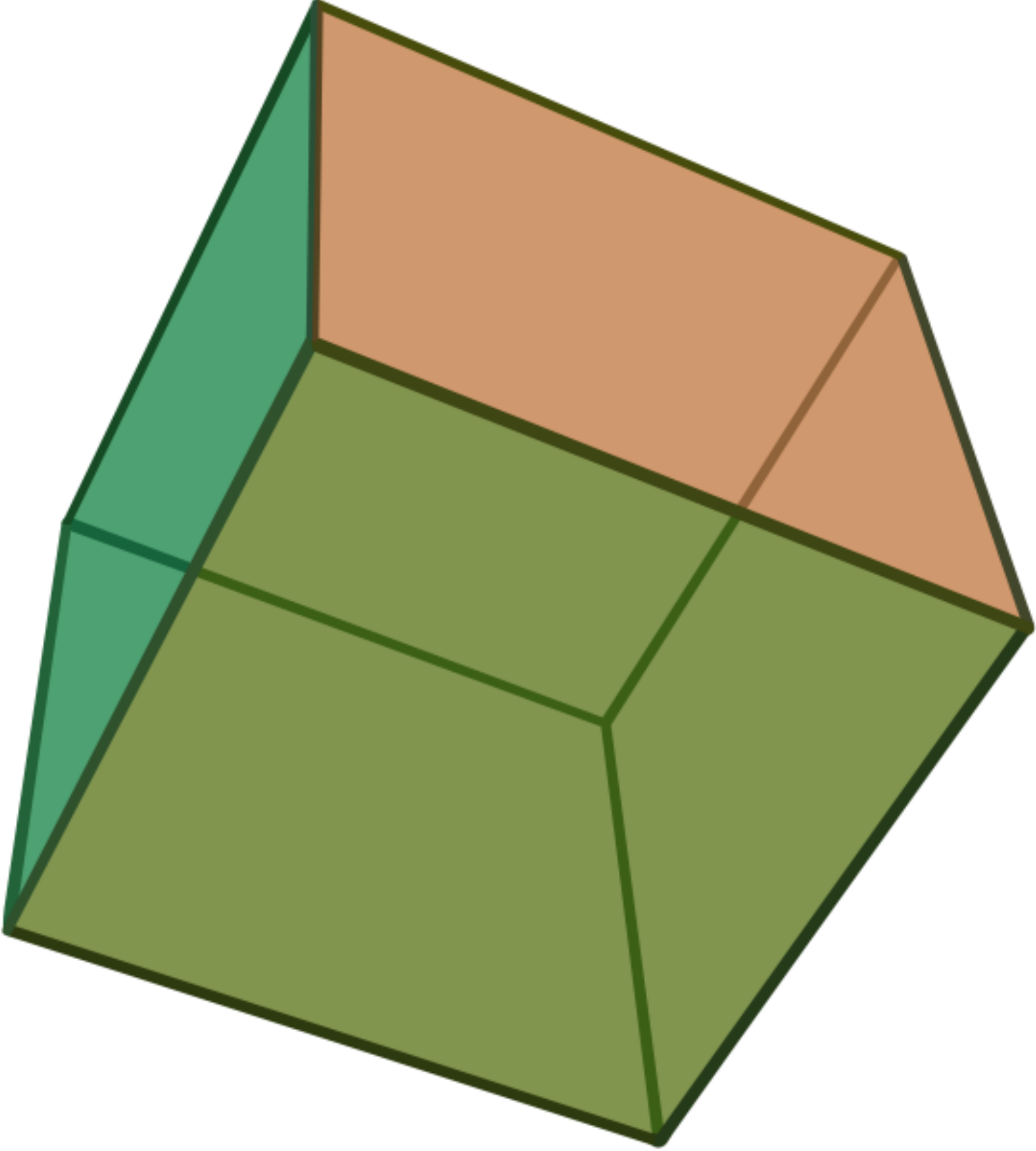}\qquad\includegraphics[height=3cm]{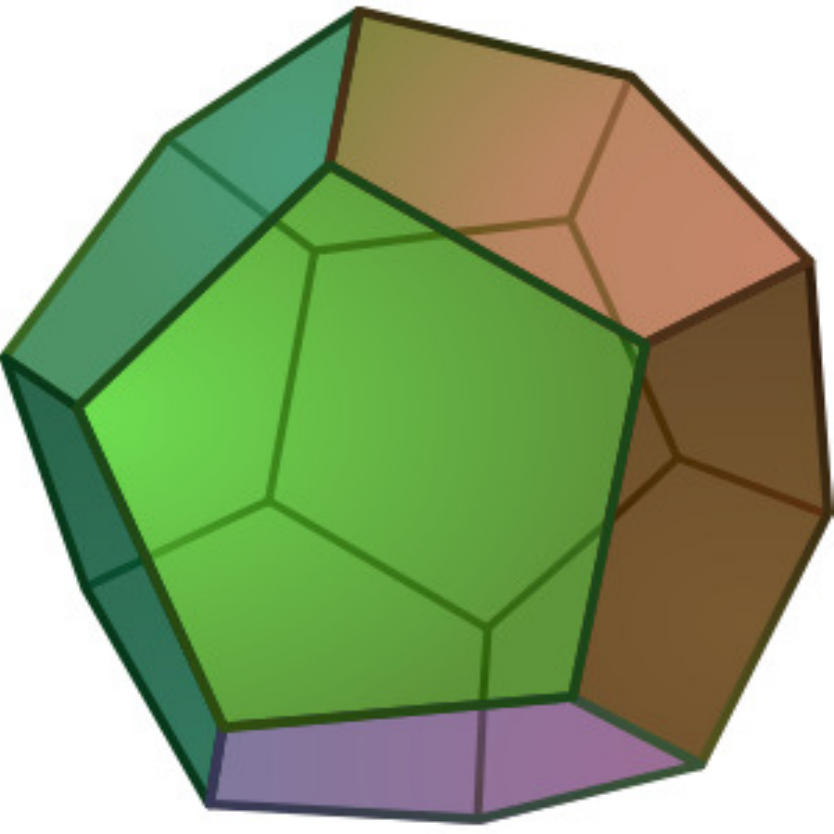}\qquad\includegraphics[height=3cm]{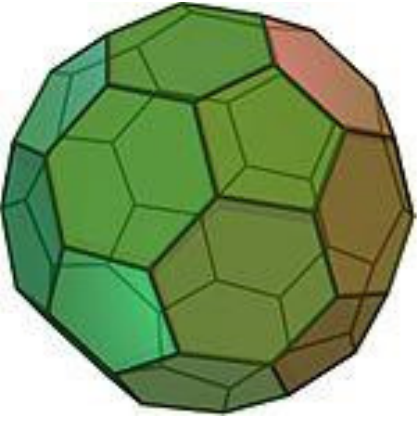}
\end{center}
\caption{Simple polytopes: cube, dodecahedron and truncated icosahedron  (www.wikipedia.org)}
\end{figure}

\begin{example}
\begin{itemlist}
\item $3$ of $5$ Platonic solids are simple.
\item $7$ of $13$ Archimedean solids are simple.
\end{itemlist}
\end{example}
\noindent{\bf Exercise:}
\begin{itemlist}
\item A simple $n$-polytope with all $2$-faces triangles is combinatorially
    equivalent to the $n$-simplex.
\item A simple $n$-polytope with all $2$-faces quadrangles is combinatorially
    equivalent to the $n$-cube.
\item A simple $3$-polytope with all $2$-faces pentagons is combinatorially
    equivalent to the dodecahedron.
\end{itemlist}
\subsection{Realization of $f$-vector}
\begin{theorem}\cite{S06} (Ernst Steinitz (1871-1928))
An integer vector $(f_0,f_1,f_2)$ is a face vector of a\index{$f$-vector}\index{polytope!$f$-vector}
{\bf three-dimensional} polytope if and only if
$$
f_0-f_1+f_2=2, \quad f_2 \leqslant 2f_0-4, \quad f_0 \leqslant 2f_2-4.
$$
\end{theorem}
\begin{corollary}
$$
f_2+4\leqslant 2f_0\leqslant 4f_2-8
$$
\end{corollary}
Well-known $g$-theorem \cite{S80,BL81} gives the criterion when an integer
vector $(f_0,\dots,f_{n-1})$ is an $f$-vector of a simple $n$-polytope (see
also \cite{Bu-Pa15}).

For general polytopes the are only partial results about $f$-vectors.

\noindent {\bf Classical problem:} For {\em four-dimensional} polytopes the
conditions characterizing the face vector $(f_0,f_1,f_2,f_3)$ are still {\em
not known} \cite{Z02}.

\subsection{Dual polytopes}
For an $n$-polytope $P\subset\mathbb R^n$ with $\mathbf{0}\in{\rm int}\, P$
the {\em dual polytope}\index{polytope!dual} $P^*$ is
$$
P^*=\{\boldsymbol{y}\in\mathbb (R^n)^*\colon \boldsymbol{y}\boldsymbol{x}+1\geqslant 0\}
$$
\begin{itemlist}
\item $i$-faces of $P^*$ are in an inclusion reversing bijection with
    $(n-i-1)$-faces of $P$.

\item $(P^*)^*=P$.
\end{itemlist}
An $n$-polytope is {\em simplicial}\index{polytope!simplicial} if any its facet is a simplex.

\begin{lemma}
A polytope dual to a simple polytope is simplicial.\\
A polytope dual to a simplicial polytope is simple.
\end{lemma}
\begin{lemma}
Let a polytope $P^n$, $n\geqslant 2$, be simple and simplicial. Then either $n=2$, or $P^n$ is
combinatorially equivalent to a simplex $\Delta^n$, $n>2$.
\end{lemma}

\begin{example}
Among $5$ Platonic solids the tetrahedron is self-dual, the cube is dual to
the octahedron, and the dodecahedron is dual to the icosahedron.

There are no simplicial polytope among Archimedean solids. Polytopes dual to
Archimedean solids are called {\em Catalan solids}, since they where first
described by E.C. Catalan (1814-1894). For example, the polytope dual to
truncated icosahedron is called {\em pentakis dodecahedron}.
\begin{figure}
\begin{center}
\includegraphics[height=3cm]{Arh-5.pdf}\qquad\includegraphics[height=3cm]{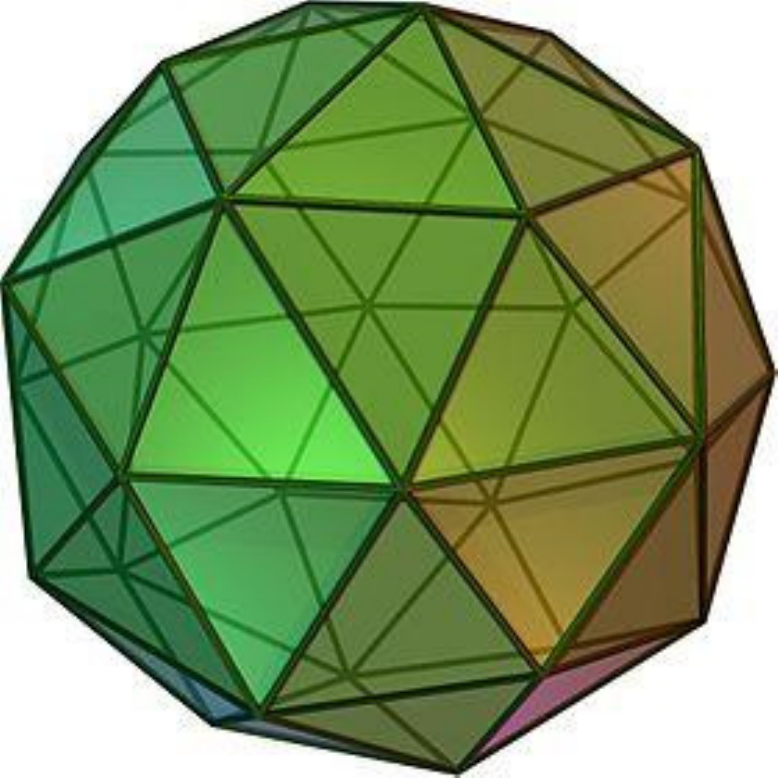}
\end{center}
\caption{Truncated icosahedron and pentakis dodecahedron (www.wikipedia.org)}\label{TI}
\end{figure}
\end{example}

\begin{figure}
\begin{center}
\includegraphics[scale=0.35]{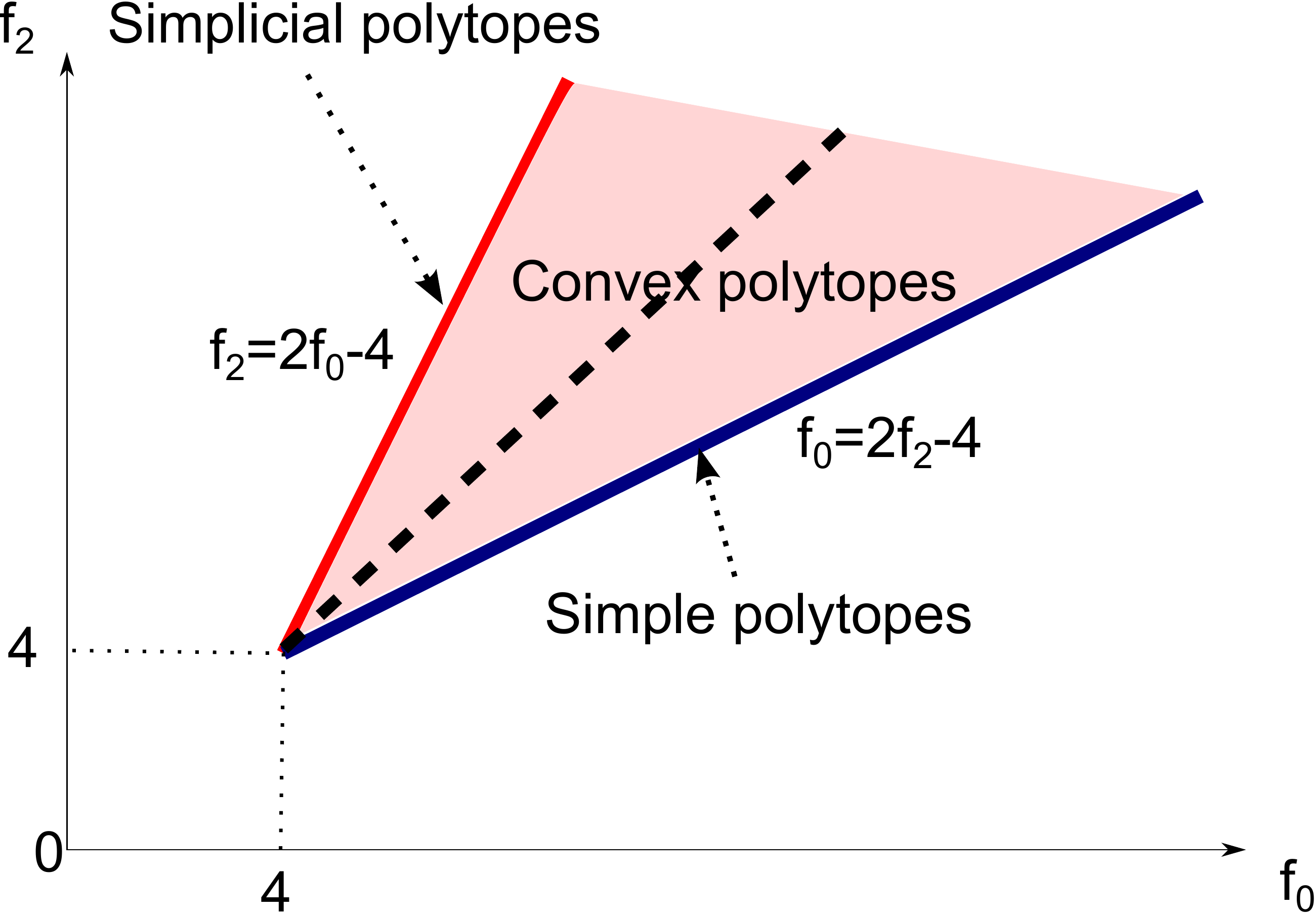}
\end{center}
\caption{By Steinitz's theorem and Euler's formula integer points inside the cone are in the one-to-one correspondence with $f$-vectors of convex $3$-polytopes
}\label{Plot}
\end{figure}

On Fig. \ref{Plot} the point $(4,4)$ corresponds to the tetrahedron. The bottom ray corresponds
to simple polytopes, the upper ray -- to simplicial. For $k\geqslant 3$
self-dual pyramids over $k$-gons give points on the diagonal.
\subsection{$k$-belts}
\begin{definition}
Let $P$ be a simple convex $3$-polytope. A {\em thick path}\index{thick path} is a sequence of facets $(F_{i_1},\dots,F_{i_k})$ with $F_{i_j}\cap F_{i_{j+1}}\ne\varnothing$ for $j=1,\dots,k-1$. A {\em $k$-loop}\index{$k$-loop} is a cyclic
sequence $(F_{i_1},\dots,F_{i_k})$ of facets,  such that $F_{i_1}\cap F_{i_2}$, $\dots$,
$F_{i_{k-1}}\cap F_{i_k}$, $F_{i_k}\cap F_{i_1}$ are edges. A $k$-loop is called {\em simple}\index{$k$-loop!simple $k$-loop}, if facets $(F_{i_1},\dots, F_{i_k})$ are pairwise different. 
\end{definition}
\begin{example}
Any vertex of $P$ is surrounded by a simple $3$-loop. Any edge is surrounded by a
simple $4$-loop. Any $k$-gonal facet is surrounded by a simple $k$-loop.
\end{example}
\begin{definition}
A {\em $k$-belt} \index{$k$-belt} is a $k$-loop,  such that $F_{i_1}\cap\dots\cap F_{i_k}=\varnothing$
and $F_{i_p}\cap F_{i_q}\ne\varnothing$ if and only if
$\{p,q\}\in\{\{1,2\},\dots,\{k-~1,k\},\{k,1\}\}$.
\end{definition}
\begin{figure}
\begin{center}
\includegraphics[height=5cm]{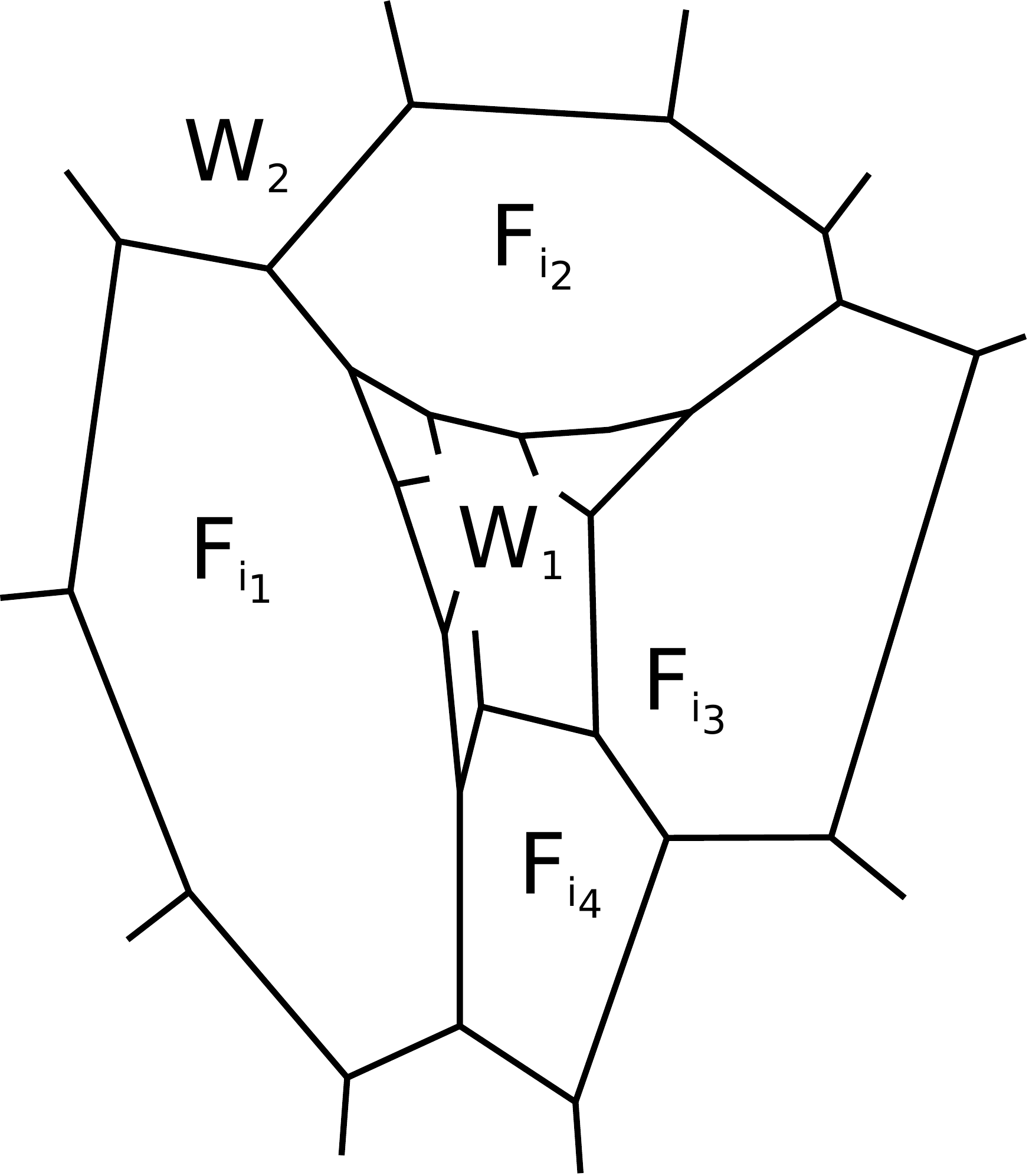}
\end{center}
\caption{$4$-belt of a simple $3$-polytope}
\end{figure}

\subsection{Simple paths and cycles} By $G(P)$ we denote a vertex-edge graph of a simple $3$-polytope $P$\index{graph!of a polytope}\index{polytope!graph}. We call it \\
the {\em graph of a polytope}. Let $G$ be a graph.
\begin{definition}
\begin{itemlist}
\item An {\em edge path}\index{edge path} is a sequence of vertices $(v_1,\dots,v_k)$,
    $k>1$ such that $v_i$ and $v_{i+1}$ are connected by some edge $E_i$
    for all $i<k$.
\item  An edge path is {\em simple} \index{edge path!simple}if it passes any vertex of $G$ at
    most once.
\item A {\em cycle} \index{cycle}is a simple edge path, such that $v_k=v_1$, where
    $k>2$. We denote a cycle by $(v_1,\dots, v_{k-1})$.
\end{itemlist}
\end{definition}
A cycle $(v_1,\dots,v_k)$ in the graph of a simplicial $3$-polytope $P$ is {\em dual}\index{cycle!dual to a $k$-belt} to a
$k$-belt in a simple $3$-polytope $P^*$ if all it's vertices do not lie in
the same face, and $v_i$ and $v_j$, are connected by an edge if and only if
$\{i,j\}\in\{\{1,2\},\dots,\{k-1,k\},\{k,1\}\}$.
\begin{figure}
\begin{center}
\includegraphics[scale=0.4]{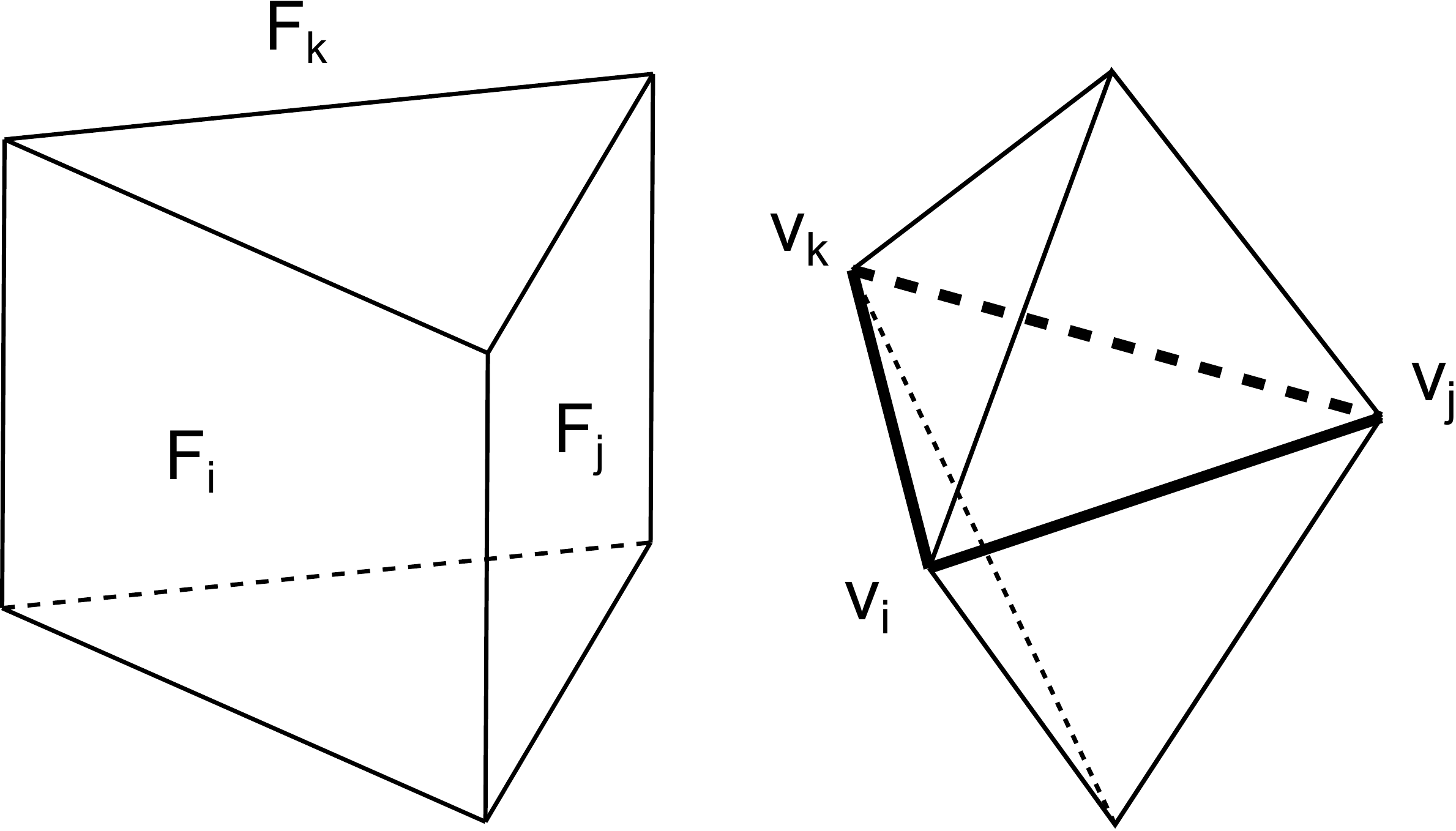}
\end{center}
\caption{$(F_i,F_j,F_k)$ is a $3$-belt\quad\quad $(v_i,v_j,v_k)$ is a
cycle dual to the $3$-belt}
\end{figure}
\begin{definition}
A {\em zigzag path} on a simple $3$-polytope is an edge path with no $3$ successive edges lying in the same facet. 
\end{definition}
Starting with one edge an choosing the second edge having with it a common vertex, we obtain a unique way to construct a zigzag.
\begin{definition}
A {\em zigzag cycle} on a simple $3$-polytope is a cycle with no $3$ successive edges lying in the same facet.
\end{definition}
\begin{figure}
\begin{center}
\includegraphics[height=3cm]{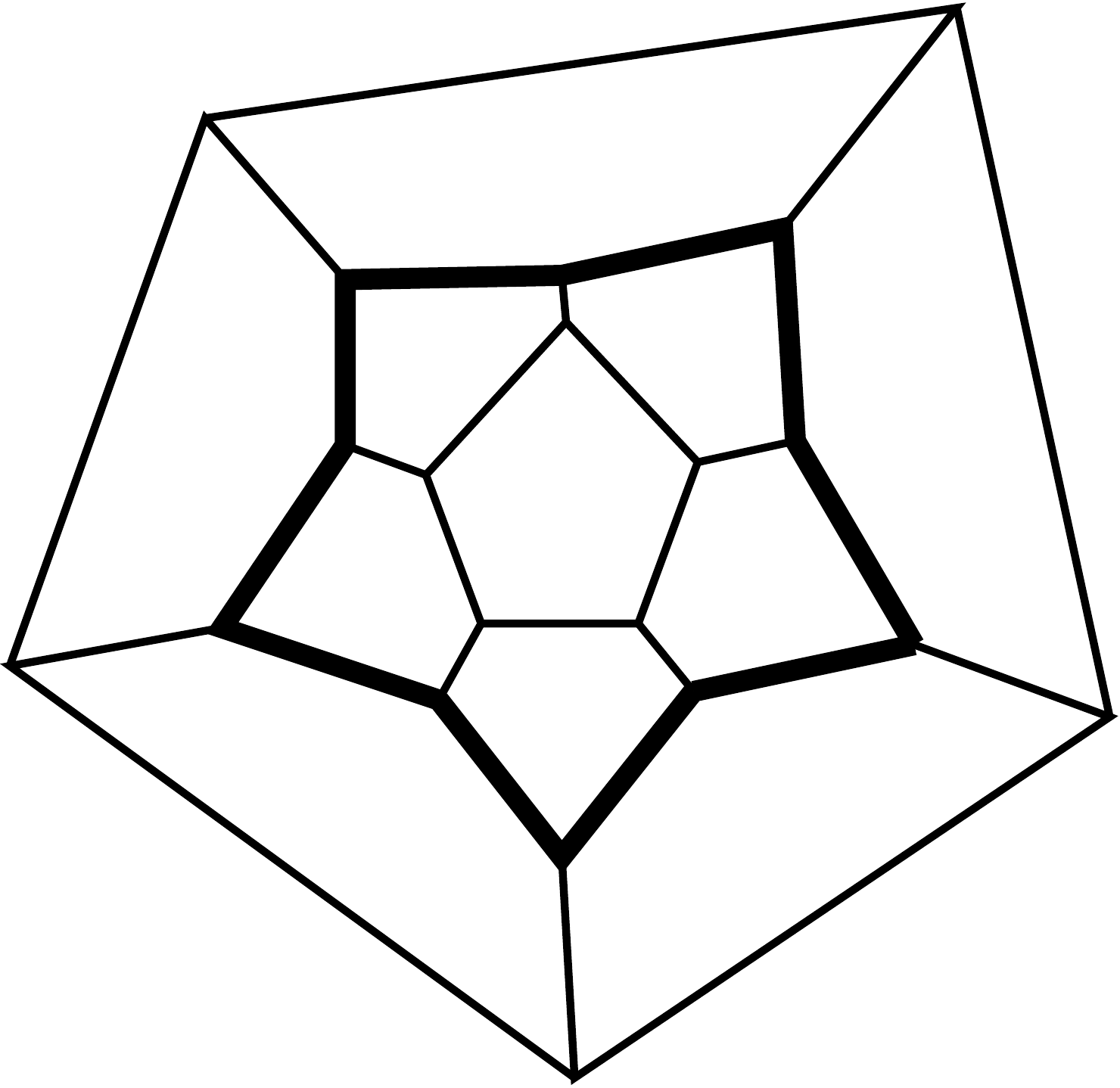}\\
\end{center}
\caption{A zigzag cycle on the Schlegel diagram of the dodecahedron}\label{zzf}
\end{figure}
\subsection{The Steinitz theorem}
\begin{definition}
A graph $G$ is called {\em simple}\index{graph!simple} if it has no loops and multiple edges. A connected graph $G$ is called $3$-connected\index{graph!$3$-connected}, if it has at least $6$ edges and deletion of any one or two vertices with all incident edges leaves $G$ connected.
\end{definition}
\begin{theorem} \label{S-Theorem}(The Steinitz theorem, see \cite{Z07})\\
\index{Steinitz theorem}\index{theorem!Steinitz}A graph $G$ is a graph of a $3$-polytope if and only if it is simple, planar and $3$-connected.
\end{theorem}
\begin{remark}\label{GSU}
Moreover, the cycles in $G$ corresponding to facets are exactly chordless simple edge cycles $C$ with $G\setminus C$ disconnected; hence the combinatorics of the embedding $G\subset S^2$ is uniquely defined. 
\end{remark} 
We will need the following version of the Jordan curve theorem\index{Jordan curve theorem}\index{theorem!Jordan curve}. It can be proved rather directly similarly to the piecewise-linear version of this theorem on the plane.
\begin{theorem}\label{JtheoremS}
Let $\gamma$  be a simple piecewise-linear (in respect to some homeomorphism $S^2\simeq \partial P$ for a $3$-polytope $P$) closed curve on the sphere $S^2$. Then
\begin{enumerate}
\item $S^2\setminus\gamma$ consists of  two connected components
    $\mathcal{C}_1$ and~ $\mathcal{C}_2$.
\item Closure $\overline{\mathcal{C}_\alpha}$ is homeomorphic to a disk for each $\alpha=1,2$.
\end{enumerate}
\end{theorem}
We will also need the following result.\index{graph!$3$-connected}
\begin{lemma}\label{3C-lemma}
Let $G\subset S^2$ be a finite simple graph with at least $6$ edges.  Then $G$ is $3$-connected if and only if all connected components of $S^2\setminus G$ are bounded by simple edge cycles and closures of any two areas (<<facets>>) either do not intersect, or intersect by a single common vertex, or intersect by a single common edge. 
\end{lemma}
\begin{proof}
Let $G$ satisfy the condition of the lemma. We will prove that $G$ is $3$-connected.
Let $v_1\ne v_2$, $u_1,u_2\notin\{v_1,v_2\}$, be vertices of $G$, perhaps $u_1=u_2$. We need to prove that  there is an edge-path from $v_1$ to $v_2$ in $G\setminus\{u_1,u_2\}$. Since $G$ is connected, there is an edge-path $\gamma$ connecting $v_1$ and $v_2$. Consider the vertex $u_\alpha$, $\alpha\in\{1,2\}$, and all facets $F_{i_1},\dots, F_{i_p}$ containing it. From the hypothesis of  the lemma $p\geqslant 3$. Since the graph is embedded to the sphere, after relabeling we obtain a simple $p$-loop $(F_{i_1},\dots, F_{i_p})$. For $j=\in\{1,\dots,p\}$ denote by $w_j$ the end of the edge $F_{i_j}\cap F_{i_{j+1}}$ different from $u_\alpha$,  where $F_{i_{p+1}}=F_{i_1}$.  Let $g_j$ be the simple edge-path connecting $w_{j-1}$ and $w_j$ in $F_{i_j}\setminus u_\alpha$ (See Fig. \ref{star-w}).
\begin{figure}
\begin{center}
\includegraphics[height=4cm]{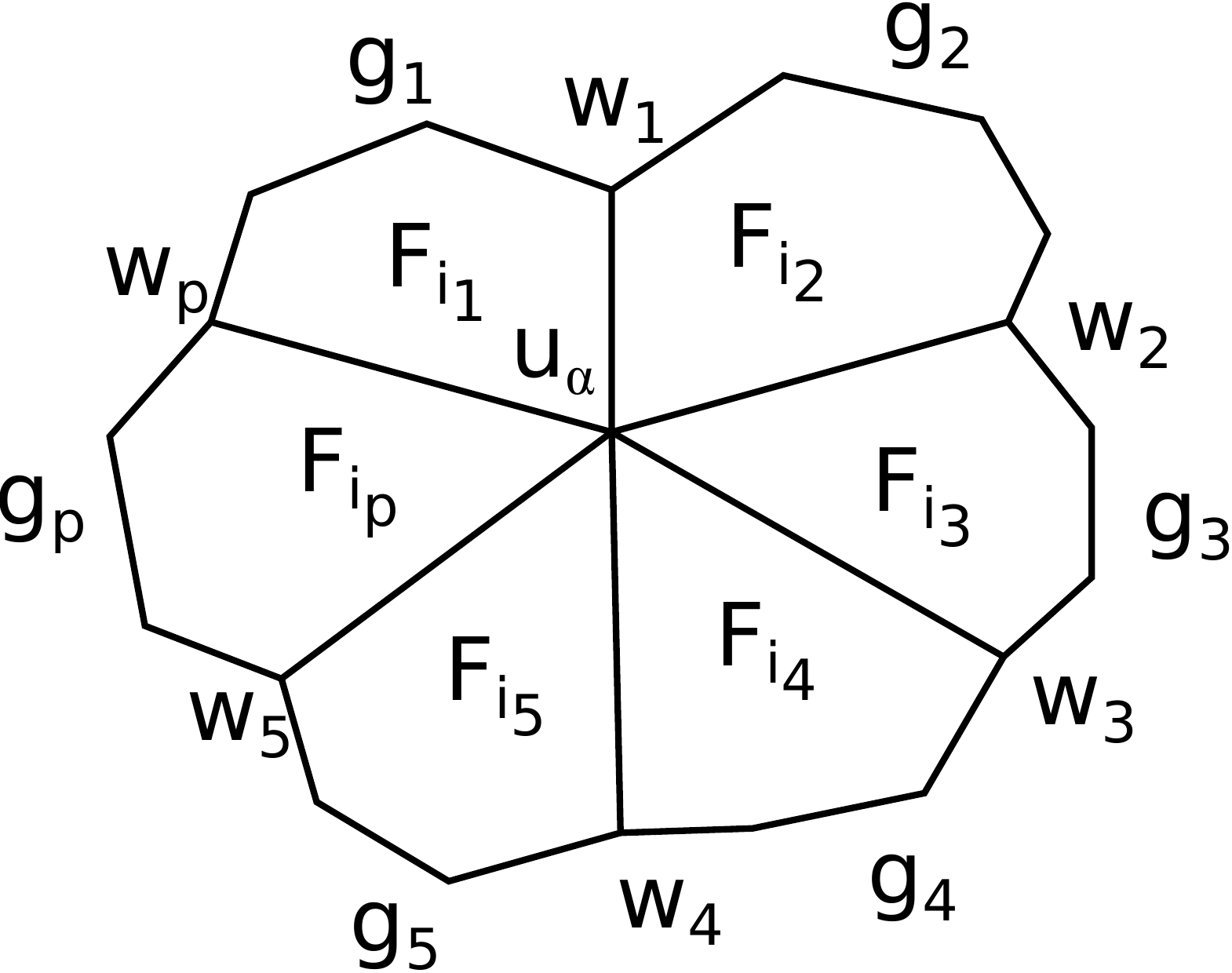}
\end{center}
\caption{Star of the vertex $u_\alpha$}\label{star-w}
\end{figure}
Then $\eta_\alpha=(g_1,g_2,\dots, g_p)$ is a simple edge-cycle. Indeed, if $g_s$ and $g_t$ have common vertex, then this vertex belongs to $F_{i_s}\cap F_{i_t}$ together with $u_\alpha$; hence it is connected with $u_\alpha$ by an edge; therefore $\{s,t\}=\{k,k+1\mod p\}$ for some $k$, and the vertex is $w_k$. If $u_1$ and $u_2$ are different and are connected by an edge $E$, then $E=F_{i_s}\cap F_{i_t}$ for some $s,t\in\{1,\dots,p\}$, $s-t=\pm1 \mod p$, and we can change $\gamma$ not to contain $E$ substituting  the simple edge-path in $F_{i_s}\setminus E$ for $E$. Now for the new path $\gamma_1$ consider all times it passes $u_\alpha$. We can remove all the fragments $(w_i,u_\alpha, w_j)$  and substitute the simple edge path in $\eta_\alpha\setminus u_\beta$ connecting $w_i$ and $w_j$ for each fragment $(w_i,u_\alpha,w_j)$. The same can be done for $u_\beta$, $\{\alpha,\beta\}=\{1,2\}$. Thus we obtain the edge-path $\gamma_2$ connecting $v_1$ and $v_2$ in $G\setminus\{u_1,u_2\}$.

Now let $G$ be $3$-connected. Consider the connected component $D$ of $S^2\setminus G$ and it's boundary $\partial D$. If there is a hanging vertex $v\in\partial D$ of $G$, then deletion of the other end of the edge containing $v$ makes the graph disconnected. Hence any vertex $v\in\partial D$ of $G$ gas valency at least $2$, and $D$ is surrounded by an edge-cycle $\eta$. If $\eta$ is not simple, then there is a vertex $v\in\eta$ passed several times.  Then the area $D$ appears several times when we walk around the vertex $v$. Since $D$ is connected, there is a simple piecewise-linear (in respect to some homeomorphism $S^2\simeq \partial P$ for a $3$-polytope $P$) closed curve $\eta$ in the closure $\overline{D}$ of $D$ with the only point $v$ on the boundary. Walking round $v$, we pass edges in both connected component of $S^2\setminus\eta$; hence the deletion of $v$ divides $G$ into several connected components. Thus the cycle $\eta$ is simple. Let facets $F_1=\overline{D_1}$, $F_2=\overline{D_2}$  have two common vertices $v_1$, $v_2$. Consider piecewise linear simple curves $\eta_1\subset F_1$, $\eta_2\subset F_2$, with ends $v_1$ and $v_2$ and all other points lying in $D_1$ and $D_2$ respectively. Then $\eta_1\cup \eta_2$ is a simple piecewise-linear closed curve; hence it separates the sphere $S^2$ into two connected components. If $v_1$ and $v_2$ are not adjacent in $F_1$ or $F_2$, then both connected components contain vertices of $G$; hence deletion of $v_1$ and $v_2$ makes the graph disconnected. Thus any two common vertices are adjacent in both facets. Moreover, since there are no multiple edges, the corresponding edges belong to both facets. Then either both facets are surrounded by a common $3$-cycle, and in this case $G$ has only $3$ edges, or any two facets either do not intersect, or intersect by a common vertex, or intersect by a common edge. This finishes the proof.
\end{proof}
Let $\mathcal{L}_k=(F_{i_1},\dots, F_{i_k})$ be a simple $k$-loop for $k\geqslant 3$. Consider midpoints $w_j$ of edges $F_{i_j}\cap F_{i_{j+1}}$, $F_{i_{k+1}}=F_{i_1}$ and segments $E_j$ connecting $w_j$ and $w_{j+1}$ in $F_{j+1}$.  Then $(E_1,\dots,E_k)$ is a simple piecewise-linear curve $\eta$ on $\partial P$. It separates $\partial P\simeq S^2$ into two areas homeomorphic to discs $D_1$ and $D_2$ with $\partial D_1=\partial D_2=\eta$.  Consider two graphs $G_1$ and $G_2$ obtained from the graph $G(P)$ of $P$ by addition of vertices $\{w_j\}_{j=1}^k$ and edges $\{E_j\}_{j=1}^k$, and deletion of all vertices and edges with interior points inside $D_1$ or $D_2$ respectively.  
\begin{lemma}\label{belt-cut-lemma}
There exist simple polytopes $P_1$ and $P_2$ with graphs $G_1=G(P_1)$ and $G_2=G(P_2)$. 
\end{lemma}
\begin{proof}
The proof is similar for both graphs; hence we consider the graph $G_1$. It has at least $6$ edges, is connected and planar. Now it is sufficient to prove that the hypothesis of Lemma \ref{3C-lemma} is valid. For this we see each facet of $G_1$ is either a facet of $P$, or it is a part of a facet $F_{i_j}$ for some $j$, or it is bounded by the cycle $\eta$. In particular, all facets are bounded by simple edge-cycles. If the facets $F_i$ and $F_j$ are both of the first two types they either do not intersect or intersect by common edge as it is in $P$. If $F$ is the facet bounded by $\eta$, then it intersects only facets $(F_{i_1},\dots,F_{i_k})$, and each intersection is an edge $F\cap F_{i_j}$, $j=1,\dots,k$.  
\end{proof}
\begin{definition}
We will call polytopes $P_1$ and $P_2$ {\em loop-cuts}\index{loop-cut} (or, more precisely, {\em $\mathcal{L}_k$-cuts}) of $P$.
\end{definition}

\newpage
\section{Lecture 2. Combinatorics of simple polytopes}
\subsection{Flag polytopes}
\begin{definition}
A simple polytope is called {\em flag}\index{flag polytope}\index{polytope!flag} if any set of pairwise intersecting
facets $F_{i_1},\dots,F_{i_k}$: $F_{i_s}\cap F_{i_t}\ne\varnothing$,
$s,t=1,\dots,k$, has a nonempty intersection $F_{i_1}\cap\dots\cap
F_{i_k}\ne\varnothing$. \end{definition}
\begin{figure}
\begin{center}
\begin{tabular}{ccc}
\includegraphics[height=3cm]{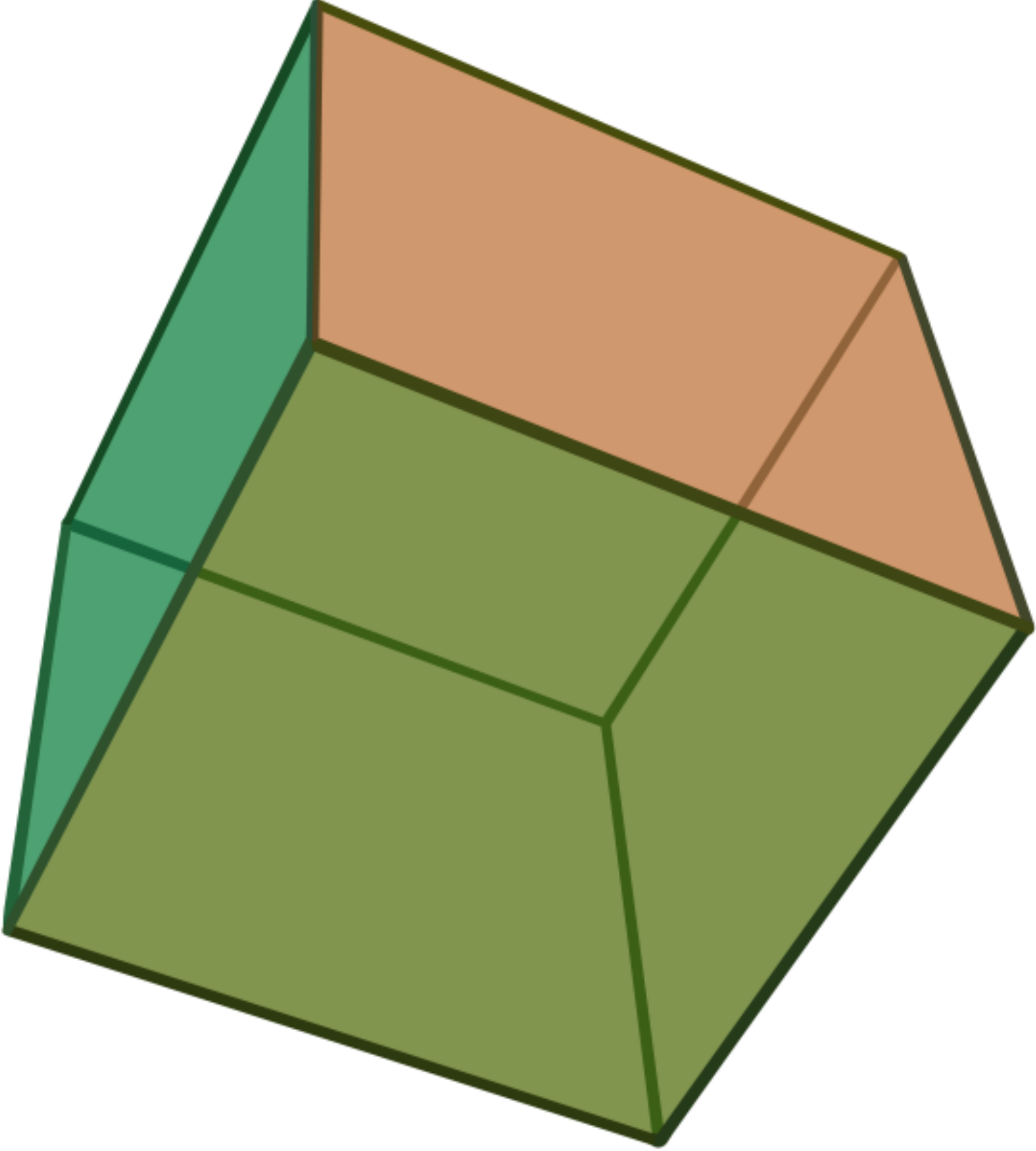}&\hspace{3cm}&\includegraphics[height=3cm]{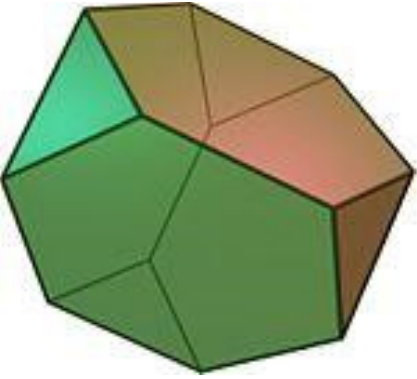}\\
a) && b)
\end{tabular}
\end{center}
\caption{a) flag polytope; b) non-flag polytope  (www.wikipedia.org)}
\end{figure}
\begin{example}
$n$-simplex $\Delta^n$ is not a flag polytope for $n\geqslant 2$.
\end{example}
\begin{proposition}\label{nflag}Simple $3$-polytope $P$ is {\bf not flag} if and only if either
    $P=\Delta^3$, or $P$ contains a  $3$-belt.
\end{proposition}
\begin{corollary}
Simple $3$-polytope $P\ne\Delta^3$ is flag if and only any $3$-loop
corresponds to a vertex.
\end{corollary}
\begin{proposition}\label{facet-belt} Simple $3$-polytope $P$ is flag if and only if any facet 
is surrounded by a $k$-belt, where $k$ is the number of it's edges.
\end{proposition}
\begin{proof}
A simplex is not flag and has no $3$-belts.

By Proposition \ref{nflag} a simple $3$-polytope $P\not\simeq\Delta^3$ is not flag if and only if it has a $3$-belt. The facet $F\subset P$ is not surrounded by a belt if and only if it belongs to a $3$-belt.   
\end{proof}
\begin{corollary} For any flag simple $3$-polytope $P$ we have $p_3=0$.
\end{corollary}
Later (see Lecture \ref{ConFul}) we will need the following result.
\begin{proposition}\label{34belts}
A flag $3$-polytope $P$ has no $4$-belts if and only if any pair of adjacent facets is surrounded by a belt.
\end{proposition}
\begin{proof}
The pair $(F_i,F_j)$ of adjacent facets is a $2$-loop and is surrounded by a simple edge-cycle. Let $\mathcal{L}=(F_{i_1}, \dots, F_{i_k})$ be the $k$-loop that borders it. If $\mathcal{L}$ is not simple, then $F_{i_a}=F_{i_b}$ for $a\ne b$. Then $F_{i_a}$ and $F_{i_b}$ are not adjacent to the same facet $F_i$ or $F_j$. Let $F_{i_a}$  be adjacent to $F_i$, and $F_{i_b}$ to $F_j$. Then $(F_i,F_j,F_{i_a})$ is a $3$-belt. A contradiction. Hence $\mathcal{L}$ is a simple loop. If it is not a belt, then $F_{i_a}\cap F_{i_b}\ne\varnothing$ for non-successive facets $F_{i_a}$ and $F_{i_b}$. From Proposition \ref{facet-belt} we obtain that $F_{i_a}$ and $F_{i_b}$ are not adjacent to the same facet $F_i$ or $F_j$. Let $F_{i_a}$  be adjacent to $F_i$, and $F_{i_b}$ to $F_j$. Then $(F_{i_a}, F_i,F_j,F_{i_b})$ is a $4$-belt. On the other hand, if there is a $4$-belt $(F_i,F_j,F_k,F_l)$, then facets $F_k$ and $F_l$ belong to the loop surrounding the pair $(F_i,F_j)$. Since $F_i\cap F_k=\varnothing=F_j\cap F_l$, they are not successive facets of this loop; hence the loop is not a belt. This finishes the proof. 
\end{proof}
In the combinatorial study of fullerenes the following version of the Jordan
curve theorem gives the important tool. It follows from the Theorem \ref{JtheoremS}.
\begin{theorem}\label{Jtheorem}
Let $\gamma$  be a simple edge-cycle on a simple $3$-polytope $P$. Then
\begin{enumerate}
\item $\partial P\setminus\gamma$ consists of  two connected components
    $\mathcal{C}_1$ and~ $\mathcal{C}_2$.
    \item Let $\mathcal{D}_\alpha=\{F_j\in\mathcal{F}_P \colon {\rm int}\,
        F_j\subset \mathcal{C}_{\alpha}\}$,
        $\alpha=1,2$. Then $\mathcal{D}_1\sqcup
        \mathcal{D}_2=\mathcal{F}_P$.
\item The closure $\overline{\mathcal{C}_\alpha}$ is homeomorphic to a disk.
    We have $\overline{\mathcal{C}_\alpha}=|\mathcal{D}_\alpha|$.
\end{enumerate}
\end{theorem}

\begin{corollary}
If we remove the $3$-belt from the surface of a simple $3$-polytope, we obtain two parts
$W_1$ and $W_2$, homeomorphic to disks.
\end{corollary}
\begin{figure}
\begin{center}
\includegraphics[scale=0.2]{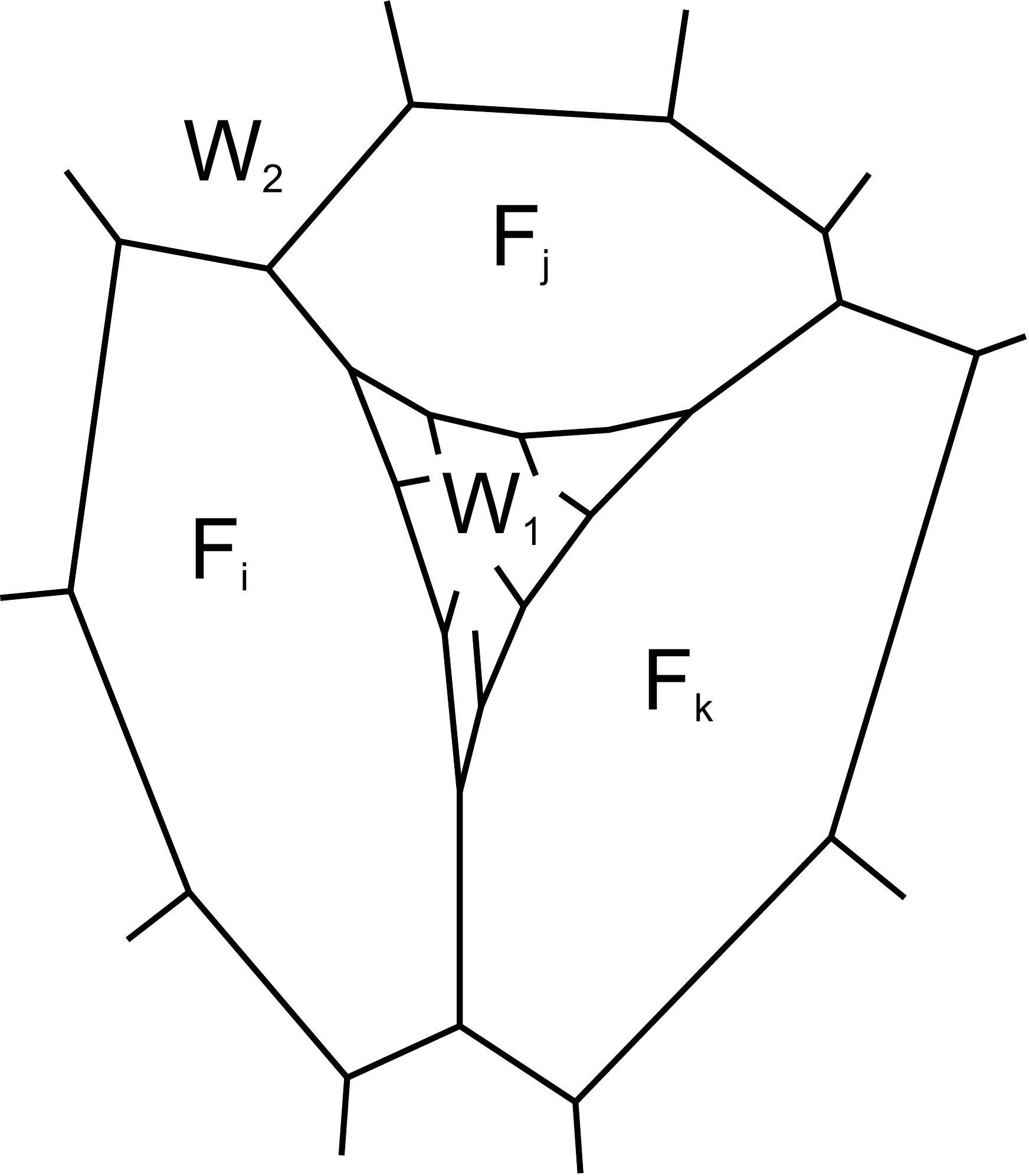}
\end{center}
\caption{$3$-belt on the surface of a simple $3$-polytope}
\end{figure}
\begin{proposition}\label{flagm6} Let $P$ be a flag simple $3$-polytope. Then $m\geqslant 6$, and $m=6$ if and only if $P$ is combinatorially equivalent to the cube $I^3$.
\end{proposition}
\begin{proof}
Take a facet $F_1$. By Proposition \ref{facet-belt} it is surrounded by a $k$-belt $\mathcal{B}=(F_{i_1},\dots,F_{i_k})$, $k\geqslant 4$. Since there is at least one facet in the connected component  $W_\alpha$ of $\partial P\setminus \mathcal{B}$, ${\rm int\,}F_1\notin W_\alpha$, we obtain $m\geqslant 2+k\geqslant 6$. If $m=6$, then $k=4$, $F_1$ is a quadrangle, and  $W_\alpha={\rm int\,}F_j$ for some facet $F_j$ Then $F_j\cap F_{i_1}\cap F_{i_2}$, $F_j\cap F_{i_2}\cap F_{i_3}$, $F_j\cap F_{i_3}\cap F_{i_4}$, $F_j\cap F_{i_4}\cap F_{i_1}$ are vertices, and $P$ is combinatorially equivalent to $I^3$.  
\end{proof}
\begin{lemma}\label{loop-cut-flag}
Let $P$ be a flag polytope, $\mathcal{L}_k$ be a simple $k$-loop, and $P_1$ and $P_2$ be $\mathcal{L}_k$-cuts of $P$. Then the following conditions are equivalent:
\begin{enumerate}
\item both polytopes $P_1$ and $P_2$ are flag;
\item $\mathcal{L}_k$ is a $k$-belt.
\end{enumerate}
\end{lemma}
\begin{proof}
Since $P$ has no $3$-belts, for $k=3$ the loop $\mathcal{L}_k$ surrounds a vertex; hence one of the polytopes $P_1$ and $P_2$ is a simplex, and it is not flag. Let $k\geqslant 4$. Then $P_1$ and $P_2$ are not simplices. There are three types of facets in $P$: lying only in $P_1$, lying only in $P_2$, and lying in $\mathcal{L}_k$. Let $\mathcal{B}_3=(F_i,F_j,F_k)$ be a $3$-loop in $P_\alpha$, $\alpha\in\{1,2\}$. Let $F_i, F_j,F_k$ correspond to facets of $P$. Since intersecting facets in $P_\alpha$ also intersect in $P$, $(F_i,F_j,F_k)$ is also a $3$-loop in $P$, and $F_i\cap F_j\cap F_k\in P$ is a vertex. Since $F_i\cap F_j\ne\varnothing$ in $P_\alpha$, either the corresponding edge of $P$ lies in $P_\alpha$, or it intersects the new facet, and $F_i$ and $F_j$ are consequent facets of $\mathcal{L}_k$. Since $k\geqslant 4$, at least one edge of $F_i\cap F_j$, $F_j\cap F_k$, and $F_k\cap F_i$ of $P$ lies in $P_\alpha$; hence $F_i\cap F_j\cap F_k\in P_\alpha$, and $\mathcal{B}_3$ is not a $3$-belt in $P_{\alpha}$. If one of the facets, say $F_i$, is a new facet of $P_\alpha$, then $F_j,F_k\in\mathcal{L}_k$, since $F_i\cap F_j,F_i\cap F_k\ne\varnothing$. Consider the edge $F_j\cap F_k$ of $P$. It intersects $F_i$ in $P_\alpha$ if and only if $F_j$ and $F_k$ are consequent facets in $\mathcal{L}_k$. Thus if $\mathcal{B}_3$ is a $3$-belt, then $\mathcal{L}_k$ is not a $k$-belt, and vice versa, if $\mathcal{L}_k$ is not a $k$-belt, then $F_j\cap F_k\ne\varnothing$ for some non-consequent facets of $\mathcal{L}_k$, and the corresponding $3$-loop $\mathcal{B}_3$ is a $3$-belt in the polytope $P_1$ or $P_2$ containing $F_j\cap F_k$.  This finishes the proof.
\end{proof}

\subsection{Non-flag $3$-polytopes as connected sums}
The existence of a $3$-belt is equivalent to the fact that $P$ is
combinatorially equivalent to a {\em connected sum}\index{connected sum of simple polytopes}\index{polytope!connected sum} $P=Q_1\#_{v_1,v_2}Q_2$ of
two simple $3$-polytopes $Q_1$ and $Q_2$ along vertices $v_1$ and $v_2$.
\begin{figure}
\begin{center}
\includegraphics[scale=0.4]{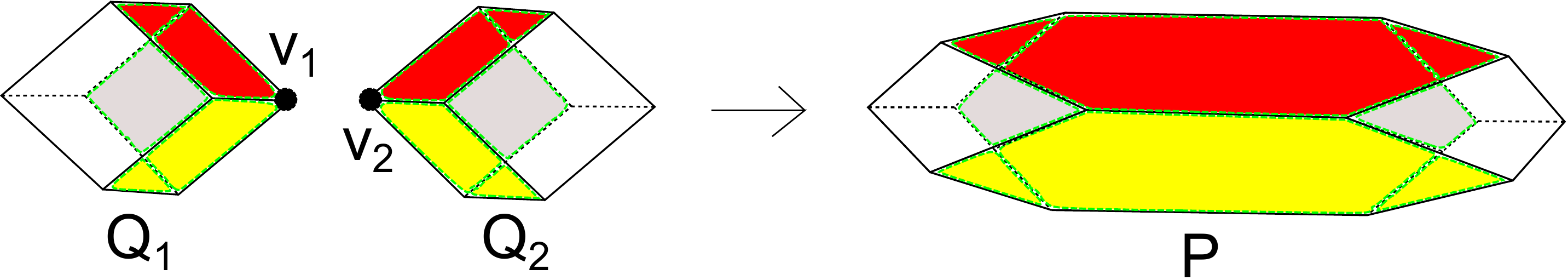}
\end{center}
\caption{Connected sum of two simple polytopes along vertices}
\end{figure}
The part $W_i$ appears if we remove from the surface of the polytope $Q_i$
the facets containing the vertex $v_i$, $i=1,2$.

\subsection{Consequence of Euler's formula for simple $3$-polytopes}
Let $p_k$ be a number of $k$-gonal facets of a $3$-polytope.\index{$p$-vector}\index{polytope!$p$-vector}
\begin{theorem}(See \cite{Gb03}) \label{pkth}For any {\bf simple} $3$-polytope~$P$
\begin{equation}\label{formulapk}
3p_3+2p_4+p_5=12+\sum\limits_{k\geqslant 7}(k-6)p_k,
\end{equation}
\end{theorem}
\begin{proof}
The number of pairs (edge, vertex of this edge) is equal, on the one hand, to $2f_1$ and, 
on the other hand (since the polytope is simple), to  $3f_0$. Then $f_0=\frac{2f_1}{3}$,
 and from the Euler formula we obtain
$2f_1=6f_2-12$. Counting the pairs (facet, edge of this facet), we have
$$
\sum\limits_{k\geqslant 3}kp_k=2f_1=6\left(\sum\limits_{k\geqslant 3}p_k\right)-12,
$$ 
which implies formula (\ref{formulapk}).
\end{proof}
\begin{corollary}
There is no simple polytope $P$ with all facets hexagons.  Moreover, if
$p_k=0$ for $k\ne 5,6$, then $p_5=12$.
\end{corollary}
\noindent {\bf Exercise:} The $f$-vector of a simple polytope is expressed in terms
of the $p$-vector by the following formulas:
$$
f_0=2\left(f_2-2\right)\quad f_1=3\left(f_2-2\right)\quad f_2=\sum_kp_k
$$

\subsection{Realization theorems}

\begin{definition}
An integer sequence $(p_k|k\geqslant 3)$ is called $3$-{\em realizable} is there is a simple $3$-polytope
$P$ with $p_k(P)=p_k$. 
\end{definition}
\begin{theorem}(Victor Eberhard \cite{Eb1891}, see \cite{Gb03})\index{theorem!Eberhard's}
For a sequence $(p_k|3\leqslant k\ne 6)$  there exists $p_6$ such that the sequence $(p_k|k\geqslant 3)$
is $3$-realizable  if and only if it satisfies formula (\ref{formulapk}) .
\end{theorem}
There arise a natural question.

\noindent {\bf Problem:}  For a given sequence $(p_k|3\leqslant k\ne 6)$ find all $p_6$ such
that the sequence $(p_k|k\geqslant 3)$ is $3$-realizable.\\
{\bf Notation:} When we write a finite sequence $(p_3,p_4,\dots, p_k)$ we mean that $p_l=0$ for $l>k$.
\begin{example}(see \cite{Gb03})
Sequences $(0,6,0,p_6)$ and $(0,0,12,p_6)$ are $3$-realizable if and only if $p_6\ne 1$. 
The sequence $(4,0,0,p_6)$ is $3$-realizable if and only if $p_6$ is an even integer different from $2$.
The sequence $(3,1,1,p_6)$  is $3$-realizable if and only if $p_6$ is an odd integer greater than $1$. 
\end{example}
Let us mention also the following results.
\begin{theorem}
For a given sequence $(p_k|3\leqslant k\ne 6)$ satisfying formula
(\ref{formulapk})
\begin{itemlist}
\item there exists $p_6\leqslant 3\left(\sum\limits_{k\ne 6}p_k\right)$
    such that the sequence $(p_k|k\geqslant 3)$ is $3$-realizable \cite{F74};
\item if $p_3=p_4=0$ then any sequence $(p_k|k\geqslant 3,
    p_6\geqslant 8)$ is $3$-realizable \cite{Gb68}.
\end{itemlist}
\end{theorem}
There are operations on simple $3$-polytopes  
that do not effect $p_k$ except for $p_6$. We call them {\em $p_6$-operations}\index{$p_6$-operation}.
As we will see later they are  important for applications.

\noindent {\bf Operation I:} \label{Op1}\index{first iterative procedure}The operation affects all edges of the polytope $P$. We present a fragment on Fig. \ref{O1F}.
\begin{figure}
\includegraphics[height=5cm]{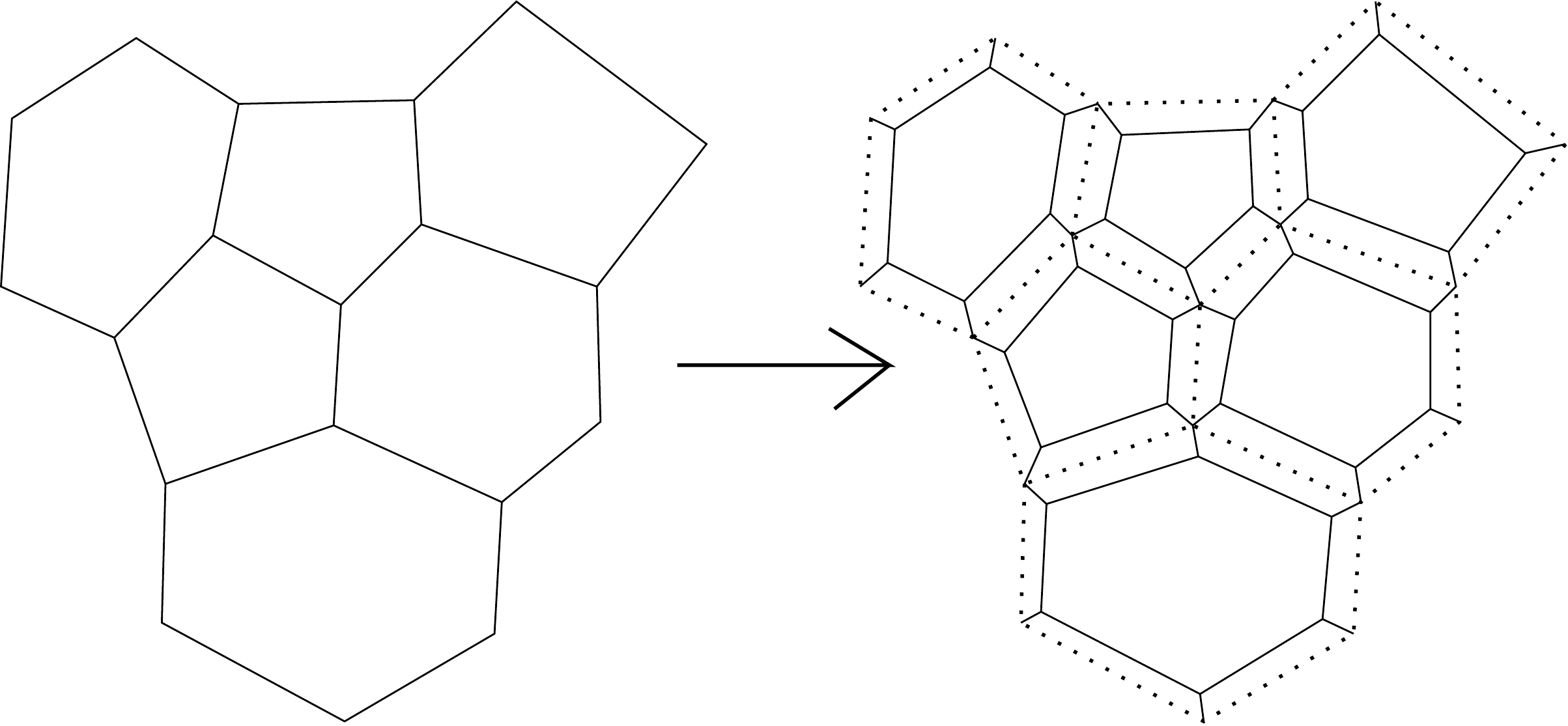}
\caption{Operation I}\label{O1F}
\end{figure}
On the right picture the initial polytope $P$ is drawn by dotted lines, while the resulting polytope -- by  solid lines. We have
$$
p_k(P')=\begin{cases}p_k(P),&k\ne 6;\\
p_6(P)+f_1(P),&k=6.
\end{cases}
$$

\noindent {\bf Operation II:} The operation\index{second iterative procedure} affects all edges of the polytope $P$. We present a fragment on Fig. \ref{O2F}.
\begin{figure}
\includegraphics[height=5cm]{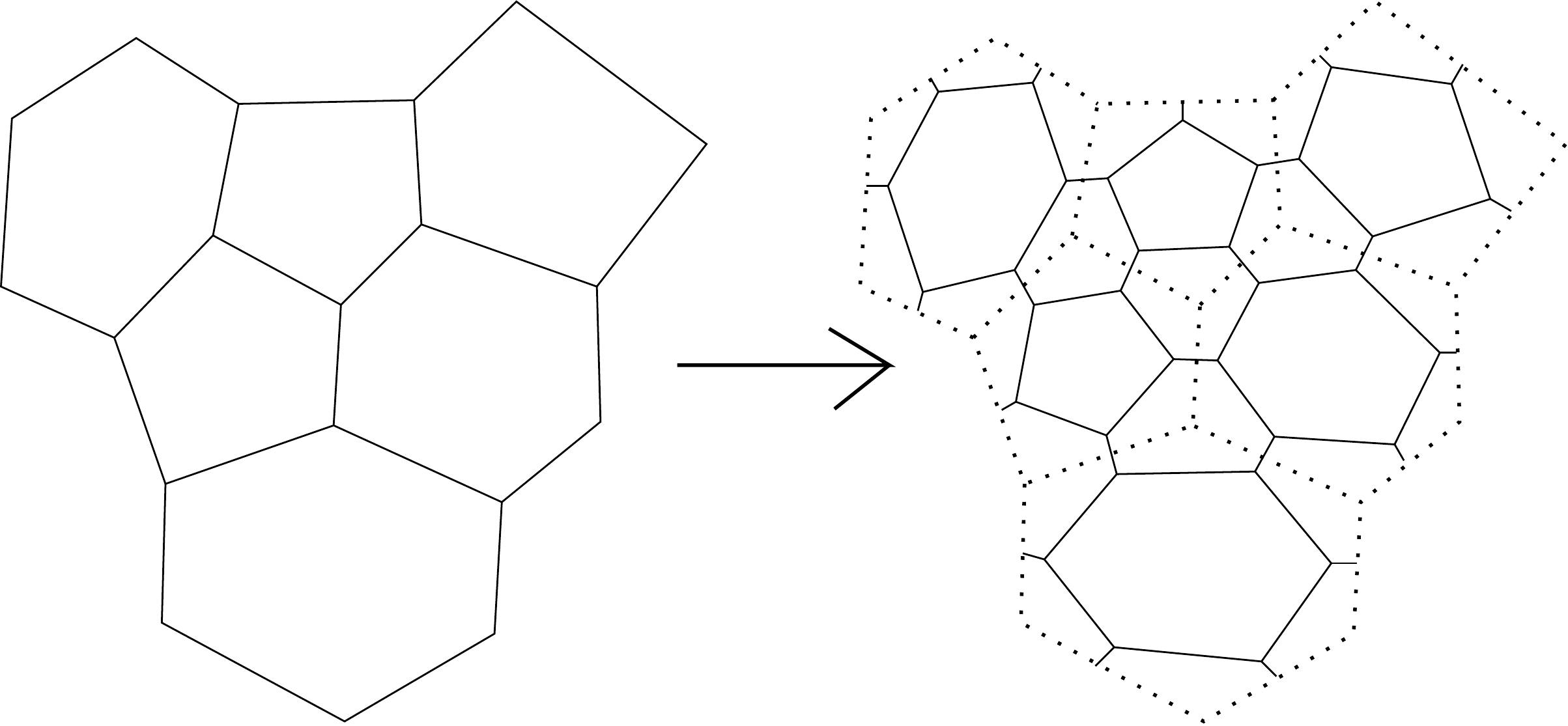}
\caption{Operation II}\label{O2F}
\end{figure}
On the right picture the initial polytope $P$ is drawn by dotted lines, while the resulting polytope -- by solid lines.
We have
$$
p_k(P')=\begin{cases}p_k(P),&k\ne 6;\\
p_6(P)+f_0(P),&k=6.
\end{cases}
$$

Operation I and Operation II are called iterative procedures\index{iterative procedures}\index{polytope!iterative procedures} (see \cite{LMR06}), since arbitrary compositions of them are well defined. 

\noindent {\bf Exercise:} Operation I and Operation II commute; therefore they define an action of the semigroup $\mathbb Z_{\geqslant 0}\times \mathbb Z_{\geqslant 0}$ on the set of all combinatorial simple $3$-polytopes, where $\mathbb Z_{\geqslant 0}$ is the additive semigroup of nonnegative integers.

\subsection{Graph-truncation of simple $3$-polytopes}
Consider a subgraph $\Gamma\subset G(P)$ without isolated vertices. For each edge 
$$
E_{i,j}=F_i\cap F_j=P\cap \{\boldsymbol{x}\in\mathbb R^3\colon
(\boldsymbol{a}_i+\boldsymbol{a}_j)\boldsymbol{x}+(b_i+b_j)=0\}
$$ 
consider the halfspace 
$$\mathcal{H}_{ij,\varepsilon}^+=\{\boldsymbol{x}\in\mathbb R^3\colon
(\boldsymbol{a}_i+\boldsymbol{a}_j)\boldsymbol{x}+(b_i+b_j)\geqslant
\varepsilon\}.$$
Set
$$
P_{\Gamma,\varepsilon}=P\cap\bigcap\limits_{E_{i,j}\in \Gamma} \mathcal{H}^+_{ij,\varepsilon}
$$
\noindent {\bf Exercise:} For small enough values of $\varepsilon$  the combinatorial type of
$P_{\Gamma,\varepsilon}$ does not depend on $\varepsilon$. 

\begin{definition} 
We will denote by $P_{\Gamma}$ the  combinatorial type of $P_{\Gamma,\varepsilon}$ for small enough values of $\varepsilon$  and call it a {\em $\Gamma$-truncation}\index{$\Gamma$-truncation} of $P$. When it is clear what is $\Gamma$ we call $P_{\Gamma}$ simply \emph{graph-truncation} \index{graph-truncation} of $P$.
\end{definition}

\begin{example} For $\Gamma=G(P)$ the polytope $P'=P_{\Gamma}$ is obtained from
$P$ by a $p_6$-operation I defined above. 
\end{example}
\begin{proposition}\label{pgflag}
Let  $P$ be a simple polytope with $p_3=0$. Then the polytope $P_{G(P)}$ is flag.
\end{proposition}
We leave the proof as an exercise.
\begin{corollary}
For a given sequence $(p_k|3\leqslant k\ne 6)$ satisfying formula
(\ref{formulapk}) there are infinitely many values of $p_6$ such that the
sequence $(p_k| k\geqslant 3)$ is $3$-realizable.
\end{corollary}

\subsection{Analog of Eberhard's theorem for flag polytopes}
\begin{theorem}\cite{Bu-Er15} For every sequence $(p_k|3\leqslant k\ne 6, p_3=0)$\index{theorem!analog of Eberhard's theorem for flag polytopes} of nonnegative integers
satisfying formula (\ref{formulapk}) there exists a value of $p_6$ such that
there is a {\bf flag simple} $3$-polytope $P^3$ with  $p_k=p_k(P^3)$
for all $k\geqslant 3$.
\end{theorem}
\begin{proof}
For a given sequence $(p_k|3\leqslant k\ne 6, p_3=0)$ satisfying formula
(\ref{formulapk}) by Eberhard's theorem there exists a simple polytope $P$
with $p_k=p_k(P)$, $k\ne 6$. Then the polytope $P'=P_{G(P)}$ is flag by
Proposition \ref{pgflag}. We have $p_k(P')=p_k(P)$, $k\ne 6$, and
$p_6(P')=p_6(P)+f_1(P)$.
\end{proof}

\newpage
\section{Lecture 3. Combinatorial fullerenes}
\subsection{Fullerenes}
A {\em fullerene}\index{fullerene}\index{polytope!fullerene} is a molecule of carbon that is topologically sphere and
any atom belongs to exactly three carbon rings, which are pentagons or
hexagons.
\begin{figure}
\begin{center}
\begin{tabular}{cc}
\includegraphics[height=6cm]{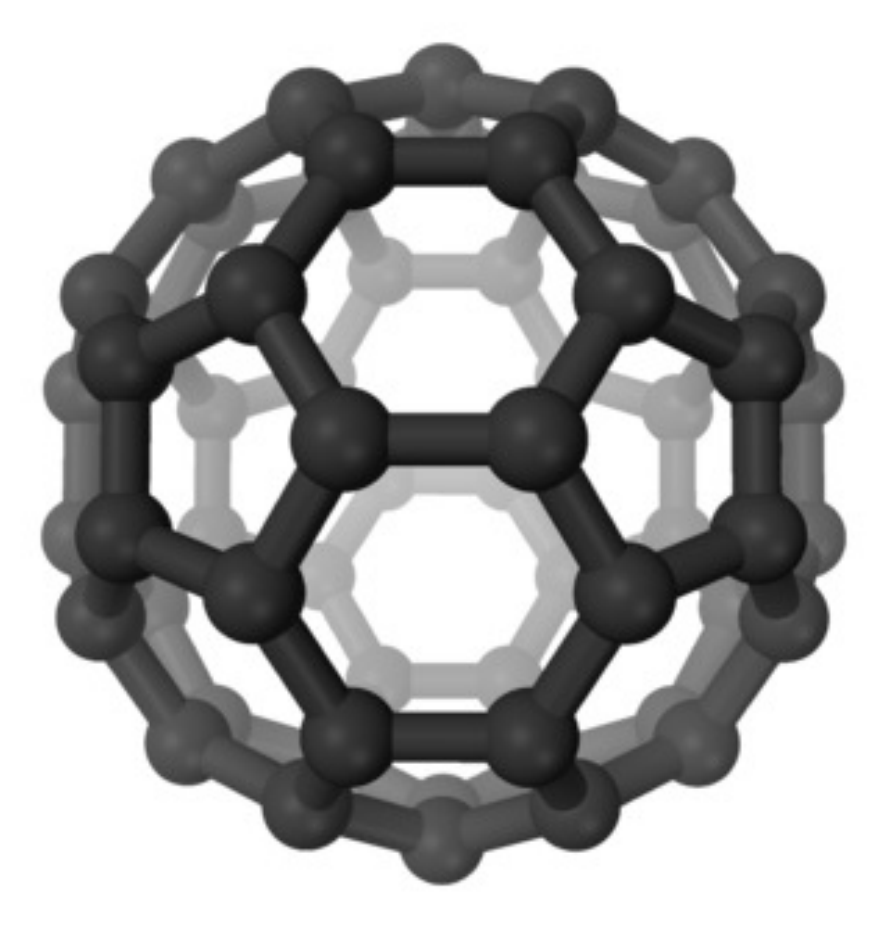}&\includegraphics[height=6cm]{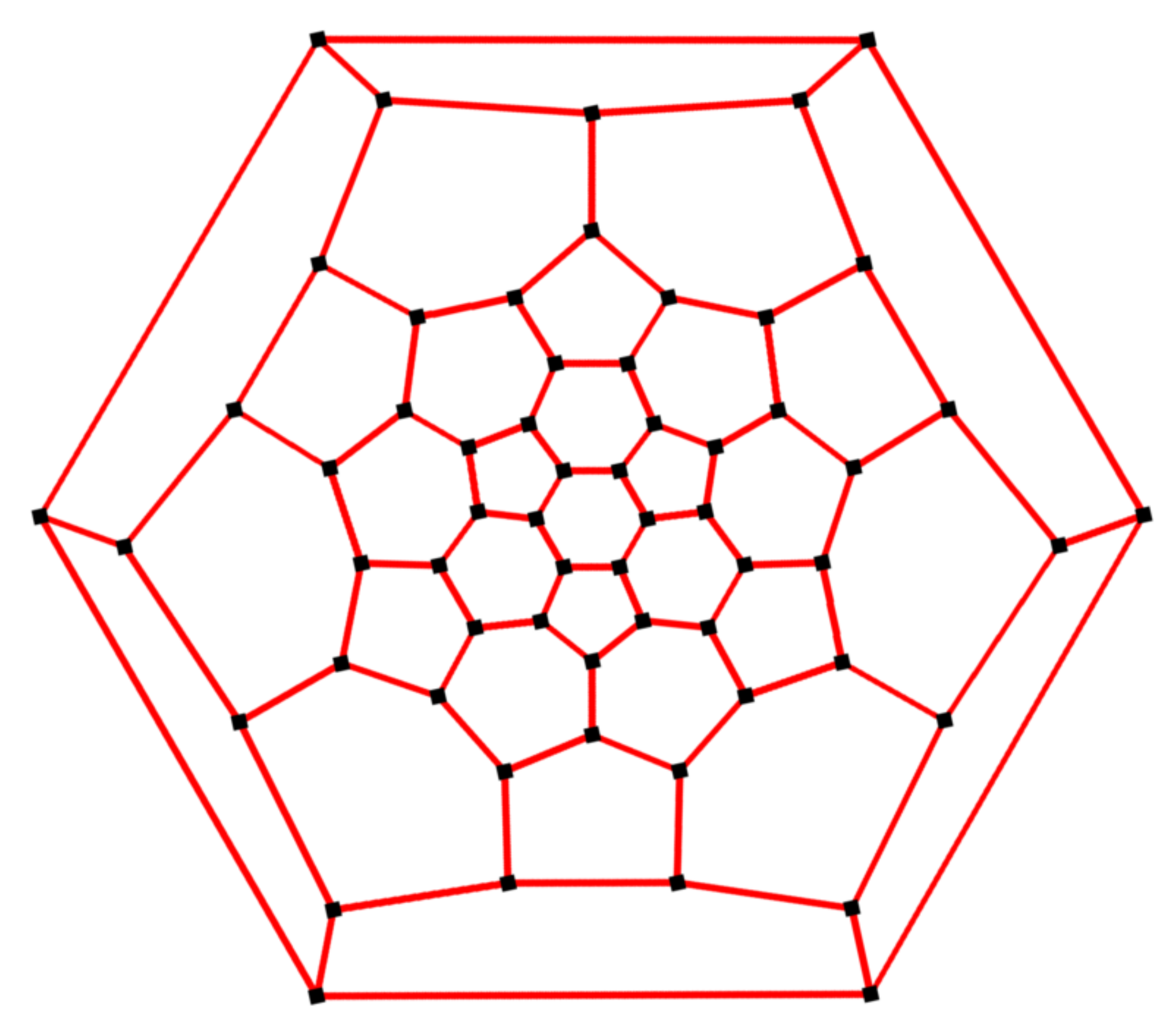}\\
Buckminsterfullerene $C_{60}$&Schlegel diagram\\
$(f_0,f_1,f_2)=(60,90,32)$&\\
$(p_5,p_6)=(12,20)$&
\end{tabular}
\end{center}
\caption{Buckminsterfullerene and it's Schlegel diagram (www.wikipedia.org)}
\end{figure}

The first fullerene $C_{60}$ was generated by chemists-theorists  Robert
Curl, Harold Kroto, and Richard Smalley in 1985 (Nobel Prize in chemistry 1996,
\cite{C96,K96,S96}). They called it {\em Buckminsterfullerene}.\index{Buckminsterfullerene}\index{fullerene!Buckminsterfullerene}

\begin{figure}
\begin{tabular}{cl}
\hspace{-0.4cm}\begin{tabular}{c}
\includegraphics[scale=0.25]{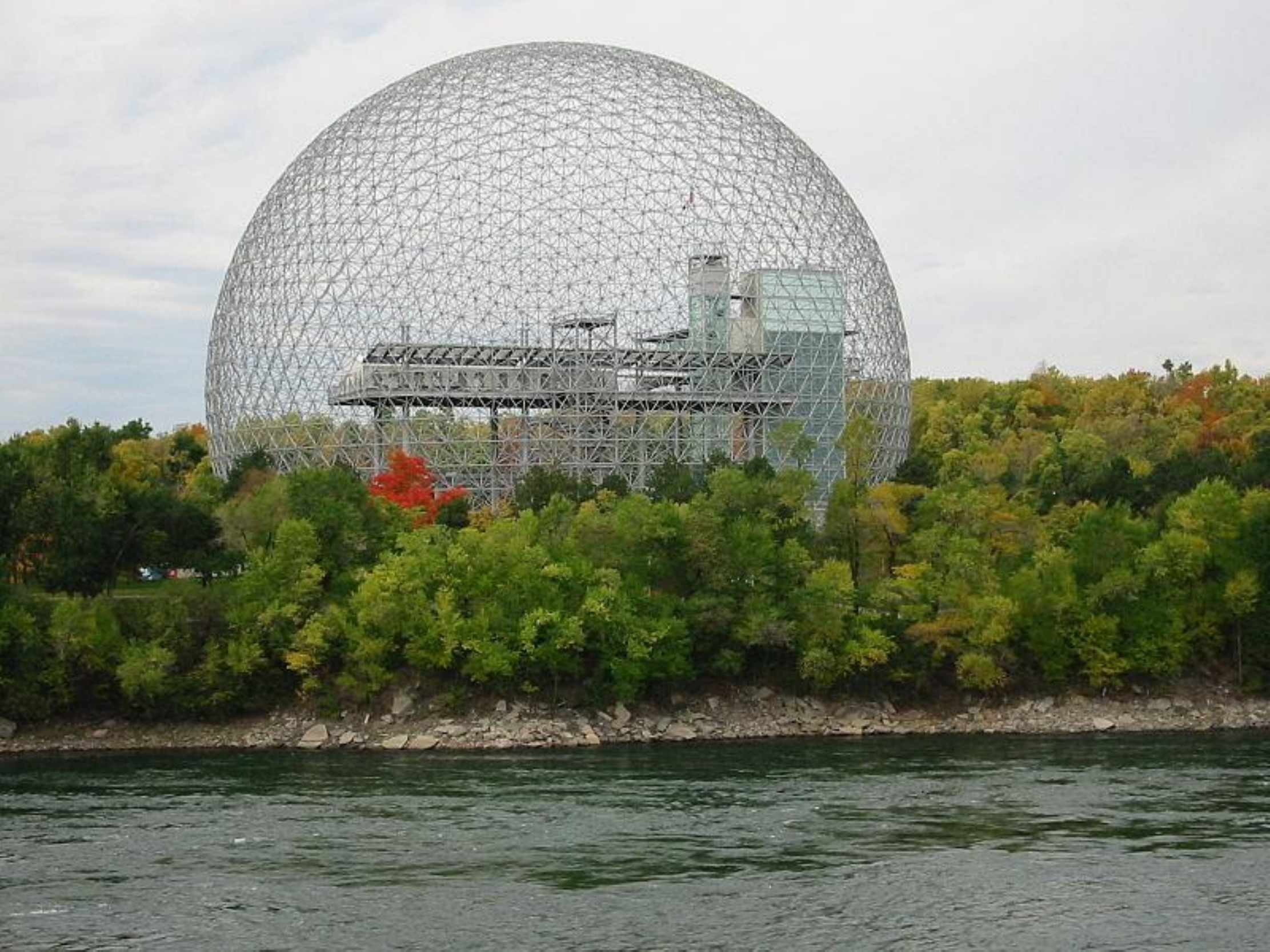}
\end{tabular}&
\hspace{-0.4cm}\begin{tabular}{l}
Fullerenes were named after\\
Richard Buckminster Fuller\\
(1895-1983) -- a famous american\\
architect, systems theorist, author,\\
designer and inventor. In 1954 he\\
patented an architectural construction\\ 
in the form of polytopal spheres\\
for roofing large areas.

\vspace{0.2cm}   \\
They are also called {\em buckyballs}.
\end{tabular}
\end{tabular}
\caption{Fuller's Biosphere, USA Pavillion on Expo-67 (Montreal, Canada) (www.wikipedia.org)}
\end{figure}
\begin{definition}
A {\em combinatorial fullerene}\index{combinatorial fullerene} is a simple $3$-polytope with all facets
pentagons and hexagons. 
\end{definition}
To be short by a fullerene below we mean a combinatorial fullerene.
\begin{figure}
\begin{center}
\begin{tabular}{ccc}
\includegraphics[height=3cm]{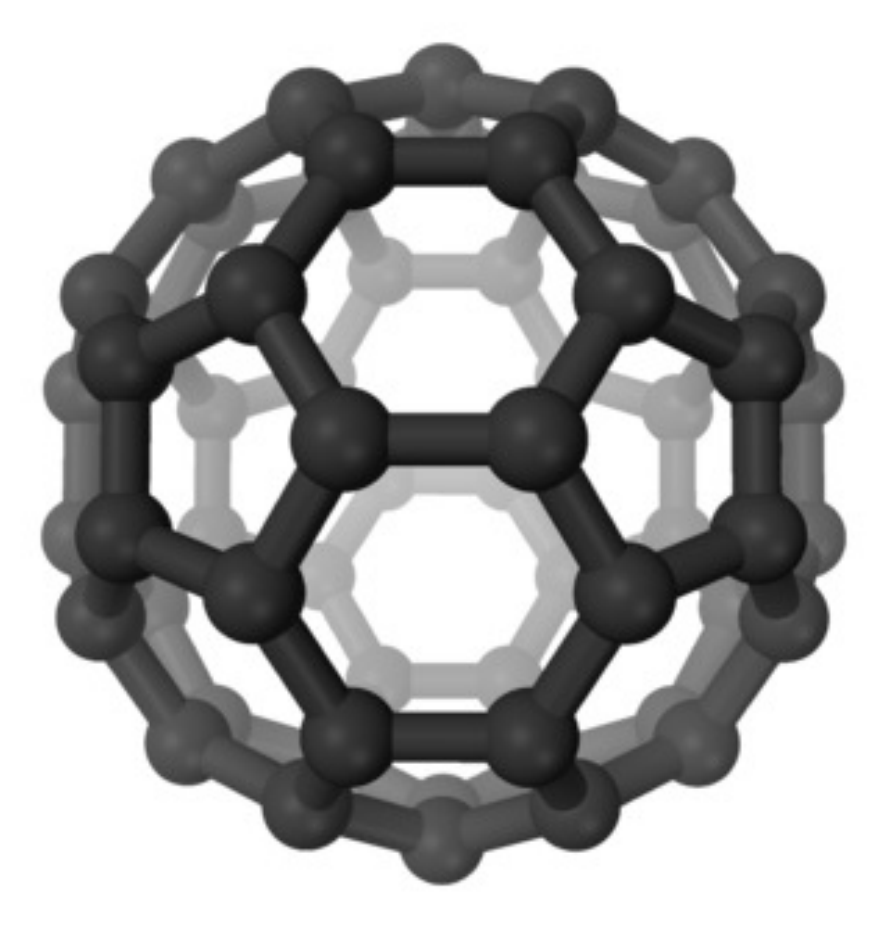}&\hspace{1cm}&\includegraphics[height=3cm]{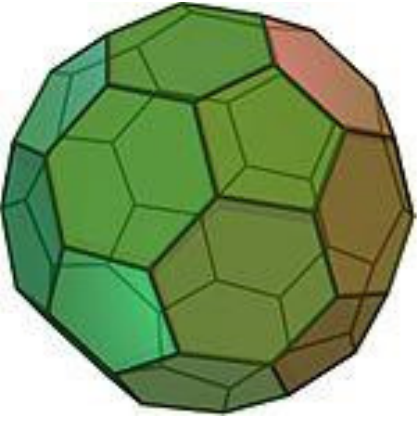}
\end{tabular}
\end{center}
\caption{Fullerene $C_{60}$ and truncated icosahedron (www.wikipedia.org)}
\end{figure}
For any fullerene  $p_5=12$, and expression of the $f$-vector in terms of the $p$-vector obtains the form 
$$
f_0=2(10+p_6),\quad f_1=3(10+p_6),\quad f_2=(10+p_6)+2
$$

\begin{remark}
Since the combinatorially chiral polytope is geometrically chiral (see Proposition \ref{Chiral}), the following problem is important for applications in the physical theory of fullerenes:\\ 

\noindent{\bf Problem:} To find an algorithm to decide if the given fullerene is combinatorially chiral.
\end{remark}
\subsection{Icosahedral fullerenes}
Operations I and II (see page \pageref{Op1}) transform fullerenes into fullerenes. The first procedure increases $f_0$ in $4$ times,
the second -- in $3$ times.

Applying operation I to the dodecahedron we obtain fullerene $ C_{80}$ with $p_6 = 30$.
In total there are $31 924$ fullerenes with $p_6 = 30$.
\begin{figure}
\begin{center}
\begin{tabular}{cc}
\includegraphics[height=3cm]{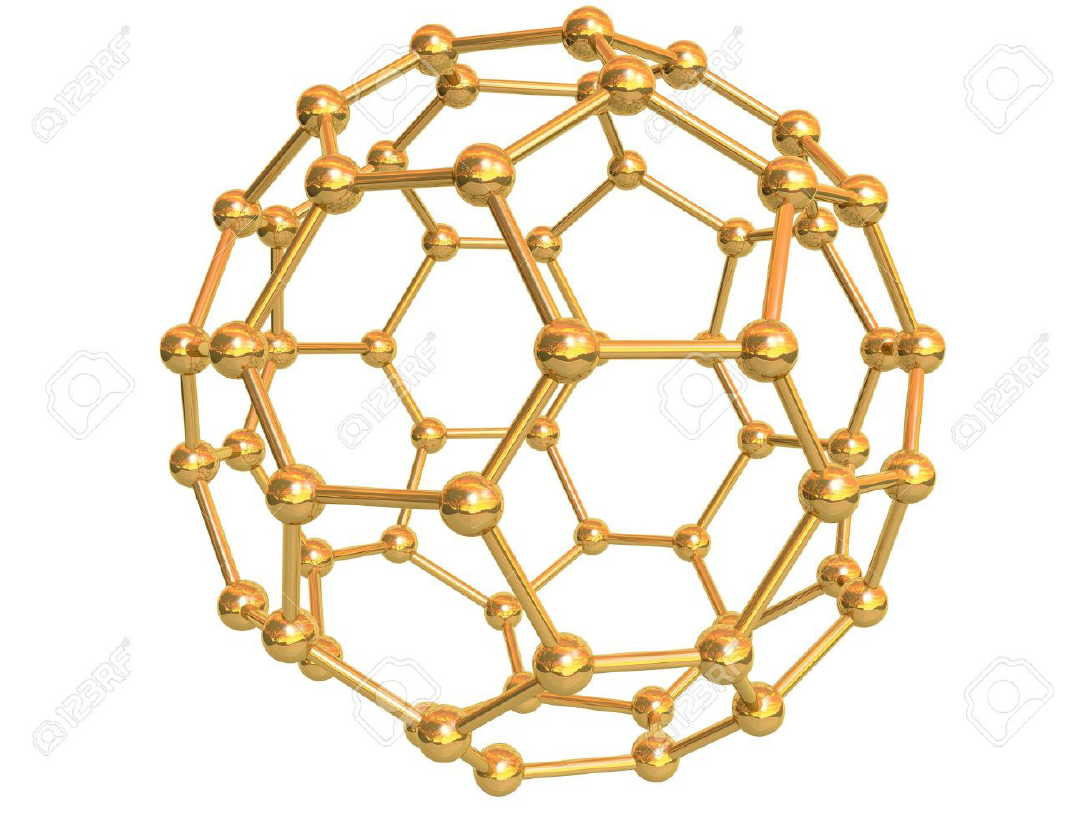}&\includegraphics[height=3cm]{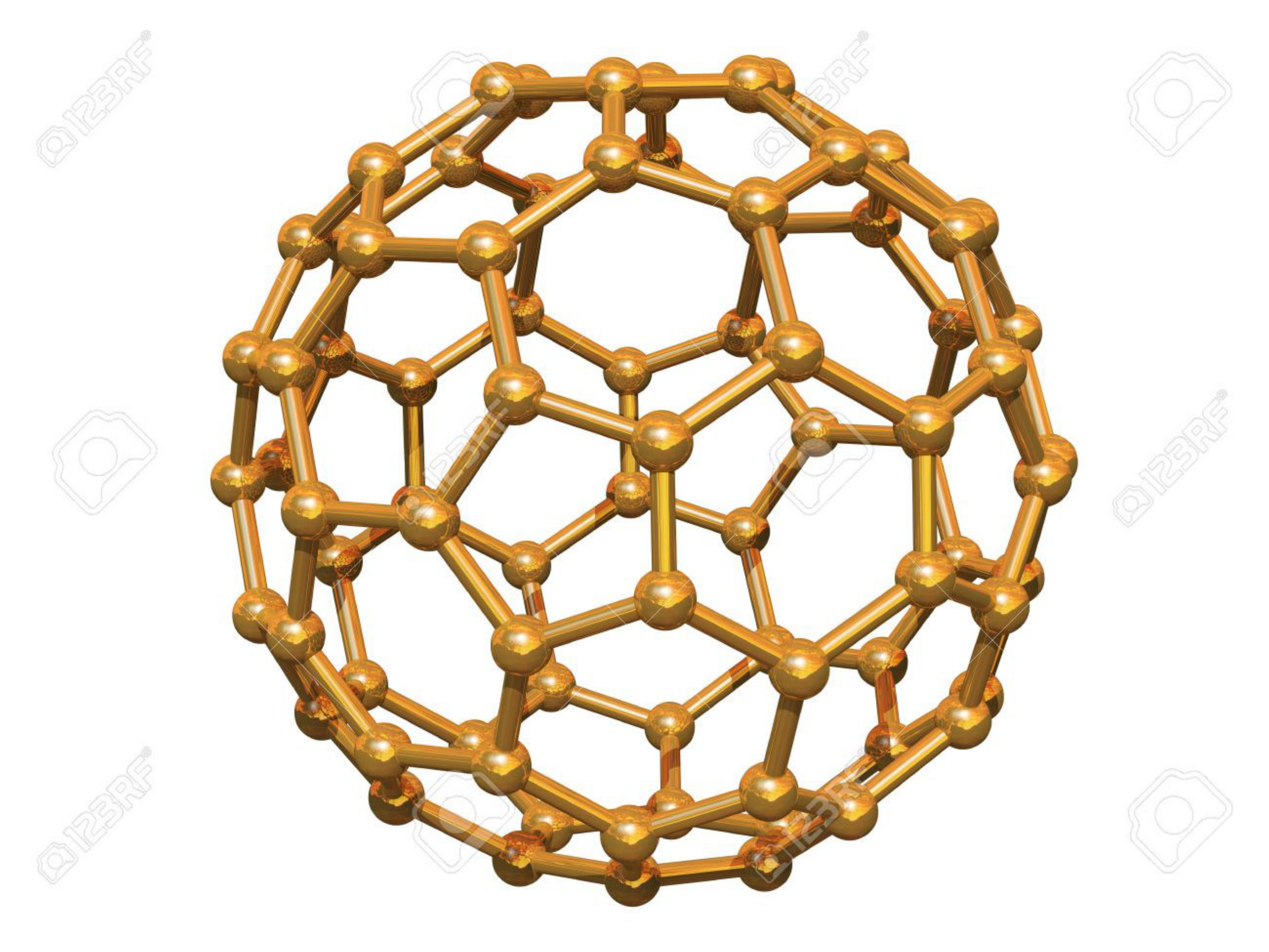}\\
$C_{60}$&$C_{80}$
\end{tabular}
\end{center}
\caption{Icosahedral fullerenes $C_{60}$ and $C_{80}$ (http://previews.123rf.com)}
\end{figure}
Applying operation II  to the dodecahedron we obtain the Buckminsterfullerene\index{Buckminsterfullerene} $C_{60}$ with $p_6 = 20$. In total there are $1812$ fullerenes with $p_6 = 20$.

\begin{definition}
Fullerene with a (combinatorial) group of symmetry of the icosahedron is
called an {\em icosahedral fullerene}\index{fullerene!icosahedral}.
\end{definition}

The construction implies that starting from the dodecahedron any combination
of the first and the second iterative procedures gives an icosahedral
fullerene.\\
{\bf Exercise:} Proof that the opposite is also true.

Denote operation $1$ by $T_1$ and operation $2$ by $T_2$. Theses operations define the action of the semigroup $\mathbb Z^2_{\geqslant 0}$ on the set of combinatorial fullerenes. 
\begin{proposition}
The operations $T_1$ and $T_2$ change  the number of hexagons of the fullerene $P$ by the following rule:
$$
p_6(T_1P)=30+4p_6(P);\quad p_6(T_2P)=20+3p_6(P).
$$
\end{proposition} 
The proof we leave as an exercise.
\begin{corollary} The $f$-vector of a fullerene is changed by the following rule:
$$
T_1(f_0,f_1,f_2)=(4f_0,4f_1,f_2+f_1);\quad T_2(f_0,f_1,f_2)=(3f_0,3f_1,f_2+f_0).
$$
\end{corollary}
\subsection{Cyclic $k$-edge cuts}
\begin{definition} Let $\Gamma$ be a graph. A {\em cyclic $k$-edge cut}\index{cyclic $k$-edge cut} is a set $E$ of $k$ edges of
$\Gamma$,  such that $\Gamma\setminus E$ consists of two connected
component each containing a cycle, and for any subset $E'\subsetneq E$ the
graph $\Gamma\setminus E'$ is connected.
\end{definition}
For any $k$-belt $(F_1,\dots, F_k)$ of the simple $3$-polytope $P$ the set of
edges $\{F_1\cap F_2,\dots, F_{k-1}\cap F_k, F_k\cap F_1\}$ is a cyclic
$k$-edge cut of the graph $G(P)$. For $k=3$ any cyclic $k$-edge cut in $G(P)$
is obtained from a $3$-belt in this way. For larger $k$ not any cyclic $k$-edge cut is obtained from a $k$-belt.

In the paper \cite{D98} it was proved that for any fullerene $P$ the graph
$G(P)$ has no cyclic $3$-edge cuts. In \cite{D03} it was proved that $G(P)$
has no cyclic $4$-edge cuts. In \cite{KM07} and \cite{KS08} cyclic $5$-edge
cuts were classified. In \cite{KS08} cyclic $6$-edge cuts were classified. In
\cite{KKLS10} degenerated cyclic $7$-edge cuts and fullerenes with
non-degenerated cyclic $7$-edge cuts were classified, where a cyclic $k$-edge
cut is called degenerated, if one of the connected components has less than
$6$ pentagonal facets, otherwise it is called non-degenerated.

\subsection{Fullerenes as flag polytopes}
Let $\gamma$ be a simple edge-cycle on a simple $3$-polytope. We say that $\gamma$ {\em borders} a $k$-loop $\mathcal{L}$
if $\mathcal{L}$ is a set of facets that appear when we walk along $\gamma$ in one of the components  $\mathcal{C}_{\alpha}$. We say that an $l_1$-loop $\mathcal{L}_1=(F_{i_1}, \dots, F_{i_{l_1}})$  {\em borders} an $l_2$-loop $\mathcal{L}_2=(F_{j_1},\dots, F_{j_{l_2}})$ \linebreak (along $\gamma$), if they border the same edge-cycle $\gamma$. If $l_2=1$, then we say that $\mathcal{L}_1$ {\em surrounds} $F_{j_1}$.

Let  $\gamma$ have $a^1_p$ successive edges corresponding to 
$F_{i_p}\in\mathcal{L}_1$, and $a^2_q$ successive edges corresponding to
$F_{j_q}\in\mathcal{L}_2$.

\begin{lemma}\label{b-lemma}
Let a loop $\mathcal{L}_1$ border a loop	 $\mathcal{L}_2$ along  $\gamma$. Then one of the following holds:
\begin{enumerate}
\item $\mathcal{L}_\alpha$ is a $1$-loop, and $\mathcal{L}_\beta$ is a
    $a^\alpha_1$-loop, for $\{\alpha,\beta\}=\{1,2\}$;
\item $l_1,l_2\geqslant 2$,
    $l_1+l_2=l_\gamma=\sum_{r=1}^{l_1}a^1_r=\sum_{r=1}^{l_2}a^2_r$.
\end{enumerate}
\end{lemma}
\begin{proof}
If $l_2=1$, then $\gamma$ is a boundary of the facet $F_{j_1}$, successive edges of $\gamma$ 
belong to different facets in $\mathcal{L}_1$, and $l_1=a^2_1$. Similar argument works for $l_1=1$.

Let $l_1,l_2\geqslant 2$. Any edge of  $\gamma$ is an intersection of a facet from $\mathcal{L}_1$ with a facet from $\mathcal{L}_2$. Successive edges of $\gamma$ belong to the same facet in $\mathcal{L}_\alpha$ if and only if they belong to successive facets in $\mathcal{L}_\beta$, $\{\alpha, \beta\}=\{1,2\}$;
therefore\linebreak
$l_\alpha=\sum_{r=1}^{l_\beta}(a^\beta_r-1)=\sum_{r=1}^{l_\beta}a_r^\beta-l_\beta$.
We have $l_\gamma=\sum_{r=1}^{l_\beta}a_r^\beta=l_1+l_2$.
\end{proof}
\begin{lemma}\label{belt-lemma}
Let $\mathcal{B}=(F_{i_1},\dots,F_{i_k})$ be a $k$-belt. Then
\begin{enumerate}
\item  $|\mathcal{B}|=F_{i_1}\cup\dots\cup F_{i_k}$ is homeomorphic to a cylinder;
\item $\partial|\mathcal{B}|$ consists of two simple edge-cycles $\gamma_1$ and
    $\gamma_2$.
\item $\partial P\setminus |\mathcal{B}|$ consists of two connected components
    $\mathcal{P}_1$ and $\mathcal{P}_2$.
\item  Let $\mathcal{W}_\alpha=\{F_j\in\mathcal{F}_P \colon {\rm int\, } F_j\subset \mathcal{P}_{\alpha}\}\subset
    \mathcal{F}_P$, $\alpha=1,2$. \\ 
    Then $\mathcal{W}_1\sqcup \mathcal{W}_2\sqcup \mathcal{B}=\mathcal{F}_P$.
\item $\overline{\mathcal{P_\alpha}}=|\mathcal{W}_\alpha|$ is homeomorphic to a disk, $\alpha=1,2$.
\item $\partial \mathcal{P}_\alpha=\partial \overline{\mathcal{P}_\alpha}=\gamma_\alpha$, $\alpha=1,2$.
\end{enumerate}
\end{lemma}
The proof is straightforward using Theorem \ref{Jtheorem}.

Let a facet $F_{i_j}\in\mathcal{B}$ has $\alpha_j$ edges in $\gamma_1$ and
$\beta_j$ edges in $\gamma_2$. If $F_{i_j}$ is an $m_{i_j}$-gon, then $\alpha_j+\beta_j=m_{i_j}-2$.

\begin{lemma}\label{loop-lemma}
Let  $P$ be a simple $3$-polytope with $p_3=0$, $p_k=0$, $k\geqslant 8$,
$p_7\leqslant 1$, and let $\mathcal{B}_k$ be a $k$-belt, $k\geqslant 3$, consisting of $b_i$ $i$-gons,
$4\leqslant i\leqslant 7$. Then one of the following holds:
\begin{enumerate}
\item $\mathcal{B}_k$ surrounds two $k$-gonal facets $F_s: \{F_s\}=\mathcal{W}_1$, and
    $F_t: \{F_t\}=\mathcal{W}_2$,\\ and all facets of $\mathcal{B}_k$ are quadrangles;
\item $\mathcal{B}_k$ surrounds a $k$-gonal facet $F_s: \{F_s\}=\mathcal{W}_\alpha$, and borders
    an $l_\beta$-loop $\mathcal{L}_\beta\subset \mathcal{W}_\beta$, $\{\alpha,\beta\}=\{1,2\}$,
    $l_\beta=b_5+2b_6+3b_7\geqslant 2$;
\item $\mathcal{B}_k$ borders an $l_1$-loop $\mathcal{L}_1\subset \mathcal{W}_1$ and an $l_2$-loop
    $\mathcal{L}_2\subset \mathcal{W}_2$, where
    \begin{enumerate}
    \item  $l_1=\sum_{j=1}^k(\alpha_j-1)\geqslant 2$, $l_2=\sum_{j=1}^k(\beta_j-1)\geqslant 2$;
    \item $l_1+l_2=2k-2b_4-b_5+b_7\leqslant 2k+1$.
    \item $\min\{l_1,l_2\}\leqslant
        k-b_4-\lceil\frac{b_5-b_7}{2}\rceil\leqslant k$.
    \item If $b_7=0$, $l_1, l_2\geqslant k$, then $l_1=l_2=k$,
        $b_4=b_5=0$, $b_6=k$.
    \end{enumerate}
\end{enumerate}
\end{lemma}
\begin{proof}
Walking round $\gamma_\alpha$ in $\mathcal{P}_\alpha$ we obtain an $l_\alpha$-loop $\mathcal{L}_\alpha\subset
\mathcal{W}_{\alpha}$.

If $\mathcal{B}_k$ surrounds two $k$-gons $F_s:\{F_s\}=\mathcal{W}_1$, and $F_t:\{F_t\}=\mathcal{W}_2$,
then all facets in $\mathcal{B}_k$ are quadrangles.

If $\mathcal{B}_k$ surrounds a $k$-gon $F_s:\{F_s\}=\mathcal{W}_\alpha$ and borders an
$l_\beta$-loop $\mathcal{L}_\beta\subset W_\beta$, $l_\beta\geqslant 2$, then from Lemma
\ref{b-lemma} we have
$$
l_\beta=\sum\limits_{j=1}^k(m_{i_j}-3)-k=\sum\limits_{j=1}^k(m_{i_j}-3-1)=\sum\limits_{j=4}^7jb_j-4\sum\limits_{j=4}^7b_j=b_5+2b_6+3b_7.
$$

If $\mathcal{B}_k$ borders an $l_1$-loop $\mathcal{L}_1$ and an $l_2$-loop $\mathcal{L}_2$,
$l_1,l_2\geqslant 2$, then (a) follows from Lemma~\ref{b-lemma}.
\begin{multline*}
l_1+l_2=\sum\limits_{j=1}^k(\alpha_j+\beta_j-2)=\sum\limits_{i=1}^k(m_{i_j}-4)=\\ 
\sum\limits_{j=4}^7jb_j-4\sum\limits_{j=4}^7b_j=b_5+2b_6+3b_7=2k-2b_4-b_5+b_7.
\end{multline*}
We have $\min\{l_1,l_2\}\leqslant \left[\frac{l_1+l_2}{2}\right]=
k-b_4-\lceil\frac{b_5-b_7}{2}\rceil\leqslant k$, since $b_7\leqslant 1$.

If $b_7=0$ and $l_1,l_2\geqslant k$, then from (3b) we have $l_1=l_2=k$,
$b_4=b_5=0$, $b_6=k$.
\end{proof}

\begin{lemma}\label{23lemma}
Let an $l_1$-loop $\mathcal{L}_1=(F_{i_1},\dots,F_{i_{l_1}})$ border an $l_2$-loop $\mathcal{L}_2$, $l_2\geqslant 2$. 
\begin{enumerate}
\item If $l_1=2$, then $l_2=m_{i_1}+m_{i_2}-4$;
\item If $l_1=3$ and $\mathcal{L}_1$ is not a $3$-belt, then $F_{i_1}\cap F_{i_2}\cap F_{i_3}$ is a vertex, and\linebreak
$l_2=m_{i_1}+m_{i_2}+m_{i_3}-9$.
\end{enumerate}
\end{lemma}
The proof is straightforward from Lemma \ref{b-lemma}.
\begin{theorem}\label{3belts-theorem}
Let $P$ be simple $3$-polytope with $p_3=0$, $p_4\leqslant 2$, $p_7\leqslant
1$, and $p_k=0$, $k\geqslant 8$. Then it has no $3$-belts. In particular, it
is a flag polytope.
\end{theorem}
\begin{proof}
Let $P$ has a $3$-belt $\mathcal{B}_3$. Since $p_3=0$, by Lemma \ref{loop-lemma} it borders an
$l_1$-loop $\mathcal{L}_1$ and $l_2$-loop $\mathcal{L}_2$, where $l_1,l_2\geqslant 2$,
$l_1+l_2\leqslant 7$. By Lemma \ref{23lemma} (1) we have
$l_1,l_2\geqslant 3$; hence $\min\{l_1,l_2\}=3$. If $\mathcal{B}_3$ contains a heptagon,
then $\mathcal{W}_1,\mathcal{W}_2$ contain no heptagons. If $\mathcal{B}_3$ contains no heptagons, then from Lemma \ref{loop-lemma} (3d)  $l_1=l_2=3$, and one of the sets $\mathcal{W}_1$ and $\mathcal{W}_2$, say $\mathcal{W}_\alpha$, contains no heptagons. In both cases we obtain a set $\mathcal{W}_\alpha$ without heptagons and a
$3$-loop $\mathcal{L}_\alpha\subset \mathcal{W}_\alpha$. Then $\mathcal{L}_\alpha$ is a $3$-belt, else by Lemma
\ref{23lemma} (2)  the belt $\mathcal{B}_3$ should have at least $4+4+5-9=4$ facets.
Considering the other boundary component of $\mathcal{L}_\alpha$ we obtain again a
$3$-belt there. Thus we obtain an infinite series of different $3$-belts
inside $|\mathcal{W}_\alpha|$. A contradiction.
\end{proof}
\begin{corollary}\label{3belts-ful}\index{fullerene!flagness}
Any fullerene is a flag polytope.
\end{corollary}
This result follows directly from the results of paper \cite{D98} about cyclic $k$-edge cuts of fullerenes. We present a different approach from \cite{Bu-Er15,BE15b} based on the notion of a $k$-belt.

\begin{corollary}\label{Cor3belts}
Let $P$ be a fullerene. Then any $3$-loop surrounds a vertex.
\end{corollary}
In what follows we will implicitly use the fact that for any flag polytope, in particular satisfying conditions of Theorem \ref{3belts-theorem}, if facets $F_i$, $F_j$, $F_k$ pairwise intersect, then $F_i\cap F_j\cap F_k$ is a vertex.

\subsection{$4$-belts and $5$-belts of fullerenes}

\begin{lemma}\label{4lemma}
Let $P$ be a flag $3$-polytope, and let a $4$-loop $\mathcal{L}_1=(F_{i_1},F_{i_2},F_{i_3},F_{i_4})$ border an $l_2$-loop $\mathcal{L}_2$, $l_2\geqslant 2$, where index $j$ of $i_j$ lies in $\mathbb Z_4=\mathbb Z/(4)$. Then one of the following holds:
\begin{enumerate}
\item $\mathcal{L}_1$ is a {\bf $4$-belt} (Fig. \ref{4loop} a);
\item $\mathcal{L}_1$ is a {\bf simple loop} consisting of facets surrounding an edge (Fig. \ref{4loop} b),
and\\
$l_2=m_{i_1}+m_{i_2}+m_{i_3}+m_{i_4}-14$;
\item $\mathcal{L}_1$ is {\bf not a simple loop}: $F_{i_j}=F_{i_{j+2}}$ for some $j$,  $F_{i_{j-1}}\cap F_{i_{j+1}}=\varnothing$ (Fig. \ref{4loop} c), and $l_2=m_{i_{j-1}}+m_{i_j}+m_{i_{j+1}}-8$.
\end{enumerate}
\end{lemma}
\begin{figure}
\begin{center}
\includegraphics[height=4cm]{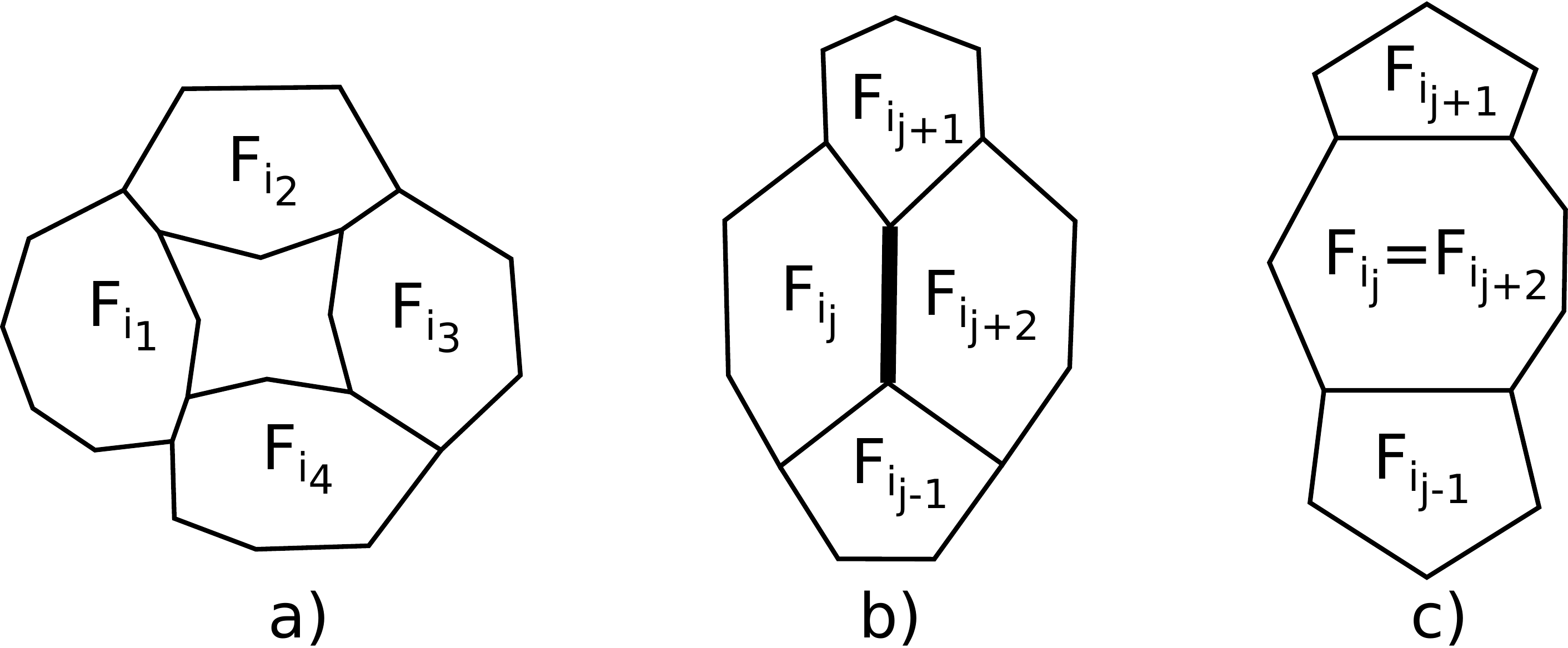}
\caption{Possibilities for a $4$-loop $\mathcal{L}_1$}\label{4loop}
\end{center}
\end{figure}
\begin{proof}
Let $\mathcal{L}_1$ be not a $4$-belt. If $\mathcal{L}_1$ is simple, then $F_{i_j}\cap F_{i_{j+2}}\ne\varnothing$ for some $j$. Then $F_{i_j}\cap F_{j_{j+1}}\cap F_{i_{j+2}}$ and $F_{i_j}\cap F_{j_{j-1}}\cap F_{i_{j+2}}$ are vertices, $\mathcal{L}_1$ surrounds the edge $F_{i_j}\cap F_{i_{j+2}}$, and  by Lemma \ref{b-lemma} we have  $l_2=(m_{i_j}-3)+(m_{i_{j+1}}-2)+(m_{i_{j+2}}-3)+(m_{i_{j-2}}-2)-4=m_{i_1}+m_{i_2}+m_{i_3}+m_{i_4}-14$.  If $\mathcal{L}_1$ is not simple, then $F_{i_j}=F_{i_{j+2}}$ for some $j$. The successive facets of $\mathcal{L}_1$ are different by definition. Let $\mathcal{L}_1$ and $\mathcal{L}_2$ border the edge cycle $\gamma$ and $\mathcal{L}_1\subset \mathcal{D}_{\alpha}$ in notations of Theorem \ref{Jtheorem}. Since $F_{i_j}$ intersects $\gamma $ by two paths,  ${\rm int }\,F_{i_{j-1}}$ and ${\rm int}\,F_{i_{j+1}}$ lie in different connected components of $\mathcal{C}_\alpha\setminus{\rm int}\, F_{i_j}$; hence $F_{i_{j-1}}\cap F_{i_{j+1}}=\varnothing$. By Lemma \ref{b-lemma} we have $l_2=(m_{i_{j-1}}-1)+(m_{i_j}-2)+(m_{i_{j+1}}-1)-4=m_{j_{j-1}}+m_{i_j}+m_{i_{j+1}}-8$.   
\end{proof}

\begin{theorem}\label{4belts-theorem}
Let $P$ be a simple polytope with all facets pentagons and hexagons with at
most one exceptional facet $F$ being a quadrangle or a heptagon.
\begin{enumerate}
\item If $P$ has no quadrangles, then $P$ has no $4$-belts.
\item If $P$ has a quadrangle $F$, then there is exactly one $4$-belt. It
    surrounds $F$.
\end{enumerate}
\end{theorem}
\begin{proof}
By Theorem \ref{3belts-theorem} the polytope $P$ is flag.

By Lemma \ref{facet-belt} a quadrangular facet is surrounded by a $4$-belt.

Let $\mathcal{B}_4$ be a $4$-belt that does not surround a quadrangular facet. By
Lemma \ref{loop-lemma} it borders an $l_1$-loop $\mathcal{L}_1$ and $l_2$-loop
$\mathcal{L}_2$, where $l_1,l_2\geqslant 2$, and $l_1+l_2\leqslant 9$. We have
$l_1,l_2\geqslant 3$, since by Lemma \ref{23lemma} (1) a $2$-loop borders a $k$-loop with $k\geqslant
4+5-4=5$. We have $l_1,l_2\geqslant 4$ by Theorem \ref{3belts-theorem} and
Lemma \ref{23lemma} (2), since a $3$-loop that is not a $3$-belt borders a
$k$-loop with $k\geqslant 4+5+5-9=5$. Also $\min\{l_1,l_2\}=4$.  If $\mathcal{B}_4$
contains a heptagon, then $\mathcal{W}_1,\mathcal{W}_2$ contain no heptagons. If $\mathcal{B}_4$
contains no heptagons, then $l_1=l_2=4$ by Lemma \ref{loop-lemma} (3d), and one of the sets $\mathcal{W}_1$ and
$\mathcal{W}_2$, say $\mathcal{W}_\alpha$, contains no heptagons. In both cases we obtain a set
$\mathcal{W}_\alpha$ without heptagons and a $4$-loop $\mathcal{L}_\alpha\subset \mathcal{W}_\alpha$.
Then $\mathcal{L}_\alpha$ is a $4$-belt, else by Lemma \ref{4lemma} the belt $\mathcal{B}_4$
should have at least $4+5+5+5-14=5$ or $4+5+5-8=6$ facets. Applying the same argument to
$\mathcal{L}_\alpha$ instead of $\mathcal{B}_k$, we have that either $\mathcal{L}_\alpha$ surrounds on the
opposite side a quadrangle, or it borders a $4$-belt and consists of
hexagons. In the first case by Lemma \ref{loop-lemma} (2) the $4$-belt $\mathcal{L}_\alpha$ consists of
pentagons. Thus we can move inside $\mathcal{W}_\alpha$ until we finish with a quadrangle.
If $P$ has no quadrangles, then we obtain a contradiction. If $P$ has a
quadrangle $F$, then it has no heptagons; therefore moving inside $\mathcal{W}_\beta$ we
should meet some other quadrangle. A contradiction.
\end{proof}
\begin{corollary}\label{4-belt-ful}\index{fullerene!absence of $4$-belts}
Fullerenes have no $4$-belts.
\end{corollary}
This result follows directly from \cite{D03}. Above we prove more general Theorems \ref{3belts-theorem} and \ref{4belts-theorem}, since we will need them in Lecture 9.
\begin{corollary}
Let $P$ be a fullerene. Then any simple $4$-loop surrounds an edge.
\end{corollary}
Now consider $5$-belts of fullerenes. Describe a special family of fullerenes.
\begin{figure}
\begin{center} 
\begin{tabular}{cc}
\includegraphics[height=3cm]{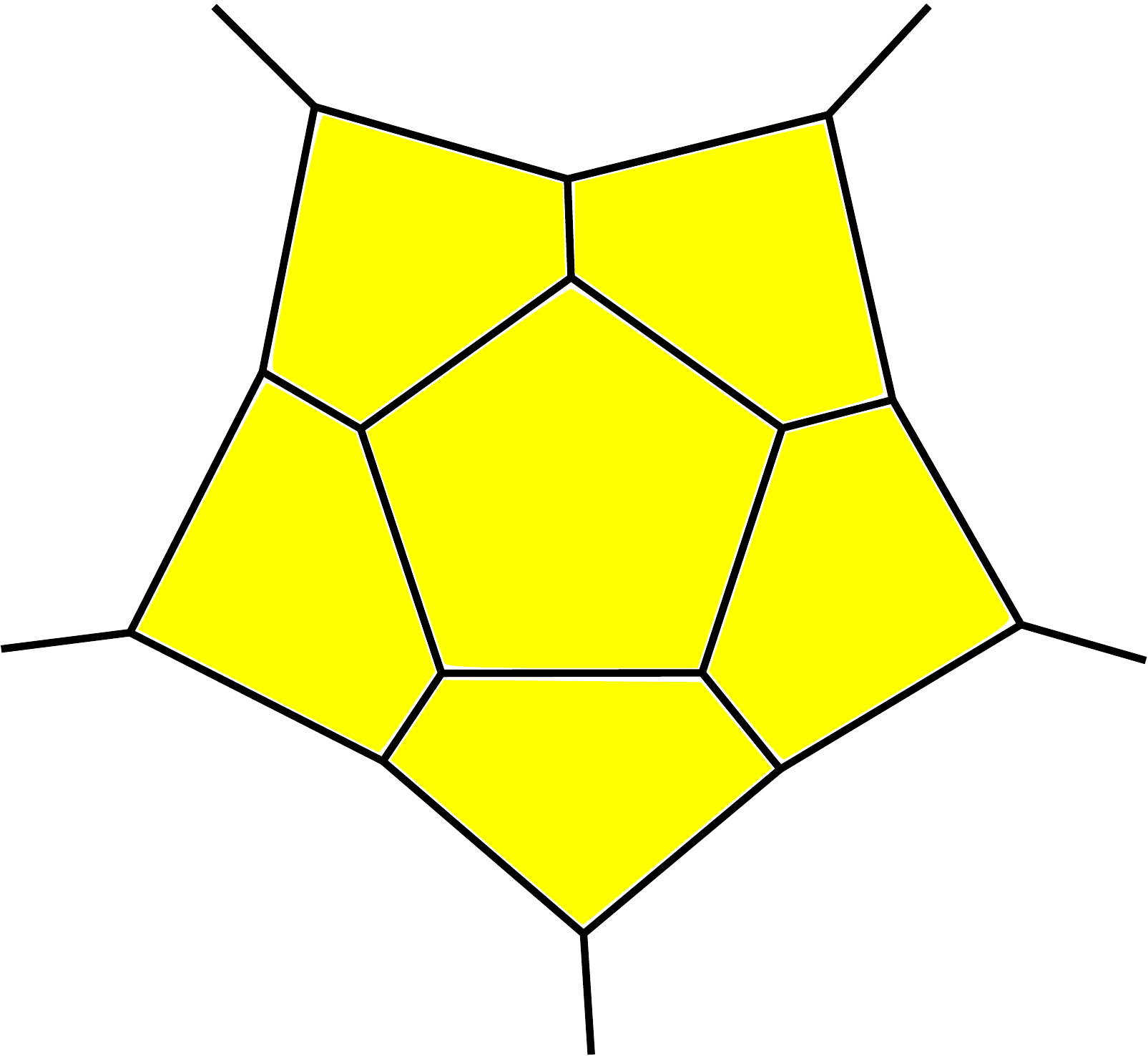}&\includegraphics[height=3cm]{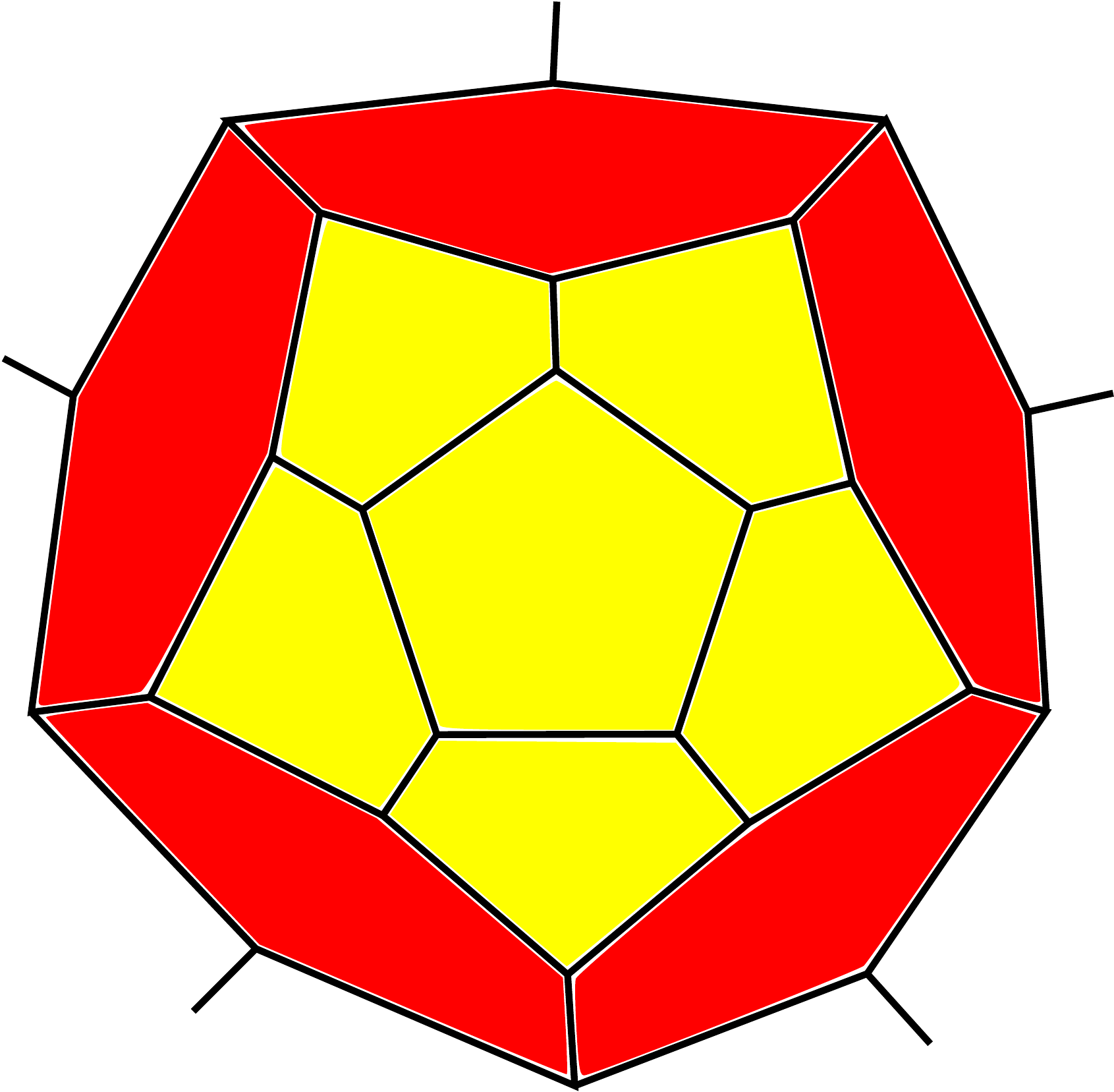}\\
cap&the first $5$-belt\\
a)&b)
\end{tabular}
\end{center}
\caption{Construction of fullerenes $D_k$}\label{dkfig}
\end{figure}

\noindent{\bf Construction (Series of polytopes $D_k$):} Denote by $D_0$ the dodecahedron. If we cut it's surface along the zigzag cycle (Fig. \ref{zzf}), we obtain two caps on Fig. \ref{dkfig}a). Insert $k$  successive $5$-belts of hexagons with hexagons intersecting neighbors by opposite edges to obtain the combinatorial description of $D_k$. \linebreak
We have $p_6(D_k)=5k$,  $f_0(D_k)=20+10k$, $k\geqslant0$.

Geometrical realization of the polytope $D_k$ can be obtained from the geometrical realization of $D_{k-1}$ by the the following sequence of edge- and two-edges truncations, represented on Fig. \ref{OP1}.
\begin{figure}
\begin{center} 
\includegraphics[height=1.8cm]{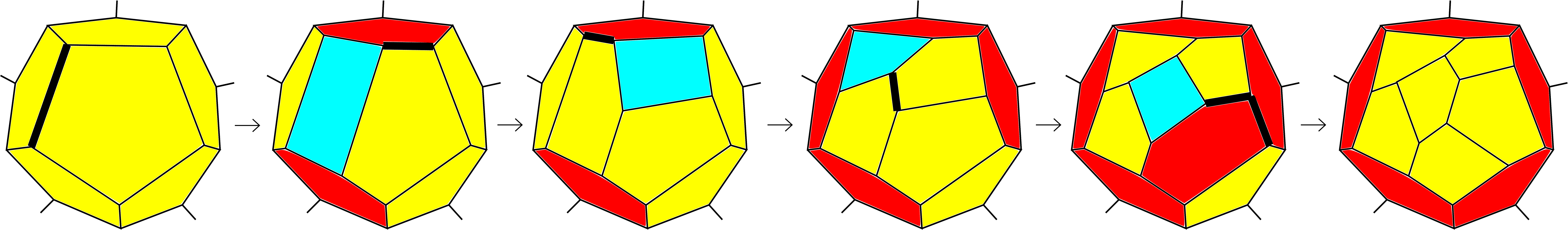}
\end{center}
\caption{Geometrical construction of a $5$-belt of hexagons}\label{OP1}
\end{figure}

The polytopes $D_k$ for $k\geqslant 1$ are exatly nanotubes of type $(5,0)$\index{$(5,0)$-nanotubes} \cite{KM07,KS08,KKLS10}.
\begin{figure}
\begin{center} 
\begin{tabular}{cc}
\includegraphics[height=3cm]{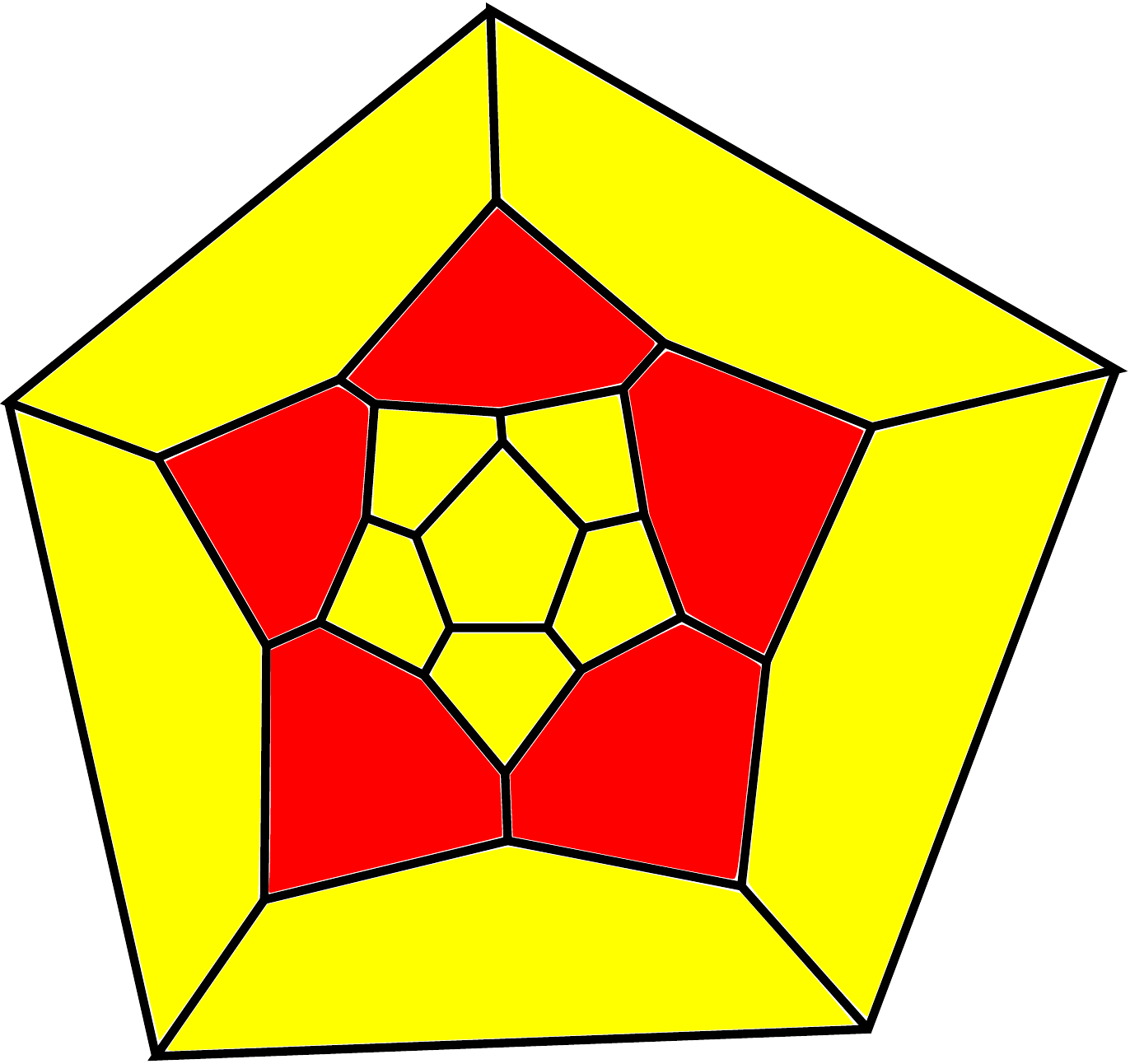}&\includegraphics[height=3cm]{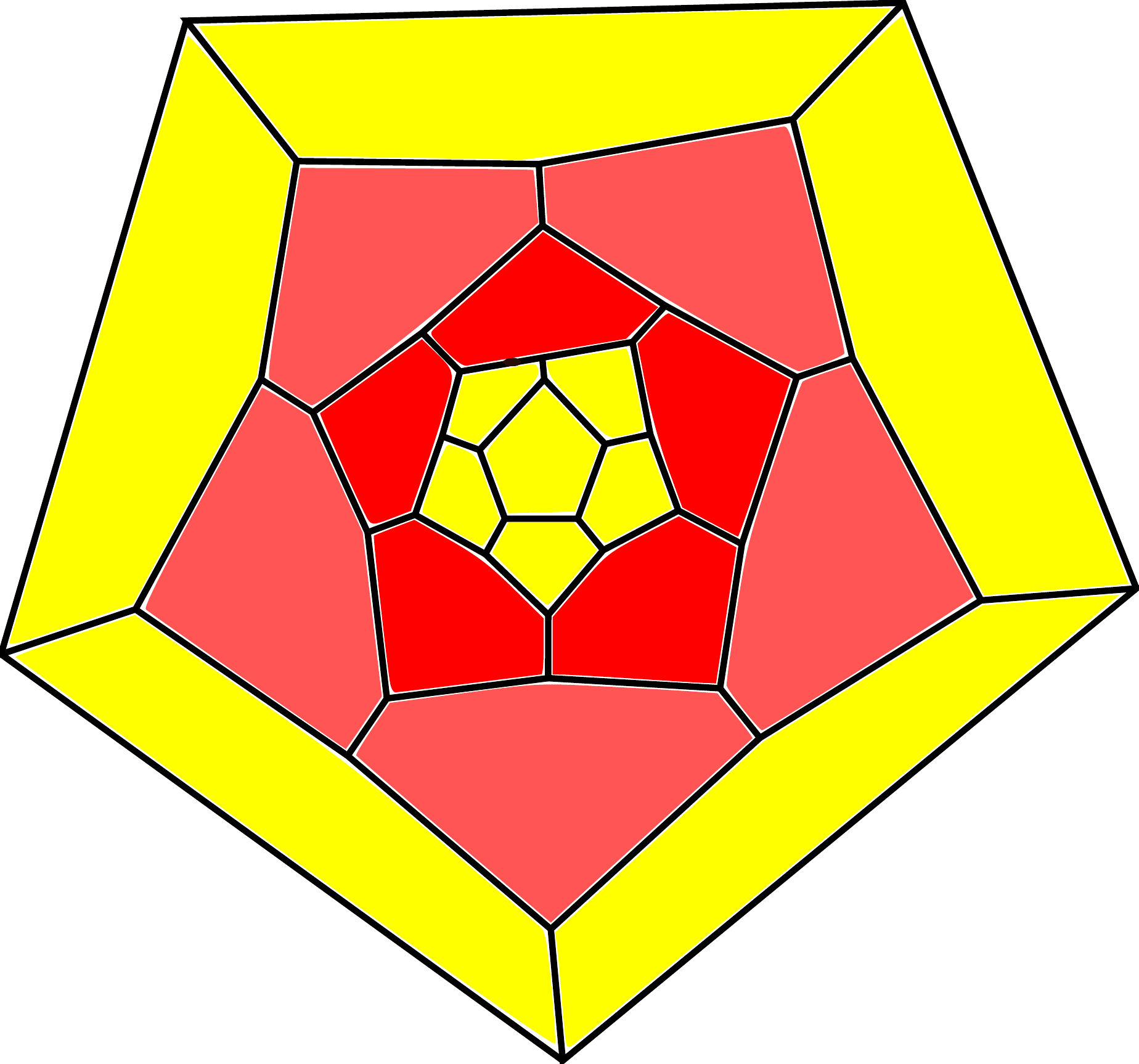}\\
\end{tabular}
\end{center}
\caption{Fullerenes $D_1$ and $D_2$}
\end{figure}
\begin{lemma}\label{5lemma}
Let $P$ be a flag $3$-polytope without $4$-belts, and let a $5$-loop\linebreak  $\mathcal{L}_1=(F_{i_1},F_{i_2},F_{i_3},F_{i_4},F_{i_5})$ border an $l_2$-loop $\mathcal{L}_2$, $l_2\geqslant 2$, where index $j$ of $i_j$ lies in $\mathbb Z_5=\mathbb Z/(5)$. Then one of the following holds:
\begin{enumerate}
\item $\mathcal{L}_1$ is a {\bf $5$-belt} (Fig. \ref{5loop}a);
\item $\mathcal{L}_1$ is a {\bf simple loop} consisting of facets surrounding two adjacent edges\linebreak (Fig. \ref{5loop}b), and $l_2=m_{i_1}+m_{i_2}+m_{i_3}+m_{i_4}+m_{i_5}-19\geqslant 6$;
\item $\mathcal{L}_1$ is {\bf not a simple loop}: $F_{i_j}=F_{i_{j+2}}$ for some $j$, $F_{i_{j-2}}\cap F_{i_{j-1}}\cap F_{i_j}$ is a vertex,  $F_{i_{j+1}}$ does not intersect $F_{i_{j-2}}$ and $F_{i_{j-1}}$ (Fig. \ref{5loop}c), and\\ $l_2=m_{i_{j-2}}+m_{i_{j-1}}+m_{i_j}+m_{i_{j+1}}-13\geqslant 7$. 
\end{enumerate}
\end{lemma}
\begin{figure}[h]
\begin{center}
\includegraphics[height=3cm]{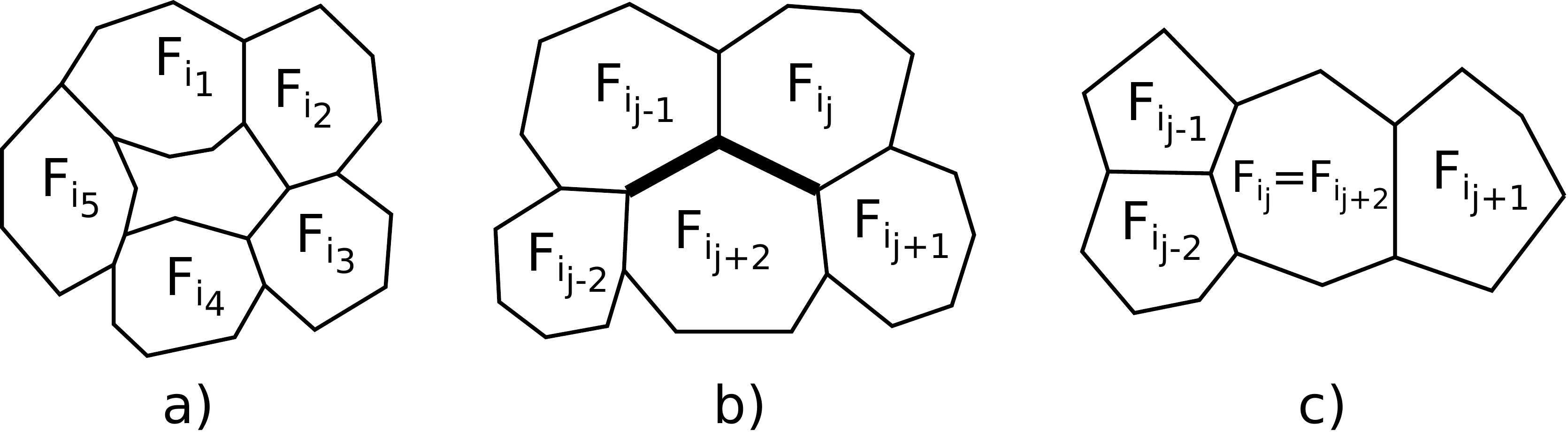}
\end{center}
\caption{Possibilities for a $5$-loop $\mathcal{L}_1$}\label{5loop}
\end{figure}

\begin{proof}
Let $\mathcal{L}_1$ be not a $5$-belt. If $\mathcal{L}_1$ is simple, then two non-successive facets $F_{i_j}$ and $F_{i_{j+2}}$ intersect. Then $F_{i_j}\cap F_{i_{j+1}}\cap F_{i_{j+2}}$ is a vertex. By Theorem \ref{4belts-theorem} the $4$-loop $(F_{i_{j-2}},F_{i_{j-1}},F_{i_j},F_{i_{j+2}})$ is not a $4$-belt; hence either $F_{i_{j-2}}\cap F_{i_j}\ne\varnothing$, or $F_{i_{j-1}}\cap F_{i_{j+2}}\ne\varnothing$. Up to relabeling in the inverse order, we can assume that  $F_{i_{j-1}}\cap F_{i_{j+2}}\ne\varnothing$. Then $F_{i_{j-1}}\cap F_{i_j}\cap F_{i_{j+2}}$ and $F_{i_{j-2}}\cap F_{i_{j-1}}\cap F_{i_{j+2}}$ are vertices.  Thus $\mathcal{L}_1$ surrounds the adjacent edges $F_{i_{j-1}}\cap F_{i_{j+2}}$ and $F_{i_j}\cap F_{i_{j+2}}$.  By Lemma \ref{b-lemma} we have  $l_2=(m_{i_{j-2}}-2)+(m_{i_{j-1}}-3)+(m_{i_j}-3)+(m_{i_{j+1}}-2)+(m_{i_{j+2}}-4)-5=m_{i_1}+m_{i_2}+m_{i_3}+m_{i_4}+m_{i_5}-19\geqslant 6$. The last inequality holds, since flag $3$-polytope without $4$-belts has no triangles and quadrangles. If $\mathcal{L}_1$ is not simple, then $F_{i_j}=F_{i_{j+2}}$ for some $j$. The successive facets of $\mathcal{L}_1$ are different by definition. Let $\mathcal{L}_1$ and $\mathcal{L}_2$ border the edge cycle $\gamma$ and $\mathcal{L}_1\subset \mathcal{D}_{\alpha}$ in notations of Theorem \ref{Jtheorem}. Since $F_{i_j}$ intersects $\gamma $ by two paths,  ${\rm int }\,F_{i_{j-2}}\cup{\rm int}\,F_{i_{j-1}}$ and ${\rm int}\,F_{i_{j+1}}$ lie in different connected components of $\mathcal{C}_\alpha\setminus{\rm int}\, F_{i_j}$; hence $F_{i_{j-2}}\cap F_{i_{j+1}}=\varnothing=F_{i_{j-1}}\cap F_{i_{j+1}}$. Since $P$ is flag, $F_{i_{j-2}}\cap F_{i_{j-1}}\cap F_{i_j}$ is a vertex, thus we obtain the configuration on Fig. \ref{5loop}c. By Lemma \ref{b-lemma} we have $l_2=(m_{i_{j-2}}-2)+(m_{i_{j-1}}-2)+(m_{i_j}-3)+(m_{i_{j+1}}-1)-5=m_{i_{j-2}}+m_{j_{j-1}}+m_{i_j}+m_{i_{j+1}}-13\geqslant 7$.   
\end{proof}
The next result follows directly from \cite{KS08} or  \cite{KM07}. We develop the approach from \cite{BE15b} based on the notion of a $k$-belt.
\begin{theorem}\label{5belts-theorem}\index{fullerene!$5$-belts} Let $P$ be a fullerene. Then the following statements hold.\\
{\bf I.} Any pentagonal facet is surrounded by a $5$-belt.
There are $12$ belts of this type.\\
{\bf II.} If there is a $5$-belt not surrounding a pentagon, then
\begin{enumerate}
\item it consists only of hexagons;
\item the fullerene is combinatorially equivalent to the polytope $D_k$, $k\geqslant 1$.
\item the number of $5$-belts is $12+k$.
\end{enumerate}
\end{theorem}
\begin{proof}
(1) Follows from Proposition \ref{facet-belt} and Corollary \ref{3belts-ful}.

(2) Let the $5$-belt $\mathcal{B}_5$ do not surround a pentagon. By Lemma
\ref{loop-lemma} it borders an $l_1$-loop $\mathcal{L}_1\subset \mathcal{W}_1$ and an
$l_2$-loop $\mathcal{L}_2\subset \mathcal{W}_2$, $l_1,l_2\geqslant 2$, $l_1+l_2\leqslant 10$. By Lemma~\ref{23lemma}~(1) we have $l_1,l_2\geqslant 3$.  From Corollary \ref{3belts-ful} and Lemma~\ref{23lemma}~(2) we obtain
$l_1,l_2\geqslant 4$. From Corollary \ref{4-belt-ful} and Lemma \ref{4lemma} we obtain
$l_1,l_2\geqslant 5$. Then $l_1=l_2=5$ and all facets in $\mathcal{B}_5$ are hexagons by Lemma \ref{loop-lemma} (3d). From Lemma \ref{5lemma} we obtain that $\mathcal{L}_1$ and $\mathcal{L}_2$ are $5$-belts. Moving inside $\mathcal{W}_1$ we obtain a series of hexagonal $5$-belts, and this series can
stop only if the last $5$-belt $\mathcal{B}_l$ surrounds a pentagon. Since $\mathcal{B}_l$ borders a $5$-belt,  Lemma \ref{loop-lemma} (2) implies that $\mathcal{B}_l$ consists of pentagons, which have $(2,2,2,2,2)$ edges on the common boundary with a
$5$-belt. We obtain the fragment on Fig. \ref{dkfig}a). Moving from this fragment backward we obtain a series of hexagonal $5$-belts including $\mathcal{B}_5$ with facets having $(2,2,2,2,2)$ edges on both boundaries. This series can finish only with fragment on Fig. \ref{dkfig}a) again. Thus any belt not surrounding a pentagon belongs to this series and the number of $5$-belts is equal to $12+k$.
\end{proof}
\begin{theorem}\label{frag-theorem}
A fullerene $P$ is combinatorially equivalent to a polytope $D_k$ for some $k\geqslant 0$ if and only if it contains the fragment on Fig. \ref{dkfig}a).
\end{theorem}
\begin{proof}
By Proposition \ref{facet-belt} the outer $5$-loop of the fragment on Fig. \ref{dkfig}a)
is a $5$-belt. By the outer boundary component it borders a $5$-loop $\mathcal{L}$. By
Lemma \ref{5lemma} it is a $5$-belt. If this belt surrounds a pentagon, then we obtain a combinatorial dodecahedron (case $k=0$). If not, then $P$ is combinatorially equivalent to $D_k$, $k\geqslant 1$, by Theorem \ref{5belts-theorem}.
\end{proof}
\begin{corollary} Any simple $5$-loop of a fullerene
\begin{enumerate}
\item either surrounds a pentagon;
\item or is a hexagonal $5$-belt of a fullerene $D_k$, $k\geqslant 1$;
\item or surrounds a pair of adjacent edges (Fig. \ref{5loop}b).
\end{enumerate}
\end{corollary}
\begin{proof}
Let $\mathcal{L}=(F_{i_1},F_{i_2},F_{i_3},F_{i_4},F_{i_5})$ be a simple $5$-loop, where index $j$ of $i_j$ lies in $\mathbb Z_5=\mathbb Z/(5)$. If $\mathcal{L}$ is a $5$-belt, then by Theorem \ref{5belts-theorem}, we obtain cases (1) or (2). Otherwise some non-successive facets intersect: $F_{i_j}\cap F_{i_{j+2}}\ne\varnothing$ for some $j$. Then $F_{i_j}\cap F_{i_{j+1}}\cap F_{i_{j+2}}$ is a vertex. Since a fullerene has no $4$-belts in a simple $4$-loop $(F_{i_{j-2}},F_{i_{j-1}},F_{i_j},F_{i_{j+2}})$ either $F_{i_{j-2}}\cap F_{i_j}\ne\varnothing$, or $F_{i_{j-1}}\cap F_{i_{j+2}}\ne\varnothing$. Up to relabeling in the inverse order, we can assume that  $F_{i_{j-1}}\cap F_{i_{j+2}}\ne\varnothing$. Then $F_{i_{j-1}}\cap F_{i_j}\cap F_{i_{j+2}}$ and $F_{i_{j-2}}\cap F_{i_{j-1}}\cap F_{i_{j+2}}$ are vertices.  Thus $\mathcal{L}_1$ surrounds the adjacent edges $F_{i_{j-1}}\cap F_{i_{j+2}}$ and $F_{i_j}\cap F_{i_{j+2}}$.
\end{proof}
\newpage

\section{Lecture 4. Moment-angle complexes and moment-angle manifolds}
We discuss main notions, constructions and results of toric topology. Details can be found in the monograph \cite{Bu-Pa15}, which we will follow.
\subsection{Toric topology}
Nowadays toric topology\index{toric topology} is a large research area. Below we discuss applications of toric topology to the mathematical theory of fullerenes based on the following correspondence. 
\begin{center}
\begin{tabular}{ccc}
\multicolumn{3}{c}{\bf Canonical correspondence}\\[.2cm]
Simple polytope $P$&&moment-angle manifold $\mathcal{Z}_P$\\
number of facets $=m$&$\longrightarrow$&canonical $T^m$-action on $\mathcal{Z}_P$\\
$\dim P=n$&&$\dim\mathcal{Z}_P=m+n$\\[.2cm]
Characteristic function &&Quasitoric manifold\\
$\{F_1,\dots, F_m\}\to\mathbb Z^n$&$\longrightarrow$&$M^{2n}=\mathcal{Z}_P/T^{m-n}$
\end{tabular}
\end{center}

Algebraic-topological invariants of moment-angle manifolds $\mathcal{Z}_P$ give combinatorial invariants of polytopes $P$. As an application we obtain  combinatorial invariants of mathematical fullerenes.


\subsection{Moment-angle complex of a simple polytope}\label{moment-angle}
Set
$$
D^2=\{z \in \mathbb{C}; |z| \leq 1\},\qquad S^1=\{z \in D^2, |z|=1\}.
$$

The multiplication of complex numbers gives the canonical action of the circle $S^1$ on the disk $D^2$ which orbit space
is the interval $\mathbb{I} = [0,1]$. 

We have the canonical projection
$$
\pi:(D^2, S^1)\rightarrow(\mathbb{I}, 1) : z \rightarrow |z|^2.
$$ 

By definition a {\em multigraded polydisk}\index{multigraded polydisk}\index{polydisk} is
$\mathbb{D}^{2m} = D^2_1 \times \ldots \times D^2_m$.

Define the {\em standard torus} $\mathbb{T}^m=S^1_1 \times \ldots \times S^1_m$.

\begin{proposition}
There is a canonical action of the torus $\mathbb{T}^m$ on the polydisk $\mathbb{D}^{2m}$ with the orbit space 
$$
\mathbb{D}^{2m}/\mathbb T^m\simeq\mathbb{I}^m = \mathbb{I}^1_1 \times \ldots \times \mathbb{I}^1_m.
$$
\end{proposition}

Consider a simple polytope $P$. Let $\{F_1,\dots, F_m\}$ be the set of facets and $\{v_1,\dots, v_{f_0}\}$ --
the set of vertices. We have the face lattice $L(P)$ of $P$.

\noindent{\bf Construction (moment-angle complex of a simple polytope \cite{BP98,Bu-Pa15}):}
For $P={\rm pt}$ set $\mathcal{Z}_P={\rm pt}=\{0\}=\mathbb D^0$. Let $\dim P>0$. For any face $F \in L(P)$ set
\begin{gather*}
\mathcal{Z}_{P, F}=\{(z_1,\dots,z_m)\in \mathbb{D}^{2m}\colon z_i\in D^2_i\text{ if }F\subset F_i, z_i\in S^1_i\text{ if }F\not\subset F_i\};\\
\mathbb{I}_{P,F}=\{(y_1,\dots,y_m)\in \mathbb{I}^{m}\colon y_i\in \mathbb{I}^1_i\text{ if }F\subset F_i, y_i=1\text{ if }F\not\subset F_i\}.
\end{gather*}
\begin{proposition}
\begin{enumerate} 
\item $\mathcal{Z}_{P,F}\simeq \mathbb D^{2k}\times \mathbb T^{m-k}$, $\mathbb I_{P,F}\simeq \mathbb I^k$, where $k=n-\dim F$.
\item $\mathcal{Z}_{P,P}=\mathbb T^m$, $\mathcal{Z}_{P,\varnothing}=\mathbb D^{2m}$.
\item If $G_1\subset G_2$, then $\mathcal{Z}_{P,G_2}\subset \mathcal{Z}_{P,G_1}$, and $\mathbb I_{P,G_2}\subset\mathbb I_{P,G_1}$.
\item $\mathcal{Z}_{P,F}$ is invariant under the action of $\mathbb T^m$, and the mapping $\pi^m\colon\mathbb D^{2m}\to \mathbb I^m$ defines the homeomorphism $\mathcal{Z}_{P,F}/\mathbb T^m\simeq \mathbb I_{P,F}$. 
\end{enumerate}
\end{proposition}
The {\em moment-angle complex}\index{moment-angle complex} of a simple polytope $P$ is a subset in $\mathbb{D}^{2m}$ of the form
$$
\mathcal{Z}_{P}=\bigcup_{F\in L(P)\setminus\{\varnothing\}} \mathcal{Z}_{P,F}=\bigcup_{v-\text{ vertex}}\mathcal{Z}_{P,v}.
$$
The cube $\mathbb I^m$ has the canonical structure of a {\em cubical complex}\index{cubical complex}. It is a cellular complex with all cells being cubes with an appropriate boundary condition. 
The {\em cubical complex}\index{polytope!cubical complex} of a simple polytope $P$ is a cubical subcomplex  in $\mathbb {I}^m$ of the form
$$
\mathbb{I}_P=\bigcup_{F\in L(P)\setminus\{\varnothing\}} \mathbb I_{P,F}=\bigcup_{v-\text{ vertex}}\mathbb{I}_{P,v}.
$$
  
From the construction of the space $\mathcal{Z}_P$ we obtain.
\begin{proposition}
\begin{enumerate}
\item The subset $\mathcal{Z}_P \subset \mathbb{D}^{2m}$ is $\mathbb{T}^m$ -- invariant; hence there is the canonical action of $\mathbb{T}^m$ on $\mathcal{Z}_P$.
\item The mapping $\pi^m$ defines the homeomorphism $\mathcal{Z}_P/\mathbb T^m\simeq \mathbb I_P$.
\item For $P_1\times P_2$ we have $\mathcal{Z}_{P_1}\times \mathcal{Z}_{P_2}$.
\end{enumerate}
\end{proposition}

\subsection{Admissible mappings}
\begin{definition}
Let $P_1$, $P_2$ be two simple polytopes. A mapping of sets of facets $\varphi\colon\mathcal{F}_{P_1}\to \mathcal{F}_{P_2}$ we call {\em admissible}\index{admissible mapping}, if $\varphi(F_{i_1})\cap\dots\cap\varphi(F_{i_k})\ne\varnothing$ for any collection $F_{i_1},\dots,F_{i_k}\in\mathcal{F}_{P_1}$ with $F_{i_1}\cap\dots\cap F_{i_k}\ne\varnothing$.
\end{definition}
Any admissible mapping $\varphi\colon\mathcal{F}_{P_1}\to \mathcal{F}_{P_2}$ induces the mapping $\varphi\colon L(P_1)\to L(P_2)$ by the rule: $\varphi(P_1)=P_2$,  $\varphi(F_{i_1}\cap\dots\cap F_{i_k})=\varphi(F_{i_1})\cap\dots\cap \varphi(F_{i_k})$. This mapping preserves the inclusion relation.
\begin{proposition}
Any admissible mapping $\varphi\colon\mathcal{F}_{P_1}\to \mathcal{F}_{P_2}$ induces the mapping of triples $\colon(\mathbb{D}^{2m_1},\mathcal{Z}_{P_1},\mathbb T^{m_1})\to(\mathbb D^{2m_2},\mathcal{Z}_{P_2},\mathbb T^{m_2})$ and the mapping $\mathbb I_{P_1}\to\mathbb I_{P_2}$, which we will denote by the same letter $\widehat\varphi$: 
$$
\widehat{\varphi}(x_1,\dots,x_{m_1})=(y_1,\dots,y_{m_2}),\,
y_j=\begin{cases}1,&\text{ if }\varphi^{-1}(j)=\varnothing,\\
\prod\limits_{i\in\varphi^{-1}(j)}x_i,&\text{else}.
\end{cases}
$$
In particular, we have the homomorphism of tori $\mathbb T^{m_1}\to\mathbb T^{m_2}$ such that the mapping $\mathcal{Z}_{P_1}\to\mathcal{Z}_{P_2}$ is equivariant.

We have the commutative diagram 
$$
\begin{CD}
  \mathcal Z_{P_1} @>\widehat{\varphi}>>\mathcal{Z}_{P_2}\\
  @VV\pi^mV\hspace{-0.2em} @VV\pi^m V @.\\
  \mathbb {I}_{P_1} @> \widehat\varphi>> \mathbb {I}_{P_2}
\end{CD}\eqno 
$$
\end{proposition}
\begin{example}
Let $P_1=\mathbb I^2$ and $P_2=\mathbb I$. Then any admissible mapping\linebreak $\mathcal{F}_{P_1}\to\mathcal{F}_{P_2}$ is a constant mapping. Indeed, there are two facets $G_1$ and $G_2$ in $\mathbb I$,  which do not intersect. $\mathbb I^2$ has four facets $F_1,F_2,F_3,F_4$, such that $F_1\cap F_2$, $F_2\cap F_3$, $F_3\cap F_4$, and $F_4\cap F_1$ are vertices. Let $\varphi(F_1)=G_i$. Then $\varphi(F_2)=G_i$, since $\varphi(F_1)\cap \varphi(F_2)=\varnothing$. By the same reason we have $\varphi(F_3)=\varphi(F_4)=G_i$. Without loss of generality let $i=1$ and $G_1=\{0\}$. Then the mapping of the moment-angle complexes
\begin{gather*}
\mathcal{Z}_{\mathbb I^2}=\{(z_1,z_2,z_3,z_4)\in\mathbb D^8\colon |z_1|=1\text{ or }|z_3|=1, \text{ and }
|z_2|=1 \text{ or }|z_4|=1\}=\\
=(S^1_1\times D^2_3\cup D^2_1\times S^1_3)\times(S^1_2\times D^2_4\cup  D^2_2\times S^1_4)\cong S^3\times S^3,\\
\mathcal{Z}_{I^1}=\{(w_1,w_2)\in\mathbb D^4\colon |w_1|=1\text{ or }|w_2|=1\}=\\
=(S^1_1\times D^2_2)\cup(D^2_1\times S^1_2)\cong S^3
\end{gather*}
is 
$$
\widehat\varphi\colon \mathcal{Z}_{\mathbb I^2}\to\mathcal{Z}_{\mathbb I^1},\qquad \widehat\varphi(z_1,z_2,z_3,z_4)=(z_1 z_2 z_3 z_4, 1).
$$
\end{example}
\begin{example}
Let $P_1=\mathbb I^2$, $P_2=\Delta^2$. Then any mapping $\varphi\colon\mathcal{F}_{P_1}\to\mathcal{F}_{P_2}$ is admissible. 
Let $\mathcal{F}_{P_1}=\{F_1,F_2,F_3,F_3\}$ as in previous example, and \\$\mathcal{F}_{P_2}=\{G_1,G_2,G_3\}$. The admissible mapping  
$$
\varphi(F_1)=G_1,\quad\varphi(F_2)=G_2,\quad \varphi(F_3)=\varphi(F_4)=G_3
$$ 
induces the mapping of face lattices 
\begin{gather*}
\varphi(\mathbb I^2)=\Delta^2,\quad \varphi(\varnothing)=\varnothing,\\
\varphi(F_1\cap F_2)=G_1\cap G_2,\quad \varphi(F_2\cap F_3)=G_2\cap G_3, \\
\varphi(F_3\cap F_4)=G_3,\quad\varphi(F_4\cap F_1)=G_3\cap G_1.
\end{gather*}

The mapping of the moment-angle complexes
\begin{gather*}
\mathcal{Z}_{\mathbb I^2}=\{(z_1,z_2,z_3,z_4)\in\mathbb D^8\colon |z_1|=1\text{ or }|z_3|=1, \text{ and }
|z_2|=1 \text{ or }|z_4|=1\}=\\
=(S^1_1\times D^2_3\cup D^2_1\times S^1_3)\times(S^1_2\times D^2_4\cup D^2_2\times  S^1_4)\cong S^3\times S^3,\\
\mathcal{Z}_{\Delta^2}=\{(w_1,w_2,w_3)\in\mathbb D^6\colon |w_1|=1,\text{ or }|w_2|=1,\text{ or }|w_3|=1\}=\\
=(S^1_1\times D^2_2\times D^2_3)\cup(D^2_1\times S^1_2\times D^2_3)\cup(D^2_1\times  D^2_2\times S^1_3)\cong S^5
\end{gather*}
is 
$$
\widehat\varphi\colon \mathcal{Z}_{\mathbb I^2}\to\mathcal{Z}_{\Delta^2},\qquad \varphi(z_1,z_2,z_3,z_4)=(z_1,z_2,z_3\cdot z_4).
$$
\end{example}

\subsection{Barycentric embedding and cubical subdivision of a simple polytope}
\noindent {\bf Construction (barycentric embedding of a simple polytope):} \index{barycentric embedding}\index{polytope!barycentric embedding}Let $P$ be a simple $n$-polytope with facets $F_1, \dots, F_m$.
For each face $G\subset P$ define a point $\boldsymbol{x}_G$ as a barycenter of it's vertices. We have $\boldsymbol{x}_G\in{\rm relint\,}G$. The points $\boldsymbol{x}_G$, $G\in L(P)\setminus\{\varnothing\}$, define a barycentric simplicial subdivision $\Delta(P)$ of the polytope $P$. The simplices of $\Delta(P)$ correspond to flags of faces $F^{a_1}\subset F^{a_2}\subset\dots\subset F^{a_k}$, $\dim F^i=i$:
$$
\Delta_{F^{a_1}\subset F^{a_2}\subset\dots\subset F^{a_k}}={\rm conv}\{\boldsymbol{x}_{F^{a_1}},\boldsymbol{x}_{F^{a_2}},\dots, \boldsymbol{x}_{F^{a_k}}\},
$$
The maximal simplices are $\Delta_{\boldsymbol{v}\subset F^1\subset F^2\subset\dots\subset F^{n-1}\subset P}$, where 
$\boldsymbol{v}$ is a vertex.
For any point $\boldsymbol{x}\in P$ the minimal simplex $\Delta(\boldsymbol{x})$ containing $\boldsymbol{x}$ can be found by the following rule. Let $G(\boldsymbol{x})=\bigcap\limits_{F_i\ni \boldsymbol{x}}F_i$. If $\boldsymbol{x}=\boldsymbol{x}_G$, then $\Delta(\boldsymbol{x})=\Delta_{G}$. Else take a ray starting in $\boldsymbol{x}_G$, passing through $\boldsymbol{x}$ and intersecting $\partial G$ in $\boldsymbol{x}_1$. Iterating the argument we obtain either $\boldsymbol{x}_1=\boldsymbol{x}_{G_1}$, and $\Delta(\boldsymbol{x})=\Delta_{G_1\subset G}$, or a new point $\boldsymbol{x}_2$. In the end we will stop when 
$\boldsymbol{x}_l=\boldsymbol{x}_{G_l}$, and $\Delta(\boldsymbol{x})=\Delta_{G_l\subset\dots\subset G_1\subset G}$.   

Define a piecewise-linear mapping $b_P\colon P\to \mathbb{I}^m$ by the rule 
$$
\boldsymbol{x}_G\to\widehat{\varepsilon}(G)=(\varepsilon_1,\dots,\varepsilon_m)\in\mathbb{I}^m, \text{where }\varepsilon_i=\begin{cases}0,&\text{if }G\subset F_i,\\1,&\text{if }G\not\subset F_i\end{cases},
$$
on the vertices of $\Delta(P)$, and for any simplex continue the mapping to the cube $\mathbb{I}^m$ via barycentric coordinates. In particular, $b_P(\boldsymbol{x}_P)=(1,1,\dots,1)$, and $b_P(\boldsymbol{x}_{\boldsymbol{v}})$ is a point with $n$ zero coordinates.
\begin{theorem}
The mapping $b_P$ defines a homeomorphism $P\simeq \mathbb I_P\subset \mathbb I^m$. 
\end{theorem}
\begin{proof}
Let $\boldsymbol{x}\in P$, and $\Delta(\boldsymbol{x})=\Delta_{G_1\subset\dots\subset G_r}$. 

We have $\boldsymbol{x}=t_1\boldsymbol{x}_{G_1}+\dots+t_r\boldsymbol{x}_{G_r}$, where $t_i>0$, and $t_1+\dots+t_r=1$. The coordinates of the vector  $b_P(\boldsymbol{x})=t_1\widehat{\varepsilon}(G_1)+\dots+t_r\widehat{\varepsilon}(G_r)=(x_1,\dots,x_m)$ belong to the interval $[0,1]$. Arrange them ascending:  
\begin{multline*}
0=x_{i_1}=\dots=x_{i_{p_1}}<x_{i_{p_1+1}}=\dots=x_{i_{p_1+p_2}}<\dots<\\
<x_{i_{p_1+\dots+p_r+1}}=\dots=x_{i_m}=1.
\end{multline*}
Then 
$$
G_1=F_{i_1}\cap\dots\cap F_{i_{p_1+\dots+p_r}},G_2=F_{i_1}\cap\dots\cap F_{i_{p_1+\dots+p_{r-1}}},\dots,G_r=F_{i_1}\cap\dots\cap F_{i_{p_1}},
$$
and
$$
t_1=1-x_{i_{p_1+\dots+p_r}},t_2=x_{i_{p_1+\dots+p_r}}-x_{i_{p_1+\dots+p_{r-1}}}, \dots,t_r=x_{i_{p_1+p_2}}.
$$
Thus the mapping $b_P$ is an embedding. Since $P$ is compact and $\mathbb I^m$ is Hausdorff, we have the homeomorphism $P\simeq b_P(P)$. In the construction above we have $x_{i_j}\ne 1$ only if $F_{i_j}\supset G_1$; hence $b_P(\boldsymbol{x})\in \mathbb I_{P,G_1}$, and $b_P(P)\subset \mathbb I_P$. On the other hand, the above formulas imply that  $\mathbb I_P\subset b_P(P)$. This finishes the proof.
\end{proof}
\begin{corollary}
The homeomorphism $b_P\colon P\to \mathbb I_P\simeq \mathcal{Z}_P/\mathbb T^m$ defines a mapping $\pi_P\colon\mathcal{Z}_P\to P$ such that the following diagram is commutative:

$$
\begin{CD}
  \mathcal Z_P @>>>\mathbb{D}^{2m}\\
  @VV\pi_PV\hspace{-0.2em} @VV V @.\\
  P @> b_P>> \mathbb{I}^m
\end{CD}\eqno 
$$
\end{corollary}
\begin{corollary}
Any admissible mapping $\varphi\colon\mathcal{F}_{P_1}\to \mathcal{F}_{P_2}$ induces the mapping of polytopes $\widehat\varphi \colon P_1 \to P_2$ such that the following diagram is commutative: 
$$
\begin{CD}
  \mathcal Z_{P_1} @>\widehat{\varphi}>>\mathcal{Z}_{P_2}\\
  @VV\pi_{P_1}V\hspace{-0.2em} @VV\pi_{P_2} V @.\\
  P_1 @> \widehat\varphi>> P_2
\end{CD}\eqno 
$$
\end{corollary}

\noindent {\bf Construction (canonical section):}\index{moment-angle complex!canonical section} The mapping
$$
s : \mathbb I \rightarrow D^2 : s(y)=\sqrt{y}
$$
induces the section $s^m\colon\mathbb I_P\to \mathcal{Z}_P$. Together with the homeomorphism $P\simeq \mathbb I_P$ this gives the {\em canonical section} $s_P= s^m\circ b_P: P \rightarrow \mathcal{Z}_P$, such that $\pi_P\circ s_P={\rm id}$.

\noindent {\bf Construction (cubical subdivision):}\index{cubical subdivision}\index{polytope!cubical subdivision}
The space $\mathbb I_P$ has the canonical partition into cubes $\mathbb I_{P,v}$, one for each vertex $v\in P^n$. The homeomorphism\linebreak  $\mathbb I_P={\rm Im\,}b_P(P)\simeq P$ gives the {\em cubical subdivision} of the polytope $P$.
\begin{example}
For $P = I$ we have an embedding $I \subset I^2$.
\begin{figure}
\begin{center}
\includegraphics[scale=0.3]{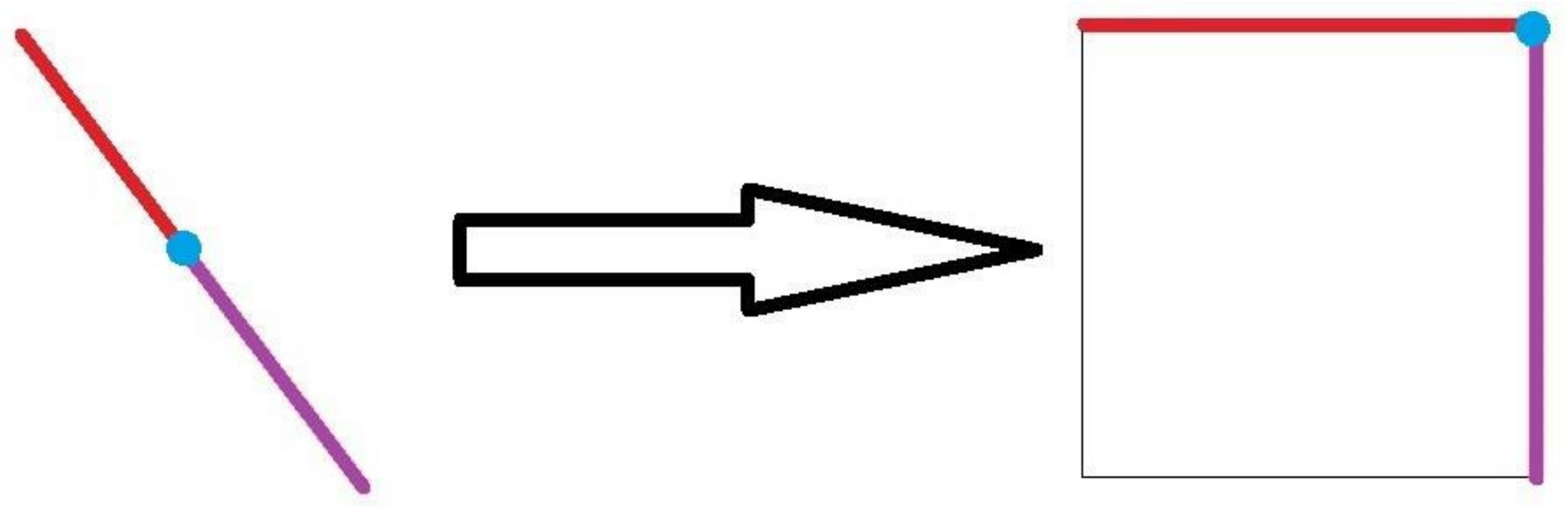}
\end{center}
\caption{Barycentric embedding and cubical subdivision of the interval}
\end{figure}
\end{example}

\begin{example}
For $P = \Delta^2$ we have an embedding $\Delta^2 \subset I^3$
\begin{figure}
\begin{center}
\includegraphics[scale=0.32]{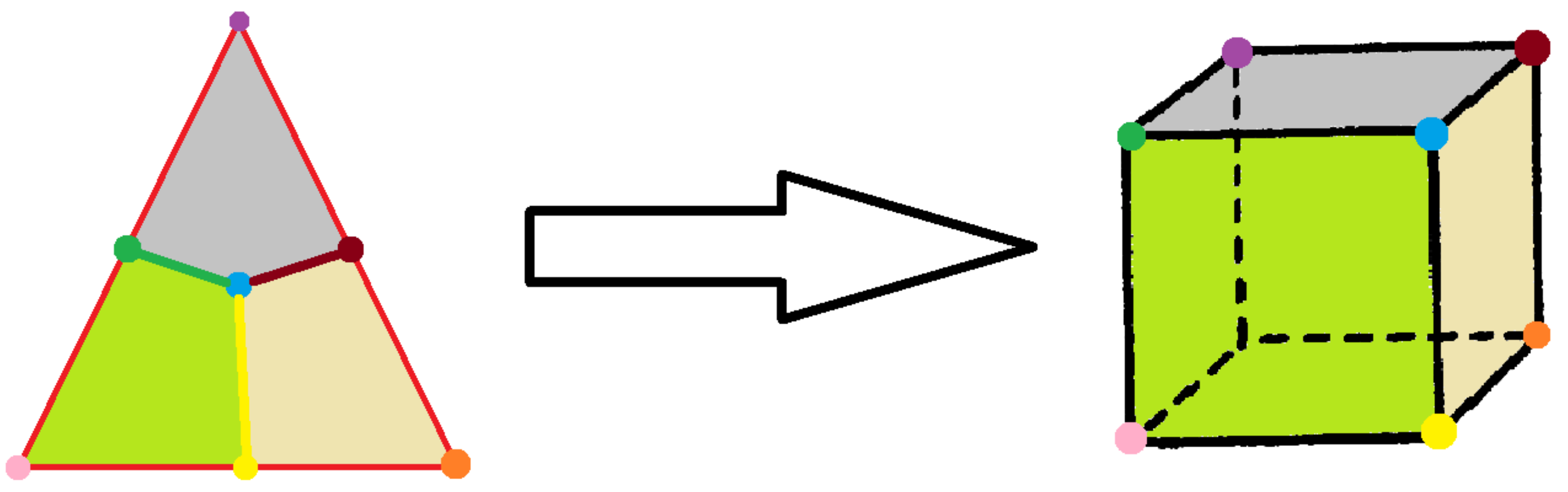}
\end{center}
\caption{Barycentric embedding and cubical  subdivision of the triangle}
\end{figure}
\end{example}

\noindent{\bf Construction (product over space):}\index{product over space}
Let $f \colon X \to Z$ and $g \colon Y \to Z$ be maps of topological spaces. The product $X\times_Z Y$ \index{product over space}over space $Z$ is described by the general pullback diagram:
\[ \begin{CD}
X\times_Z Y @ >>> X\\
@ VVV @ VVfV\\
Y @ >g>> Z
\end{CD} \]
where $X\times_Z Y = \big\{ (x,y) \in X\times Y\;:\; f(x) =g(y) \big\}$.
\begin{proposition}
We have $\mathcal{Z}_P=\mathbb {D}^{2m}\times_{\mathbb{I}^m}P$.\index{moment-angle complex!product over cube}
\end{proposition}

\subsection{Pair of spaces in the power of a simple polytope}
\noindent{\bf Construction (raising to the power of a simple polytope):} \index{raising to the power of a simple polytope}\index{polytope!raising to the power}Let $P$ be a simple polytope $P$ with the face lattice $L(P)$ and the set of facets $\{F_1,\dots, F_m\}$. For $m$ pairs of topological spaces $\{(X_i,W_i),i=1,\dots,m\}$ set $(\boldsymbol{X},\boldsymbol{W})=\{(X_i,W_i),i=1,\dots,m\}$. For a face $F\in L(P)\setminus\{\varnothing\}$ define 
$$
(\boldsymbol{X},\boldsymbol{W})^P_F=\{(y_1,\dots,y_m)\in X_1\times\dots\times X_m\colon y_i\in X_i\text{ if }F\in F_i, y_i\in W_i\text{ if }F\notin F_i\}.
$$

The set of pairs $(\boldsymbol{X},\boldsymbol{W})$ in degree of a simple polytope $P$ is
$$
(\boldsymbol{X},\boldsymbol{W})^P=\bigcup_{F\in L(P)\setminus\{\varnothing\}}(\boldsymbol{X},\boldsymbol{W})^P_F
$$
\begin{example}
\begin{enumerate}
\item Let $W_i=X_i$ for all $i$. Then  $(\boldsymbol{X},\boldsymbol{W})^P=X_1\times\dots\times X_m$ for any $P$.
\item Let $W_i=*_i$ -- a fixed point in $X_i$, $i=1,2$, and $P=I$. Then  $(\boldsymbol{X},\boldsymbol{W})^I=X_1 \vee X_2$ is the wedge of  the spaces $X_1$ and $X_2$.
\end{enumerate}
\end{example}

\noindent{\bf Construction (pair of spaces in the power of a simple polytope):} \index{pair of spaces in the power of a simple polytope}
In the case $W_i=W$, $X_i=X$,  $i=1,\dots,m$, the space $(\boldsymbol{X},\boldsymbol{W})^P$ is called a pair of spaces $(X,W)$ in the power of a simple polytope $P$ and is denoted by $(X,W)^P$.

\begin{example}
The space $(D^2,S^1)^P$ is the moment-angle complex $\mathcal{Z}_P$ of the polytope $P$ (see Subsection 
\ref{moment-angle}).
\end{example}
\begin{example}
The space $(\mathbb I,1)^P$ is the image $\mathbb I_P=b_P(P)$ of the barycentric embedding of the polytope $P$ (see Subsection 
\ref{moment-angle}).
\end{example}

\noindent {\bf{Exercise:}} Describe the space $(X,W)^P$, where $P$ is a $5$-gon.

Let us formulate properties of the construction. The proof we leave as an exercise.
\begin{proposition}
\begin{enumerate}
\item Let $P_1$ and $P_2$ be simple polytopes. Then
$$
(X,W)^{{P_1 \times P_2}} = (X,W)^{P_1} \times (X,W)^{P_2}
$$
\item Let $\{v_1,\dots, v_{f_0}\}$ be the set of vertices of $P$. There is a homeomorphism
$$
(X,W)^P \cong \bigcup_{k = 1}^{f_0} (X,W)^P_{v_k}
$$
\item Any mapping $f \colon (X_1,W_1) \to (X_2,W_2)$ gives the commutative diagram
\[
\begin{CD}
  (X_1,W_1)^{P} @>{f^P}>>(X_2,W_2)^{P}\\
  @V{\cap}VV @VV{\cap }V @.\\
  (X_1,X_1)^P=X_1^m @>f^m>> X_2^{m}=(X_2,X_2)^P
\end{CD}
\]
\item We have $id^P=id$. \\
For $f_1 \colon (X_1,W_1) \to (X_2,W_2)$, $f_2 \colon (X_2,W_2) \to (X_1,W_1)$ we have 
$$(f_2\circ f_1)^P=f_2^P\circ f_1^P. $$ 
\end{enumerate}
\end{proposition}

\subsection{Davis-Januszkiewicz' construction}
\noindent {\bf Davis-Januszkiewicz' construction \cite{DJ91}:}\index{moment-angle complex!Davis-Januszkiewicz' construction} \index{moment-angle manifold!Davis-Januszkiewicz' construction} For $\boldsymbol{x}\in P$ we have the face\linebreak 
 $G(x)=\bigcap\limits_{F_i\supset \boldsymbol{x}}F_i\in L(P)$.  For a face $G\in L(P)$ define the subgroup $\mathbb T^G\subset\mathbb T^m$ as 
$$
\mathbb T^G=(S^1,1)_G^P=\{(t_1,\dots, t_m)\in \mathbb T^m\colon t_j=1,\text{ if }F_j\not\ni G\}
$$
Set
$$
\widetilde{\mathcal{Z}_P}=P \times \mathbb{T}^m / \sim,
$$
where $(\boldsymbol{x}_1, \boldsymbol{t}_1) \sim (\boldsymbol{x}_2, \boldsymbol{t}_2) \Leftrightarrow \boldsymbol{x}_1=\boldsymbol{x}_2=\boldsymbol{x}$, and $\boldsymbol{t}_1\boldsymbol{t}_2^{-1}\in T^{G(\boldsymbol{x})}$.

There is a canonical action of $\mathbb T^m$ on $\widetilde{\mathcal{Z}_P}$ induced by the action of $\mathbb T^m$ on the second factor.

\begin{theorem}\label{3mameq}
The canonical section $s_P\colon P\to \mathcal{Z}_P$ induces the $\mathbb T^m$-equivariant homeomorphism 
$$
\widetilde{\mathcal{Z}_P} \longrightarrow\mathcal{Z}_P
$$
defined by the formula $(\boldsymbol{x},\boldsymbol{t})\to \boldsymbol{t}s_P(\boldsymbol{x})$.
\end{theorem}

\subsection{Moment-angle manifold of a simple polytope}\label{facesubmzp}
\noindent{\bf Construction (moment-angle manifold of a simple polytope \cite{BPR07,Bu-Pa15}):} \index{moment-angle manifold}Take a simple polytope
$$
P=\{\boldsymbol{x}\in \mathbb R^n\colon \boldsymbol{a}_i\boldsymbol{x}+b_i\geqslant 0, i=1,\dots,m\}.
$$
We have ${\rm rank\,} A=n$, where $A$ is the $m\times n$-matrix with rows $\boldsymbol{a}_i$. 
Then there is an embedding
\[ j_P \colon P \longrightarrow \mathbb{R}^m_{\geq} \,:\,j_P(\boldsymbol{x})=(y_1,\ldots,y_m), \]
where $y_i=\boldsymbol{a}_i\boldsymbol{x} + b_i$, and we will consider $P$ as the
subset in $\mathbb{R}^m_{\geqslant}$.\\
A {\em moment-angle manifold} $\widehat{\mathcal{Z}_P}$ is the subset in $\mathbb{C}^m$ defined as $\rho^{-1}\circ j_P(P)$, where $\rho(z_1, \ldots, z_m) = \big(|z_1|^2, \ldots, |z_m|^2\big)$. The action of $\mathbb T^m$ on $\mathbb {C}^m$ induces the action of $\mathbb T^m$ on $\widehat{\mathcal{Z}_P}$.

For the embeddings $j_{\mathcal{Z}}\colon\widehat{\mathcal{Z}_P}\subset \mathbb C^m$ and $j_P\colon P\subset\mathbb R^m_{\geqslant}$ we have the commutative diagram:
\[ \begin{CD}
\widehat{\mathcal{Z}_P} @ >{j_{\mathcal{Z}}}>> \mathbb{C}^m\\
@ V{\rho_P}VV @ VV{\rho}V\\
P @>{j_P}>> \mathbb{R}_{\geqslant}^m
\end{CD} \]

\begin{proposition}\label{0zp}
We have $\widehat{\mathcal{Z}_P}\subset\mathbb C^m\setminus\{0\}$.
\end{proposition}
\begin{proof}
If $0\in\widehat{\mathcal{Z}_P}$, then $0=\rho(0)\in j_P(P)$. This corresponds to a point $\boldsymbol{x}\in P$ such that $\boldsymbol{a}_i\boldsymbol{x}+b_i=0$ for all $i$. This is impossible, since any point of a simple $n$-polytope lies in at most $n$ facets. 
\end{proof}

\begin{definition}
For the set of vectors $(\boldsymbol{x}_1,\dots,\boldsymbol{x}_m)$ spanning $\mathbb R^n$, the set of vectors $(\boldsymbol{y}_1,\dots,\boldsymbol{y}_m)$ spanning $\mathbb R^{m-n}$ is called {\em Gale dual}\index{Gale duality}, if for the matrices $X$ and $Y$ with column vectors $\boldsymbol{x}_i$ and $\boldsymbol{y}_j$ we have $XY^T=0$.
\end{definition}

Take an $((m-n)\times m)$-matrix $C$ such that $CA=0$ and ${\rm rank\,}C=m-n$. Then the vectors $\boldsymbol{a}_i$ and the column vectors $\boldsymbol{c}_i$ of $C$ are Gale dual to each other. Let $\boldsymbol{c}_i=(c_{1,i},\dots,c_{m-n,i})$.
\begin{proposition}
We have 
$$
\widehat{\mathcal{Z}_P}=\{\boldsymbol{z}\in\mathbb C^m\colon c_{i,1}|z_1|^2+\dots+c_{i,m}|z_m|^2=c_i\},
$$
where $c_i=c_{i,1}b_1+\dots+c_{i,m}b_m$.
\end{proposition}
Denote $\Phi_i(\boldsymbol{z})=c_{i,1}|z_1|^2+\dots+c_{i,m}|z_m|^2-c_i$.
Consider the mapping
$$
\Phi : \mathbb{C}^m \to \mathbb R^{m-n} : \; \Phi(\boldsymbol{z})=(\Phi_1(\boldsymbol{z}),\ldots,\Phi_{m-n}(\boldsymbol{z})).
$$
It is the $\mathbb T^m$-equivariant quadratic mapping with respect to the trivial
action of $\mathbb T^m$ on $\mathbb R^{m-n}$.

\begin{proposition}\label{ZPeq}
\begin{enumerate}
\item $\widehat{\mathcal{Z}_P}$ is a {\bf complete intersection} of {\bf
    real quadratic hypersurfaces} in $\mathbb R^{2m} \cong \mathbb C^m$:
$$
\mathcal{F}_k=\{ \boldsymbol{z}\in \mathbb C^m : \Phi_k(\boldsymbol{z})=0 \}, \; k=1, \ldots, m-n.
$$

\item There is a canonical {\bf trivialisation} of the normal bundle of the $\mathbb T^m$-equivariant embedding $\widehat{\mathcal{Z}_P} \subset \mathbb{C}^m$, that is $\widehat{\mathcal{Z}_P}$ has the canonical structure of a {\bf framed} manifold.
\end{enumerate}
\end{proposition}

\begin{proof}
We have $\widehat{\mathcal{Z}_P} = \Phi^{-1}(0)$, where $\Phi \colon \mathbb{R}^{2m} \cong \mathbb{C}^m \to \mathbb{R}^{m-n}$.
Next step is an exercise.

\noindent {\bf Exercise:} Differential
$d\Phi\large|_y \colon \mathbb{R}^{2m} \to \mathbb{R}^{m-n}$
is an epimorphism for any point of $y\in \Phi^{-1}(0)$.
\end{proof}
\begin{corollary} For an appropriate choice of $C$
$$
\widehat{\mathcal{Z}_P}=\bigcap_{k=1}^{m-n}\mathcal{F}_k
$$
where any surface $\mathcal{F}_k \subset \mathbb{R}^{2m}$ is a $(2m-1)$-dimensional smooth $\mathbb{T}^m$-manifold.
\end{corollary}
\begin{proof}
We just need to find such $C$ that the vector $Cb$ has all coordinates nonzero. For any $C$ above $Cb$ has a nonzero coordinate since $0\notin \widehat{\mathcal{Z}_P}$ by Proposition \ref{0zp}. Then we can obtain from it the matrix we need by elementary transformations of rows.
\end{proof}
\noindent {\bf{Exercise:}} Describe the orbit space $\mathcal{F}_k/\mathbb{T}^m$.

{\noindent}{\bf Construction (canonical section):} The projection $\rho$ has the canonical section 
$$
s\colon\mathbb R^m_{\geqslant}\to \mathbb{C}^m,\qquad s(x_1,\dots,x_m)=(\sqrt{x_1},\dots,\sqrt{x_m}),
$$
which gives a {\em canonical section}\index{moment-angle manifold!canonical section} $\widehat{s_P}\colon P\to \widehat{\mathcal{Z}_P}$ by the formula  $\widehat{s_P}=s\circ j_P$.

\begin{theorem}\label{thBPR}({\bf Smooth structure on the moment-angle complex}, \index{moment-angle complex!smooth structure}\cite{BPR07})
The section $\widehat{s_P}\colon P\to \widehat{\mathcal{Z}_P}$ induces the $\mathbb T^m$-equivariant homeomorphism 
$$
\widetilde{\mathcal{Z}_P} \longrightarrow\widehat{\mathcal{Z}_P}
$$
defined by the formula $(\boldsymbol{x},\boldsymbol{t})\to \boldsymbol{t}\widehat{s_P}(\boldsymbol{x})$.

Together with the $\mathbb T^m$-equivariant homeomorphism $\widetilde{\mathcal{Z}_P}\to\mathcal{Z}_P$ this gives a smooth structure on the moment-angle complex $\mathcal{Z}_P$. 
\end{theorem}

Thus in what follows we identify $\widehat{\mathcal{Z}_P}$ and $\mathcal{Z}_P$.

\noindent {\bf{Exercise:}} Describe the manifold $\mathcal{Z}_P$ for
$P = \{ \boldsymbol{x}\in \mathbb{R}^{2} : A\boldsymbol{x}+b\geqslant 0 \}$, where
\begin{align*}
1.\quad A^\top &=
\begin{pmatrix}
1&0&-1&1\\
0&1&0&-1
\end{pmatrix},
&b^\top &= (0,0,1,1)\\[10pt]
2.\quad A^\top &=
\begin{pmatrix}
1&0&-1&1&-1\\
0&1&0&-1&-1
\end{pmatrix},
&b^\top &= (0,0,1,1,2)
\end{align*}
\noindent {\bf{Exercise:}} 
Let $G\subset P$ be a face of codimension~$k$ in a simple
$n$-polytope~$P$, let $\mathcal{Z}_P$ be the corresponding moment-angle
manifold with the quotient projection $p\colon\mathcal{Z}_P\to P$. Show that
$p^{-1}(G)$ is a smooth submanifold of $\mathcal{Z}_P$ of codimension $2k$.
Furthermore, $p^{-1}(G)$ is diffeomorphic to $\mathcal Z_G\times
T^\ell$, where $\mathcal Z_G$ is the moment-angle manifold
corresponding to~$G$ and $\ell$ is the number of facets of $P$ not
intersecting~$G$.

\subsection{Mappings of the moment-angle manifold into spheres}\label{SSMS}
For any set $\omega=\{j_1,\dots, j_k\}\subset \{1,\dots,m\}$ define \index{moment-angle manifold!mappings into spheres}
\begin{gather*}
\mathbb C^{m-k}_\omega=\{(z_1,\dots,z_m)\in \mathbb C^m\colon z_j=0, j\in \omega\};\\
S^{2m-2k-1}_\omega=\{(z_1,\dots,z_m)\in \mathbb C^m\colon z_j=0, j\in \omega,\sum_{j\notin\omega}|z_j|^2=1\};\\
\mathbb R^{m-k}_\omega=\{(y_1,\dots,y_m)\in \mathbb R^m\colon y_j=0, j\in \omega\}.\\
\end{gather*}
\noindent{\bf Exercise:} For $k\geqslant 1$ the sphere $S^{2k-1}_{[m]\setminus \omega}$ is a deformation retract of\linebreak $S^{2m-1}\setminus S^{2m-2k-1}_ \omega$.
\begin{proposition}
\begin{enumerate}
\item The embedding $\mathcal{Z}_P\subset\mathbb C^m$ induces the embedding $\mathcal{Z}_P\subset S^{2m-1}$ via projection $\mathbb C^m\setminus\{0\}\to S^{2m-1}$.
\item For any set $\omega$, $|\omega|=k$, such that $\bigcap \limits_{j\in \omega}F_j=\varnothing$ the image of the embedding $\mathcal{Z}_P\subset S^{2m-1}$ lies in $S^{2m-1}\setminus S^{2m-2k-1}_{\omega}$; hence the embedding is homotopic to the mapping $\varphi_\omega\colon\mathcal{Z}_P\to S^{2k-1}_{[m]\setminus \omega}$, induced by the projection $\mathbb C^m\to\mathbb C^k_{[m]\setminus \omega}$.
\end{enumerate}
\end{proposition}
\begin{proof}
(1) follows from Proposition \ref{0zp}.

(2) follows from the fact that if   $\bigcap \limits_{j\in \omega}F_j=\varnothing$, then there is no $\boldsymbol{x}\in P$ such that $\boldsymbol{a}_j\boldsymbol{x}+b_j=0$ for all $j\in \omega$.
\end{proof}

We have the commutative diagram 
$$
\begin{CD}
\mathcal{Z}_P @ >>> \mathbb{C}^m\setminus \mathbb C^{m-k}_\omega@>\xi_\omega>>S^{2k-1}_{[m]\setminus \omega}\subset \mathbb C^{k}_{[m]\setminus \omega}\\
@ VVV @ VV{\rho}V@VVV\\
P @>{A\boldsymbol{x}+b}>> \mathbb{R}_{\geqslant}^m\setminus\mathbb R^{m-k}_\omega@>\pi_\omega>>{}\Delta^{k-1}\subset\mathbb R^k_{\geqslant}
\end{CD} 
$$
where
\begin{align*}
\xi_\omega(z_1,\dots,z_m)=\frac{\boldsymbol{z}_\omega}{|\boldsymbol{z}_{\omega}|}, \quad\boldsymbol{z}_\omega=(z_{j_1},\dots,z_{j_k}),\quad |\boldsymbol{z}_\omega|=\sqrt{|z_{j_1}|^2+\dots+|z_{j_k}|^2}.\\
\pi_{\omega}(y_1,\dots,y_m)=\frac{\boldsymbol{y}_{\omega}}{d_{\omega}},\quad \boldsymbol{y}_{\omega}=(y_{j_1},\dots,y_{j_k}),  \quad d_{\omega}=|y_{j_1}|+\dots+|y_{j_k}|. 
\end{align*}
\begin{example} 
For any pair of facets $F_i, F_j$, such that $F_i\cap F_j=\varnothing$, there is a mapping 
$\mathcal{Z}_P\to S^3_{[m]\setminus \{i,j\}}$.
\end{example}
\begin{definition}
The class $a\in H^k(X,\mathbb Z)$ is called {\em cospherical}\index{cospherical class} if there is a mapping $\varphi\colon X\to S^k$ such that $\varphi^*\left(\left[S^k\right]\right)=a$. 
\end{definition}
\begin{corollary}\label{cosph} 
For each $\omega\subset[m]$, $|\omega|=k$, such that $\bigcap\limits_{i\in \omega}F_i=\varnothing$ we have the cospherical class $\varphi_\omega^*\left(\left[S^{2k-1}_{[m]\setminus\omega}\right]\right)$ in $H^{2k-1}(\mathcal{Z}_P)$. 
\end{corollary}
\subsection{Projective moment-angle manifold}
\noindent{\bf Construction (projective moment-angle manifold):} \index{projective moment-angle manifold}\index{moment-angle manifold!projective}Let $S_\Delta^1$ be the diagonal subgroup in $\mathbb{T}^m$. We have the {\em free action} of $S_\Delta^1$ on $\mathcal{Z}_P$ and therefore the {\em smooth} manifold
$$
\mathcal{P}\mathcal{Z}_P = \mathcal{Z}_P/S_\Delta^1
$$
is the projective version of the moment-angle manifold $\mathcal{Z}_P$.
\begin{definition}
For actions of the commutative group $G$ on spaces $X$ and $Y$ define:
$$
X\times_GY=X\times Y/\left\{gx,gy)\sim (x,y)\, \forall x\in X,y\in Y,g\in G\right\}. 
$$
\end{definition}
\begin{corollary}
For any simple polytope $P$ there exists the {\bf smooth} manifold
$$
W=\mathcal{Z}_{P}\times_{S_\Delta^1}D^2
$$
such that $\partial W =\mathcal{Z}_{P}$\index{moment-angle manifold!as a boundary}.
\end{corollary}
We have the fibration $W \longrightarrow \mathcal{P}\mathcal{Z}_P$ with the fibre $D^2$.\\
\noindent{\bf Exercise:} $P=\Delta^n \Longleftrightarrow \mathcal{Z}_P=S^{2n+1}
    \Longrightarrow \mathcal{P}S^{2n+1}=\mathbb{C}P^n$.

The constructions of the subsection \ref{SSMS} respect the diagonal action of $S^1$;  hence we obtain the following results.  

For $k\geqslant 1$ the set $\mathbb CP^{k-1}_{[m]\setminus \omega}$ is a deformation retract of $\mathbb CP^{m-1}\setminus \mathbb CP^{m-k-1}_ \omega$.
\begin{proposition}
\begin{enumerate}
\item The embedding $\mathcal{Z}_P\subset\mathbb C^m$ induces the embedding $\mathcal{P}\mathcal{Z}_P\subset \mathbb CP^{m-1}$.
\item For any set $\omega$, $|\omega|=k$, such that $\bigcap \limits_{j\in \omega}F_j=\varnothing$ the image of the embedding $\mathcal{P}\mathcal{Z}_P\subset\mathbb{C}P^m$ lies in $\mathbb CP^{m-1}\setminus \mathbb CP^{m-k-1}_{\omega}$; hence the embedding is homotopic to the mapping $\mathcal{P}\mathcal{Z}_P\to\mathbb{C}P^{k-1}_{[m]\setminus \omega}$, induced by the projection $\mathbb C^m\to\mathbb C^k_{[m]\setminus \omega}$.
\end{enumerate}
\end{proposition}

\newpage
\section{Lecture 5. Cohomology of a moment-angle manifold} 
When we deal with homology and cohomology, if it is not specified, the notation $H^*(X)$ and $H_*(X)$ means that we consider integer coefficients. 
\subsection{Cellular structure}
Define a cellular structure on $D^2$ consisting of $3$ cells: 
$$
p=\{1\},\quad  U=S^1\setminus\{1\},\quad V=D^2\setminus S^1. 
$$
Set on $D^2$ the standard orientation, with $(1,0)$ and $(0,1)$ being the positively oriented basis, and on $S^1$ the counterclockwise orientation induced from $D^2$. Then in the chain complex $C_*(D^2)$ we have  
$$
dp=0,\quad dU=0,\quad dV=U.
$$
The coboundary operator $\partial\colon C^i(X)\to C^{i+1}(X)$ is defined by the rule $\langle \partial \varphi,a\rangle=\langle \varphi,da\rangle$. For a cell $E$ let us denote by $E^*$ the cochain such that $\langle E^*, E'\rangle=\delta(E,E')$ for any cell $E'$. Denote $p^*=1$. Then the coboundary operator in $C^*(D^2)$ has the form
$$
\partial 1=0,\quad\partial U^*=V^*,\quad\partial V^*=0. 
$$
By definition the {\em multigraded polydisk} $\mathbb{D}^{2m}$ has the canonical multigraded cellular structure\index{polydisk!cellular structure} , which is a product of cellular structures of disks, with cells corresponding to pairs of sets $\sigma,\omega$, $\sigma\subset\omega\subset[m]=\{1,2,\dots,m\}$.
$$
C_{\sigma,\omega}=\tau_1\times\dots\times \tau_m,\tau_j=\begin{cases}V_j,&j\in\sigma,\\
U_j,&j\in\omega\setminus\sigma,\\
p_j,&j\in[m]\setminus\omega
\end{cases}, \; mdeg\,  C_{\sigma,\omega}=(-i, 2 \omega),
$$
where $i=|\omega\setminus\sigma|$.
Then the cellular chain complex $C_*(\mathbb D^{2m})$ is the tensor product of $m$ chain complexes $C_*(D^2_i)$, $i=1,\dots,m$. The boundary operator $d$ of the chain complex respects the multigraded structure and can be considered as a multigraded operator of $mdeg\; d=(-1,0)$. It can be calculated on the elements of the tensor product by the the Leibnitz rule
$$
d (a\times b)=(d a)\times b+ (-1)^{\dim a}a\times (db). 
$$
For cochains the $\times$-operation  $C^i(X)\times C^j(Y)\to C^{i+j}(X\times Y)$ is defined by the rule $\langle\varphi\times\psi,a\times b\rangle=\langle \varphi, a\rangle\langle \psi, b\rangle$. Then 
$$
\langle\psi_1\times\dots\times\psi_m, a_1\times\dots\times a_m\rangle=\langle\psi_1,a_1\rangle\dots\langle\psi_m,a_m\rangle.
$$
The basis in $C^*(\mathbb{D}^{2m})$ is formed by the cochains $C^*_{\sigma,\omega}=\tau_1^*\times \dots\times \tau _m^*$, where $C_{\sigma,\omega}=\tau_1\times\dots\times \tau_m$. 

The coboubdary operator $\partial$ is also multigraded. It has multidegree $mdeg\;\partial=(1,0)$. It can be calculated on the elements of the  tensor algebra $C^*(\mathbb D^{2m})$ by the rule $\partial (\varphi\times \psi)=(\partial \varphi)\times \psi+ (-1)^{\dim \varphi}\varphi\times (\partial \psi)$.

\begin{proposition}
The moment-angle complex $\mathcal{Z}_P$ has the canonical structure of a multigraded subcomplex\index{moment-angle complex!cellular structure} in the multigraded cellular structure of $\mathbb{D}^{2m}$. The projection $\mathbb \pi^m\colon \mathcal{Z}_P\to\mathbb I_P$ is cellular. 
\end{proposition}

\begin{theorem}
There is a multigraded structure in the cohomology group:\index{moment-angle complex!multigraded structure in cohomology}
$$
H^n(\mathcal{Z}_P,\mathbb Z)\simeq\bigoplus\limits_{2|\omega|=n+i}H^{-i,2\omega}(\mathcal{Z}_P,\mathbb Z),
$$ 
where for $\omega=\{j_1,\dots,j_k\}$, we have  $|\omega|=k$. 
\end{theorem}
\begin{proof}
The multigraded structure in cohomology is induced by the multigraded cellular structure described above. 
\end{proof}
\begin{example}
Let  $P=\Delta^n$, then $\mathcal{Z}_P=S^{2n+1}$. In the case $n=1$ the simplex $\Delta^1$ is an interval  $I$, and we have the decomposition $\mathcal{Z}_I=S^3=S^1\times D^2\cup D^2\times S^1$. The space $\mathcal{Z}_I$ consists of $8$ cells
\begin{gather*}
p_1\times p_2\\
p_1\times U_2,\quad U_1\times p_2\\
p_1\times V_2,\quad U_1\times U_2,\quad V_1\times p_2,\\
U_1\times V_2,\quad V_1\times U_2
\end{gather*}
We have 
\begin{gather*}
H^*(S^3)=H^{0,2\varnothing}(S^3)\oplus H^{-1,2\{1,2\}}(S^3).
\end{gather*}
\end{example}

\subsection{Multiplication}

Now\index{multiplication in cohomology} following \cite{Bu-Pa15} we will describe the cohomology ring of a moment-angle complex in terms of the cellular structure defined above. This result is non-trivial, since  the problem to define the multiplication in cohomology in terms of cellular cochains in general case is unsolvable. The reason is that the diagonal mapping used in the definition of the cohomology product is not cellular, and a cellular approximation can not be made functorial with respect  to arbitrary cellular mappings. We construct a canonical cellular diagonal approximation $\widetilde\Delta\colon \mathcal{Z}_P\to\mathcal{Z}_P\times\mathcal{Z}_P$, which is functorial with respect to mappings induced by admissible mapping of sets of facets of polytopes.

Remind, that the product in the cohomology of a cell complex $X$ is defined as follows. Consider the composite mapping of cellular cochain complexes
\begin{equation}\label{mult}
\mathcal{C}^*(X)\otimes\mathcal{C}^*(X)\stackrel{\times}{\longrightarrow}\mathcal{C}^*(X\times X)\stackrel{\widetilde{\Delta}^*}{\longrightarrow}\mathcal{C}^*(X).
\end{equation}
Here the mapping $\times$ sends a cellular cochain $c_1\otimes c_2\in C^{q_1}(X)\otimes C^{q_2}(X)$ to the cochain $c_1\times c_2\in C^{q_1+q_2}(X\times X)$, whose value on a cell $e_1\times e_2\in C_*(X\times X)$ is
$\langle c_1, e_1\rangle\langle c_2,e_2\rangle$. 
The mapping $\widetilde{\Delta}^*$ is induced by a cellular mapping $\widetilde{\Delta}$ (a cellular {\em diagonal approximation}) homotopic to the diagonal $\Delta\colon X\to X\times X$. In cohomology, the mapping (\ref{mult}) induces a multiplication $H^*(X)\otimes H^*(X)\to H^*(X)$ which does not depend on the choice of a cellular approximation and is functorial. However, the mapping (\ref{mult}) itself is not functorial because there is no choice of a cellular approximation compatible with arbitrary cellular mappings. 

Define polar coordinated in $D^2$ by $z=\rho e^{i\varphi}$.
\begin{proposition}\label{DeltaApp}
\begin{enumerate}
\item The mapping $\Delta_t \colon \mathbb I\times D^2\to D^2\times D^2$: $\rho e^{i\varphi}\to$
$$
\to\begin{cases}\left((1-\rho)t+\rho e^{i(1+t)\varphi},(1-\rho)t+\rho e^{i(1-t)\varphi}\right),&\varphi\in[0,\pi],\\
\left((1-\rho)t+\rho e^{i(1-t)\varphi+2\pi i t},(1-\rho)t+\rho e^{i(1+t)\varphi-2\pi i t}\right),&\varphi\in[\pi,2\pi]\end{cases}
$$
defines the homotopy of mappings of pairs 
$(D^2, S^1)\to(D^2\times D^2, S^1\times S^1)$. 
\item The mapping $\Delta_0$ is the diagonal mapping 
$\Delta\colon D^2\to D^2\times D^2$.
\item The mapping $\Delta_1$ is
$$
\rho e^{i\varphi}\to\begin{cases}((1-\rho)+\rho e^{2i\varphi},1),&\varphi\in[0,\pi],\\
(1,(1-\rho)+\rho e^{2i\varphi}),&\varphi\in[\pi,2\pi]\end{cases}
$$
It is cellular and sends the pair $(D^2,S^1)$ to the pair of wedges\\
$(D^2\times 1\vee 1\times D^2,S^1\times 1\vee 1\times S^1)$ in the point $(1,1)$. Hence it is  a cellular approximation of $\Delta$.  
\item We have 
$$
(\Delta_1)_*p=p\times p, \,(\Delta_1)_*U=U\times p+p\times U,\,(\Delta_1)_*V=V\times p+p\times V;
$$ 
hence 
$$
(U^*)^2=\langle U^*\times U^*,(\Delta_1)_*V\rangle V^*=\langle U^*\times U^*,V\times p+p\times V\rangle V^*=0,
$$ and the multiplication of cochains in $\mathcal{C}^*(D^2)$ induced by $\Delta_1$ is trivial:\label{DA1}
$$
1\cdot X=X=X\cdot 1,\quad (U^*)^2=U^*V^*=V^*U^*=(V^*)^2=0.
$$
\end{enumerate}
\end{proposition}
The proof we leave as an exercise. 

Using the properties of the construction of the moment-angle complex we obtain the following result.
\begin{corollary}\label{DAC}\index{moment-angle complex!multiplication in cohomology}
\begin{enumerate}
\item For any simple polytope $P$ with $m$ facets there is a homotopy 
$$
\Delta_t^m\colon (\mathbb D^{2m},\mathcal{Z}_P)\to (\mathbb D^{2m}\times\mathbb D^{2m}, \mathcal{Z}_P\times\mathcal{Z}_P),
$$
where $\Delta_0^m$ is the diagonal mapping and $\Delta_1^m$ is a cellular mapping\index{moment-angle complex!cellular approximation of the diagonal mapping}.
\item In the cellular cochain complex of $\mathbb D^{2m}=D^2\times \dots\times D^2$ the multiplication defined by $\Delta_1^m$ is the tensor product of multiplications of the factors defined by the rule $(\varphi_1\times\varphi_2)(\psi_1\times\psi_2)=(-1)^{\dim \varphi_2\dim\psi_1}\varphi_1\psi_1\times\varphi_2\psi_2$, and 
$$
(\varphi_1\times\dots\times \varphi_m)(\psi_1\times\dots\times \psi_m)=(-1)^{\sum\limits_{i>j}\dim\varphi_i\dim\psi_j}\varphi_1\psi_1\times\dots\times\varphi_m\psi_m,
$$
and respects the multigrading.        
\item The multiplication in $\mathcal{C}^*(\mathcal{Z}_P)$ given by $\Delta^m_1$ is defined from the inclusion\\
 $\mathcal{Z}_P\subset\mathbb D^{2m}$ as a multigraded cellular subcomplex\index{moment-angle complex!multiplication in cohomology}.
\end{enumerate}
\end{corollary}

\subsection{Description in terms of the Stanley-Reisner ring}
\begin{definition}
Let $\{F_1,\dots, F_m\}$ be the set of facets of a simple polytope $P$. Then
a {\em Stanley-Reisner ring}\index{Stanley-Reisner ring}\index{polytope!Stanley-Reisner ring} of $P$ over $\mathbb Z$ is defined as a monomial ring
$$
\mathbb{Z}[P]=\mathbb{Z}[v_1,\dots,v_m]/J_{SR}(P),
$$
where
$$
J_{SR}(P) = (v_{i_1}\dots v_{i_k},\text{ if }F_{i_1}\cap \dots\cap F_{i_k}=\varnothing)
$$
is the Stanley-Reisner ideal.
\end{definition}
\begin{example}
$\mathbb{Z}[\Delta^2]=\mathbb{Z}[v_1,v_2,v_3]/(v_1 v_2 v_3)$
\end{example}

\begin{theorem} (see \cite{BG09})
Two polytopes are combinatorially equivalent if and only if their
Stanley-Reisner rings are isomorphic.
\end{theorem}
\begin{corollary}
Fullerenes $P_1$ and $P_2$ are combinatorially equivalent if and only if
there is an isomorphism $\mathbb Z[P_1] \cong \mathbb Z[P_2]$.
\end{corollary}

\begin{theorem}
The Stanley-Reisner ring \index{Stanley-Reisner ring} of a flag polytope is a monomial quadratic ring:
$$
J_{SR}(P)= \{v_i v_j: F_i\cap F_j=\varnothing\}.
$$
\end{theorem}
\begin{figure}
\begin{center}
\includegraphics[scale=0.55]{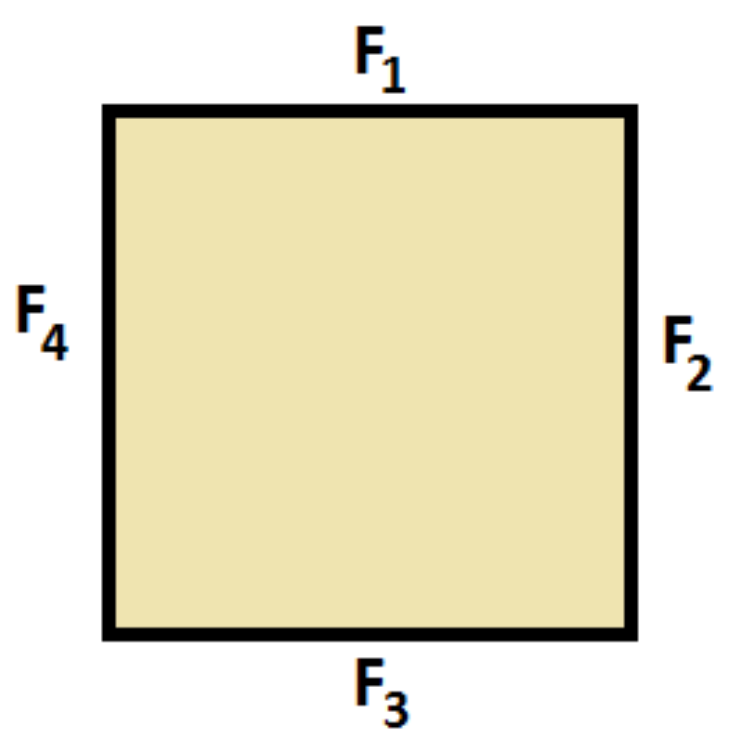}
\end{center}
\caption{Cube $I^2$. We have $J_{SR}(I^2)=\{v_1 v_3, v_2 v_4\}$}
\end{figure}
Each fullerene is a simple flag polytope (Theorem \ref{3belts-theorem}).

\begin{corollary}
The Stanley-Reisner ring of a fullerene is monomial quadratic.
\end{corollary}

\noindent{\bf Construction (multigraded complex):} For a set $\sigma\subset[m]$ define $G(\sigma)=\bigcap\limits_{i\in\sigma}F_i$. Conversely, for a face $G$ define $\sigma(G)=\{i\colon G\subset F_i\}\subset[m]$. Then $\sigma(G(\sigma))=\sigma$, and $G(\sigma(G))=G$. Let
\begin{gather*}
R^*(P)=\Lambda[u_1,\dots, u_m]\otimes \mathbb Z[P]/(u_iv_i,v_i^2),\\
{\rm mdeg}\;u_i=(-1,2\{i\}), {\rm mdeg}\;v_i=(0,2\{i\}), du_i=v_i, dv_i=0
\end{gather*} be a multigraded differential algebra. It is additively generated by monomials $v_{\sigma}u_{\omega\setminus\sigma}$, where $v_{\sigma}=\prod\limits_{i\in\sigma}v_i$, $G(\sigma)\ne\varnothing$, and $u_{\omega\setminus\sigma}=u_{j_1}\wedge\dots\wedge u_{j_l}$ for $\omega\setminus\sigma=\{j_1,\dots,j_l\}$. 
\begin{theorem}\cite{Bu-Pa15}\label{HRth}
We have a mutigraded ring isomorphism
$$
H[R^*(P),d]\simeq H^*(\mathcal{Z}_P,\mathbb Z)
$$
\end{theorem}
\begin{proof}
Define the mapping $\zeta\colon R^*(P)\to C^*(\mathcal{Z}_P)$ by the rule $\zeta(v_{\sigma}u_{\omega\setminus\sigma})=C^*_{\sigma,\omega}$. 
It is a graded ring isomorphism from Proposition \ref{DeltaApp}(\ref{DA1}), and Corollary \ref{DAC}. The formula $\zeta(dv_{\sigma}u_{\omega\setminus\sigma})=\partial C^*_{\sigma,\omega}$ follows from the Leibnitz rule.
\end{proof}

\noindent{\bf Exercise:} Prove that for the cospherical class $\varphi_\omega^*\left(\left[S^{2k-1}_{[m]\setminus\omega}\right]\right)$, $\omega=\{i_1,\dots,i_k\}$, (see Corollary \ref{cosph}) we have $\varphi_\omega^*\left(\left[S^{2k-1}_{[m]\setminus\omega}\right]\right)=\pm [u_{i_1}v_{i_2}\dots v_{i_k}]\in H[R^*(P),d]$.
\subsection{Description in terms of unions of facets}\index{moment-angle complex!description of cohomology in terms of unions of facets}
Let $P_\omega=\bigcup\limits_{i\in \omega}F_i$ for a subset $\omega\subset
[m]$.  By definition $P_{\varnothing}=\varnothing$, and $P_{[m]}=\partial P$.

\begin{definition}
For two sets $\sigma,\tau\subset[m]$ define $l(\sigma,\tau)$ to be the number of pairs $\{(i,j)\colon i\in\sigma,j\in\tau, i>j\}$. We write $l(i,\tau)$ and $l(\sigma,j)$ for $\sigma=\{i\}$ and $\tau=\{j\}$ respectively. 
\end{definition}
\noindent{\bf Comment:} The number $(-1)^{l(\sigma,\omega)}$ is used for definition of the multiplication of cubical chain complexes 
(see \cite{S51}). 
In the discrete mathematics the number $l(\sigma,\tau)$ is a characteristic of two subsets $\sigma,\tau$ of an ordered set. 

\begin{proposition} We have
\begin{enumerate}
\item $l(\sigma,\tau)=\sum\limits_{i\in\sigma}l(i,\tau)=\sum\limits_{j\in\tau}l(\sigma,j)=\sum\limits_{i\in\sigma,j\in\tau}l(i,j)$.
\item $l(\sigma,\tau_1\sqcup \tau_2)=l(\sigma,\tau_1)+l(\sigma,\tau_2)$, $l(\sigma_1\sqcup\sigma_2,\tau)=l(\sigma_1,\tau)+l(\sigma_2,\tau)$.
\item $l(\sigma,\tau)+l(\tau,\sigma)=|\sigma||\tau|-|\sigma\cap\tau|$. \\
In particular, if $\sigma\cap \tau=\varnothing$, then $l(\sigma,\tau)+l(\tau,\sigma)=|\tau||\sigma|$.
\end{enumerate}
\end{proposition}
\begin{definition}
Set 
$$
\mathbb I_{P,\omega}=\bigcup\limits_{G\ne\varnothing\colon\sigma(G)\subset \omega}\mathbb{I}_{P,G}=
\{(x_1,\dots,x_m)\in \mathbb I_P\colon x_i=1,i\notin \omega\}. 
$$
\end{definition}
\begin{theorem}\cite{Bu-Pa15}\label{Pwtheorem}
For any $\omega\subset[m]$ there is an isomorphism:
$$
H^{-i,2\omega}(\mathcal{Z}_P,\mathbb Z)\cong {H}^{|\omega|-i}(P,P_{\omega},\mathbb Z),
$$
\end{theorem}
\begin{proof}
For  subsets $A\subset \mathbb I^m$ and  $\omega\subset[m]$ define 
$$
A^{\omega}=\{(y_1,\dots,y_m)\in A\colon y_i=0 \text{ for some }i\in\omega\}, \quad A^0=A^{[m]}.
$$
We have $A^{\varnothing}=A$.
There is a homeomorphism of pairs $(P,P_{\omega})\simeq (\mathbb I_P,\mathbb I_P^\omega)$. 

The homotopy $r^{\omega}_t\colon \mathbb I^m\to\mathbb I^m$:
$$
r^\omega_t(y_1,\dots,y_m)=(y_1',\dots,y_m'), y_j'=
\begin{cases}
(1-t)y_j+t,&j\notin \omega;\\
y_j,&j\in\omega,
\end{cases}
$$
gives a deformation retraction $r^\omega=r^\omega_1\colon (\mathbb I_P,\mathbb I_P^\omega)\to (\mathbb I_{P,\omega}, \mathbb I_{P,\omega}^0)$. 

There is a natural multigraded cell structure on the cube $\mathbb I^m$, induced by the cell structure on $\mathbb I$ consisting of $3$ cells: $\mathbf{0}=\{0\}$, $\mathbf{1}=\{1\}$ and $J=(0,1)$. 
All the sets $\mathbb I_P$, $\mathbb I_{P,G}$, $\mathbb I_{P}^\omega$, $\mathbb I_{P,\omega}$, $\mathbb I_{P,\omega}^0$ are cellular subcomplexes. 
There is a natural orientation in $J$ such that $\mathbf 0$ is the beginning, and $\mathbf 1$ is the end. We have 
\begin{gather*}
d \mathbf0=d\mathbf1=0,\quad dJ=\mathbf1-\mathbf0;\\
\partial \mathbf1^*=-\partial \mathbf 0^*=J^*,\quad\partial J^*=0.
\end{gather*}
The cells in $\mathbb I^m$ has the form $\eta_1\times\dots\times \eta_m$, $\eta_i\in\{\mathbf0_i,\mathbf1_i,J_i\}$. There is natural cellular approximation for the diagonal mapping $\Delta\colon \mathbb I\to \mathbb I\times\mathbb I$  by the mapping $\Delta_1$:
$$
\Delta_1(x)=\begin{cases}(2x,1),&x\in[0,\frac{1}{2}],\\
(1,2x-1),&x\in[\frac{1}{2},1],
\end{cases}
$$
connected with $\Delta$ by the homotopy $\Delta_t=(1-t)\Delta+t\Delta_1$. Then 
$$
(\Delta_1)_*\mathbf0=\mathbf0\times \mathbf0, \quad(\Delta_1)_*\mathbf1=\mathbf1\times \mathbf1,\quad(\Delta_1)_*J=J\times \mathbf0+\mathbf1\times J,
$$
and for the induced multiplication we have 
\begin{gather*}
(\mathbf0^*)^2=\mathbf0^*,\quad (\mathbf1^*)^2=\mathbf1^*,\quad\mathbf0^*\mathbf1^*=\mathbf1^*\mathbf0^*=0,\quad \mathbf0^*J^*=J^*\mathbf1^*=0,\\
J^*\mathbf0^*=\mathbf1^*J^*=J^*,\quad (J^*)^2=0.
\end{gather*}

The cells in $\mathbb I_{P,\omega}\setminus\mathbb I_{P,\omega}^0$ have the form 
$$
E_{\sigma}=\eta_1\times\dots\times \eta_m,\quad\eta_j=\begin{cases}J_j,&j\in\sigma,\\
1,&j\notin\sigma,\end{cases},
$$
where $\sigma\subset\omega$, and $G(\sigma)\ne\varnothing$.
Then $E^*_{\sigma}=\eta_1^*\times\dots\times\eta_m^*$. 

Now define the mapping $\xi_{\omega}\colon R^{-i,2\omega}\to C^{|\omega|-i}(\mathbb I_{P,\omega}, \mathbb I_{P,\omega}^0)$ by the rule
$$
\xi_{\omega}(u_{\omega\setminus\sigma}v_{\sigma})=(-1)^{l(\sigma,\omega)} E^*_{\sigma}.
$$
By construction $\xi_{\omega}$ is an additive isomorphism.  For $\sigma\subset\omega$ we have 
$$
\partial\xi_{\omega}(v_{\sigma}u_{\omega\setminus\sigma})=\partial\left((-1)^{l(\sigma,\omega)}E^*_{\sigma}\right)=(-1)^{l(\sigma,\omega)}\!\!\!\!\!\!\!\!\!\!\sum\limits_{j\in\omega\setminus\sigma,G(\sigma\sqcup\{j\})\ne\varnothing}(-1)^{l(j,\sigma)}E^*_{\sigma\sqcup\{j\}},
$$
On the other hand,
\begin{multline*}
\xi_{\omega}(dv_{\sigma}u_{\omega\setminus\sigma})=\xi_{\omega}\left(\sum\limits_{j\in\omega\setminus\sigma, G(\sigma\sqcup\{j\})\ne\varnothing}(-1)^{l(j,\omega\setminus\sigma)}v_{\sigma\sqcup\{j\}}u_{\omega\setminus(\sigma\sqcup\{j\})}\right)=\\
=\sum\limits_{j\in\omega\setminus\sigma, G(\sigma\sqcup\{j\})\ne\varnothing}(-1)^{l(\sigma\sqcup\{j\},\omega)}(-1)^{l(j,\omega\setminus\sigma)}E^*_{\sigma\sqcup\{j\}}
\end{multline*}
Now the proof follows from the formula 
$$
l(\sigma\sqcup\{j\},\omega)+l(j,\omega\setminus\sigma)=l(\sigma,\omega)+l(j,\omega)+l(j,\omega\setminus\sigma)=l(\sigma,\omega)+l(j,\sigma)+2l(j,\omega\setminus\sigma)
$$
\end{proof}

\begin{corollary}\cite{Bu-Pa15}\label{Pocor}
For any $\omega\subset[m]$ there is an isomorphism:
$$
H^{-i,2\omega}(\mathcal{Z}_P,\mathbb Z)\cong \widetilde{H}^{|\omega|-i-1}(P_{\omega},\mathbb Z),
$$
where by definition $\widetilde{H}^{-1}(\varnothing,\mathbb Z)=\mathbb Z$.
\end{corollary}
The proof follows from the long exact sequence in the reduced cohomology of the pair $(P,P_{\omega})$, since $P$ is contractible.

\subsection{Multigraded Betti numbers and the Poincare duality}
\begin{definition} Define {\em multigraded Betti numbers} \index{multigraded Betti numbers}$\beta^{-i,2\omega}={\rm rank\,} H^{-i,2\omega}(\mathcal{Z}_P)$. We have
$$
\beta^{-i,2\omega}={\rm rank\,}H^{|\omega|-i}(P,P_{\omega})={\rm rank\,}\widetilde H^{|\omega|-i-1}(P_\omega,\mathbb Z).
$$
\end{definition}
From Proposition \ref{ZPeq} the manifold $\widehat{\mathcal{Z}_P}$ is oriented. 
\begin{proposition} We have\index{moment-angle manifold!multigraded Poincare duality}
$$
\beta^{-i,2\omega}=\beta^{-(m-n-i),2([m]\backslash \omega)}.
$$
\end{proposition}
\begin{proof} From the Poincare duality theorem the 
bilinear form $H^*(\mathcal{Z}_P)\otimes H^*(\mathcal{Z}_P)\to\mathbb Z$ defined by 
$$
\langle\varphi,\psi\rangle=\langle \varphi\psi,[\mathcal{Z}_P]\rangle,
$$
where $[\mathcal{Z}_P]$ is a fundamental cycle, is non-degenerate if we factor out the torsion. This means that there is a basis for which the matrix of the bilinear form has determinant $\pm1$.
For mutligraded ring this means that the matrix consists of blocks corresponding to the forms 
$$
H^{-i,2\omega}(\mathcal{Z}_P)\otimes H^{-(m-n-i),2([m]\backslash \omega)}(\mathcal{Z}_P)\to\mathbb Z.
$$
Hence all blocks are squares and have determinant $\pm1$, which finishes the proof. 
\end{proof}
Let the polytope $P$ be given in the irredundant form $\{\boldsymbol{x}\in\mathbb R^n\colon A\boldsymbol{x}+b\geqslant 0\}$. 
For the vertex $v=F_{i_1}\cap \dots\cap F_{i_n}$ define the submatrix $A_v$ in $A$ corresponding to the rows $i_1,\dots, i_n$. 
\begin{proposition}
The fundamental cycle $[\mathcal{Z}_P]$ can be represented by the following element in $C_{-(m-n),[m]}(\mathcal{Z}_P)$:
$$
Z=\sum\limits_{v\text{ -- vertex}}(-1)^{l(\sigma(v),[m])}{\rm sign} (\det A_v) C_{\sigma(v),[m]}.
$$ 
Then the form
$$
C^{-i,\omega}(\mathcal{Z}_P)\otimes C^{-(m-n-i),[m]\setminus\omega}(\mathcal{Z}_P)\to\mathbb Z
$$
is defined by the property  
$$
\langle u_{[m]\setminus\sigma(v)}v_{\sigma(v)},Z\rangle=(-1)^{l(\sigma(v),m)}{\rm sign}(\det A_v).
$$
\end{proposition}
The idea of the proof is to use the Davis--Januszkiewicz' construction. The space $P^n\times \mathbb T^m$ has the orientation defined by orientations of $P^n$ and $S^1$. Then the mapping 
$$
P^n\times \mathbb T^m\to \mathcal{Z}_P\colon (\boldsymbol{x},\boldsymbol{t})\to \boldsymbol{t}s_P(\boldsymbol{x})
$$
defines the orientation of the cells $C_{\sigma(v),[m]}$.
\subsection{Multiplication in terms of unions of facets}
For pairs of spaces define the direct product as
$$
(X,A)\times(Y,B)=(X\times Y,A\times Y\cup X\times B).
$$
There is a canonical multiplication in the cohomology of cellular pairs 
$$
H^k(X,A)\otimes H^l(X,B)\to H^{k+l}(X, A\cup B)
$$
defined in the cellular cohomology by the rule
$$
H^k(X,A)\otimes H^l(X,B)\xrightarrow{\times} H^{k+l}((X,A)\times(X,B))\xrightarrow{\widetilde\Delta^*}H^{k+l}(X,A\cup B),
$$
where $\widetilde\Delta$ is a cellular approximation of the diagonal mapping 
$$
\Delta\colon(X,A\cup B)\to(X,A)\times(X,B).
$$
Thus for any simple polytope $P$ and subsets $\omega_1,\omega_2\subset[m]$, we have the canonical multiplication
$$
H^{k}(P,P_{\omega_1})\otimes H^{l}(P, P_{\omega_2})\to
H^{k+l}(P, P_{\omega_1\cup \omega_2}).
$$

\begin{theorem}\label{MPTheorem}
There is the ring isomorphism
$$
H^*(\mathcal{Z}_P)\simeq \bigoplus\limits_{\omega\subset [m]}H^*(P, P_\omega)
$$
where the multiplication on the right hand side 
$$
H^{|\omega_1|-k}(P,P_{\omega_1})\otimes H^{|\omega_2|-l}(P, P_{\omega_2})\to
H^{|\omega_1|+|\omega_2|-k-l}(P, P_{\omega_1\cup \omega_2})
$$
is trivial if $\omega_1\cap\omega_2\ne\varnothing$, and for the case $\omega_1\cap \omega_2=\varnothing$ is given by the rule 
$$
a\otimes b\to (-1)^{l(\omega_2,\omega_1)+|\omega_1|l}ab,
$$ 
where  $a\otimes b\to ab$ is the canonical multiplication.
\end{theorem}
\noindent{\bf Comment:} The statement of the theorem presented in \cite{Bu-Pa15} as Exercise 3.2.14 does not contain the specialization of the sign. 
\begin{proof}
We will identify $(P,P_{\omega})$ with $(\mathbb I_P,\mathbb I_{P}^\omega)$ and $H^*(\mathcal{Z}_P)$ with $H[R^*(P),d]$.
If $\omega_1\cap \omega_2\ne\varnothing$, then the multiplication 
$$
H^{-k,2\omega_1}(\mathcal{Z}_P)\otimes H^{-l,2\omega_2}(\mathcal{Z}_P)\to H^{-(k+l),2(\omega_1\cup\omega_2)}(\mathcal{Z}_P)
$$ 
is trivial by Theorem \ref{HRth}. Let $\omega_1\cap \omega_2=\varnothing$. 
We have the commutative diagram of mappings
$$
\begin{CD}
(\mathbb I_{P,\omega_1\sqcup \omega_2},\mathbb I_{P,\omega_1\sqcup \omega_2}^0) @>i_{\omega_1,\omega_2}>>(\mathbb I_{P,\omega_1},\mathbb I_{P,\omega_1}^0)\times (\mathbb I_{P,\omega_2},\mathbb I_{P,\omega_2}^0)\\
  @Ar^{\omega_1\sqcup\omega_2}AA @AAr^{\omega_1}\times r^{\omega_2}A @.\\
 (\mathbb I_P, \mathbb I_P^{\omega_1\sqcup \omega_2}) @>\Delta>>(\mathbb I_P,\mathbb I_P^{\omega_1})\times (\mathbb I_P,\mathbb I_P^{\omega_2})
\end{CD}
$$
which gives the commutative diagram
$$
\begin{CD}
H^*\left((\mathbb I_{P,\omega_1},\mathbb I_{P,\omega_1}^0)\times (\mathbb I_{P,\omega_2},\mathbb I_{P,\omega_2}^0)\right)@>i^*_{\omega_1,\omega_2}>>H^*\left(\mathbb I_{P,\omega_1\sqcup \omega_2},\mathbb I_{P,\omega_1\sqcup \omega_2}^0\right) \\
  @V(r^{\omega_1}\times r^{\omega_2})^*VV @VV(r^{\omega_1\sqcup\omega_2})^*V @.\\
H^*\left((\mathbb I_P,\mathbb I_P^{\omega_1})\times (\mathbb I_P,\mathbb I_P^{\omega_2})\right)@>\Delta^*>> H^*\left(\mathbb I_P,\mathbb I_P^{\omega_1\sqcup \omega_2}\right)
\end{CD}
$$
where the vertical mappings are isomorphisms.
Together with the functoriality  of the $\times$-product in cohomology this proves the theorem provided the commutativity of the diagram
{\tiny
$$
\begin{CD}
C^{-k,2\omega_1}(\mathcal{Z}_P)\otimes C^{-l,2\omega_2}(\mathcal{Z}_P)@>\widetilde{\Delta}\circ\times>>C^{-(k+l),2(\omega_1\sqcup\omega_2)}(\mathcal{Z}_P)\\
@V\xi_{\omega_1}\otimes\xi_{\omega_2}VV @VV\xi_{\omega_1\sqcup\omega_2}V @.\\
C^{|\omega_1|-k}\left(\mathbb I_{P,\omega_1},\mathbb I_{P,\omega_1}^0\right)\otimes C^{|\omega_2|-l}\left(\mathbb I_{P,\omega_2},\mathbb I_{P,\omega_2}^0\right)@>i_{\omega_1,\omega_2}^*\circ\times>>C^{|\omega_1|+|\omega_2|-k-l}\left(\mathbb I_{P,\omega_1\sqcup \omega_2},\mathbb I_{P,\omega_1\sqcup \omega_2}^0\right)
\end{CD}
$$}
where the lower arrow is the composition of two mappings:{\tiny
\begin{gather*}
C^{|\omega_1|-k}\left(\mathbb I_{P,\omega_1},\mathbb I_{P,\omega_1}^0\right)\otimes C^{|\omega_2|-l}\left(\mathbb I_{P,\omega_2},\mathbb I_{P,\omega_2}^0\right)\xrightarrow{\times}C^{|\omega_1|+|\omega_2|-k-l}\left((\mathbb I_{P,\omega_1},\mathbb I_{P,\omega_1}^0)\times (\mathbb I_{P,\omega_2},\mathbb I_{P,\omega_2}^0)\right)\\
C^{|\omega_1|+|\omega_2|-k-l}\left((\mathbb I_{P,\omega_1},\mathbb I_{P,\omega_1}^0)\times (\mathbb I_{P,\omega_2},\mathbb I_{P,\omega_2}^0)\right)\xrightarrow{i_{\omega_1,\omega_2}^*} C^{|\omega_1|+|\omega_2|-k-l}\left(\mathbb I_{P,\omega_1\sqcup \omega_2},\mathbb I_{P,\omega_1\sqcup \omega_2}^0\right)
\end{gather*}}
For this we have 
\begin{multline*}
\xi_{\omega_1\sqcup\omega_2}((u_{\omega_1\setminus\sigma_1}v_{\sigma_1})(u_{\omega_2\setminus\sigma_2}v_{\sigma_2}))=\\
=(-1)^{l(\omega_1\setminus\sigma_1,\omega_2\setminus\sigma_2)}\xi_{\omega_1\sqcup\omega_2}(u_{(\omega_1\sqcup\omega_2)\setminus(\sigma_1\sqcup\sigma_2)}v_{\sigma_1\sqcup\sigma_2})=\\
=(-1)^{l(\omega_1\setminus\sigma_1,\omega_2\setminus\sigma_2)}(-1)^{l(\sigma_1\sqcup\sigma_2,\omega_1\sqcup\omega_2)}E^*_{\sigma_1\sqcup\sigma_2}.
\end{multline*}
On the other hand
\begin{multline*}
i_{\omega_1,\omega_2}^*\left(\xi_{\omega_1}(u_{\omega_1\setminus\sigma_1}v_{\sigma_1})\times\xi_{\omega_2}(u_{\omega_2\setminus\sigma_2}v_{\sigma_2})\right)=\\
=(-1)^{l(\sigma_1,\omega_1)}(-1)^{l(\sigma_2,\omega_2)}i_{\omega_1,\omega_2}^*(E^*_{\sigma_1}\times E^*_{\sigma_2})=\\
=(-1)^{l(\sigma_1,\omega_1)}(-1)^{l(\sigma_2,\omega_2)}(-1)^{l(\sigma_1,\sigma_2)}E^*_{\sigma_1\sqcup\sigma_2},
\end{multline*}
where the last equality follows from the the following calculation:
\begin{gather*}
(i_{\omega_1,\omega_2})_*(E_{\sigma_1\sqcup\sigma_2})=(-1)^{l(\sigma_1,\sigma_2)}E_{\sigma_1}\times E_{\sigma_2}.
\end{gather*}
Now let us calculate the difference of signs:
\begin{multline*}
\left(l(\omega_1\setminus\sigma_1,\omega_2\setminus\sigma_2)+l(\sigma_1\sqcup\sigma_2,\omega_1\sqcup\omega_2)\right)-\\
\left(l(\sigma_1,\omega_1)+l(\sigma_2,\omega_2)+l(\sigma_1,\sigma_2)\right)\mod 2=\\
=l(\omega_1\setminus\sigma_1,\omega_2\setminus\sigma_2) + l(\sigma_1,\omega_2) + l(\sigma_2,\omega_1) + l(\sigma_1,\sigma_2)\mod 2=\\
=l(\omega_1\setminus\sigma_1,\omega_2\setminus\sigma_2)+l(\sigma_1,\omega_2\setminus\sigma_2)+l(\sigma_2,\omega_1)\mod 2=\\
=l(\omega_1,\omega_2\setminus\sigma_2)+l(\sigma_2,\omega_1)\mod 2=\\
=l(\omega_1,\omega_2\setminus\sigma_2)+l(\omega_1,\sigma_2)+|\sigma_2||\omega_1|\mod 2=l(\omega_1,\omega_2)+|\sigma_2||\omega_1|\mod 2=\\
=l(\omega_2,\omega_1)+|\omega_1||\omega_2|+|\omega_1|(|\omega_2|-l)\mod 2=l(\omega_2,\omega_1)+|\omega_1|l\mod 2.
\end{multline*}
\end{proof}
\subsection{Description in terms of related simplicial complexes}\index{moment-angle complex!description of cohomology in terms of related simplicial complexes}
\begin{definition}
An {\em (abstract) simplicial complex}\index{abstract simplicial complex} $K$ on the vertex set $[m]=\{1,\dots,m\}$ is the set of subsets $K\subset 2^{[m]}$ such that
\begin{enumerate}
\item $\varnothing \in K$;
\item $\{i\}\in K$ for $i=1,\dots,m$;
\item If $\sigma\subset\tau$ and $\tau\in K$, then $\sigma\in K$.
\end{enumerate}
The sets $\sigma\in K$ are called {\em simplices} \index{abstract simplicial complex!simplex}. For an abstract simplicial complex $K$ there is a {\em geometric realization}\index{abstract simplicial complex!geometric realization} $|K|$ as a subcomplex in the simplex $\Delta^{m-1}$ with the vertex set $[m]$.   
\end{definition}
For a simple polytope $P$ define an abstract simplicial complex $K$ on the vertex set $[m]$ by the rule 
$$
\sigma\in K\text{ if and only if }\sigma=\sigma(G)=\{i\colon G\subset F_i\}\text{ for some }G\in L(P)\setminus\{\varnothing\}.
$$ 
We have the combinatorial equivalence $K\simeq \partial P^*$. For any subset $\omega\subset[m]$ define the full subcomplex $K_{\omega}=\{\sigma\in K\colon \sigma\subset\omega\}$. 

\begin{definition}
For two simplicial complexes $K_1$ and $K_2$ on the vertex sets ${\rm vert}(K_1)$ and ${\rm vert}(K_2)$ {\em join} \index{join} \index{abstract simplicial complex!join} $K_1*K_2$ is the simplicial complex on the vertex set ${\rm vert}(K_1)\sqcup{\rm vert}(K_2)$ with simplices $\sigma_1\sqcup\sigma_2$, $\sigma_1\in K_1$, $\sigma_2\in K_2$.
\end{definition}
A cone $CK_{\omega}$ is by definition $\{0\}* K_{\omega}$, where $\{0\}$ is the simplicial complex with one vertex $\{0\}$.
\begin{proposition}
For any $\varnothing\ne\omega\subset[m]$ we have a homeomorphism of pairs 
$$
(\mathbb I_{P,\omega},\mathbb I_{P,\omega}^0)\simeq(C|K_{\omega}|,|K_{\omega}|).
$$
\end{proposition}
\begin{proof}
For any simplex $\sigma\in K$ consider it's barycenter $\boldsymbol{y}_{\sigma}\in |K|$. Then we have a barycentric subdivision of $K$ consisting of simplices 
$$
\Delta_{\sigma_1\subset\dots\subset\sigma_k}={\rm conv}\{\boldsymbol{y}_{\sigma_1},\dots,\boldsymbol{y}_{\sigma_k}\}, \,k\geqslant 1.
$$
Define the mapping 
$c_K\colon K\to \mathbb I^m$ as 
$$
c_K(\boldsymbol{y}_{\sigma_k})=(y_1,\dots,y_m),y_i=\begin{cases}0,&i\in\sigma,\\
1,&i\notin\sigma\end{cases}
$$
on the vertices of the barycentric subdivision, $c_K(\{0\})=(1,\dots,1)$, and on the simplices and cones on simplices by linearity.
This defines the piecewise linear homeomorphisms of pairs 
$$
(C|K|,|K|)\to (\mathbb I_P,\mathbb I_P^0), \text{ and }(C|K_{\omega}|,|K_{\omega}|)\to (\mathbb I_{P,\omega},\mathbb I_{P,\omega}^0).
$$
\end{proof}
\begin{corollary}\label{PKw}
We have the homotopy equivalence $P_{\omega}\sim |K_{\omega}|$.
\end{corollary} 

For the simplicial complex  $K_{\omega}$ we have the simplicial chain complex with the free abelian groups of chains $C_i(K_{\omega})$, $i\geqslant-1$, generated by simplices $\sigma\in K_{\omega}$, $|\sigma|=i+1$, (including the empty simplex $\varnothing$, $|\varnothing|=0$), and the boundary homomorphism 
$$
d\colon C_i(K_{\omega})\to C_{i-1}(K_{\omega}),\quad d\sigma=\sum\limits_{i\in\sigma}(-1)^{l(i,\sigma)}(\sigma\setminus \{i\}).
$$
There is the cochain complex of groups $C^i(K_{\omega})={\rm Hom}(C_i(K_{\omega}),\mathbb Z)$. Define the cochain $\sigma^*$ by the rule $\langle \sigma^*,\tau\rangle=\delta(\sigma,\tau)$. The coboundary homomorphism $\partial =d^*$ can be calculated by the rule
$$
\partial \sigma^*=\sum\limits_{j\in\omega\setminus\sigma,\sigma\sqcup\{j\}\in K_{\omega}}(-1)^{l(j,\sigma)}(\sigma\sqcup\{j\})^*
$$
The homology groups of the chain and cochain complexes are $\widetilde{H}_{i}(K_{\omega})$ and $\widetilde{H}^{i}(K_{\omega})$  respectively.
The following result is proved similarly to Theorem \ref{Pwtheorem} and Theorem \ref{MPTheorem}. 
\begin{theorem}\index{moment-angle complex!description of cohomology in terms of related simplicial complexes}
For any $\omega\subset[m]$ the mapping 
$$
\widehat\xi_{\omega}\colon R^{-i,2\omega}\to C^{|\omega|-i-1}(K_{\omega}),\quad \widehat\xi_{\omega}(u_{\omega\setminus\sigma}v_{\sigma})=(-1)^{l(\sigma,\omega)}\sigma^*
$$ 
is the isomorphism of cochain complexes $\{C^{-i,2\omega}(\mathcal{Z}_P)\}_{i\geqslant 0}$
and $\{C^{|\omega|-i-1}(K_{\omega})\}_{i\geqslant 0}$. It induces the isomorphism  $H^{-i,2\omega}(\mathcal{Z}_P)\simeq \widetilde{H}^{|\omega|-i-1}(K_{\omega}) $ and the isomorphism of rings 
$$
H^*(\mathcal{Z}_P)\simeq \bigoplus\limits_{\omega\subset [m]}\widetilde{H}^*(K_\omega)
$$
where the multiplication on the right hand side 
$$
\widetilde{H}^p(K_{\omega_1})\otimes \widetilde H^{q}(K_{\omega_2})\to
\widetilde H^{p+q+1}(K_{\omega_1\cup \omega_2})
$$
is trivial if $\omega_1\cap\omega_2\ne\varnothing$, and for the case $\omega_1\cap \omega_2=\varnothing$ is given by the mapping of cochains defined by the rule 
$$
\sigma_1^*\otimes\sigma_2^*\to(-1)^{l(\omega_1,\omega_2)+l(\sigma_1,\sigma_2)+|\omega_1||\sigma_2|}(\sigma_1\sqcup\sigma_2)^*.
$$ 
\end{theorem}

\subsection{Description in terms of unions of facets modulo boundary}

The embeddings $b_P\colon P\to \mathbb I_P$ and $c_K\colon K\to \mathbb I_P^0$ define the simplicial isomorphism of barycentric subdivisions of $\partial P$ and $K$: the vertex $\boldsymbol{y}_{\sigma}$, $\sigma\ne\varnothing$, is mapped to the vertex $\boldsymbol{x}_{G(\sigma)}$ and on simplices we have the linear isomorphism. Then $K_{\omega}$ is embedded into  $P_{\omega}$.

For the set $P_{\omega}$ considered in the space $\partial P$ the boundary $\partial P_{\omega}$ consists of all points $\boldsymbol{x}\in P_{\omega}$ such that $\boldsymbol{x}\in F_j$ for some $j\notin \omega$. Hence  $\partial P_{\omega}$ consists of all faces $G\subset P$ such that $\sigma(G)\cap\omega\ne\varnothing$ and $\sigma(G)\not\subset \omega$.

Define on $P$ the orientation induced from $\mathbb R^n$, and on $\partial P$ the orientation induced from $P$ by the rule: a basis $(\boldsymbol{e}_1,\dots,\boldsymbol{e}_{n-1})$ in $\partial P$ is positively oriented if and only if the basis $(\boldsymbol{n},\boldsymbol{e}_1,\dots,\boldsymbol{e}_{n-1})$ is positively oriented, where $\boldsymbol{n}$ is the outer normal vector.

We have the orientation of simplices in $K_{\omega}$ defined by the canonical order of the vertices of the set 
$\omega\subset [m]$. We have the cellular structure on $P_{\omega}$ defined by the faces of $P$. Fix some orientation of faces in $P$ such that for facets the orientation coincides with $\partial P$. For a cell $E$ with fixed orientation in some cellular or simplicial structure it is convenient to consider the chain  $-E$ as a cell with an opposite orientation. Then the boundary operator just sends the cell to the sum of cells on the boundary with induced orientations. 

\begin{lemma}
The orientation of the simplex $\sigma=\{i_1,\dots,i_l\}\in |K_{\omega}|$ coincides with the orientation of the simplex
$$
{\rm conv}\{\boldsymbol{y}_{\sigma}, \boldsymbol{y}_{\sigma\setminus\{i_1\}},\boldsymbol{y}_{\sigma\setminus\{i_1,i_2\}},\dots,\boldsymbol{y}_{\{i_l\}}\}
$$  
\end{lemma}
The proof we leave as an exercise.

Now we establish the Poincare duality\index{Poincare duality} between the groups $\widetilde H^{i}(K_{\omega})$ and $H_{n-i-1}(P_{\omega},\partial P_{\omega})$. 

\begin{definition}
For a face $G\subset P_{\omega}$, $G\not\subset \partial P_{\omega}$, with a positively oriented basis $(\boldsymbol{e}_1,\dots,\boldsymbol{e}_k)$  and a simplex 
$\sigma\in K_{\omega}$ define the {\em intersection index}\index{intersection index} 
$$
C_*(P_{\omega},\partial P_{\omega})\otimes C_*(K_{\omega})\to\mathbb Z
$$ 
by the rule 
$$
\langle G,\sigma\rangle=
\begin{cases}0,&\text{ if }G(\sigma)\ne G;\\
1,&\text{ if }G(\sigma)=G,\text{ and the basis } (\boldsymbol{e}_1,\dots,\boldsymbol{e}_k, \boldsymbol{h}_1,\dots,\boldsymbol{h}_l)\text{ is positive};\\
-1,&\text{ if }G(\sigma)=G,\text{ and the basis } (\boldsymbol{e}_1,\dots,\boldsymbol{e}_k, \boldsymbol{h}_1,\dots,\boldsymbol{h}_l)\text{ is negative},
\end{cases}
$$
where $l=n-k-1$, and $(\boldsymbol{h}_1,\dots,\boldsymbol{h}_l)$ is any basis defining the orientation of any maximal simplex in the barycentric subdivision of $\sigma\subset P_\omega$ consistent with the orientation of $\sigma$, for example
$$
(\boldsymbol{h}_1,\dots,\boldsymbol{h}_l)=(\boldsymbol{y}_{\sigma\setminus\{i_1\}}-\boldsymbol{y}_{\sigma},\boldsymbol{y}_{\sigma\setminus\{i_1,i_2\}}-\boldsymbol{y}_{\sigma},\dots,\boldsymbol{y}_{\{i_l\}}-\boldsymbol{y}_{\sigma}) 
$$
\end{definition}
\begin{proposition} \label{pind}
We have $\langle dG,\tau\rangle=(-1)^{\dim G}\langle G,d\tau\rangle$.
\end{proposition}
\begin{proof}
Both left and right sides are equal to zero, if $\tau\ne \sigma(G)\sqcup\{j\}$ for some $j\in\omega\setminus\sigma$. Let $\tau=\sigma(G)\sqcup \{j\}$. Then $\tau=\sigma(G_j)$ for $G_j=G\cap F_j$. 
Let $\sigma=\sigma(G)$. The vector corresponding to $\boldsymbol{u}_j=\boldsymbol{y}_{\sigma\sqcup \{j\}}-\boldsymbol{y}_{\sigma}$ and the outer normal vector to the facet $\sigma$ of the simplex $\sigma\sqcup\{j\}$ look to opposite sides of ${\rm aff}\sigma$ in ${\rm aff}(\sigma\sqcup \{j\})$ in the geometric realization of $K$; hence the orientation of the basis $(\boldsymbol{u}_j,\boldsymbol{h}_1,\dots,\boldsymbol{h}_l)$ is negative in $\sigma\sqcup \{j\}$. On the other hand, $\boldsymbol{u}_j=\boldsymbol{x}_{G(\sigma\sqcup \{j\})}-\boldsymbol{x}_{G(\sigma)}$; hence this vector looks to the same side of ${\rm aff}(G_j)$ in ${\rm aff}(G)$ with the outer normal vector to $G_j$, the orientation of the basis $(\boldsymbol{u}_j,\boldsymbol{g}_1,\dots, \boldsymbol{g}_{k-1})$ is positive for the basis $(\boldsymbol{g}_1,\dots,\boldsymbol{g}_{k-1})$ defining the induced orientation of $G_j$.
Hence for the induced orientations of $G_j$ and $\sigma$ we have 
\begin{itemize}
\item$\langle G\cap F_j,\sigma\sqcup \{j\}\rangle$ is opposite to the sign of the orientation of $(\boldsymbol{g}_1,\dots,\boldsymbol{g}_{k-1},\boldsymbol{u}_j,\boldsymbol{h}_1,\dots,\boldsymbol{h}_l)$;
\item$\langle G,\sigma\sqcup \{j\}\rangle$ coinsides with the sign of the orientation of $(\boldsymbol{u}_j,\boldsymbol{g}_1,\dots, \boldsymbol{g}_{k-1},\boldsymbol{h}_1,\dots,\boldsymbol{h}_l)$;
\end{itemize}
Hence these numbers differ by the sign $(-1)^{k}$. 
\end{proof}
\begin{definition}
Set 
$$
\widehat{H}_i(P_{\omega},\partial P_{\omega})=
\begin{cases}
H_i(P_{\omega},\partial P_{\omega}),&0\leqslant i\leqslant n-2;\\
H_{n-1}(P_{\omega},\partial P_{\omega})/ ([\sum\limits_{i\in\omega}F_i]),&i=n-1.
\end{cases}.
$$
\end{definition}
\begin{theorem}\label{PoiTh}
The mapping $G\to \langle G,\sigma(G)\rangle\sigma(G)^*$ induces the isomorphism 
$$
\widehat{H}_{n-i-1}(P_{\omega},\partial P_{\omega})\simeq \widetilde{H}^i(K_{\omega}), \;0\leqslant i\leqslant n-1,\;\omega\ne\varnothing. 
$$
Moreover, for $\omega_1\cap \omega_2=\varnothing$ the multiplication 
$$
\widehat{H}_{n-p-1}(P_{\omega_1},\partial P_{\omega_1})\otimes \widehat{H}_{n-q-1}(P_{\omega_2},\partial P_{\omega_2})\to \widehat{H}_{n-(p+q)-2}(P_{\omega_1\sqcup \omega_2},\partial P_{\omega_1\sqcup\omega_2})
$$
induced by the isomorphism is defined by the rule 
$$
G_1\otimes G_2\to\frac{\langle G_1,\sigma(G_1)\rangle\langle G_2,\sigma(G_2)\rangle}{\langle G_1\cap G_2,\sigma(G_1\cap G_2)\rangle}(-1)^{l(\omega_1,\omega_2)+ |\omega_1|(n-\dim G_2)+l(\sigma(G_1),\sigma(G_2))}G_1\cap G_2
$$
\end{theorem}
The proof follows directly from  Proposition \ref{pind}.
\subsection{Geometrical interpretation of the cohomological groups} Let $P$ be a simple polytope. 
From Corollary \ref{PKw} we obtain the following results
\begin{proposition}\label{Pcc}
\begin{enumerate}
\item If $\omega=\varnothing$, then $P_{\omega}=\varnothing$; hence
$$
H^{-i,2\varnothing}(\mathcal{Z}_P)=\widetilde{H}^{-i-1}(P_{\omega})=\begin{cases} \mathbb Z,&i=0,\\
0,&\text{ otherwise }.\end{cases}
$$
\item If $G(\omega)\ne\varnothing$, then $P_{\omega}$ is contractible; hence
$$
H^{-i,2\omega}(\mathcal{Z}_P)=\widetilde{H}^{|\omega|-i-1}(P_{\omega})=0\text{ for all }i.
$$
In particular, this is the case for $|\omega|=1$.

\item If  $\omega=\{p,q\}$, then either $P_{\omega}$ is contractible, if $F_p\cap F_q\ne\varnothing$,\\
 or $P_{\omega}=F_p\sqcup F_q$,  where both $F_p$ and $F_q$ are contractible, if $F_p\cap F_q=\varnothing$. Hence
$$
H^{-i,2\{p,q\}}(\mathcal{Z}_P)=\widetilde{H}^{1-i}(P_{\omega})\begin{cases}\mathbb Z,&i=1, F_p\cap F_q\ne\varnothing,\\
0,&\text{otherwise}.\end{cases}
$$

\item If $G(\omega)=\varnothing$ and $\omega\ne\varnothing$, then $\dim K_{\omega}\leqslant\min\{n-1,|\omega|-2\}$; hence
$$
H^{-i,2\omega}(\mathcal{Z}_P)=\widetilde{H}^{|\omega|-i-1}(P_{\omega})=0\text{ for } |\omega|-i-1>\min\{n-1,|\omega|-2\}.
$$
\item If $\omega=[m]$, then $P_{[m]}=\partial P\simeq S^{n-1}$; hence 
$$
H^{-i,2[m]}(\mathcal{Z}_P)=\widetilde{H}^{m-i-1}(P_{\omega})=\begin{cases}\mathbb Z,&i=m-n,\\0,&\text{ otherwise.}\end{cases}
$$
\item $P_{\omega}$ is a subcomplex in $\partial P\simeq S^{n-1}$; hence 
$$
\widetilde{H}^{n-1}(P_{\omega})=\begin{cases}\mathbb Z,&\omega=[m],\\
0,&\text{ otherwise.}\end{cases}
$$
\item $H^{0,2\omega}(\mathcal{Z}_P)=\widetilde{H}^{|\omega|-1}(K_{\omega})=\begin{cases}\mathbb Z,&\omega=\varnothing,\\
0,&\text{ otherwise.}\end{cases}$
\end{enumerate}
\end{proposition}
\begin{corollary}\label{zpgc}
For $k\geqslant 0$ we have 
$$
H^k(\mathcal{Z}_P)=\bigoplus_{\omega}\widetilde{H}^{k-1-|\omega|}(P_{\omega}).
$$
More precisely, 
$$
H^0(\mathcal{Z}_P)=\widetilde{H}^{-1}(\varnothing)=\mathbb Z=\widetilde{H}^{n-1}(P_{[m]})=H^{m+n}(\mathcal{Z}_P),
$$ 
and for $0<k<m+n$ we have
$$
H^k(\mathcal{Z}_P)=\bigoplus\limits_{\max\{\lceil\frac{k+1}{2}\rceil,k-n+1\}\leqslant|\omega|\leqslant \min\{k-1,m-1\}, G(\omega)=\varnothing}\widetilde{H}^{k-1-|\omega|}(P_{\omega}).
$$
In particular,
\begin{gather*}
H^1(\mathcal{Z}_P)=H^2(\mathcal{Z}_P)=0=H^{m+n-2}(\mathcal{Z}_P)=H^{m+n-1}(\mathcal{Z}_P);\\
H^3(\mathcal{Z}_P)\simeq \bigoplus\limits_{|\omega|=2}\widetilde H^0(P_\omega)=\bigoplus\limits_{F_i\cap F_j=\varnothing}\mathbb Z\simeq H^{m+n-3}(\mathcal{Z}_P);\\
H^4(\mathcal{Z}_P)\simeq \bigoplus\limits_{|\omega|=3}\widetilde H^0(P_\omega)\simeq H^{m+n-4}(\mathcal{Z}_P);\\
H^5(\mathcal{Z}_P)\simeq \bigoplus\limits_{|\omega|=3}\widetilde H^{1}(P_\omega)+\bigoplus\limits_{|\omega|=4}\widetilde H^{0}(P_\omega)\simeq H^{m+n-5}(\mathcal{Z}_P);\\
H^6(\mathcal{Z}_P)\simeq \bigoplus\limits_{|\omega|=4}\widetilde H^{1}(P_\omega)+\bigoplus\limits_{|\omega|=5}\widetilde H^{0}(P_\omega);\\
H^7(\mathcal{Z}_P)\simeq \bigoplus\limits_{|\omega|=4}\widetilde H^{2}(P_\omega)+\bigoplus\limits_{|\omega|=5}\widetilde H^{1}(P_\omega)+\bigoplus\limits_{|\omega|=6}\widetilde H^{0}(P_\omega).
\end{gather*}
\end{corollary}
\begin{proof}
From Proposition \ref{Pocor} we obtain 
$$
H^k(\mathcal{Z}_P)=\bigoplus\limits_{2|\omega|-i=k}H^{-i,2\omega}(\mathcal{Z}_P)\simeq\bigoplus\limits_{2|\omega|-i=k}\widetilde{H}^{|\omega|-i-1}(P_{\omega})=\bigoplus\limits_{|\omega|\leqslant k}\widetilde{H}^{k-|\omega|-1}(P_{\omega}).
$$
If $|\omega|=0$, then  $\widetilde{H}^{k-|\omega|-1}(P_{\omega})=\widetilde{H}^{k-1}(\varnothing)=\begin{cases}\mathbb Z,&k=0,\\
0,&\text{ otherwise}.\end{cases}$\\
If $|\omega|=k$, then $\widetilde{H}^{k-|\omega|-1}(P_{\omega})=\widetilde{H}^{-1}(P_{\omega})=\begin{cases}\mathbb Z,&k=0,\\
0,&\text{ otherwise}.\end{cases}$\\
Thus we have $H^0(\mathcal{Z}_P)=\widetilde{H}^{-1}(\varnothing)=\mathbb Z$, and for $k>0$ nontrivial summands appear only for $0<|\omega|<k$, and $k-1-|\omega|\leqslant \dim K_{\omega}\leqslant\min\{n-1,|\omega|-2\}$. Hence $|\omega|\geqslant\max\{k-n,\left\lceil\frac{k+1}{2}\right\rceil\}$.\\
If $|\omega|=k-n$, then $\widetilde{H}^{k-|\omega|-1}(P_{\omega})=\widetilde{H}^{n-1}(P_{\omega})=\begin{cases}\mathbb Z,
&|\omega|=m, k=m+n,\\0,&\text{ otherwise}.\end{cases}$ \\   
If $k=m+n$, then $|\omega|\geqslant m$; hence $|\omega|=m$, \,$H^{m+n}(\mathcal{Z}_P)=\widetilde{H}^{n-1}(\partial P)=\mathbb Z$. \\
If $|\omega|=m$, then $\widetilde{H}^{k-|\omega|-1}(P_{\omega})=\widetilde{H}^{k-m-1}(\partial P)=\begin{cases}\mathbb Z,&k=m+n,\\
0,&\text{ otherwise}.\end{cases}$\\

Thus, for $0<k<m+n$ nontrivial summands appear only for 
$$
\max\left\{k-n+1,\left\lceil\frac{k+1}{2}\right\rceil\right\}\leqslant |\omega|\leqslant \min\{k-1,m-1\}.
$$ 
If $|\omega|=1$, then $\widetilde{H}^{k-|\omega|-1}(P_{\omega})=0$ for all $k$.\\
If $|\omega|=2$, then $
\widetilde{H}^{k-|\omega|-1}(P_{\omega})=\begin{cases}\mathbb Z,&k=3\text{ and }G(\omega)=\varnothing,\\
0,&\text{ otherwise}.\end{cases}$\\
Thus, for $k=3,4,5,6,7$ we have the left parts of formulas above; in particular the corresponding cohomology groups have no torsion. From the universal coefficients formula the homology groups $H_k(\mathcal{Z}_P)$, $k\leqslant 5$, have no torsion. Then the right parts follow from the Poincare duality.
\end{proof}
\begin{corollary}
If the group $H^k(\mathcal{Z}_P)$ has torsion, then  $7\leqslant k\leqslant m+n-6$.
\end{corollary}

\newpage
\section{Lecture 6. Moment-angle manifolds of $3$-polytopes}
\subsection{Corollaries of general results}
From Corollary \ref{zpgc} for a $3$-polytope $P$ we have 
\begin{proposition}
\begin{gather*}
H^0(\mathcal{Z}_P)=\widetilde{H}^{-1}(\varnothing)=\mathbb Z=\widetilde{H}^{2}(P_{[m]})=H^{m+3}(\mathcal{Z}_P);\\
H^1(\mathcal{Z}_P)=H^2(\mathcal{Z}_P)=0=H^{m+1}(\mathcal{Z}_P)=H^{m+2}(\mathcal{Z}_P);\\
H^3(\mathcal{Z}_P)\simeq \bigoplus\limits_{|\omega|=2}\widetilde H^0(P_\omega)=\bigoplus\limits_{F_i\cap F_j=\varnothing}\mathbb Z\simeq H^m(\mathcal{Z}_P);\\
H^4(\mathcal{Z}_P)\simeq \bigoplus\limits_{|\omega|=3, G(\omega)=\varnothing}\widetilde{H}^0(P_\omega)\simeq H^{m-1}(\mathcal{Z}_P);\\
H^k(\mathcal{Z}_P)\simeq\oplus\bigoplus\limits_{|\omega|=k-2}\widetilde{H}^{1}(P_{\omega})\oplus \bigoplus\limits_{|\omega|=k-1}\widetilde{H}^{0}(P_{\omega}),   5\leqslant k\leqslant m-2.
\end{gather*}
In particular, $H^*(\mathcal{Z}_P)$ has no torsion, and so $H^k(\mathcal{Z}_P)\simeq H^{m+3-k}(\mathcal{Z}_P)$.
\end{proposition}
\begin{proposition}
For a $3$-polytope $P$ nonzero Betti numbers could be
\begin{gather*}
{\rm rank\,} \widetilde{H}^{-1}(\varnothing)=\beta^{0,2\varnothing}=1=\beta^{-(m-3),2[m]}={\rm rank\,}\widetilde{H}^2(\partial P);\\
={\rm rank\,}\widetilde{H}^0(P_{\omega})=\beta^{-i,2\omega}=\beta^{-(m-3-i),2([m]\setminus\omega)}={\rm rank\,}\widetilde{H}^1(P_{[m]\setminus\omega}),\\
|\omega|=i+1, i=1,\dots,m-4.
\end{gather*}
\end{proposition}
The proof we leave as an exercise. 

For $|\omega|=i+1$ the number $\beta^{-i,2\omega}+1$ is equal to the number\\
of connected components of the set $P_{\omega}\subset P$.
\begin{definition}
{\em Bigraded Betti numbers}\index{bigraded Betti numbers} are defined as 
$$
\beta^{-i,2j}={\rm rank\,}H^{-i,2j}(\mathcal{Z}_P)=\sum\limits_{|\omega|=j}\beta^{-i,2\omega}.
$$
\end{definition}
\noindent{\bf Exercise:} $\beta^{-1,4}=\frac{m(m-1)}{2}-f_1=\frac{(m-3)(m-4)}{2}$.

\begin{proposition}
Let $\omega\subset[m]$ and $P_{\omega}$ be connected. Then topologically $P_{\omega}$ is a sphere with $k$ holes bounded by connected components $\eta_i$ of $\partial P_{\omega}$, which are simple edge cycles. 
\end{proposition}
\begin{proof}
It is easy to prove that $P_{\omega}$ is an orientable $2$-manifold with boundary, which proves the statement.
\end{proof}
Let the $3$-polytope $P$ have the standard orientation induced from $\mathbb R^3$, and the boundary $\partial P$ have the orientation induced from $P$ by the rule: the basis $(\boldsymbol{e}_1,\boldsymbol{e}_2)$ in $\partial P$ is positively oriented if and only if the basis 
$(n,\boldsymbol{e}_1,\boldsymbol{e}_2)$ is positively oriented in $P$, where $n$ is the outer normal vector. Then any set $P_{\omega}$ is an oriented surface with the boundary $\partial P_{\omega}$ consisting of simple edge cycles. Describe the Poincare duality given by Theorem  \ref{PoiTh}.
We have the orientation of simplices in $K_{\omega}$ defined by the canonical order of the vertices induced from the set $\omega\subset [m]$. We have the cellular structure on $P_{\omega}$ defined by vertices, edges and facets of $P$.
Orient the faces of $P$ by the following rule: 
\begin{itemize}
\item facets $F_i$ orient similarly to $\partial P$;
\item for $i<j$ orient the edge $F_i\cap F_j$ in such a way that the pair of vectors $(F_i\cap F_j,\boldsymbol{y}_{\{j\}}-\boldsymbol{y}_{\{i,j\}})$ has positive orientation in $F_j$; 
\item for $i<j<k$ assign <<$+$>> to the vertex   $F_i\cap F_j\cap F_k$, if the pair of vectors $(\boldsymbol{y}_{\{j,k\}}-\boldsymbol{y}_{\{i,j,k\}},\boldsymbol{y}_{\{k\}}-\boldsymbol{y}_{\{i,j,k\}})$ has positive orientation in $F_k$, and <<$-$>> otherwise.
\end{itemize}
\begin{corollary}
The mapping 
$$
C^i(K_{\omega})\to C_{2-i}(P_{\omega},\partial P_{\omega}),\quad \sigma^*\to G(\sigma)
$$ 
defines an isomorphism  
$$
\widetilde{H}^i(K_{\omega})\simeq \widehat{H}_{2-i}(P_{\omega},\partial P_{\omega}).
$$ 
\end{corollary}

We have the following computations. 
\begin{proposition}\label{B3}
For the set $\omega$ let $P_{\omega}=P_{\omega^1}\sqcup\dots\sqcup P_{\omega^s}$
be the decomposition into connected components. Then
\begin{enumerate}
\item $H_0(P_{\omega},\partial P_{\omega})=0$ for $\omega\ne [m]$, and $H_0(\partial P,\varnothing)=\mathbb Z$ for $\omega=[m]$ with the basis $[v]$, where $v\in P$ is any vertex with the orientation <<$+$>>.
\item $H_1(P_{\omega},\partial P_{\omega})=\bigoplus\limits_{i=1}^s H_1(P_{\omega^i},\partial P_{\omega^i})$, and $H_1(P_{\omega^i},\partial P_{\omega^i})\simeq \mathbb Z^{q_i-1}$, where $q_i$ is the number of cycles in $\partial P_{\omega^i}$. The basis is given by any set of edge paths in $P_{\omega^i}$ connecting one fixed boundary cycle with other boundary cycles.
\item $H_2(P_{\omega},\partial P_{\omega})/(\sum\limits_{i\in\omega}[F_i])\simeq \mathbb Z^s/(1,1,\dots,1)$, where $\mathbb Z^s$ has the basis\\ $e_{\omega^j}=[\sum\limits_{i\in\omega^j}F_i]$.
\end{enumerate}
\end{proposition}
The nontrivial multiplication is defined by the following rule. Each set $P_{\omega^j}$ is a sphere with holes. If $\omega_1\cap \omega_2=\varnothing$, then $P_{\omega_1^i}\cap P_{\omega_2^j}$ is the intersection of a boundary cycle in $\partial P_{\omega_1^i}$ with a boundary cycle in $\partial P_{\omega_2^j}$, which is the union $\gamma_1\sqcup\dots\sqcup\gamma_l$  of edge-paths. 
\begin{proposition}
We have $e_{\omega_1^i} \cdot e_{\omega_2^j}=0$, if $P_{\omega_1^i}\cap P_{\omega_2^j}=\varnothing$. Else up to the sign $(-1)^{l(\omega_1,\omega_2)+|\omega_1|}$ it is the sum of the elements $[\gamma_i]$ given by the paths with the orientations such that an edge on the path and the transversal edge lying in one facet and oriented from $P_{\omega_1^i}$ to $P_{\omega_2^j}$ form positively oriented pair of vectors. 
\end{proposition}
\begin{proof}    
For the facets $F_i\in P_{\omega_1}$ and $F_j\in P_{\omega_2}$ we have 
$$
F_i\otimes F_j\to (-1)^{l(\omega_1,\omega_2)+|\omega_1|}(-1)^{l(i,j)}F_i\cap F_j,
$$
where the pair of vectors $\left((-1)^{l(i,j)}F_i\cap F_j, \boldsymbol{y}_j-\boldsymbol{y}_{\{i,j\}}\right)$ is positively oriented in $F_j$. 
\end{proof}
\begin{proposition}
Let $\omega_1\sqcup\omega_2=[m]$, and let the element $[\gamma]$ correspond to the oriented edge path $\gamma$, connecting two boundary cycles of $P_{\omega_2^j}$. Then $e_{\omega^i_1}\cdot [\gamma]=0$, if $P_{\omega_1^i}\cap \gamma=\varnothing$,  and up to the sign $(-1)^{l(\omega_1,\omega_2)}$ it is $+1$, if $\gamma$ starts at $P_{\omega^i_1}$, and $-1$, if $\gamma$ ends at $P_{\omega^i_1}$.  
\end{proposition}
\begin{proof}
$$
F_i\otimes (F_j\cap F_k)\to(-1)^{l(\omega_1,\omega_2)}(-1)^{l(i,\{j,k\})}F_i\cap F_j\cap F_k,
$$
where $(-1)^{l(i,\{j,k\})}F_i\cap F_j\cap F_k$ is the vertex $F_i\cap F_j\cap F_k$ with the sign $+$,  if  $F_j\cap F_k$ starts at $F_i$, and $-$, if $F_j\cap F_k$ ends at  $F_i$. 
\end{proof}
\subsection{$k$-belts and Betti numbers}
\begin{definition} For any $k$-belt $\mathcal{B}_k=\{F_{i_1},\dots,F_{i_k}\}$ define $\omega(\mathcal{B}_k)=\{i_1,\dots,i_k\}$, and   
$\widetilde{\mathcal{B}_k}$ to be the generator in the group 
$$
\mathbb Z\simeq H^{-(k-2),2\omega}(\mathcal{Z}_P)\simeq H^1(P_{\omega})\simeq H^1(K_{\omega})\simeq H_1(P_{\omega},\partial P_{\omega}),
$$
where $\omega=\omega(\mathcal{B}_k)$. 
\end{definition}
\begin{remark}
It is easy to prove that $\mathcal{B}_k$ is a $k$-belt if and only if $K_{\omega(\mathcal{B}_k)}$ is combinatorially equivalent to the boundary of a $k$-gon.
\end{remark}
Let $P$ be a simple $3$-polytope with $m$ facets.
\begin{proposition}\label{3bb}
Let $\omega=\{i,j,k\}\subset[m]$. Then 
$$
H^{-1,2\omega}(\mathcal{Z}_P)=\begin{cases}\mathbb Z,&(F_i,F_j,F_k) \text{ is a $3$-belt},\\
0,&\text{ otherwise }.
\end{cases}
$$ 
In particular, $\beta^{-1,6}$ is equal to the number of $3$-belts, and the set of elements $\{\widetilde{\mathcal{B}_3}\}$ is a basis in $H^{-1,6}(\mathcal{Z}_P)$.
\end{proposition}
\begin{proof}
We have $H^{-1,2\omega}(\mathcal{Z}_P)\simeq \widetilde{H}^1(K_{\omega})$. Consider all possibilities for the simplicial complex
$K_{\omega}$ on $3$ vertices. If $\{i,j,k\}\in K_{\omega}$, then $K_{\omega}$ is a $3$-simplex, and it is contractible. Else $K_{\omega}$ is a graph. If $K_{\omega}$ has no cycles, then each connected component is a tree, else $K_{\omega}$ is a cycle with $3$ vertices. This proves the statement.
\end{proof}
\begin{proposition}\label{4bb}
Let $P$ be a simple $3$-polytope without $3$-belts,
and $\omega\subset[m]$, $|\omega|=4$. Then 
$$
H^{-2,2\omega}(\mathcal{Z}_P)=\begin{cases}\mathbb Z,&\omega=\omega(\mathcal{B}) \text{ for some $4$-belt $\mathcal{B}$},\\
0,&\text{ otherwise},
\end{cases}
$$ 
where the belt $\mathcal{B}$ is defined in a unique way (we will denote it $\mathcal{B}(\omega)$).
In particular, $\beta^{-2,8}$ is equal to the number of $4$-belts, and the set of elements $\{\widetilde{\mathcal{B}_4}\}$ is a basis in $H^{-2,8}(\mathcal{Z}_P)$.
\end{proposition}
\begin{proof}
We have $H^{-2,2\omega}(\mathcal{Z}_P)\simeq \widetilde{H}^1(K_{\omega})$. Consider the $1$-skeleton $K^1_{\omega}$.
If it has no cycles, then $K_{\omega}=K^1_{\omega}$ is  a disjoint union of trees. If $K^1_{\omega}$ has a $3$-cycle on vertices $\{i,j,k\}$, then $\{i,j,k\}\in K_{\omega}$. Let $l=\omega\setminus\{i,j,k\}$. $l$ is either disconnected from $\{i,j,k\}$, or connected to it by one edge, or connected to it by two edges, say $\{i,l\}$ and $\{j,l\}$, with $\{i,j,l\}\in K_{\omega}$, or connected to it by three edges with $K_{\omega}\simeq\partial \Delta^3$. In all these cases $\widetilde{H}^1(K_{\omega})=0$. If $K^1_{\omega}$ has no $3$-cycles, but has a $4$-cycle $\{i,j\},\{j,k\},\{k,l\},\{l,i\}$, then $K_{\omega}$ coincides with this cycle and $(F_i,F_j,F_k,F_l)$ is a $4$-belt. This proves the statement.
\end{proof}

\begin{theorem}\label{5bb}
Let $P$ be a simple $3$-polytope without $3$-belts and $4$-belts,
and $\omega\subset[m]$, $|\omega|=5$. Then 
$$
H^{-3,2\omega}(\mathcal{Z}_P)=\begin{cases}\mathbb Z,&\omega=\omega(\mathcal{B}) \text{ for some $5$-belt $\mathcal{B}$},\\
0,&\text{ otherwise},
\end{cases}
$$ 
where the belt $\mathcal{B}$ is defined in a unique way (we will denote it $\mathcal{B}(\omega)$).
In particular, $\beta^{-3,10}$ is equal to the number of $5$-belts, and the set of elements $\{\widetilde{\mathcal{B}_5}\}$ is a basis in $H^{-3,10}(\mathcal{Z}_P)$.
\end{theorem}
\begin{proof}
We have $H^{-3,2\omega}(\mathcal{Z}_P)\simeq \widetilde{H}^1(K_{\omega})$. Since $\widetilde{H}^1(K_{\omega})=0$ for $|\omega|\leqslant 2$, from Propositions \ref{3bb} and \ref{4bb} we have $\widetilde{H}^1(K_{\omega})=0$, if $K_{\omega}$ is disconnected. Let it be connected. Consider the sphere with holes  $P_{\omega}$. If  $H^1(P_{\omega})\ne0$, then there are at least two holes. Consider a simple edge cycle $\gamma$ bounding one of the holes. Walking round $\gamma$ we obtain a $k$-loop $\mathcal{L}_k=(F_{i_1},\dots, F_{i_k})$, $k\geqslant 3$ in $P_{\omega}$. If $k=3$, then the absence of $3$-belts implies that $F_{i_1}\cap F_{i_2}\cap F_{i_3}$ is a vertex; hence $P_{\omega}=\{F_{i_1},F_{i_2}, F_{i_3}\}$, which is a contradiction. If $k=4$, then the absence of $4$-belts implies that $F_{i_1}\cap F_{i_3}\ne\varnothing$, or  $F_{i_2}\cap F_{i_4}\ne\varnothing$. Without loss of generality let $F_{i_1}\cap F_{i_3}\ne\varnothing$. Then $F_{i_1}\cap F_{i_2}\cap F_{i_3}$ and $F_{i_3}\cap F_{i_4}\cap F_{i_1}$ are vertices; hence $P_{\omega}=\{F_{i_1},F_{i_2}, F_{i_3},F_{i_4}\}$, which is a contradiction. Let $k=5$. If $\mathcal{L}_5$ is not a $5$-belt, then some two nonsuccessive facets intersect. They are adjacent to some facet of $\mathcal{L}_5$. Without loss of generality let it be $F_{i_2}$, and $F_{i_1}\cap F_{i_3}\ne\varnothing$.  Then $F_{i_1}\cap F_{i_2}\cap F_{i_3}$ is a vertex. The absence of $4$-belts implies  that $F_{i_3}\cap F_{i_5}\ne\varnothing$, or $F_{i_4}\cap F_{i_1}\ne\varnothing$. Without loss of generality let $F_{i_3}\cap F_{i_5}\ne\varnothing$. Then $F_{i_3}\cap F_{i_4}\cap F_{i_5}$ and $F_{i_1}\cap F_{i_3}\cap F_{i_5}$ are vertices, and $P_{\omega}$ is a disc bounded by $\gamma$. A contradiction.Thus $\mathcal{L}_5$ is a $5$-belt, and $H^1(P_{\omega})\simeq\mathbb Z$ generated by $\widetilde{\mathcal{L}_5}$. This proves the statement.
\end{proof}
\begin{proposition}\label{345b}
Any simple $3$-polytope $P\ne \Delta^3$ has either a $3$-belt, or a $4$-belt, or a $5$-belt.
\end{proposition}
\begin{proof}
If $P\ne\Delta^3$ has no $3$-belts, then it is a flag polytope and any facet of $P$ is surrounded by a belt. Theorem \ref{pkth} implies that any flag simple $3$-polytope  has a quadrangular or pentagonal facet. This finishes the proof.  
\end{proof}

\begin{corollary}
For a fullerene $P$
\begin{itemlist}
\item $\beta^{-1,6}=0$ -- the number of $3$-belts;
\item $\beta^{-2,8}=0$ -- the number of $4$-belts;
\item $\beta^{-3,10}=12+k$, $k\geqslant 0$, -- the number of $5$-belts. If $k>0$, then $p_6=5k$;
\item the product mapping $H^3(\mathcal{Z}_P)\otimes H^3(\mathcal{Z}_P) \to
H^6(\mathcal{Z}_P)$ is trivial.
\end{itemlist}

\end{corollary}

\subsection{Relations between Betti numbers}\index{polytope!relations between bigraded Betti numbers of $3$-polytopes}
\begin{theorem}(Theorem 4.6.2, \cite{Bu-Pa15}) For any simple polytope $P$ with $m$ facets
$$
(1-t^2)^{m-n}(h_0+h_1t^2+\dots+h_n t^{2n})=\sum\limits_{-i,2j}(-1)^i\beta^{-i,2j}t^{2j},
$$
where $h_0+h_1t+\dots+h_nt^n=(t-1)^n+f_{n-1}(t-1)^{n-1}+\dots+f_0$.
\end{theorem}

\begin{corollary}
Set $h=m-3$. For a simple $3$-polytope $P\ne\Delta^3$ with $m$ facets
\begin{multline*}
(1-t^2)^h(1+ht^2+ht^4+t^6)=\\
1-\beta^{-1,4}t^4+\sum\limits_{j=3}^{h}(-1)^{j-1}(\beta^{-(j-1),2j}-\beta^{-(j-2),2j})t^{2j}+\\
(-1)^{h-1}\beta^{-(h-1),2(h+1)}t^{2(h+1)}+(-1)^ht^{2(h+3)}.
\end{multline*}
\end{corollary}

\noindent{\bf Exercise:} For any simple $3$-polytope $P$ we have:
\begin{itemlist}
\item $\beta^{-1,4}$ -- the number of pairs $(F_i,F_j)$, $F_i\cap
    F_j=\varnothing$;
\item $\beta^{-1,6}$ -- the number of $3$-belts;
\item $\beta^{-2,6}=\sum\limits_{i<j<k}s_{i,j,k}$, where $s_{i,j,k}+1$ is
    equal to the number of connected components of the set $F_i\cup
    F_j\cup F_k$;
\item  $\beta^{-3,8}=\sum\limits_{i<j<k<r}s_{i,j,k,r}$, where
    $s_{i,j,k,r}+1$ is equal to the number of connected components of the set
    $F_i\cup F_j\cup F_k\cup F_r$.
\end{itemlist}

\begin{corollary}
For any simple $3$-polytope $P$
\begin{itemlist}
\item $\beta^{-1,4}=\frac{h(h-1)}{2}$;
\item $\beta^{-2,6}-\beta^{-1,6}=\frac{(h^2-1)(h-3)}{3}$;
\item $\beta^{-3,8}-\beta^{-2,8}=\frac{(h+1)h(h-2)(h-5)}{8}$.
\end{itemlist}
\end{corollary}
\begin{corollary}
For any fullerene
\begin{itemlist}
\item $\beta^{-1,4}=\frac{(8+p_6)(9+p_6)}{2}$;
\item $\beta^{-2,6}=\frac{(6+p_6)(8+p_6)(10+p_6)}{3}$;
\item $\beta^{-3,8}=\frac{(4+p_6)(7+p_6)(9+p_6)(10+p_6)}{8}$.
\end{itemlist}
\end{corollary}

\newpage
\section{Lecture 7. Rigidity for $3$-polytopes}\label{LRid}
\subsection{Notions of cohomological  rigidity}\index{cohomological rigidity}
\begin{definition}
Let $\mathfrak{P}$ be some set of polytopes.

We call a property of simple polytopes {\em rigid}\index{rigid!property} in $\mathfrak{P}$ if for any polytope $P\in \mathfrak{P}$ with this property the isomorphism of graded rings $H^*(\mathcal{Z}_P)\simeq H^*(\mathcal{Z}_Q)$, $Q\in \mathfrak{P}$ implies that $Q$ also has this property.

We call a set $\mathfrak{S}_P\subset H^*(\mathcal{Z}_P)$ defined for any polytope $P\in\mathfrak{P}$ of polytopes {\em rigid}\index{rigid!set} in $\mathfrak{P}$ if for any isomorphism $\varphi$ of graded rings $H^*(\mathcal{Z}_P)\simeq H^*(\mathcal{Z}_Q)$, $P,Q\in \mathfrak{P}$ we have $\varphi(\mathfrak{S}_P)=\mathfrak{S}_Q$.

We call a polytope {\em rigid} \index{rigid!polytope}\index{polytope!rigid}(or {\em $B$-rigid}) \index{polytope!$B$-rigid}\index{$B$-rigidity} in $\mathfrak{P}$, if any isomorphism of graded rings $H^*(\mathcal{Z}_P)\simeq H^*(\mathcal{Z}_Q)$, $Q\in \mathfrak{P}$, implies that $Q$ is combinatorially equivalent to $P$.
\end{definition}
In this lecture we follow mainly the works \cite{FW15} and \cite{FMW15}. Some results we mention without proof with the appropriate reference, for some results we give new proofs, and some results we prove in strengthened form.

\subsection{Straightening along an edge}
For any edge $F_p\cap F_q$  of a simple $3$-polytope $P$ there is an operation of {\em straightening along the edge}\index{straightening along an edge}\index{polytope!straightening along an edge} (see Fig. \ref{Unfold}). In this subsection we discuss its properties we will need below. 
\begin{figure}[h]
\includegraphics[height=4cm]{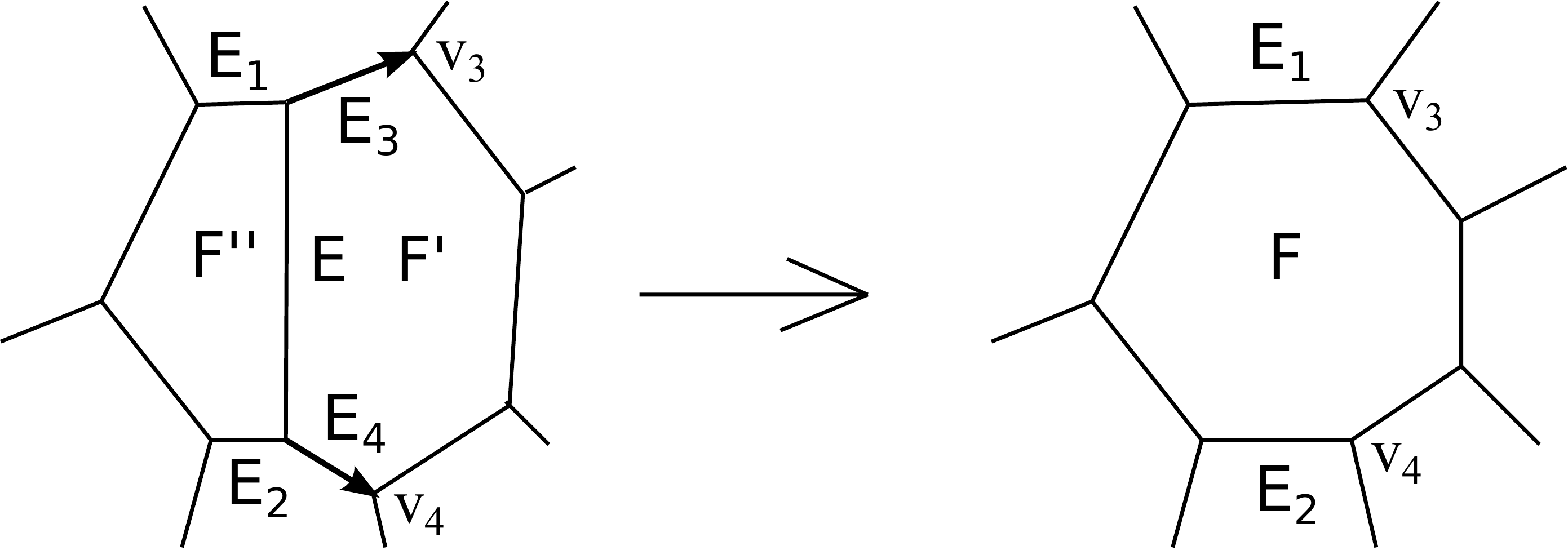}
\caption{Straightening along the edge}\label{Unfold}
\end{figure}
\begin{lemma}\cite{Bu-Er15}\label{cor-3-s}
For $P\simeq\Delta^3$ no straightening operations are defined. Let
$P\not\simeq \Delta^3$ be a simple polytope. The operation of straightening
along $E=F_i\cap F_j$ is not defined if and only if there is a $3$-belt
$(F_i,F_j,F_k)$ for some $F_k$.
\end{lemma}
\begin{proof}
For $P=\Delta^3$ straightening along any edge transforms triangle into double edge; hence it is not allowed. 

Let $P\not\simeq\Delta^3$ be a simple polytope. We have $2f_1=3f_0$; hence from the Euler formula we have $\frac{2f_1}{3}-f_1+f_2=2$, and $f_1=3(f_2-2)$. In particular $f_1\geqslant 9$ for $P\not\simeq \Delta^3$, and the graph $G'\subset S^2$ arising from the graph $G(P)$ under straightening along the edge has at least $6$ edges. We will use Lemma \ref{3C-lemma} to establish whether  $G'$ corresponds to a polytope (which will be simple by construction). Let $F_i\cap F_j$ be an edge we want to straighten along. Let $F_i\cap F_j\cap F_p$ and $F_i\cap F_j\cap F_q$ be it's vertices (see Fig. \ref{Fijpq}). 
\begin{figure}[h]
\begin{center}
\includegraphics[height=4cm]{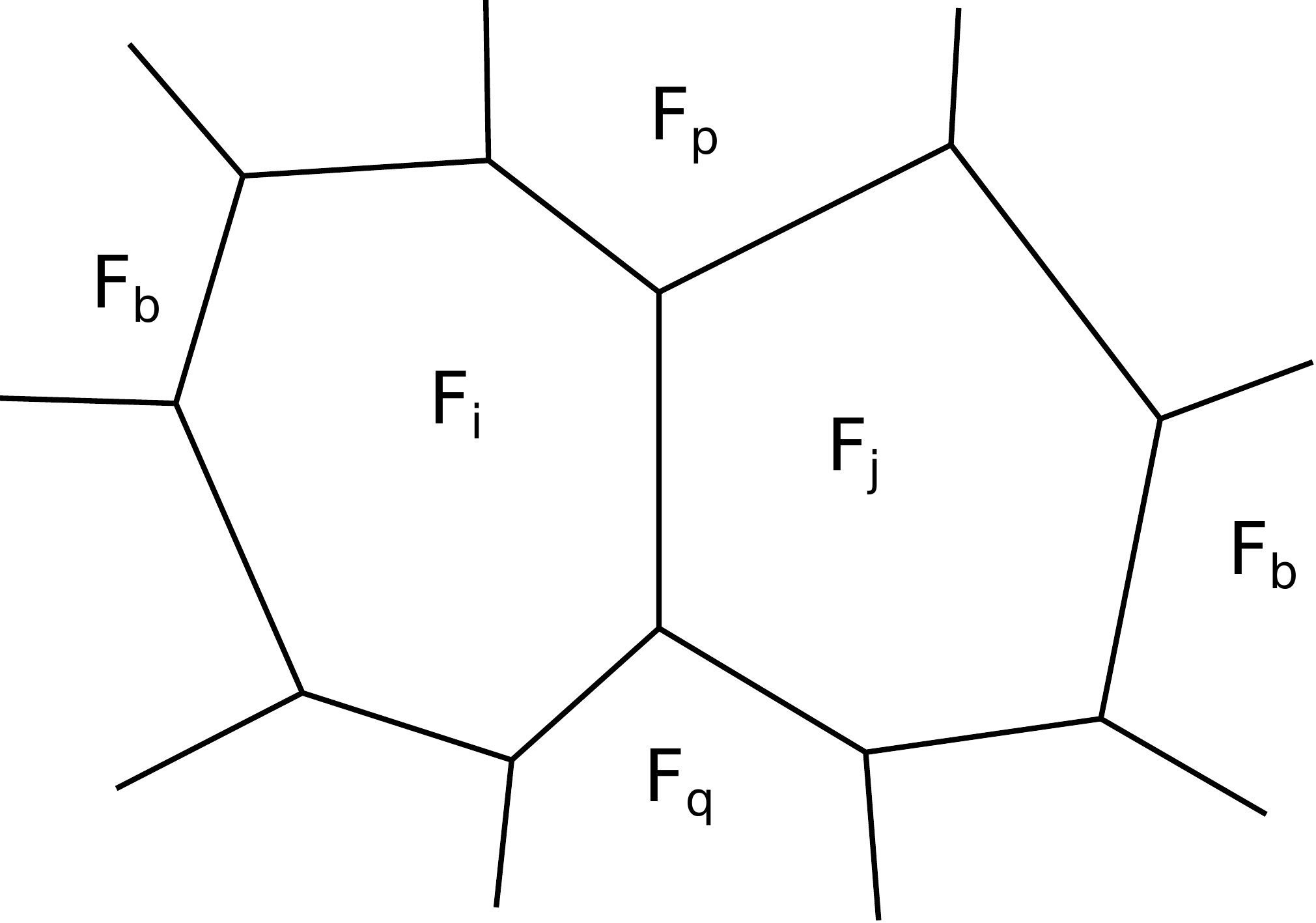}
\end{center}
\caption{Neighborhood of the edge $F_i\cap F_j$}\label{Fijpq}
\end{figure}
Then after straightening the number of edges of $F_p$ and $F_q$ decreases;  hence $F_p$ and $F_q$ should have at least $4$ edges. From construction any facet of $G'$, except for $F$, is surrounded by a simple edge-cycle. Since $F_i$ and $F_j$ intersect by a single edge, $F$ also is surrounded by a simple edge cycle. Let $F_a'$ and $F_b'$ be facets of $G'$. If $F_a',F_b'\ne F$, then  we have $F_a'\cap F_b'=F_a\cap F_b$, where $F_a$, $F_b$ are the corresponding facets of $P$. Hence this is either an empty set, or an edge. If $F_a'=F$, then we have $F'_a\cap F_b'$ consists of more than one edge if and only if the corresponding facet $F_b$ of $P$ intersects both $F_i$ and $F_j$, and $F_b\ne F_p,F_q$. This is equivalent to the fact that $(F_i,F_j,F_b)$ is a $3$-belt.  
\end{proof}
\begin{lemma}\cite{Bu-Er15}\label{Uflag}
The polytope $Q$ obtained by straightening a flag polytope $P$ along the
edge $F_i\cap F_j$ is not flag if and only if there is a $4$-belt
$(F_i,F_j,F_k,F_l)$.
\end{lemma} 
\begin{proof}
Since $P$ is flag, it has no triangles. Hence $Q\not\simeq\Delta^3$. Then $Q$ is not flag if and only if it contains a $3$-belt  $(F_a', F_b',F_c')$. If $F\notin \{F_a',F_b',F_c'\}$, then $(F_a,F_b,F_c)$ is a $3$-loop of $P$; hence $F_a\cap F_b\cap F_c$ is a vertex. Then $F_a'\cap F_b'\cap F_c'$ is also a vertex. A contradiction. Let $F_a'=F$. Then both $F_b$, $F_c$ intersect $F_i\cup F_j$, and $F_b\cap F_c\ne\varnothing$ in $P$.  If $(F_\alpha, F_b,F_c)$ is a $3$-loop in $P$ for $\alpha\in\{i,j\}$, then $F_\alpha\cap F_b\cap F_c$ is a vertex different from the ends of $F_i\cap F_j$. This vertex is also a vertex in $G'$ and is the intersection $F_a'\cap F_b'\cap F_c'$. A contradiction. Thus each of the facets $F_i$ and $F_j$ intersects only one facet among $F_b$ and $F_c$, and these facets are different. This is possible if and only if $(F_i,F_j,F_b,F_c)$ or $(F_i,F_j,F_c,F_b)$ is a $4$-belt.   
\end{proof}
\begin{remark}
Straightenings along opposite edges of the quadrangle give the same combinatorial polytope.
\end{remark}
\begin{lemma}\cite{Bu-Er15}\label{4gonF}
Let $P$ be a flag simple $3$-polytope and $F$ be its quadrangular facet. If $P\not\simeq I^3$, then  one of the $2$ combinatorial polytopes obtained by straightening along edges of $F$ is flag.
\end{lemma}
\begin{proof}
Let $F$ be surrounded by a $4$-belt $(F_{i_1},F_{i_2},F_{i_3},F_{i_4})$. If straightening along any edge of $F$ gives non-flag polytopes, then by Lemma \ref{Uflag} there are some $4$-belts $\mathcal{B}_1=(F_{i_1},F,F_{i_3},F_{i_5})$ and $\mathcal{B}_2=(F_{i_2},F,F_{i_4},F_{i_6})$. Since $\partial P\setminus\mathcal{B}_1$ consists of two connected components, one of them containing $F_{i_2}$ and the other containing $F_{i_4}$, the belt $\mathcal{B}_2$ should intersect the belt $\mathcal{B}_1$ by one more facet except for $F$. Hence $F_{i_5}=F_{i_6}$. Since $P$ is flag, we obtain vertices $F_{i_1}\cap F_{i_2}\cap F_{i_5}$, $F_{i_2}\cap F_{i_3}\cap F_{i_5}$, $F_{i_3}\cap F_{i_4}\cap F_{i_5}$, and $F_{i_4}\cap F_{i_1}\cap F_{i_5}$; hence all facets are quadrangles, and $P\simeq I^3$  
\end{proof}
\subsection{Rigidity of the property to be a flag polytope}
\begin{proposition} \label{Fab}(Lemma 5.2, \cite{FW15})
Let $P$ be a flag simple $3$-polytope. Then for any three different facets $\{F_i,F_j,F_k\}$ with $F_i\cap F_j=\varnothing$ there exist $l\geqslant 4$ and an $l$-belt $\mathcal{B}_l$ such that $F_i,F_j\in \mathcal{B}_l$ and $F_k\notin \mathcal{B}_l$.
\end{proposition}
\begin{proof} We will prove this statement by induction on the number of facets of $P$. By Proposition \ref{flagm6} for flag $3$-polytopes we have $m\geqslant 6$, and $m=6$ if and only if $P\simeq I^3$. In this case $F_i\cap F_j=\varnothing$ means that $F_i$ and $F_j$ are opposite facets. Then one of the two $4$-belts passing through $F_i$ and $F_j$ does not pass $F_k$.

Now let the statement be true for all flag $3$-polytopes with less than $m$ facets and let $P$ be a flag simple polytope with $m$ facets.  If $F_k\cap F_i\ne\varnothing$ and $F_k\cap F_j\ne\varnothing$, then we can take $\mathcal{B}_l$ to be the belt surrounding $F_k$. Now let at least one of the facets $F_i$, $F_j$ do not intersect $F_k$, say $F_i$. 

Consider facets  $F_p\notin\{F_i,F_j\}$ adjacent to $F_k$. If for some $p$ the facets $(F_p,F_k)$ do not belong to a $4$-belt, then straighten along $F_p\cap F_k$ to obtain a flag polytope $Q$ with $m-1$ facets. Then there is an $l$-belt $\mathcal{B}_l$, $F_i,F_j\in \mathcal{B}_l\not\ni F_k$ in $Q$. This belt corresponds to the belt on $P$ we need. Let for any $p\notin \{i,j\}$ there is a $4$-belt containing $(F_p,F_k)$. Since $F_k\not\simeq \Delta^2$, there are at least $3$ values of $p$; hence there are at least two $4$-belts. If any of these belts surrounds a quadrangle, then each quadrangle is adjacent to $F_k$. Consider the quadrangle $F_q$ different from $F_i$ and $F_j$ (it exists, since $F_k\cap F_i=\varnothing$ or $F_k\cap F_j=\varnothing$ by assumption).  By inductive hypothesis $m>6$,  $P\not\simeq I^3$; hence by Lemma \ref{4gonF} straightening along $F_q\cap F_k$, or any of the adjacent to $F_q\cap F_k$ edges of $F_q$ gives a flag polytope. By assumption the first case is not allowed. Consider the edge $F_p\cap F_q$ adjacent to $F_q\cap F_k$ in $F_q$ with $F_p\ne F_i,F_j$. After straightening along $F_p\cap F_q$ we obtain a flag polytope $Q$, which has a $k$-belt containing $F_i$, $F_j$ and not containing $F_k$. This belt corresponds either to a $k$-belt, or to a $(k+1)$-belt on $P$ with the same properties. 

At last consider the case when there is a $4$-belt $\mathcal{B}=(F_p,F_k,F_q,F_r)$, $p\notin\{i,j\}$, not surrounding a quadrangle.  By Lemma \ref{loop-cut-flag} $\mathcal{B}$-cuts $P_1$ and $P_2$ are flag polytopes. Since $\mathcal{B}$ is not surrounding a quadrangle, they have less facets than $P$.   If $F_i$ and $F_j$ belong to one of the polytopes, say $P_1$, then by the induction hypothesis there is an $l$-belt $\mathcal{B}_l$ with $F_i,F_j\in \mathcal{B}_l\not\ni F_k$. Consider the new facet $F$ of $P_1$. If $F\notin\mathcal{B}_l$, then $\mathcal{B}_l$ is an $l$-belt on $P$ and the Lemma is proved. Else since $F_k\notin \mathcal{B}_l$, the $l$-belt $\mathcal{B}_l$ contains the fragment $(F_p,F,F_q)$ and does not contain $F_r$. By induction hypothesis there is an $l'$-belt $\mathcal{B}_{l'}'$ on $P_2$, $F_p,F_q,\in \mathcal{B}_{l'}'\not\ni F_k$. Then the segment of $\mathcal{B}_{l'}'$ with ends $F_p$ and $F_q$ does not containing the new facet $F'$ and $F_r$, and the segment $\mathcal{B}_l\setminus\{F\}$ together form a belt we need. If any of the polytopes $P_1$ and $P_2$ contains exactly one of the facets $F_i$ and $F_j$, say $F_i\in P_1$, $F_j\in P_2$, then consider an $l_1$-loop $\mathcal{B}_1\subset P_1$, $F,F_i\in\mathcal{B}_1\not\ni F_k$, and an $l_2$-loop $\mathcal{B}_2\subset P_2$, $F',F_j\in\mathcal{B}_2\not\ni F_k$. Then $F_r\notin\mathcal{B}_1,\mathcal{B}_2$; hence $\left(\mathcal{B}_1\setminus(F_p,F,F_q)\right)\cup\left(\mathcal{B}_2\setminus\{F'\}\right)$ is a belt we need. This finishes the proof.

\end{proof}
\begin{corollary}\label{epi1} (Proposition 5.4, \cite{FW15})
Let $P$ be a flag simple $3$-polytope. Then for any $\omega\subset[m]$, $\omega\ne\varnothing$, the mapping 
$$
\bigoplus\limits_{\omega_1\sqcup\omega_2=\omega}\widehat{H}_2(P_{\omega_1},\partial P_{\omega_1})\otimes \widehat{H}_2(P_{\omega_2},\partial P_{\omega_2})\to H_1(P_{\omega},\partial P_{\omega})
$$
is an epimorphism.
\end{corollary}
\begin{proof}
We will use notations of Lemma \ref{belt-lemma}. Let $P_{\omega}=P_{\omega^1}\sqcup\dots\sqcup P_{\omega^s}$ be the decomposition into connected components. Consider $P_{\omega^r}$, $r\in\{1,\dots,s\}$. Let  $\partial P_{\omega^r}=\eta_1\sqcup\dots\sqcup \eta_t$ be the decomposition into boundary components. If $t=1$, then $P_{\omega^r}$ is a disk and is contractible. Let $t\geqslant 2$. Take $a\ne b$, and facets $F_{i_1}$ and $F_{i_2}$ in $\partial P\setminus P_{\omega^r}$ intersecting $\eta_a$ and $\eta_b$ respectively. By Proposition \ref{Fab} there is an $l$-belt $\mathcal{B}_l$  of the form $(F_{j_1},\dots, F_{j_l})$ with $F_{j_1}=F_{i_1}$, and $F_{j_p} = F_{i_2}$ for some $p$, $3\leqslant p\leqslant l-1$.   Set $\Pi_1=(F_{j_1},\dots, F_{j_p})$. Take 
\begin{gather*}
\omega_1=\{j\colon F_j\in \mathcal{B}_l\cap P_{\omega^r}\}, \;\omega_2=\omega\setminus\omega_1,\\
A=[\sum\limits_{F_j\in P_{\omega^r}\cap \Pi_1}F_j]\in\widehat H_2(P_{\omega_1},\partial P_{\omega_1}), \;B=[\sum\limits_{F_k\in P_{\omega^r}\cap \mathcal{W}_1}F_k]\in \widehat H_2(P_{\omega_2},\partial P_{\omega_2}).
\end{gather*} 
\begin{figure}
\begin{center}
\includegraphics[height=7cm]{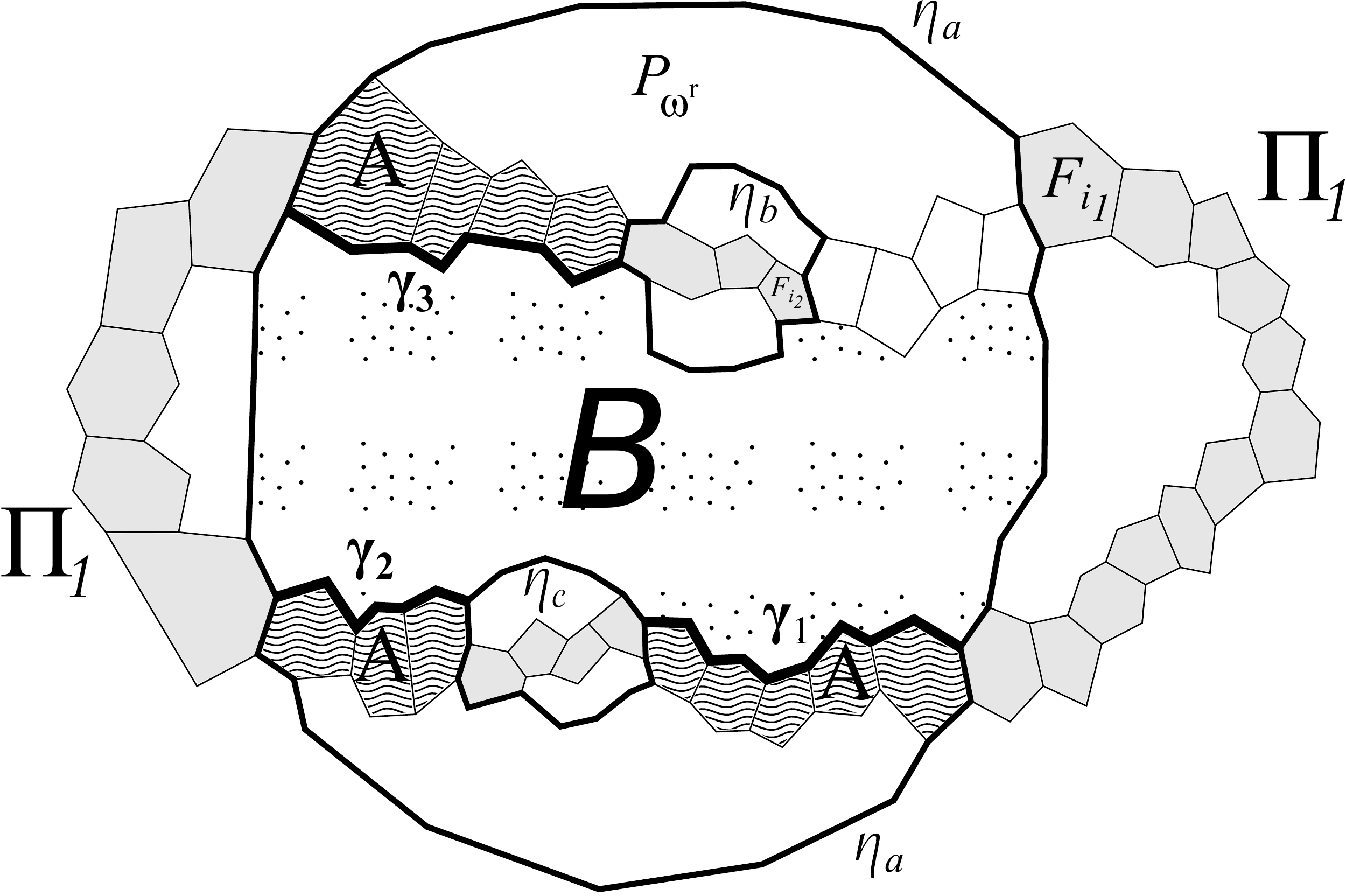}
\caption{The belt $\mathcal{B}_l$ intersecting $P_{\omega^r}$}
\end{center}
\end{figure}
Then $A\cdot  B=[\gamma_1]+\dots+[\gamma_q]$, where $\gamma_i$ is an edge path in $P_{\omega^r}$ that starts at $\eta_{\alpha_{i-1}}$ and ends at $\eta_{\alpha_i}$, $\alpha_j\in [s]$, $j=0,\dots, q$, $i=1,\dots,q$, and $\{\alpha_0,\alpha_q\}=\{a,b\}$. This element corresponds to a path connecting $\eta_a$ and $\eta_b$ in $P_{\omega^r}$. Thus we can realize any element from the basis given by Proposition \ref{B3}. 
\end{proof}
The following simple result is well-known.
\begin{lemma}\label{SR}
Simplex $\Delta^3$ is rigid in the class of all simple $3$-polytopes.
\end{lemma}
\begin{proof}
This is equivalent to the fact that any two facets intersect, that is $H^3(\mathcal{Z}_P)=0$.
\end{proof}
The following result follows from Theorem 5.7 in \cite{FW15}. We will give another proof here.
\begin{theorem}\label{Flagth}
The polytope  $P\ne\Delta^3$ is flag if and only if 
$$
H^{m-2}(\mathcal{Z}_P)\subset(\widetilde{H}^*(\mathcal{Z}_P))^2.
$$
\end{theorem}
\begin{proof}
The polytope $P\ne\Delta^3$ is not flag if and only if it has a $3$-belt. This corresponds to an element of a basis in $H^{-1,2\omega}(\mathcal{Z}_P)\simeq H_1(P_{\omega},\partial P_{\omega})$, $|\omega|=3$. By the Poincare duality this element corresponds to an element of a basis in $H^{-(m-4),2([m]\setminus\omega)}(\mathcal{Z}_P)\simeq \widehat{H}_2(P_{\omega},\partial P_{\omega})$. The latter element belongs to $H^{m-2}(\mathcal{Z}_P)$ but does not belong to $(\widetilde{H}^*(\mathcal{Z}_P))^2$. 

If the polytope is flag, then it has no $3$-belts, and by Proposition \ref{3bb} 
$$
H^5(\mathcal{Z}_P)=\bigoplus\limits_{|\omega|=4}H^{-3,2\omega}(\mathcal{Z}_P)=\bigoplus\limits_{|\omega|=4}\widehat{H}_2(P_{\omega},\partial P_{\omega}).
$$ 
Hence by the Poincare duality 
$$
H^{m-2}(\mathcal{Z}_P)=\bigoplus\limits_{|\omega|=m-4}H_1(P_{\omega},\partial P_{\omega}).
$$
By Corollary \ref{epi1} we have $H^{m-2}(\mathcal{Z}_P)\subset(\widetilde{H}^*(\mathcal{Z}_P))^2$. 

By Lemma \ref{SR} the simplex is a rigid polytope. This finishes the proof.
\end{proof}
\begin{corollary}\label{Flagcor}\index{rigidity!of the property to be a flag $3$-polytope}
The property to be a flag polytope is rigid in the class of simple $3$-polytopes.
\end{corollary}

\subsection{Rigidity of the property to have a $4$-belt}\index{rigidity!of the property to have a $4$-belt}
Remind that for any set $\omega=\{i,j\}\subset[m]$ we have 
$$
H^{-1,2\omega}(\mathcal{Z}_P)=\widehat{H}_2(P_{\omega},\partial P_{\omega})=\begin{cases}\mathbb Z\text{ with generator } [F_i]=-[F_j],&F_i\cap F_j=\varnothing,\\
0,&F_i\cap F_j\ne\varnothing,\end{cases}
$$ 
and 
$$
H^3(\mathcal{Z}_P)=\bigoplus\limits_{\{i,j\}\colon F_i\cap F_j=\varnothing}\mathbb Z
$$
\begin{definition}
The set $\{F_{i_1},\dots, F_{i_k}\}$ with $F_{i_1}\cap\dots\cap F_{i_k}=\varnothing$ is called a {\em nonface}\index{nonface}\index{polytope!nonface} of $P$, and the corresponding set $\{i_1,\dots,i_k\}$ -- a {\em nonface}\index{simplicial complex!nonface} of $K_P$. A nonface minimal by inclusion is called a {\em minimal}\index{nonface!minimal}\index{polytope!minimal nonface} nonface\index{simplicial complex!minimal nonface}.
Define $N(K)$ to be the set of all minimal nonfaces of the simplicial complex $K$.
\end{definition}
For any nonface $\omega=\{i,j\}$  choose a generator $\widetilde\omega\in H^{-1,2\omega}(\mathcal{Z}_P)$.

\begin{proposition}\label{4bpr}
The multiplication $H^3(\mathcal{Z}_P)\otimes H^3(\mathcal{Z}_P)\to H^6(\mathcal{Z}_P)$ is trivial if and only if $P$ has no $4$-belts. 
\end{proposition}
\begin{proof}
For $\omega_1=\{i,j\},\omega_2=\{p,q\}\in N(K_P)$, $\omega_1\cap \omega_2=\varnothing$, the simplicial complex $K_{\omega_1\sqcup\omega_2}$ has no $2$-simplices; hence it is at most $1$-dimensional and can be considered as a graph. Moreover, this graph has no $3$-cycles. If it has a $4$-cycle, then $K_{\omega_1\sqcup\omega_2}$ is a boundary of a $4$-gon, $(F_i,F_p,F_j,F_q)$ is a $4$-belt, and $\widetilde{\omega_1}\cdot\widetilde{\omega_2}$ is a generator of $H_1(P_{\omega_1\sqcup\omega_2},\partial P_{\omega_1\sqcup\omega_2})$. If $K_{\omega_1\sqcup\omega_2}$ has no $4$-cycles, then it has no cycles at all,  $H_1(P_{\omega_1\sqcup\omega_2},\partial P_{\omega_1\sqcup\omega_2})\simeq\widetilde{H}^1(K_{\omega_1\sqcup\omega_2})=0$, and $\widetilde{\omega_1}\cdot\widetilde{\omega_2}=0$.  This proves the statement.
\end{proof}
\begin{corollary}\label{4beltcor}
The property to have a $4$-belt is rigid in the class of all simple $3$-polytopes.
\end{corollary}
\subsection{Rigidity of flag $3$-polytopes without $4$-belts}
First we prove the following technical result, which we will need below.
\begin{proposition}  \label{Fabc}(Lemma 3.2, \cite{FMW15})
Let $P$ be a flag $3$-polytope without $4$-belts. Then for any three different facets $\{F_i,F_j,F_k\}$ with $F_i\cap F_j=\varnothing$ there exist $l\geqslant 5$ and an $l$-belt $\mathcal{B}_l$ such that $F_i,F_j\in \mathcal{B}_l$, $F_k\notin \mathcal{B}_l$, and $F_k$ does not intersect at least one of the two connected components of $\mathcal{B}_l\setminus\{F_i,F_j\}$.
\end{proposition}
\begin{remark} In \cite{FMW15} only the sketch of the proof is given. It contains several additional assumptions. We give the full prove following the same idea. 
\end{remark}
\begin{proof}
From Proposition \ref{Fab} there is an $s$-belt $\mathcal{B}_1$, with $F_i,F_j\in\mathcal{B}_1\not\ni F_k$. We have $\mathcal{B}_1=(F_i,F_{i_1},\dots,F_{i_p},F_j,F_{j_1},\dots,F_{j_q})$, $s=p+q+2$, $p,q\geqslant 1$. According to Lemma \ref{belt-lemma} the belt $\mathcal{B}_1$ divides the surface $\partial P\setminus\mathcal{B}_1$ into two connected components $\mathcal{P}_1$ and $\mathcal{P}_2$, both homeomorphic to disks. Consider the component $\mathcal{P}_{\alpha}$ containing ${\rm int}\,F_k$. Set $\beta=3-\alpha$. Then either $\partial \mathcal{P}_\alpha=\partial F_k$, or $\partial\mathcal{P}_\alpha\cap\partial F_k$ consists of finite set of disjoint edge-segments $\gamma_1$, $\dots$, $\gamma_d$. 
\begin{figure}[h]
\begin{center}
\includegraphics[height=4cm]{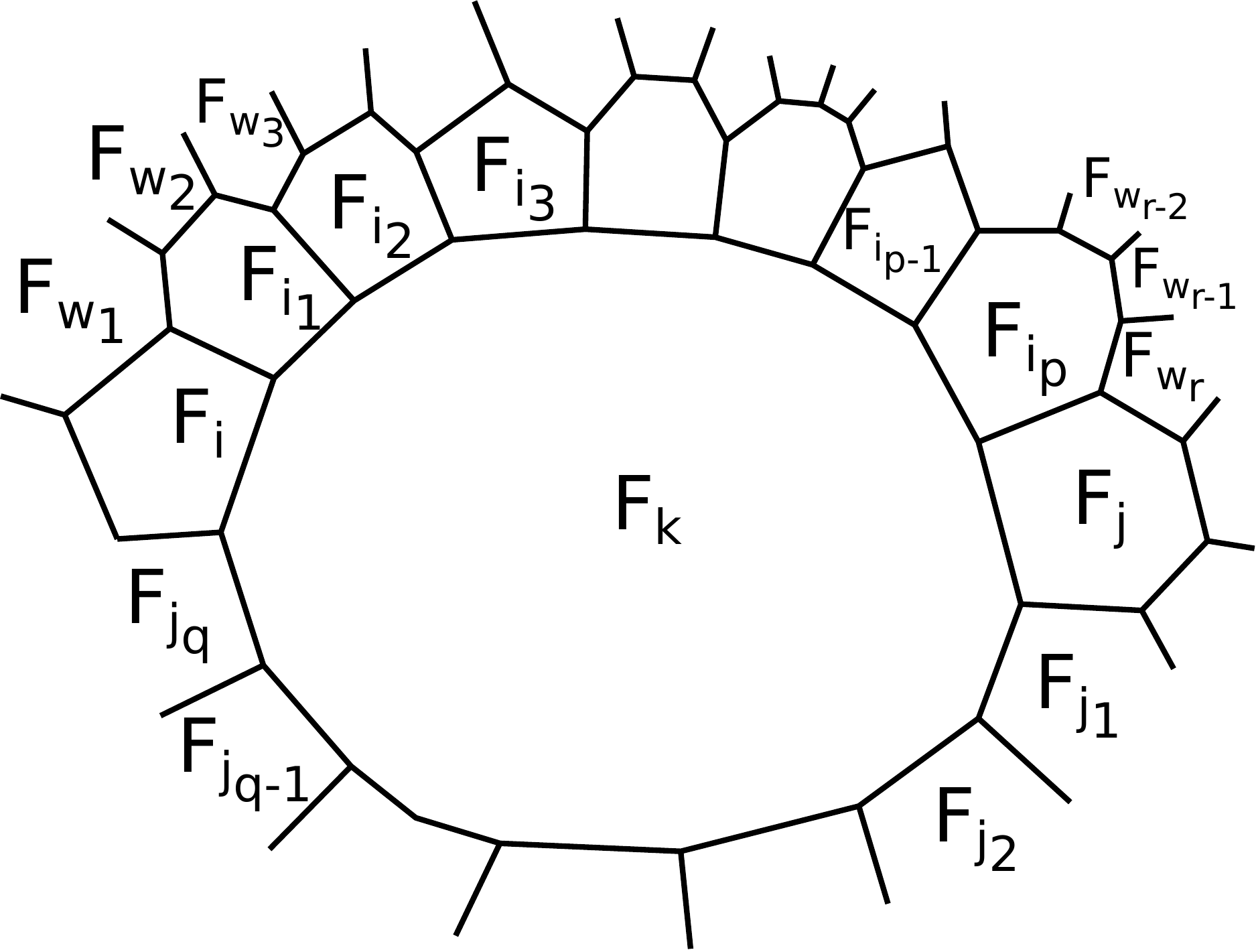}
\end{center}
\caption{Case 1}\label{bijk-1}
\end{figure}

Consider the first case. Then $\mathcal{B}_1$ surrounds $F_k$, and $F_i$ and $F_j$ are adjacent to $F_k$. Consider all facets $\{F_{w_1},\dots, F_{w_r}\}$ in $\mathcal{W}_{\beta}$ (in the notations of Lemma \ref{belt-lemma}),  adjacent to facets in $\{F_{i_1},\dots, F_{i_p}\}$ (see Fig. \ref{bijk-1}), in the order we meet them while walking round $\partial \mathcal{B}_1$ from $F_i$ to $F_j$. Then $F_{w_a}\cap F_{j_b}=\varnothing$ for any $a,b$, else $(F_k,F_{j_b},F_{w_a},F_{i_c})$ is a $4$-belt for any $i_c$ with $F_{i_c}\cap F_{w_a}\ne\varnothing$, since $F_k\cap F_{w_a}=\varnothing$ (because ${\rm int}\, F_{w_a}\subset \mathcal{P}_{\beta}$) and $F_{j_b}\cap F_{i_c}=\varnothing$. We have a thick path $(F_i, F_{w_1},\dots, F_{w_r},F_j)$. Consider the shortest thick path of the form $(F_i,F_{w_{s_1}},\dots,F_{w_{s_t}},F_j)$. If two facets of this path intersect, then they are successive, else there is a shorter thick path. Thus we have a belt $(F_i,F_{w_{s_1}},\dots, F_{w_{s_t}},F_j,F_{j_1},\dots, F_{j_q})$ containing $F_i$, $F_j$, not containing $F_k$, and the segment $(F_{w_{s_1}},\dots,F_{w_{s_t}})$ does not intersect $F_k$.

Now consider the second case. We can assume that $F_i\cap F_k=\varnothing$ or $F_j\cap F_k=\varnothing$, say $F_i\cap F_k\ne\varnothing$, else consider the belt $\mathcal{B}_1$ surrounding $F_k$ and apply the arguments of the first case. Let $\gamma_a=(F_k\cap F_{u_{a,1}},\dots,F_k\cap F_{u_{a,l_a}})$.   Set $\mathcal{U}_a=(F_{u_{a,1}},\dots, F_{u_{a,l_a}})$. The segment $(F_{s_{a,1}},\dots,F_{s_{a,t_a}})$ of $\mathcal{B}_1$ between $\mathcal{U}_a$ and $\mathcal{U}_{a+1}$ denote $\mathcal{S}_a$.  Then  $\mathcal{B}_1=(\mathcal{U}_1,\mathcal{S}_1,\mathcal{U}_2,\dots,\mathcal{U}_d,\mathcal{S}_d)$ for some $d$. 

Consider the thick path $W_a=(F_{w_{a,1}},\dots,F_{w_{a,r_a}})\subset \mathcal{W}_\beta$ (see notation in Lemma \ref{belt-lemma}) arising while walking round the facets in $\mathcal{W}_\beta$ intersecting facets in $\mathcal{U}_a$ (see Fig. \ref{bijk-2}). Then $W_a\cap W_b=\varnothing$ for $a\ne b$, else $(F_w,F_{u_{a,j_1}},F_k,F_{u_{b,j_2}})$ is a $4$-belt for any $F_w\in W_a\cap W_b$ such that $F_w\cap F_{u_{a,j_1}}\ne\varnothing$, $F_w\cap F_{u_{b,j_2}}\ne\varnothing$. Also $F_{w_{a,j_1}}\ne F_{w_{a,j_2}}$ for $j_1\ne j_2$. This is true for facets adjacent to the same facet $F_{u_{a,i}}$. Let $F_{w_{a,j_1}}= F_{w_{a,j_2}}$. If the facets are adjacent to the successive facets $F_{u_{a,i}}$ and $F_{u_{a,i+1}}$, then the flagness condition implies that $j_1=j_2$ and $F_{w_{a,j_1}}$ is the facet in $\mathcal{W}_\beta$ intersecting  $F_{u_{a,i}}\cap F_{u_{a,i+1}}$. If the facets are adjacent to non-successive facets $F_{u_{a,i}}$ and $F_{u_{a,j}}$, then $(F_{w_{a,j_1}},F_{u_{a,i}},F_k,F_{u_{a,j}})$ is a $4$-belt, which is a contradiction. 

Now consider the thick path $\mathcal{V}_b=(F_{v_{b,1}},\dots,F_{v_{b,c_b}})$ arising while walking round the facets in $\mathcal{W}_\alpha$ intersecting facets in $\mathcal{S}_b$ (see Fig. \ref{bijk-2}). Then $\mathcal{V}_a\cap\mathcal{V}_b=\varnothing$ for $a\ne b$, and  $W_a\cap \mathcal{V}_b=\varnothing$ for any $a,b$, since interiors of the corresponding facets lie in different connected components of $\partial P\setminus(\mathcal{B}_1\cup F_k)$, moreover by the same reason we have $F_{v_{a,j}}\cap F_{v_{b,j}}=\varnothing$ for any $i,j$, and $a\ne b$. 
\begin{figure}[h]
\begin{center}
\includegraphics[height=9cm]{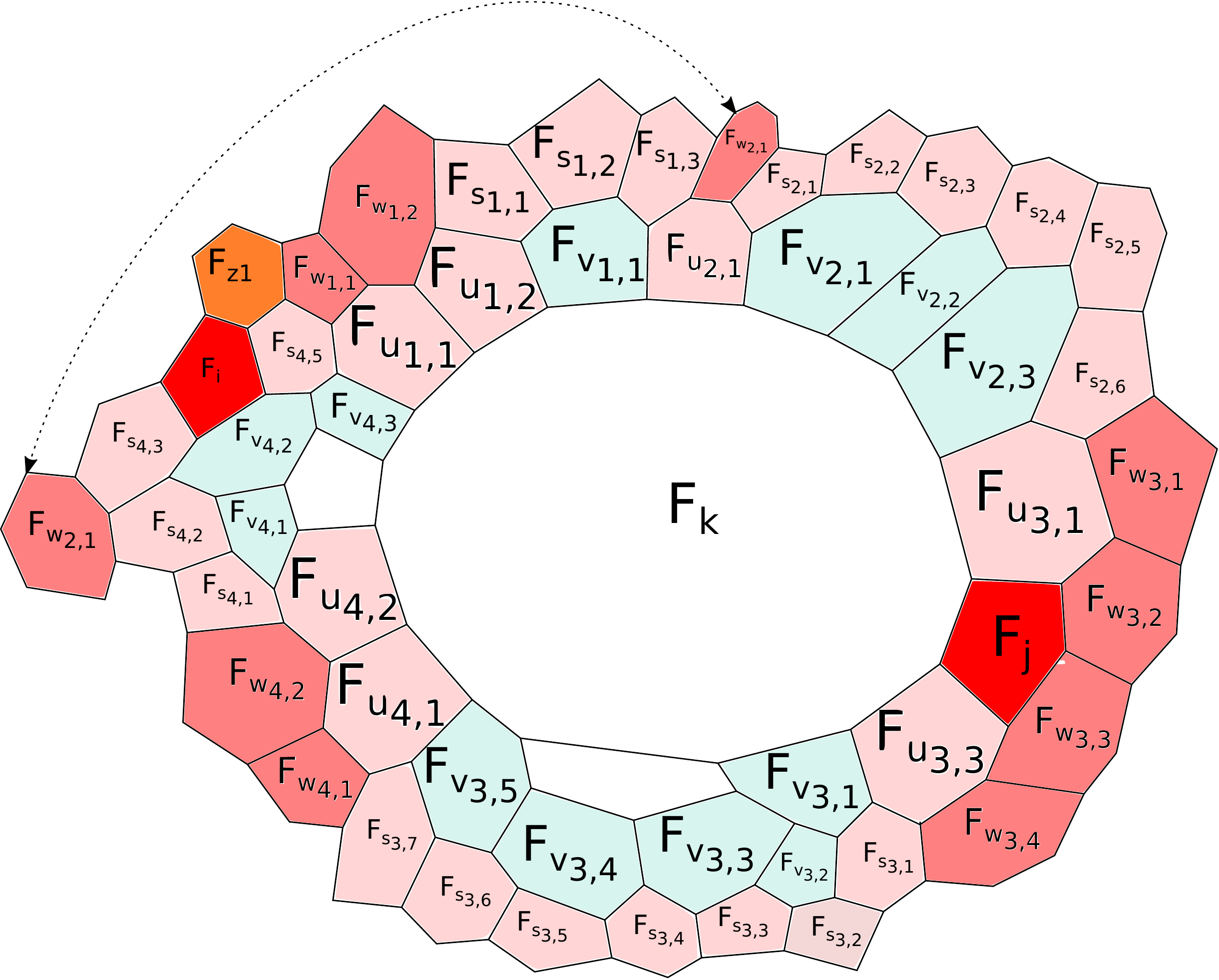}
\end{center}
\caption{Case 2}\label{bijk-2}
\end{figure}

Now we will deform the segments $\mathcal{I}=(F_{i_1},\dots,F_{i_p})$ and $\mathcal{J}=(F_{j_1},\dots,F_{j_q})$ of the belt $\mathcal{B}_1$ to obtain a new belt $(F_i,\mathcal{I}',F_j,\mathcal{J}')$ with $\mathcal{I}'$ not intersecting $F_k$. First substitute the thick path $W_a$ for each segment $\mathcal{U}_a\subset \mathcal{I}$ and the thick path $\mathcal{V}_b$ for each segment $\mathcal{S}_b\subset \mathcal{J}$. Since $F_{s_{a,t_a}}\cap F_{w_{a+1,1}}\ne\varnothing$, $F_{w_{a,r_a}}\cap F_{s_{a,1}}\ne\varnothing$, $F_{v_{a,c_a}}\cap F_{u_{a+1,1}}\ne\varnothing$, and $F_{u_{a,l_a}}\cap F_{s_{a,1}}\ne\varnothing$ for any $a$ and $a+1$ considered $\mod d$, we obtain a loop $\mathcal{L}_1=(F_i,\mathcal{I}_1,F_j,\mathcal{J}_1)$ instead of $\mathcal{B}_1$.  

Since $F_i\cap F_k=\varnothing$, we have $F_i=F_{s_{a_i,f_i}}$ for some $a_i,f_i$. If $F_j=F_{s_{a_j,f_j}}$ for some $a_j,f_j$, then we can assume that $a_i\ne a_j$, else the facets in $\mathcal{I}$ or $\mathcal{J}$ already do not intersect $F_k$, and $\mathcal{B}_1$ is the belt we need. If $F_j=F_{u_{a_j,f_j}}$ for some $a_j$ and some $f_j>1$, then substitute the thick path $(F_{w_{a_j,1}},\dots, F_{w_{a_j,g_j}})$, where $g_j$ -- the first integer with $F_{w_{a_j,g_j}}\cap F_j\ne\varnothing$ (then $F_j\cap F_{u_{a_j,f_j-1}}\cap F_{w_{a_j,g_j}}$ is a vertex), for the segment $(F_{u_{a_j,1}},\dots, F_{u_{a_j,f_j-1}})$ to obtain a loop $\mathcal{L}_2=(F_i,\mathcal{I}_2,F_j,\mathcal{J}_1)$ (else set $\mathcal{L}_2=\mathcal{L}_1$) with facets in $\mathcal{I}_2$ not intersecting $F_k$.  If $f_j<l_{a_j}$, then $F_{w_{a_j,g_j}}\cap F_{u_{a_j,f_j+1}}=\varnothing$, else $(F_k,F_{u_{a_j,f_j-1}},F_{w_{a_j,g_j}},F_{u_{a_j,f_j+1}})$ is a $4$-belt. Then $F_{w_{a,l}}\cap F_{u_{a_j,r}}=\varnothing$ for any $r\in \{f_j+1,\dots,l_{a_j}\}$ and $a,l$, such that either $a\ne a_j$, or $a=a_j$, and $l\in\{1,\dots,g_j\}$. Hence facets of the segment $(F_{u_{a_j,f_j+1}},\dots, F_{u_{a_j,l_{a_j}}})$ do not intersect facets in $\mathcal{I}_2$. 

Now a facet $F_{i_a'}$ of  $\mathcal{I}_2$ can intersect a facet $F_{j_b'}$ of $\mathcal{J}_1$ only if $F_{i_a'}=F_{w_{c,h}}$ for some $c,h$, and $F_{j_b'}=F_{s_{a_i,l}}$ for $l<f_i$, or $F_{j_b'}=F_{s_{a_j,l}}$ for $F_j=F_{s_{a_j,f_j}}$ and $l>f_j$.  
In the first case take the smallest  $l$ for all $c,h$, and the correspondent facet $F_{w_{c,h}}$. Consider the facet $F_{u_{b,g}}=F_{i_e}\in \mathcal{I}$ with $F_{u_{b,g}}\cap F_{w_{c,h}}\ne\varnothing$. Then $\mathcal{L}'=(F_{s_{a_i,l}},F_{s_{a_i,l+1}},\dots,F_i,F_{i_1},\dots,F_{i_e},F_{w_{c,h}})$ is a simple loop. If $f_i<t_{a_i}$, then consider the thick path $\mathcal{Z}_1=(F_{z_{1,1}},\dots,F_{z_{1,y_1}})$ arising while walking along the boundary of $\mathcal{B}_1$ in $\mathcal{W}_\beta$ from the facet $F_{z_{1,1}}$ intersecting $F_i\cap F_{i_1}$ by the vertex, to the facet $F_{z_{1,y_1}}$ preceding $F_{w_{a_i+1,1}}$. Consider the thick path $\mathcal{X}_1=(F_{v_{a_i,1}},\dots, F_{v_{a_i,x_1}})$ with $x_1$ being the first integer with $F_{v_{a_i,x_1}}\cap F_i\ne\varnothing$. Consider the simple curve $\eta\subset\partial P$ consisting of segments connecting the midpoints of the successive edges of intersection of the successive facets of $\mathcal{L}'$.  It divides $\partial P$ into two connected components $\mathcal{E}_1$ and $\mathcal{E}_2$ with $\mathcal{J}_1\setminus(F_{s_{a_i,l}},\dots,F_{s_{a_i,f_i-1}})$ lying in one connected component $\mathcal{E}_\alpha$, and $\mathcal{Z}_1$ -- in $\mathcal{E}_\beta\cup F_{w_{c,h}}$, $\beta=3-\alpha$.
Now substitute $\mathcal{X}_1$ for the segment $(F_{s_{a_i,1}},\dots, F_{s_{a_i,f_i-1}})$ of $\mathcal{J}_1$. If $f_i<t_{a_i}$
substitute $\mathcal{Z}_1$ for the segment $(F_{s_{a_i,f_i+1}},\dots, F_{s_{a_i,t_{a_i}}})$ of $\mathcal{I}_2$ to obtain a new loop $(F_i,\mathcal{I}_3,F_j,\mathcal{J}_2)$ with facets in $\mathcal{I}_3$ not intersecting $F_k$. A facet $F_{i_a}'' $ in $\mathcal{I}_3$ can intersect a facet $F_{j_b}''$ in $\mathcal{J}_2$ only if  $F_{i_a''}=F_{w_{c',h'}}$ for some $c',h'$, $F_j=F_{s_{a_j,f_j}}$, and $F_{j_b''}=F_{s_{a_j,l}}$ for $l>f_j$. 
The thick path $\mathcal{Z}_1$ lies in $\mathcal{E}_\beta\cup F_{w_{c,h}}$ and the segment $(F_j=F_{s_{a_j,f_j}},\dots, F_{s_{a_j,t_{a_j}}})$ lies in $\mathcal{E}_\alpha$; hence intersections of facets in $\mathcal{I}_3$ with facets in $\mathcal{J}_2$ are also intersections of the same facets in $\mathcal{I}_2$ and $\mathcal{J}_1$, and $F_{w_{c',h'}}$ is either $F_{w_{c,h}}$, or lies in $\mathcal{E}_\alpha$. We can apply the same argument for $\mathcal{S}_{a_j}$ as for $\mathcal{S}_{a_i}$ to obtain a new loop $\mathcal{L}_4=(F_i,\mathcal{I}_4,F_j,\mathcal{J}_3)$ with facets in $\mathcal{I}_4$ not intersecting $F_k$ and facets in $\mathcal{J}_3$. Then take the shortest thick path from $F_i$ to $F_j$ in $F_i\cup \mathcal{I}_4\cup F_j$ and the shortest thick path from $F_j$ to $F_i$ in $F_j\cup \mathcal{J}_3\cup F_i$ to obtain the belt we need. 

\begin{figure}[h]
\begin{center}
\includegraphics[height=9cm]{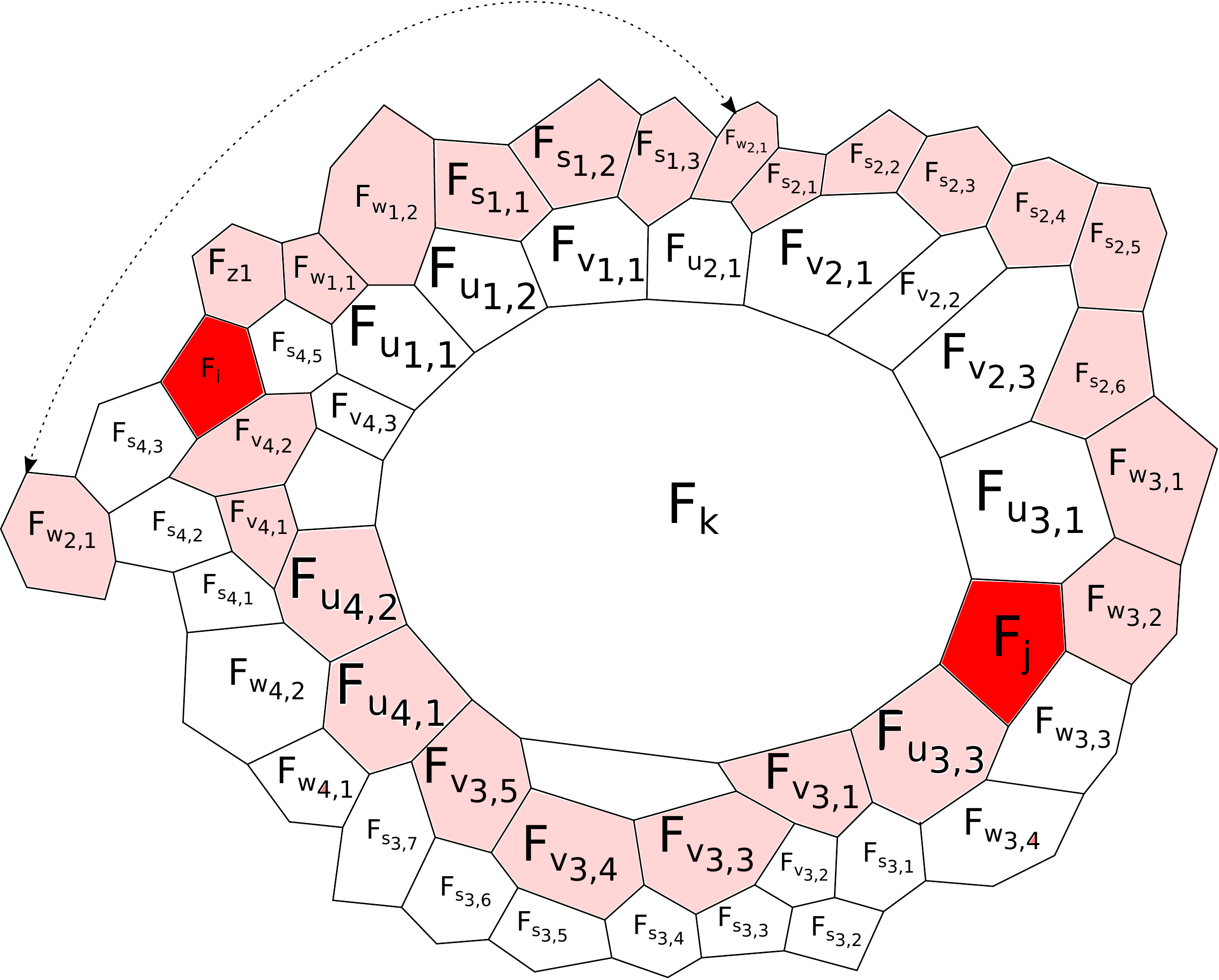}
\end{center}
\caption{Modified belt}\label{bijk-2}
\end{figure}

\end{proof}
\begin{definition}
An {\em annihilator}\index{annihilator} of an element $r$ in a ring $R$ is defined as 
$$
{\rm Ann}_R(r)=\{s\in R\colon rs=0\}
$$
\end{definition}
\begin{proposition}\label{2rigid}
The set of elements in $H^3(\mathcal{Z}_P)$ corresponding to 
$$
\bigcup\limits_{\{i,j\}\colon F_i\cap F_j=\varnothing}\{[F_i],[F_j]\in \widehat{H}_2(P_{\{i,j\}},\partial P_{\{i,j\}})\}
$$ 
is rigid in the class of all simple flag $3$-polytopes without $4$-belts.
\end{proposition}
\begin{proof}
Since the group $H^*(\mathcal{Z}_P)$ has no torsion, we have the isomorphism $H^*(\mathcal{Z}_P,\mathbb Q)\simeq H^*(\mathcal{Z}_P)\otimes\mathbb Q$ and the embedding $H^*(\mathcal{Z}_P)\subset H^*(\mathcal{Z}_P)\otimes\mathbb Q$. For polytopes $P$ and $Q$ the isomorphism $H^*(\mathcal{Z}_P)\simeq H^*(\mathcal{Z}_Q)$ implies the isomorphism over $\mathbb Q$. For the cohomology over $\mathbb Q$ all  theorems about structure of $H^*(\mathcal{Z}_P,\mathbb Q)$ are still valid. In what follows we consider cohomology over $\mathbb Q$. Set $H=H^*(\mathcal{Z}_P,\mathbb Q)$. We will need the following result.
\begin{lemma}\label{Ann2lemma}
For an element  
$$
\alpha=\sum\limits_{\omega\in N(K_P), |\omega|=2}r_{\omega}\widetilde\omega\quad\text{with }|\{\omega\colon r_{\omega}\ne 0\}|\geqslant 2
$$ 
we have 
$$
\dim {\rm Ann}_H (\alpha)<\dim {\rm Ann}_H (\widetilde \omega),\text { if }r_{\omega}\ne 0.
$$
\end{lemma}
\begin{proof} 
Choose a complementary subspace $C_{\omega}$ to ${\rm Ann}_H (\widetilde \omega)$ in $H$ as a direct sum of complements $C_{\omega,\tau}$ to ${\rm Ann}_H (\widetilde \omega)\cap \widehat{H}_*(P_{\tau},\partial P_{\tau})$ in $\widehat{H}_*(P_{\tau},\partial P_{\tau})$ for all $\tau\subset[m]\setminus\omega$. Then for any $\beta\in C_{\omega}\setminus\{0\}$ we have $\beta \widetilde{\omega}\ne0$, which is equivalent to the fact that $\beta=\sum\limits\beta_{\tau}$, $\beta_{\tau}\in C_{\omega,\tau}$, $\tau\subset[m]\setminus\omega$, with $\beta_{\tau_\beta}\widetilde\omega\ne0$ for some $\tau_\beta\subset[m]\setminus\omega$. Moreover for any $\omega'\ne\omega$ with $r_{\omega'}\ne 0$ and $\tau\subset[m]\setminus\omega$, $\tau\ne\tau_\beta$, we have $\tau_\beta\sqcup\omega\notin\{\tau\cup \omega',\tau_\beta\cup\omega',\tau\sqcup\omega\}$; hence $(\beta\cdot\alpha)_{\tau_{\beta}\sqcup\omega}= r_{\omega}\beta_{\tau_\beta}\cdot\widetilde{\omega}\ne0$, and $\beta\alpha\ne 0$. Then $C_{\omega}$ forms a direct sum with ${\rm Ann}(\alpha)$. Now consider some $\omega'\ne\omega$, $|\omega'|=2$, $r_{\omega'}\ne 0$. Let $\omega=\{p,q\}$, $\omega'=\{s,t\}$, $q\notin\omega'$. By Proposition \ref{Fabc} there is an $l$-belt $\mathcal{B}_l$ such that $F_s,F_t\in \mathcal{B}_l$,  $F_q\notin\mathcal{B}_l$, and $F_q$ does not intersect one of the two connected components $B_1$ and $B_2$ of $\mathcal{B}_l\setminus\{F_s,F_t\}$, say $B_1$. Take $\xi=[\sum\limits_{i\colon F_i\subset B_1}F_i]\in \widehat{H}_2(P_{\tau},\partial P_{\tau}),\,\tau=\{i\colon F_i\in \mathcal{B}_l\setminus\{F_s,F_t\}\}$,\linebreak and $[F_s]\in \widehat{H}_2(P_{\omega'},\partial P_{\omega'})$. Then $\xi \cdot [F_s]$ is a generator in $H_1(\mathcal{B}_l,\partial\mathcal{B}_l)\simeq \mathbb Z$. On the other hand, take $[F_q]\in\widehat{H}_2(P_{\omega},\partial P_{\omega})$. Then either $F_p\in \mathcal{B}_l\setminus\{F_s,F_t\}$, and $\xi \cdot\widetilde\omega=0$, since $\tau\cap \omega\ne\varnothing$,  or $F_p\notin \mathcal{B}_l\setminus\{F_s,F_t\}$, and $\pm \xi\cdot\widetilde\omega=\xi\cdot [F_q]=0$, since $F_q$ does not intersect $B_1$. In both cases  $\xi\in{\rm Ann}(\widetilde\omega)$ and $\xi\cdot\widetilde{\omega'}\ne 0$. Then $\xi\cdot \alpha\ne 0$, since $\tau\sqcup\omega'\ne\tau\sqcup\omega_1$ for $\omega_1\ne\omega'$. Consider any $\beta=\sum\limits_{\tau\subset[m]\setminus\omega}\beta_\tau\in C_{\omega}\setminus\{0\}$. We have  $(\beta\cdot\alpha)_{\tau_\beta\sqcup\omega}\ne 0$. If $(\xi\cdot\alpha)_{\tau_\beta\sqcup\omega}\ne 0$, then since $\xi$ is a homogeneous element, $(\xi\cdot\alpha)_{\tau_\beta\sqcup\omega}=r_{\omega_1}\xi\cdot\widetilde{\omega_1}$ for $\omega_1=(\tau_\beta\sqcup\omega)\setminus\tau=\{q,r\}$, $r\in[m]$. We have $\xi\cdot\widetilde{\omega_1}=\pm\xi\cdot[F_q]=0$, since $F_q$ does not intersect $B_1$. A contradiction. Thus, $((\xi+\beta)\cdot\alpha)_{\tau_\beta\sqcup\omega}=(\beta\cdot\alpha)_{\tau_\beta\sqcup\omega}\ne 0$; hence $(\xi+\beta)\cdot \alpha\ne0$, and the space $\langle\xi\rangle\oplus C_{\omega}$ forms a direct sum with ${\rm Ann}_H(\alpha)$. This finishes the proof. 
\end{proof}
Now let us prove Proposition \ref{2rigid}. Let $\varphi\colon H^*(\mathcal{Z}_P,\mathbb Z)\to H^*(\mathcal{Z}_Q,\mathbb Z)$ be an isomorphism of graded rings for flag simple $3$-polytopes $P$ and $Q$ without $4$-belts. Let $\omega\in N(K_P)$, $|\omega|=2$, and 
$$
\varphi(\widetilde\omega)=\alpha=\sum\limits_{\omega'\in N(K_Q),|\omega'|=2}r_{\omega'}\widetilde{\omega'}
\text{ with }|\{\omega'\colon r_{\omega'}\ne 0\}|\geqslant 2.
$$ 
Then there is some $\omega'$ such that $r_{\omega'}\ne 0$ and 
 $\varphi^{-1}(\widetilde{\omega'})=\alpha'=\sum\limits_{\omega''\in N(K_P),|\omega''|=2}r'_{\omega''}\widetilde{\omega''}$ with $r'_{\omega}\ne 0$. Now consider all the mappings in cohomology over $\mathbb Q$. Since dimension of annihilator of an element is invariant under isomorphisms,\linebreak  Lemma \ref{Ann2lemma} gives a contradiction:
$$
\dim  {\rm Ann}(\widetilde{\omega})=\dim {\rm  Ann}(\alpha)<\dim {\rm Ann}(\widetilde{\omega'})=\dim{\rm Ann}(\alpha')<\dim {\rm Ann}(\widetilde{\omega}).
$$
Thus $\varphi(\widetilde{\omega})=r_{\omega'}\widetilde{\omega'}$ for some $\omega'$. Since the isomorphism is over $\mathbb Z$, we have $r_{\omega'}=\pm1$. This finishes the proof. 
\end{proof}
\begin{definition}
Following \cite{FW15} and \cite{FMW15} for a graded algebra $A=\bigoplus\limits_{i\geqslant 0}A^i$ over the field $k$, and a nonzero element $\alpha\in A$ define a  {\em $p$-factorspace}\index{$p$-factorspace} $V$ to be a vector subspace in $A^p$ such that for any $v\in V\setminus\{0\}$ there exists $u_v\in A$ with $vu_v=\alpha$. A {\em $p$-factorindex}\index{$p$-factorindex} ${\rm ind}_p(\alpha)$ is defined to be the maximal dimension of $p$-factorspaces of $\alpha$.
\end{definition}
\begin{definition}
Define $\mathbf{B}_k=\bigoplus\limits_{\mathcal{B}_k-\text{$k$-belt}}H_1(\mathcal{B}_k,\partial \mathcal{B}_k)$ to be the subgroup in $H^{k+2}(\mathcal{Z}_P)$ generated by all elements $\widetilde{\mathcal{B}_k}$ corresponding to $k$-belts.  
\end{definition} 
\begin{definition}
For the rest of the Section let $\{\omega_i\}_{i=1}^{N(P)}$ be the set of all {\em missing edges}\index{missing edge} of the complex $K_P$ of the polytope $P$.
\end{definition}
\begin{proposition} \label{bfBkrigid}\index{rigidity!of the group generated by $k$-belts}Let $P$ be a simple $3$-polytope. Then
\begin{enumerate}
\item for any element $\alpha\in H^{k+2}(\mathcal{Z}_P,\mathbb Q)$, $4\leqslant k\leqslant m-2$, we have ${\rm ind}_3(\alpha)\leqslant\frac{k(k-3)}{2}$, and the equality ${\rm ind}_3(\alpha)=\frac{k(k-3)}{2}$ implies  $\alpha\in(\mathbf{B}_k\otimes\mathbb Q)\setminus\{0\}$. 
\item for any $k$-belt $\mathcal{B}_k$, $4\leqslant k\leqslant m-2$, we have ${\rm ind}_3(\widetilde{\mathcal{B}_k})=\frac{k(k-3)}{2}$;
\end{enumerate}
In particular, the group $\mathbf{B}_k\subset H^{k+2}(\mathcal{Z}_P,\mathbb Z)$, $4\leqslant k\leqslant m-2$, is $B$-rigid in the class of all simple $3$-polytopes.
\end{proposition}   
\begin{proof}
(1) We have 
$$
\alpha=\sum\limits_{\omega}\alpha_\omega\in\bigoplus\limits_{|\omega|=k} H_1(P_{\omega},\partial P_{\omega},\mathbb Q)\oplus\bigoplus\limits_{|\omega|=k+1}\widehat{H}_2(P_{\omega},\partial P_{\omega},\mathbb Q),
$$
Let $0\ne\beta=\sum\limits_{i=1}^{N(P)}\lambda_i\widetilde\omega_i$ be the divisor of $\alpha$.  Then there exists 
$$
\gamma=\sum\limits_{\eta}\gamma_\eta\in\bigoplus\limits_{|\eta|=k-3} H_1(P_{\eta},\partial P_{\eta},\mathbb Q)\oplus\bigoplus\limits_{|\eta|=k-2}\widehat{H}_2(P_{\eta},\partial P_{\eta},\mathbb Q),
$$
with $\beta\cdot\gamma=\alpha$. Then $\alpha_\omega=0$, for all $\omega$ with $|\omega|=k+1$, $\gamma_\eta=0$ for all $\eta$ with $|\eta|=k-3$, and $\alpha_\omega=\sum\limits_{\omega_i\subset\omega}\lambda_i\widetilde{\omega_i}\cdot \gamma_{\omega\setminus\omega_i}=\left(\sum\limits_{\omega_i\subset\omega}\lambda_i\widetilde{\omega_i}\right)\cdot\left(\sum\limits_{\eta\subset\omega,|\eta|=k-2}\gamma_\eta\right)$. 
Thus for any $3$-factorspace $V$ of $\alpha$ and any $\omega$ with $\alpha_\omega\ne 0$ the linear mapping 
$$
\varphi_{\omega}\colon V\to H^3(\mathcal{Z}_P,\mathbb Q)\colon \beta\to\beta_{\omega}=\sum\limits_{\omega_i\subset\omega}\lambda_i\widetilde{\omega_i}
$$ 
is a monomorphism; hence it is a linear isomorphism of $V$ to the factorspace $\varphi_\omega(V)$ of $\alpha_\omega$. Let $P_{\omega}=P_{\omega^1}\sqcup\dots\sqcup P_{\omega^s}$ be the decomposition into the connected components. Then $H_1(P_{\omega},\partial P_{\omega})=\bigoplus\limits_{l=1}^s H_1(P_{\omega^l},\partial P_{\omega^l})$, and $\alpha_\omega=\sum\limits_{l=1}^s\alpha_{\omega^l}$. Let $\omega_i=\{p,q\}$, with $p\in \omega^a$, $q\in\omega^b$. If $a\ne b$, then $\widetilde{\omega_i}\cdot\gamma_{\omega\setminus\omega_i}=0$, since $\widetilde{\omega_i}=\pm[F_p]=\mp[F_q]$, and the cohomology class $\widetilde{\omega_i}\cdot\gamma_{\omega\setminus\omega_i}$ should lie in $H_1(P_{\omega^a},\partial P_{\omega^a})\cap H_1(P_{\omega^b},\partial P_{\omega^b})=0$.  
Consider  $\omega_i=\{p,q\}\subset\omega^a$.  Each connected component of $P_{\omega\setminus\omega_i}$ lies in some $P_{\omega^l}$. We have $\gamma_{\omega\setminus\omega_i}=\sum\limits_{l=1}^s\gamma_{\omega^l\setminus\omega_i}$, where each summand  corresponds to the connected components lying in $\omega^l\setminus\omega_i$.  Since $\widetilde{\omega_i}\cdot\gamma_{\omega^l\setminus\omega_i}=0$ for $l\ne a$, we have
\begin{multline*}
\sum\limits_{l=1}^s\alpha_{\omega^l}=\alpha_\omega=\left(\sum\limits_{l=1}^s\sum\limits_{\omega_i\subset\omega^l}\lambda_i\widetilde{\omega_i}+\sum\limits_{\omega_i\not\subset\omega^l\forall l}\lambda_i\widetilde{\omega_i}\right)\cdot\left(\sum\limits_{\omega_i\subset\omega}\gamma_{\omega\setminus\omega_i}\right)=\\
\left(\sum\limits_{l=1}^s\sum\limits_{\omega_i\subset\omega^l}\lambda_i\widetilde{\omega_i}\right)\cdot\left(\sum\limits_{l=1}^s\sum\limits_{\omega_j\subset\omega^l}\gamma_{\omega^l\setminus\omega_j}\right)=
\sum\limits_{l=1}^s\left(\sum\limits_{\omega_i\subset\omega^l}\lambda_i\widetilde{\omega_i}\right)\cdot\left(\sum\limits_{\omega_j\subset\omega^l}\gamma_{\omega^l\setminus\omega_j}\right);
\end{multline*}
hence for any $\alpha_{\omega^l}\ne 0$ the projection $\psi_l\colon\sum\limits_{\omega_i\subset\omega}\lambda_i\widetilde{\omega_i}\to\sum\limits_{\omega_i\subset\omega^l}\lambda_i\widetilde{\omega_i}$ sends the space $\varphi_\omega(V)$ isomorphically to the $3$-factorspace $\psi_l\varphi_{\omega}(V)$ of $\alpha_{\omega^l}$. 
Now consider the connected space $P_{\omega^l}$. 

Let the graph $K^1_{\omega^l}$ have a hanging vertex $a$. Then the facet $F_a$ intersects only one facet among $\{F_t\}_{t\in\omega^l\setminus\{a\}}$, say $F_b$. Then for any $\omega_i=\{a,r\}\subset\omega^l$ we have $\widetilde{\omega_i}\cdot \gamma_{\omega^l\setminus\omega_i}=\pm [F_a]\cdot \gamma_{\omega^l\setminus\omega_i}$ is equal up to a scalar to the class  in $H_1(P_{\omega^l},\partial P_{\omega^l},\mathbb Q)$ of the single edge $F_a\cap F_b$ connecting two points on the same boundary cycle of $P_{\omega^l}$. Hence $\widetilde{\omega_i}\cap \gamma_{\omega^l\setminus\omega_i}=0$. Thus we have 
\begin{multline*}
\!\!\!\!\!\!\!\!\!\alpha_{\omega^l}=\left(\sum\limits_{\omega_i\subset\omega^l}\lambda_i\widetilde{\omega_i}\right)\left(\sum\limits_{\omega_j\subset\omega^l}\gamma_{\omega^l\setminus\omega_j}\right)=
\left(\sum\limits_{\omega_i\subset\omega^l\setminus\{a\}}\lambda_i\widetilde{\omega_i}\right)\left(\sum\limits_{\omega_j\subset\omega^l\setminus\{a\}}\gamma_{\omega^l\setminus\omega_j}\right).
\end{multline*}
Hence the mapping $\xi_a\colon \sum\limits_{\omega_i\subset\omega^l}\lambda_i\widetilde{\omega_i}\to\sum\limits_{\omega_i\subset\omega^l\setminus\{a\}}\lambda_i\widetilde{\omega_i}$ sends any nonzero vector in $\psi_l\varphi_{\omega}(V)$ to a nozero vector; therefore the $3$-factorspace $\psi_l\varphi_{\omega}(V)$ of $\alpha_{\omega^l}$ is mapped isomorphically to the $3$-factorspace $\xi_a\psi_l\varphi_{\omega}(V)\subset\bigoplus\limits_{\omega_i\subset \omega^l\setminus\{a\}}\widehat{H}_2(P_{\omega_i},\partial P_{\omega_i})$ of $\alpha_{\omega^l}$. This space has the dimension at most the number of missing edges in $K^1_{\omega^l\setminus\{a\}}$. Let $r=|\omega^l\setminus\{a\}|$. Since $\alpha_{\omega^l}\ne0$, $r\geqslant 3$. Since $P_{\omega^l\setminus\{a\}}$ is connected, the graph $K^1_{\omega^l\setminus\{a\}}$ has at least $r-1$ edges. Then the number of missing edges is at most $\frac{r(r-1)}{2}-(r-1)=\frac{(r-1)(r-2)}{2}$. Thus we have $\dim V=\dim \xi_a\psi_l\varphi_{\omega}(V)\leqslant \frac{(r-1)(r-2)}{2}\leqslant\frac{(k-2)(k-3)}{2}<\frac{k(k-3)}{2}$, since $r\leqslant k-1$. 

Now let the graph $K^1_{\omega^l}$ have no hanging vertices. Set $l$ to be the number of its edges and $r=|\omega^l|$. We have $r\leqslant k$.  Then $\dim V\leqslant\frac{r(r-1)}{2}-l$. Since the graph is connected and has no hanging vertices,  $r\geqslant 3$ and $l\geqslant r$. Therefore $\dim V\leqslant\frac{r(r-1)}{2}-r=\frac{r(r-3)}{2}\leqslant \frac{k(k-3)}{2}$. If the equality holds, then $r=k=l$,  and $\varphi_\omega(V)=\mathbb Q\langle\widetilde{\omega_i}\colon\omega_i\subset \omega\rangle$. Then $K^1_{\omega}$ is connected, has no hanging vertices and $l=k=|\omega|$ edges. We have $2k$ is the sum of $k$ vertex degrees of $K^1_{\omega}$, each degree being at least $2$. Then each degree is exactly $2$; therefore $K_{\omega}$ is a chordless cycle; hence $P_{\omega}$ is a $k$-belt. This holds for any $\omega$ with $\alpha_{\omega}\ne 0$; hence $\alpha\in (\mathbf{B}_k\otimes\mathbb Q)\setminus\{0\}$.  

(2) For a $k$-belt $\mathcal{B}_j$, $k\geqslant 4$, the space $\mathbb Q\langle \widetilde{\omega_i}\colon \omega_i\subset\omega(\mathcal{B}_j)\rangle$ is a $\frac{k(k-3)}{2}$-dimensional $3$-factorspace of $\widetilde{\mathcal{B}_j}$. Indeed, for any $\omega_i\subset\omega(\mathcal{B}_j)$ take $\gamma_{i,j}$ to be the fundamental cycle in $\widehat{H}_2(P_{\omega(\mathcal{B}_j)\setminus\omega_i},\partial P_{\omega(\mathcal{B}_j)\setminus\omega_i},\mathbb Q)$. Then  $\widetilde{\omega}_p\cdot\gamma_{q,j}=\pm\delta_{p,q}\widetilde{\mathcal{B}_j}$ for any $\omega_p,\omega_q\subset\omega(\mathcal{B}_j)$, and for a combination $\tau=\sum_{\omega_i\subset\omega(\mathcal{B}_j)}\lambda_i\widetilde{\omega_i}$ with $\lambda_p\ne 0$ we have $\tau\cdot(\pm\frac{1}{\lambda_p}\gamma_{p,j})=\widetilde{\mathcal{B}_j}$. 

Now for any graded isomorphism $\varphi\colon H^*(\mathcal{Z}_P,\mathbb Z)\to H^*(\mathcal{Z}_Q,\mathbb Z)$ we have the graded isomorphism $\widehat\varphi\colon H^*(\mathcal{Z}_P,\mathbb Q)\to H^*(\mathcal{Z}_Q,\mathbb Q)$ with the embeddings $H^*(\mathcal{Z}_P,\mathbb Z)\subset H^*(\mathcal{Z}_P,\mathbb Q)$, and $H^*(\mathcal{Z}_P,\mathbb Z)\subset H^*(\mathcal{Z}_P,\mathbb Q)$. For any $\alpha\in H^{k+2}(\mathcal{Z}_P,\mathbb Q)$ the isomorphism $\widehat\varphi$ induces the bijection between the $3$-factorspaces of $\alpha$ and  $\widehat{\varphi}(\alpha)$; hence ${\rm ind}_3(\alpha)={\rm ind}_3(\widehat{\varphi}(\alpha))$. In particular, for any $k$-belt $\mathcal{B}_k$, $4\leqslant k\leqslant m-2$, we have $\frac{k(k-3)}{2}={\rm ind}_3(\widetilde{\mathcal{B}_k})={\rm ind}_3(\widehat{\varphi}(\widetilde{\mathcal{B}_k}))$; hence (1) implies that $\widehat{\varphi}(\widetilde{\mathcal{B}_k})=\sum_j\mu_j\widetilde{\mathcal{B}_{k,j}'}$ for $k$-belts $\mathcal{B}_{k,j}'$ of $Q$. Since $\widehat\varphi(\widetilde{\mathcal{B}_k})=\varphi(\widetilde{\mathcal{B}_k})$, we have $\mu_j\in\mathbb Z$, $\varphi(\widetilde{\mathcal{B}_k})\in \mathbf{B}_k(Q)$; hence $\varphi(\mathbf{B}_k(P))\subset\mathbf{B}_k(Q)$. The same argument for the inverse isomorphism implies that $\varphi(\mathbf{B}_k(P))=\mathbf{B}_k(Q)$.
\end{proof}

\begin{proposition} \index{rigidity!of the set of elements corresponding to $k$-belts}
For any $k$, $5\leqslant k\leqslant m-2$, the set 
$$
\{\pm\widetilde{\mathcal{B}_k}\colon \mathcal{B}_k\text{ is a $k$-belt }\}\subset H^{k+2}(\mathcal{Z}_P)
$$ 
is $B$-rigid in the class of flag simple $3$-polytopes without $4$-belts. 
\end{proposition}
\begin{proof}
Let $P$ and $Q$ be flag $3$-polytopes without $4$-belts, and $\varphi\colon H^*(\mathcal{Z}_P,\mathbb Z)\to H^*(\mathcal{Z}_Q,\mathbb Z)$ be a graded isomorphism. From Proposition \ref{bfBkrigid} we have $\varphi(\widetilde{\mathcal{B}_k})=\sum\limits_{j}\mu_j\widetilde{\mathcal{B}_{k,j}'}$ for $k$-belts $\mathcal{B}_{k,j}'$ of $Q$. Then for any $\omega_i\subset \omega(\mathcal{B}_k)$ we have  $\widetilde{\omega_i}\gamma_{\omega(\mathcal{B}_k)\setminus\omega_i}=\widetilde{\mathcal{B}_k}$ for some $\gamma_{\omega(\mathcal{B}_k)\setminus\omega_i}$. Then $\varphi(\widetilde{\omega_i})\varphi(\gamma_{\omega(\mathcal{B}_k)\setminus\omega_i})=\sum\limits_{j}\mu_j\widetilde{\mathcal{B}_{k,j}'}$. 
\begin{lemma}\label{2div}
Let $\alpha\in H^{k+2}(\mathcal{Z}_P,\mathbb Z)$, $4\leqslant k\leqslant m-2$,
$$
\alpha=\sum\limits_\omega\alpha_\omega\in \bigoplus\limits_{|\omega|=k} H_1(P_{\omega},\partial P_{\omega},\mathbb Z)\oplus\bigoplus\limits_{|\omega|=k+1} \widehat{H}_2(P_{\omega},\partial P_{\omega},\mathbb Z).
$$
If $\beta\in \widehat{H}_2(P_{\tau},\partial P_{\tau},\mathbb Z)$, $\tau\ne\varnothing$, divides $\alpha$, then condition $\alpha_\omega\ne 0$ implies that $|\omega|=k$, $\tau\subset\omega$, and $\beta$ divides $\alpha_{\omega}$.
\end{lemma}
\begin{proof}
Let $\beta\gamma=\alpha$, where $\gamma=\sum\limits_{\eta}\gamma_{\eta}$. Then from the multiplication rule we have $\alpha_\omega=0$ for $|\omega|=k+1$, and  $\beta\gamma_{\omega\setminus\tau}=\alpha_{\omega}$ for each nonzero $\alpha_\omega$. 
\end{proof}
Proposition \ref{2rigid} implies that $\varphi(\widetilde{\omega_i})=\pm\widetilde{\omega_j'}$; therefore 
by Lemma \ref{2div} the element $\widetilde{\omega_j'}$ is a divisor of any $\widetilde{\mathcal{B}_{k,j}'}$ with $\mu_j\ne 0$. But for a $k$-belt $\mathcal{B}_{k,j}'$ the element $\widetilde{\omega_j'}$ is a divisor if and only if $\omega_j'\subset\omega(\mathcal{B}_{k,j}')$. We see that the isomorphism $\varphi$ maps the set $\{\pm \widetilde{\omega_i}\colon\omega_i\subset\omega(\mathcal{B}_k)\}$ bijectively to the corresponding set of any $\mathcal{B}_{k,j}'$ with $\mu_j\ne 0$.  But such a set defines uniquely the $k$-belt; hence we have only one nonzero $\mu_j$, which should be equal to $\pm 1$. This finishes the proof.
\end{proof}
\begin{proposition}\label{facebelt}\index{rigidity!of belts surrounding facets}
For any $k$, $5\leqslant k\leqslant m-2$ the set 
$$
\{\pm\widetilde{\mathcal{B}_k}\colon \mathcal{B}_k\text{ is a $k$-belt surrounding a facet}\}\subset H^{k+2}(\mathcal{Z}_P)
$$ 
is $B$-rigid in the class of flag simple $3$-polytopes without $4$-belts. 
\end{proposition}
\begin{proof}
Let the $k$-belt $\mathcal{B}_k=(F_{i_1},\dots,F_{i_k})$ surround a facet $F_j$ of a flag simple $3$-polytope $P$ without $4$-belts. Consider any facet $F_l$, $l\notin\{i_1,\dots,i_k, j\}$. If $F_l\cap F_{i_p}\ne\varnothing$, and $F_l\cap F_{i_q}\ne\varnothing$, then $F_{i_p}\cap F_{i_q}\ne\varnothing$, else  $(F_j,F_{i_p},F_l,F_{i_q})$ is a $4$-belt. Then $F_{i_p}\cap F_{i_q}\cap F_l$ is a vertex, since $P$ is flag. Then $p-q=\pm1\mod k$, and $F_l\cap F_{i_r}=\varnothing$ for any $r\ne\{p,q\}$. Thus either $F_l$ does not intersect facets in $\mathcal{B}_k$, or it intersects exactly one facet in $\mathcal{B}_k$, or it intersects two successive facets in $\mathcal{B}_k$ by their common vertex.  Consider all elements  $\beta\in H^{k+3}(\mathcal{Z}_P,\mathbb Z)$ such that $\beta$ is divided by any $\widetilde{\omega_i}$ with $\omega_i\subset\omega(\mathcal{B}_k)$. By Lemma \ref{2div} we have $\beta=\sum\limits_{|\omega|=k+1}\beta_{\omega}$. Moreover, since any $\omega_i\subset\omega(\mathcal{B}_k)$ lies in $\omega$, we have $\omega(\mathcal{B}_k)\subset \omega$; hence $\omega=\omega(\mathcal{B}_k)\sqcup\{s\}$ for some $s$. Since $P_{j\sqcup\omega(\mathcal{B}_k)}$ is contractible, we have $s\notin j\sqcup\omega(\mathcal{B}_k)$.
\begin{lemma}\label{wBk}
If $F_l$ either does not intersect facets in the $k$-belt $\mathcal{B}_k$, or intersects exactly one facet in $\mathcal{B}_k$, or intersects exactly two successive facets in $\mathcal{B}_k$ by their common vertex, then the generator of the group $H_1(P_{\omega(\mathcal{B}_k)\sqcup\{l\}},\partial P_{\omega(\mathcal{B}_k)\sqcup\{l\}},\mathbb Z)\simeq\mathbb Z$ is divisible by $\widetilde{\omega_i}$ for any $\omega_i\subset \omega(\mathcal{B}_k)$.
\end{lemma}
\begin{proof}
Let $\omega_i=\{i_p,i_q\}$. Since the facets $F_{i_p}$ and $F_{i_q}$ are not successive in $\mathcal{B}_k$, one of the facets $F_{i_p}$ and $F_{i_q}$ does not intersect $F_l$, say $F_{i_p}$. The facet $F_l$ can not intersect both connected components of $P_{\omega(\mathcal{B}_k)\setminus\{i_p,i_q\}}$; hence $P_{\omega(\mathcal{B}_k)\sqcup\{l\}\setminus\{i_p,i_q\}}$ is disconnected. Let $\gamma$ be the fundamental cycle of the connected component intersecting $F_{i_p}$. Then $\widetilde{\omega_i}\cdot \gamma=\pm[F_{i_p}]\gamma$ is a single-edge path connecting two boundary components of $P_{\omega(\mathcal{B}_k)\sqcup\{l\}}$; hence it is a generator of $H_1(P_{\omega(\mathcal{B}_k)\sqcup\{l\}},\partial P_{\omega(\mathcal{B}_k)\sqcup\{l\}},\mathbb Z)$. This finishes the proof.  
\end{proof}
From Lemma \ref{wBk} we obtain that the are exactly $m-k-1$ linearly independent elements in $H^{k+3}(\mathcal{Z}_P,\mathbb Z)$ divisible by all $\widetilde{\omega_i}$, $\omega_i\subset\omega(\mathcal{B}_k)$.

Now let $\varphi\colon H^*(\mathcal{Z}_P,\mathbb Z)\to H^*(\mathcal{Z}_Q,\mathbb Z)$ be the isomorphims of graded rings for a flag $3$-polytope $Q$ without $4$-belts, and let $\varphi(\widetilde{\mathcal{B}_k})=\pm\widetilde{\mathcal{B}_k}'$ for $\mathcal{B}_k'=(F_{j_1}',\dots,F_{j_k}')$. Assume that $\mathcal{B}_k'$ does not surround any facet. If there is a facet $F_l', l\notin\omega(\mathcal{B}_k')$ such that $F_l'\cap F_{j_p}'\ne\varnothing$, $F_l'\cap F_{j_q}'\ne\varnothing$, and $F_{j_p}'\cap F_{j_q}'=\varnothing$ for some $p,q$, then without loss of generality assume that $p<q$, and $F_l'\cap F_{j_t}' =\varnothing$ for all $t\in\{p+1,\dots,q-1\}$. Then $\mathcal{B}_r'=(F_l',F_{j_p}',F_{j_{p+1}}',\dots,F_{j_q}')$ is an $r$-belt for $r=q-p+2\leqslant k$, and there are $\frac{r(r-3)}{2}-(r-3)=\frac{(r-2)(r-3)}{2}$ common divisors of $\widetilde{\mathcal{B}_r'}$ and $\widetilde{\mathcal{B}_k'}$ of the form $\widetilde{\omega_i'}$. We have $\varphi^{-1}(\widetilde{\mathcal{B}_r'})=\pm \widetilde{\mathcal{B}_r}$ for some $r$-belt $\mathcal{B}_r$ with $\widetilde{\mathcal{B}_r}$ having $\frac{(r-2)(r-3)}{2}$ common divisors of the form $\widetilde{\omega_i}$ with $\widetilde{\mathcal{B}_k}$. Since $\mathcal{B}_k\ne\mathcal{B}_r$, there is $F_u\in\mathcal{B}_r\setminus\mathcal{B}_k$; hence $\widetilde{\mathcal{B}_k}$ and $\widetilde{\mathcal{B}_r}$ have at most  $\frac{(r-2)(r-3)}{2}$ common divisors of the form $\widetilde{\omega_i}$, and the equality holds if and only if $\mathcal{B}_r\setminus\{F_u\}\subset\mathcal{B}_k$. Then $F_u\ne F_j$. Let $F_u$ follows $F_v=F_{i_s}$ and is followed by $F_w=F_{i_t}$ in $\mathcal{B}_r$. Then $F_u\cap F_{i_s}\ne\varnothing$, $F_u\cap F_{i_t}\ne\varnothing$, and $F_{i_s}\cap F_{i_t}=\varnothing$.  We have the $4$-belt $(F_j,F_{i_s},F_u,F_{i_t})$. A contradiction. Hence any facet $F_l'$, $l\in[m]\setminus\omega(\mathcal{B}_k')$ does not intersect two non-successive facets of $\mathcal{B}_k'$; hence either it does not intersect $\mathcal{B}_k'$, or intersects in exactly one facet, or intersects exactly two successive facets by their common vertex. Then by Lemma \ref{wBk} we obtain $m-k$ linearly independent elements in $H^{k+3}(\mathcal{Z}_Q,\mathbb Z)$ divisible by all $\widetilde{\omega_i'}$, $\omega_i'\subset\omega(\mathcal{B}_k')$. A contradiction. This proves that $\mathcal{B}_k'$ surrounds a facet.     
\end{proof} 
\begin{proposition}\label{adjpr}\index{rigidity!of the pair of belts surrounding adjacent facets}
Let $\varphi\colon H^*(\mathcal{Z}_P,\mathbb Z)\to H^*(\mathcal{Z}_Q,\mathbb Z)$ be an isomorphism of graded rings, where $P$ and $Q$ are flag simple $3$-polytopes without $4$-belts. If $\mathcal{B}_1$ and $\mathcal{B}_2$ are belts surrounding adjacent facets, and $\varphi(\widetilde{\mathcal{B}_i})=\pm \widetilde{\mathcal{B}_i'}$, $i=1,2$, then the belts $\mathcal{B}_1'$ and $\mathcal{B}_2'$ also surround adjacent facets.
\end{proposition}
\begin{proof}
The proof follows directly from the following result.
\begin{lemma}
Let $P$ be a flag simple $3$-polytope without $4$-belts. Let a belt $\mathcal{B}_1$ surround a facets $F_p$, and a belt  $\mathcal{B}_2$ surrounds a facet $F_q\ne F_p$. Then $F_p\cap F_q\ne\varnothing$ if and only if $\widetilde{\mathcal{B}_1}$ and $\widetilde{\mathcal{B}_2}$ have exactly one common divisor among $\widetilde{\omega_i}$. 
\end{lemma}
\begin{proof}
If $F_p\cap F_q\ne\varnothing$, then, since $P$ is flag, $\mathcal{B}_1\cap \mathcal{B}_2$ consists of two facets which do not intersect. On the other hand, let $F_p\cap F_q=\varnothing$, and   $\{u,v\}\subset\omega(\mathcal{B}_1)\cap\omega(\mathcal{B}_2)$ with $F_u\cap F_v=\varnothing$. Then $(F_u,F_p,F_v,F_q)$ is a $4$-belt, which is a contradiction.
\end{proof}
\end{proof}
Now let us prove the main theorem.
\begin{theorem}\index{rigidity!of flag $3$-polytopes without $4$-belts}
Let $P$ be a flag simple $3$-polytope without $4$-belts, and $Q$ be a simple $3$-polytope. Then the isomorphism of graded rings  $\varphi\colon H^*(\mathcal{Z}_P,\mathbb Z)\simeq H^*(\mathcal{Z}_Q,\mathbb Z)$ implies the combinatorial equivalence $P\simeq Q$. In  other words, any flag simple $3$-polytope without $4$-belts is $B$-rigid in the class of all simple $3$-polytopes. 
\end{theorem}
\begin{proof}
By Corollaries \ref{Flagcor} and \ref{4beltcor} the polytope $Q$ is also flag and has no $4$-belts. Since $P$ is flag, any it's facet is surrounded by a belt. By Proposition \ref{facebelt} for any belt $\mathcal{B}_k$ surrounding a facet $\varphi(\widetilde{\mathcal{B}_k})=\pm\widetilde{\mathcal{B}_k'}$ for a belt $\mathcal{B}_k'$ surrounding a facet. 
\begin{lemma}
Any belt $\mathcal{B}_k$ surrounds at most one facet of a flag simple $3$-polytope without $4$-belts.
\end{lemma}
\begin{proof}
If a belt $\mathcal{B}_k=(F_{i_1},\dots,F_{i_k})$ surrounds on both sides facets $F_p$ and $F_q$, then $(F_{i_1},F_p,F_{i_3},F_q)$ is a $4$-belt, which is a contradiction.
\end{proof}
From this lemma we obtain that the correspondence $\mathcal{B}_k\to\mathcal{B}_k'$ induces a bijection between the facets of $P$ and the facets of $Q$. Then Proposition \ref{adjpr} implies that this bijection is a combinatorial equivalence.
\end{proof}

\newpage
\section{Lecture 8. Quasitoric manifolds}

\subsection{Finely ordered polytope}
Every face of codimension $k$ may be written uniquely as
\[ G(\omega) = F_{i_1}\cap \ldots \cap F_{i_k} \]
for some subset $\omega=\{ i_1, \ldots, i_k \}\subset[m]$. Then faces $G(\omega)$ may be
ordered lexicographically for each $1 \leqslant k \leqslant n$.

By permuting the facets of $P$ if necessary, we may assume
that the intersection $F_{1}\cap \ldots \cap F_{n}$ is a vertex
$v$.
In this case we describe $P$ as \emph{finely ordered}\index{polytope!finely ordered}, and refer~to~$v$ as
the \emph{initial vertex}\index{polytope!initial vertex}, since it is the first vertex of $P$ with respect
to the lexicographic ordering.

Up to an affine transformation we can assume that $\boldsymbol{a}_1=\boldsymbol{e}_1, \ldots, \boldsymbol{a}_n=\boldsymbol{e}_n$.

\subsection{Canonical orientation}\index{canonical orientation}
We consider $\mathbb{R}^n$ as the standard real $n$-dimensional Euclidean
space with the \emph{standard basis} consisting of vectors
$\boldsymbol{e}_j=(0,\ldots,1,\ldots,0)$ with $1$ on the $j$-th place, for $1
\leqslant j \leqslant n$;
and similarly for $\mathbb{Z}^n$ and $\mathbb{C}^n$.
The standard basis gives rise to the \emph{canonical  orientation} of
$\mathbb{R}^n$.

We identify $\mathbb{C}^n$ with $\mathbb{R}^{2n}$, sending\, $\boldsymbol{e}_j$\, to \,
$\boldsymbol{e}_{2j-1}$ and\, $i\boldsymbol{e}_j$\, to\, $\boldsymbol{e}_{2j}$\, for\,
$1\leqslant j\leqslant n$.
This provides the \emph{canonical orientation} for $\mathbb{C}^n$.

Since $\mathbb{C}$-linear maps from $\mathbb{C}^n$ to $\mathbb{C}^n$ preserve
the canonical orientation, we may also regard an arbitrary complex vector
space as canonically oriented.

We consider $\mathbb T^n$ as the standard $n$-dimensional torus
$\mathbb{R}^n/\mathbb{Z}^n$ which we identify with the product of $n$ unit
circles in $\mathbb{C}^n$:
\[
  \mathbb T^n=\{(e^{2\pi i\varphi_1},\ldots,e^{2\pi i\varphi_n})\in\mathbb{C}^n\},
\]
where $(\varphi_1,\ldots,\varphi_n) \in \mathbb{R}^n$.
The torus $\mathbb T^n$ is also {\em canonically oriented}.

\subsection{Freely acting subgroups}\index{moment-angle complex!freely acting subgroup}
Let $H\subset \mathbb{T}^m$ be a subgroup of dimension $r\leqslant
m-n$. Choosing a basis, we can write it in the form
\[   H=\bigl\{ (e^{2\pi i(s_{11}\varphi_1+\dots+s_{1r}\varphi_r)},
\ldots, e^{2\pi i(s_{m1}\varphi_1 + \dots + s_{mr}\varphi_r)}) \in
\mathbb{T}^m \bigr\}, \] where $\varphi_i\in\mathbb{R}, \; i=1,\ldots,r$
and $S=(s_{ij})$ is an integer $(m\times r)$-matrix which defines a
monomorphism $\mathbb{Z}^r\to\mathbb{Z}^m$ onto a direct summand. For any subset \linebreak $\omega = \{i_1,\ldots,i_n\}\subset[m]$
denote by $S_{\widehat{\omega}}$ the ($(m-n) \times r$)-submatrix of $S$ obtained by deleting
the rows $i_1,\ldots,i_n$.

Write each vertex $v\in P^n$ as $v_\omega$ if $v=F_{i_1}\cap\ldots\cap
F_{i_n}$.

\noindent{\bf{Exercise:}} The subgroup~$H$ acts freely on $\mathcal{Z}_{P}$ if and only if for every
vertex $v_\omega$ the submatrix $S_{\widehat{\omega}}$ defines a
monomorphism $\mathbb{Z}^r\hookrightarrow\mathbb{Z}^{m-n}$ onto a direct
summand.

\begin{corollary}\label{zpbun}
The subgroup $H$ of rank $r=m-n$ acts freely on $\mathcal{Z}_{P}$
if and only if for any vertex $v_\omega\in P$ we have:
\[ \det S_{\widehat{\omega}}=\pm 1. \]
\end{corollary}

\subsection{Characteristic mapping}
\begin{definition} An $(n\times m)$-matrix $\varLambda$ gives a {\em characteristic mapping}\index{characteristic mapping}\index{polytope!characteristic mapping}
\[ \ell \colon \{F_1,\ldots,F_m\} \longrightarrow \mathbb{Z}^n \]
for a given simple polytope $P^n$ with facets $\{F_1, \ldots, F_m\}$
if the columns\linebreak $\ell(F_{j_1})=\lambda_{j_1},\ldots,\ell(F_{j_n})=\lambda_{j_n}$ of $\varLambda$
corresponding to any vertex $v_\omega$ form a basis for $\mathbb{Z}^n$.
\end{definition}

\noindent{\bf{Example:}} For a pentagon $P^2_5$ we have a matrix $ \varLambda\;=\;\begin{pmatrix}
  1&0&1&0&1\\
  0&1&0&1&1
\end{pmatrix}
$\\
\vspace{-4mm}
\begin{figure}
\begin{center}
\includegraphics[scale=0.25]{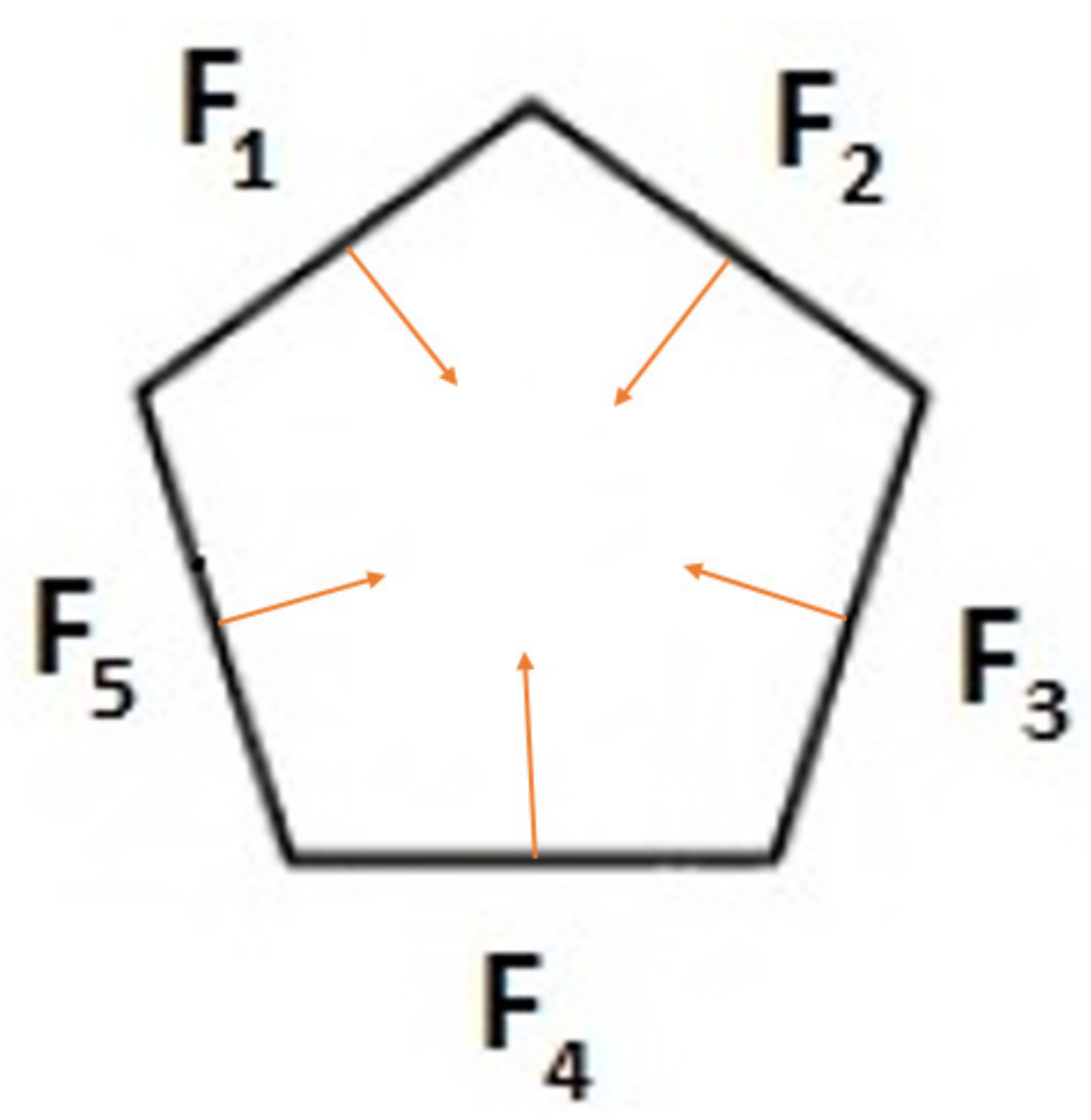}
\end{center}
\caption{Pentagon with normal vectors}
\end{figure}

\noindent {\bf Problem:} For any  simple $n$-polytope $P$ find all integral $(n\times m)$-matrices
$$
\varLambda\;=\;\begin{pmatrix}
  1&0&\ldots&0&\lambda_{1,n+1}&\ldots&\lambda_{1,m}\\
  0&1&\ldots&0&\lambda_{2,n+1}&\ldots&\lambda_{2,m}\\
  \vdots&\vdots&\ddots&\vdots&\vdots&\ddots&\vdots\\
  0&0&\ldots&1&\lambda_{n,n+1}&\ldots&\lambda_{n,m}
\end{pmatrix},
$$
in which the column $\lambda_j = (\lambda_{1,j},\ldots, \lambda_{n,j})$ corresponds
to the facet $F_j, \; j=1,\ldots,m$, and the columns $\lambda_{j_1}$, \ldots,
$\lambda_{j_n}$ corresponding to any vertex\linebreak $v_\omega = F_{j_1}\cap\dots\cap F_{j_n}$ form a basis for $\mathbb{Z}^n$.

Note that there are simple $n$-polytopes, $n\geqslant 4$, admitting no characteristic functions.  

\noindent{\bf Exercise:} Let $C^n(m)$ be a combinatorial type of a {\em cyclic polytope}\index{cyclic polytope}\index{polytope!cyclic} built as follows: take real numbers $t_1<\dots<t_m$ and $$
C^n(t_1,\dots,t_m)={\rm conv}\{(t_i,t_i^2,\dots,t_i^n),i=1,\dots,m\}.
$$ 
Prove that 
\begin{enumerate}
\item the combinatorial type of $C^n(t_1,\dots,t_m)$ does not depend on $t_1<\dots<t_m$;
\item the polytope $C^n(m)$ is simplicial;
\item for $n\geqslant 4$ any two vertices of $C^n(m)$ are connected by an edge;
\end{enumerate}
Conclude that for large $m$ the dual simple polytope $C^n(m)^*$   admits no characteristic functions.
\subsection{Combinatorial data}

\begin{definition} The {\em combinatorial quasitoric data}\index{combinatorial quasitoric data}
$(P,\varLambda)$ consists of an {\em oriented} combinatorial simple polytope $P$ and
an integer $(n\times m)$-matrix $\varLambda$ with the properties above.
\end{definition}

The matrix~$\varLambda$ defines an epimorphism
$$\ell
\colon \mathbb T^m \to \mathbb T^n.
$$
The kernel of $\ell$ (which we denote $K(\varLambda)$) is isomorphic to
$\mathbb T^{m-n}$.

\noindent{\bf Exercise:} The action of $K(\varLambda)$ on $\mathcal{Z}_P$ is {\em
free} due to the condition on the minors of~$\varLambda$.

\subsection{Quasitoric manifold with the $(A,\Lambda)$-structure}

\noindent{\bf Construction:} The quotient $M=\mathcal{Z}_P/K(\varLambda)$ is a
$2n$-dimensional {\em smooth} manifold with an action of the $n$-dimensional
torus $T^n=\mathbb
T^m/K(\varLambda)$. We denote this action by $\alpha$. It satisfies the \emph{Davis--Januszkiewicz conditions}\index{Davis--Januszkiewicz conditions}:
\begin{enumerate}
\item $\alpha$ is locally isomorphic to the standard coordinatewise
    representation of $\mathbb T^n$ in $\mathbb{C}^n$,
\item there is a projection $\pi\colon M\to P$ whose fibres are orbits
    of~$\alpha$.
\end{enumerate}
We refer to $M=M(P,\varLambda)$ as the \emph{quasitoric manifold associated
with the combinatorial data} \index{quasitoric manifold}$(P,\varLambda)$.

Let
$$ P=\{\boldsymbol{x}\in\mathbb{R}^n\colon A \boldsymbol{x} + b\geqslant 0 \}.$$
\begin{definition}
The manifold  $M=M(P,\varLambda)$ is called the quasitoric manifold
with {\em $(A,\varLambda)$-structure}\index{quasitoric manifold!
with $(A,\varLambda)$-structure}.
\end{definition}

\noindent {\bf{Exercise:}} Suppose the $(n\times m)$-matrix $\varLambda = (I_n,\, \varLambda_*)$, where $I_n$ is the unit matrix, gives a
characteristic mapping
\[ \ell \colon \{F_1,\ldots,F_m\} \longrightarrow \mathbb{Z}^n \]
Then the matrix $S=\begin{pmatrix}-\varLambda_*\\I_{m-n}\end{pmatrix}$ gives the $(m-n)$-dimensional
subgroup
\[ H=\bigl\{ (e^{2\pi i\psi_1},
\ldots, e^{2\pi i\psi_m}) \in \mathbb{T}^m \bigr\},
\] where
\[ \psi_k = - \sum_{j=n+1}^m\lambda_{k,j}\varphi_{j-n},\; k=1,\ldots,n;\;
\psi_k = \varphi_{k-n},\; k=n+1,\ldots,m, \] acting freely on  $\mathcal{Z}_P$.

\begin{example} Take  $P = \Delta^2$.
Let us describe the matrices $A$ and $\varLambda$:
\[ A= \begin{pmatrix}
1&0\\0&1\\ a_{31}&a_{32}
\end{pmatrix}, \qquad \varLambda= \begin{pmatrix}
1&0&\lambda_{13}\\0&1&\lambda_{23}
\end{pmatrix},\qquad a_{31},a_{32},\lambda_{13},\lambda_{23} \in \mathbb{Z}.
\]
Since the normal vectors are oriented inside the polytope,  $a_{31}<0,\; a_{32}<0$.\\
Thus, up to combinatorial equivalence, one can take $a_{31} = a_{32}=-1$.\\
The conditions on the characteristic mapping give
\[ \begin{vmatrix}
0&\lambda_{13}\\1&\lambda_{23}
\end{vmatrix}=\pm\,1,\quad \begin{vmatrix}
1&\lambda_{13}\\0&\lambda_{23}
\end{vmatrix}=\pm\,1,\quad \Rightarrow \quad \lambda_{13}=\pm\,1,\quad
\lambda_{23}=\pm\,1.
\]
Therefore we have 4 structures $(A,\varLambda)$.
\end{example}

\noindent{\bf Exercise:} Let $P = \Delta^2$ and $\mathbb{C} P^2$ be the complex projective space with\\
the canonical action of torus $\mathbb{T}^3$: $(t_1, t_2, t_3)(z_1 : z_2 :
z_3) = (t_1 z _1 : t_2 z_ 2 : t_3 z_3)$.

\begin{enumerate}
\item  describe $\mathbb{C} P^2$ as $(S^5 \times_{\mathbb{T}^3} \mathbb{T}^2)$;
\item describe the structure $(A, \varLambda)$ such that $M(A, \varLambda)$ is
$\mathbb{C} P^2$;
\end{enumerate}

\subsection{A partition of a quasitoric manifold}\index{quasitoric manifold!partition into disks}

We have the homeomorphism
$$
\mathcal{Z}_P \simeq \bigcup_{v_{\omega}-\text{ vertex}} {\mathcal{Z}_{P, v_\omega}},
$$
where
$$
\mathcal{Z}_{P, v_\omega}=\prod_{j \in \omega} D^2_j \times \prod_{j \in [m]\setminus\omega} S^1_j \subset \mathbb{D}^{2m}.
$$

\noindent {\bf{Exercise:}} $\mathcal{Z}_{P, v_\omega} / K(\varLambda) \simeq \mathbb{D}^{2n}_\omega$.

\begin{corollary} We have the partition:
$$
M(P, \varLambda) = \bigcup_{v_{\omega}-\text{ vertex}} \mathbb{D}^{2n}_\omega.
$$
\end{corollary}
\subsection{Stably complex structure and characteristic classes}
Denote by $\mathbb{C}_i$ the space of the 1-dimensional complex
representation of the torus~$\mathbb{T}^m$ induced from the standard
representation in $\mathbb{C}^m$ by the projection $\mathbb{C}^m\to\mathbb{C}_i$ onto the
$i$th coordinate. Let $\mathcal{Z}_P\times\mathbb{C}_i\to\mathcal{Z}_P$ be the trivial complex
line bundle; we view it as an equivariant $\mathbb{T}^m$-bundle with the
diagonal action of~$\mathbb{T}^m$. By taking the quotient with respect to
the diagonal action of $K=K(\varLambda)$ we obtain a
$T^n$-equivariant complex line bundle
\begin{equation}\label{rhoi}
  \rho_i\colon \mathcal{Z}_P\times_K\mathbb{C}_i\to\mathcal{Z}_P/K=M(P,\varLambda)
\end{equation}
over the quasitoric manifold $M=M(P,\varLambda)$. Here 
$$
\mathcal{Z}_P\times_K\mathbb{C}_i=\mathcal Z_{P}\times\mathbb{C}_i\left/\right.(\boldsymbol{t}\boldsymbol{z},\boldsymbol{t}w)\sim(\boldsymbol{z},w)\text{ for any }\boldsymbol{t}\in K, \boldsymbol{z}\in\mathcal{Z}_P,w\in\mathbb C_i.
$$ 

\begin{theorem}\label{taum} (Theorem 6.6, \cite{DJ91})\index{quasitoric manifold!stably-complex structure}
There is an isomorphism of real $T^n$-bundles over
$M=M(P,\varLambda)$:
\begin{equation}\label{stabsplitqt}
  {\mathcal
  T}\!M\oplus\underline{\mathbb{R}}^{2(m-n)}\cong\rho_1\oplus\cdots\oplus\rho_m;
\end{equation}
here $\underline{\mathbb{R}}^{2(m-n)}$ denotes the trivial real
$2(m-n)$-dimensional $T^n$-bundle over~$M$.
\end{theorem}\index{quasitoric manifold!characteristic classes}
For the proof see (Theorem 7.3.15, \cite{Bu-Pa15}) .
\begin{corollary} Let $v_i=c_1(\rho_i)\in H^2(M(P,\Lambda),\mathbb Z)$. Then for the total Chern class we have
$$
C(M(P,\Lambda))=1+c_1+\dots+c_n=(1+v_1)\dots(1+v_m),
$$
and for the total Pontryagin class we have
$$
P(M(P,\Lambda))=1+p_1+\dots+p_{\left[\frac{n}{2}\right]}=(1+v_1^2)\dots(1+v_m^2).
$$
\end{corollary}
\subsection{Cohomology ring of the quasitoric manifold}

\begin{theorem}\cite{DJ91}\index{quasitoric manifold!cohomology}
We have 
$$
H^*(M(P,\Lambda))=\mathbb Z[v_1,\dots,v_m]/(J_{SR}(P)+I_{P,\Lambda}),
$$ 
where $v_i=c_1(\rho_i)$, $J_{SR}(P)$ is the Stanley-Reisner ideal generated by monomials $\{v_{i_1}\dots v_{i_k}\colon F_{i_1}\cap \dots\cap F_{i_k}=\varnothing\}$, and  
$I_{P,\Lambda}$ is the ideal generated by the linear forms $\lambda_{i,1}v_1+\dots+\lambda_{i,m}v_m$ arising from the equality
$$
\ell(F_1)v_1+\dots+\ell(F_m)v_m=0.
$$
\end{theorem}
For the proof see (Theorem 7.3.28, \cite{Bu-Pa15}).
\begin{corollary}
If $\Lambda=(I_n,\Lambda_*)$, then 
$$
H^2(M(P,\Lambda))=\mathbb Z^{m-n}
$$
with the generators $v_{n+1},\dots,v_m$.
\end{corollary}
\begin{corollary}
\begin{enumerate}
\item The group $H^k(M(P,\Lambda))$ is nontrivial only for $k$ even;
\item $M(P,\Lambda)$ is even dimensional and orientable, hence the group $H_k(M(P,\Lambda))$ is nontrivial only for $k$ even;  
\item from the universal coefficients formula the abelian groups $H^*(M(P,\Lambda))$ and $H_*(M(P,\Lambda))$ have no torsion.
\end{enumerate}
\end{corollary}
\begin{corollary}\label{flagqtm} Let $P$ be a flag polytope and $\ell$ be its characteristic function. Then 
$$
H^*(M(P,\Lambda))=\mathbb Z[v_1,\dots,v_m]/(J_{SR}+I_{P,\Lambda}),
$$
where the ideal $J_{SR}$ is generated by monomials $v_iv_j$, where $F_i\cap F_j=\varnothing$, and $I_{P,\Lambda}$ is generated by  linear forms  $\lambda_{i,1}v_1+\dots+\lambda_{i,m}v_m$.
\end{corollary}
\begin{corollary} For any $l=1,\dots,n$, the cohomology group $H^{2l}(M(P,\Lambda),\mathbb Z)$ is generated by monomials $v_{i_1}\dots v_{i_l}$, $i_1<\dots<i_l$, corresponding to $(n-l)$-faces $F_{i_1}\cap\dots\cap F_{i_l}$. 
\end{corollary}
\begin{proof}
We will prove this by induction on characteristic $\delta=\sum_{p_i>1}p_i$ of a monomial $v_{i_1}^{p_1}\dots v_{i_s}^{p_s}$ with $i_1<\dots<i_s$.  Due to the relations from the ideal $J_{SR}$ nonzero monomials correspond to faces $F_{i_1}\cap\dots\cap F_{i_s}\ne\varnothing$. If $\delta=0$, then we have the monomial we need. If $\delta>0$,  then take a vertex $v$ in $F_{i_1}\cap\dots\cap F_{i_s}\ne\varnothing$. Let $\Lambda_v$ be the submatrix of $\Lambda$ corresponding to the columns $\{j\colon v\in F_j\}$. Then by definition of a characteristic function $\det \Lambda_v=\pm 1$. By integer elementary transformations of rows of the matrix $\Lambda$ (hence of linear relations in the ideal $I_{P,\Lambda}$) we can make $\Lambda_v=E$. Let $p_k>1$. The variable $v_{i_k}$ can be expressed as a linear combination  $v_{i_k}=\sum\limits_{j\notin\{i_1,\dots,i_s\}}a_jv_j$. Then 
$$
v_{i_1}^{p_1}\dots v_{i_s}^{p_s}=\sum\limits_{j\notin\{i_1,\dots,i_s\}}a_jv_{i_1}^{p_1}\dots v_{i_k}^{p_k-1}\dots v_{i_s}^{p_s}v_j,
$$ 
where on the right side we have the sum of monomials with less value of $\delta$. This finishes the proof.   
\end{proof}

For any $\xi=(i_1,\ldots,i_{n-1})\subset [m]$ set $\xi_i=(\xi,i),\; i\notin \xi$. 

\noindent{\bf Exercise:}\\
1. For any $\xi=(i_1,\ldots,i_{n-1})\subset [m]$  there are the
relations
\begin{equation} \label{1}
\sum_{j=1}^m \varepsilon(\xi_j)v_j=0
\end{equation}
where $\varepsilon(\xi_j) = \det|\ell(F_{i_1}),
\ldots,\ell(F_{i_{n-1}}),\ell(F_j)|$.\\
2. There is a graded ring isomorphism
\[
H^*(M(P,\Lambda)) = \mathbb{Z}[P]/J
\]
where $J$ is the ideal generated by the relations (\ref{1}).

\noindent{\bf Exercise:} For any vertex $v_\omega=F_{i_1}\bigcap\dots\bigcap F_{i_n},\;
\omega=(i_1,\ldots,i_n)$, there are the relations
\[
v_{i_n}=-\varepsilon(\xi_{i_n})\sum_{j}\varepsilon(\xi_j)\, v_j
\]
where $\xi=(i_1,\ldots,i_{n-1}),\; j\in[m\backslash \xi_{i_n}]$.

\noindent{\bf Exercise:} For any vertex $v_\omega=F_{i_1}\bigcap\dots\bigcap F_{i_n},\;
\omega=(i_1,\ldots,i_n)$, there are the relations
\[
v_{i_n}^2=-\varepsilon(\xi_{i_n})\sum_{j}\varepsilon(\xi_j)\,
v_{i_n} v_j
\]
where $ j\in[m\backslash \xi_{i_n}]$, but $F_{i_n}\bigcap
F_j\neq\emptyset$.

\subsection{Geometrical realization of cycles of quasitoric manifolds}\index{quasitoric manifold!geometric realization of cycles}
The fundamental notions of algebraic topology were introduced in the classical work by Poincare \cite{P1895}. Among them there were notions of cycles and homology. Quasitoric manifolds give nice examples of manifolds such that original notions by Poincare obtain explicit geometric realization.

Let $M^k$ be a smooth oriented manifold such that the groups $H_*(M^k,\mathbb Z)$ have no torsion. There is the classical Poincare duality $H_i(M^k,\mathbb Z)\simeq H^{k-i}(M^k,\mathbb Z)$. Moreover, according to the Milnor-Novikov theorem \cite{miln60,novi60,novi62}  for any cycle $a\in H_l(M^k,\mathbb Z)$ there is a smooth oriented manifold $N^l$ and a continuous mapping $f\colon N^l\to M^k$ such that $f_*[N^l]=a$. For the homology groups of any quasitoric manifold there is the following remarkable geometrical interpretation of this result.  Note that odd homology groups of any quasitoric manifold are trivial.
\begin{theorem}
\begin{enumerate}
\item The homology group $H_{2n-2}(M(P,\Lambda),\mathbb Z)$ of the quasitoric manifold $M^{2n}(P,\Lambda)$  is generated by embedded quasitoric manifolds $M^{2n-2}_i(P,\Lambda)$, $i=1,\dots,m$, of facets of $P$.  The embedding of the manifold $M^{2n-2}_i(P,\Lambda)\subset M(P,\Lambda)$ gives the geometric realization of the cycle Poincare dual to the cohomology class $v_i\in H^2(M(P,\Lambda),\mathbb Z)$ defined above.
\item  For any $i$ the homology group $H_{2i}(M(P,\Lambda),\mathbb Z)$ is generated by embedded quasitoric manifolds corresponding to all $i$-faces $F_{j_1}\cap\dots \cap F_{j_{n-i}}$ of the polytope $P$.  These manifolds can be described as complete intersections of manifolds $M^{2n-2}_{j_1}(P,\Lambda)$, $\dots$, $M^{2n-2}_{j_{n-i}}(P,\Lambda)$.
\end{enumerate}
\end{theorem}
The proof of the theorem follows directly from the above results on the cohomology of quasitoric manifolds and geometric interpretation of the Poincare duality in terms of Thom spaces \cite{thom54}.
\subsection{Four colors problem}\index{four colors problem}
{\bf Classical formulation}: Given any partition of a plane into {\em contiguous regions}, producing a figure called a map, two regions are called {\em adjacent} if they share a common boundary that is not a corner, where corners are the points shared by three or more regions. 

\noindent {\bf Problem:} {\em \index{$4$-colors problem}No more than four colors are required to color the regions of the map so that no two adjacent regions have the same color.} 

The problem was first proposed on October 23, 1852, when Francis Guthrie, while trying to color the map of counties of England, noticed that only four different colors were needed. 

The four colors problem became well-known in 1878 as a hard problem when Arthur Cayley suggested it for discussion during
the meeting of the London mathematical society.

The four colors problem was solved in 1976 by Kenneth Appel and Wolfgang Haken. It became the first major problem solved using a computer. For the details and the history of the problem see \cite{W14}. One of the central topics of this monograph is <<how the problem was solved>>.

\begin{example} Platonic solids.\\
The octahedron can be colored in $2$ colors.\\
The cube and the icosahedron can be colored into $3$ colors.\\
The tetrahedron and the dodecahedron can be colored into $4$ colors.
\end{example}
\begin{figure}
\begin{center}
\includegraphics[height=5cm]{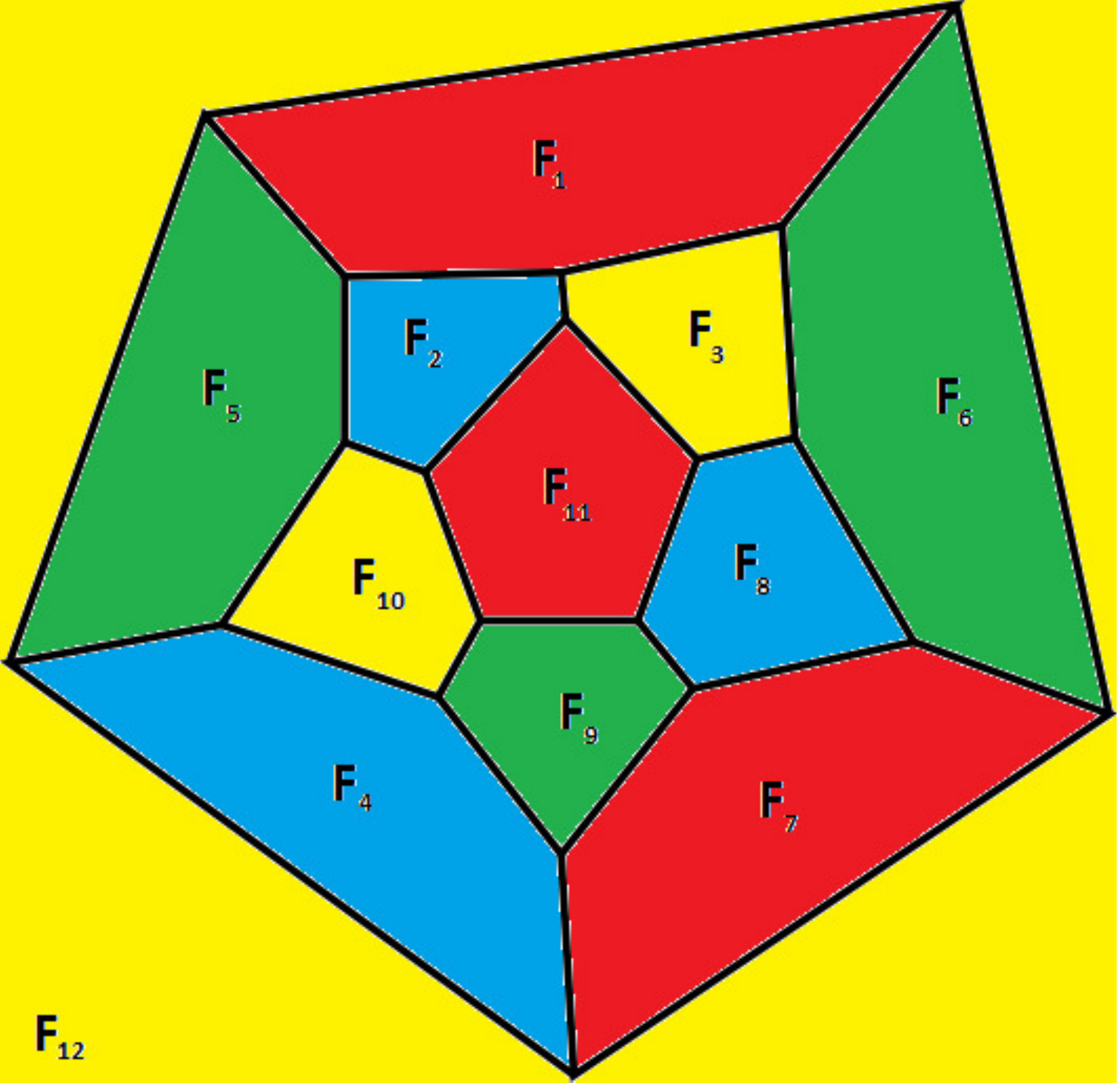}
\end{center}
\caption{Coloring of the dodecahedron}
\end{figure}
\noindent{\bf Exercise:} Color all the Archimedean solids.

\subsection{Quasitoric manifolds of $3$-dimensional polytopes}

Let $P$ be a simple $3$-polytope. Then $\partial P$ is homeomorphic to the sphere $S^2$ partitioned into polygons 
$F_1,\dots,F_m$. By the four colors theorem there is a coloring $\varphi\colon\{F_1,\dots,F_m\}\to\{1,2,3,4\}$ such that adjacent facets have different colors.

Let $\boldsymbol{e}_1,\boldsymbol{e}_2,\boldsymbol{e}_3$ be the standard basis for $\mathbb Z^3$, and $\boldsymbol{e}_4=\boldsymbol{e}_1+\boldsymbol{e}_2+\boldsymbol{e}_3$.
\begin{proposition}
The mapping
$
\ell\colon\{F_1,\dots,F_m\}\to\mathbb Z^3\colon \ell(F_i)=\boldsymbol{e}_{\varphi(F_i)}
$ 
is a characteristic function. 
\end{proposition}
\begin{corollary} \index{quasitoric manifold!existence for $3$-polytopes}
\begin{itemize}
\item Any simple $3$-polytope $P$ has combinatorial data $(P,\Lambda)$ and the quasitoric manifold $M(P,\Lambda)$;
\item Any fullerene has a quasitoric manifold.   
\end{itemize}
\end{corollary}
Since a fullerene is a flag polytope, the cohomology ring of any its quasitoric manifold is described by Corollary \ref{flagqtm}.

\noindent{\bf{Exercise:}} Find a $4$-coloring of the barrel (Fig. \ref{Barrelfig}).
\begin{figure}
\begin{center}
\includegraphics[height=5cm]{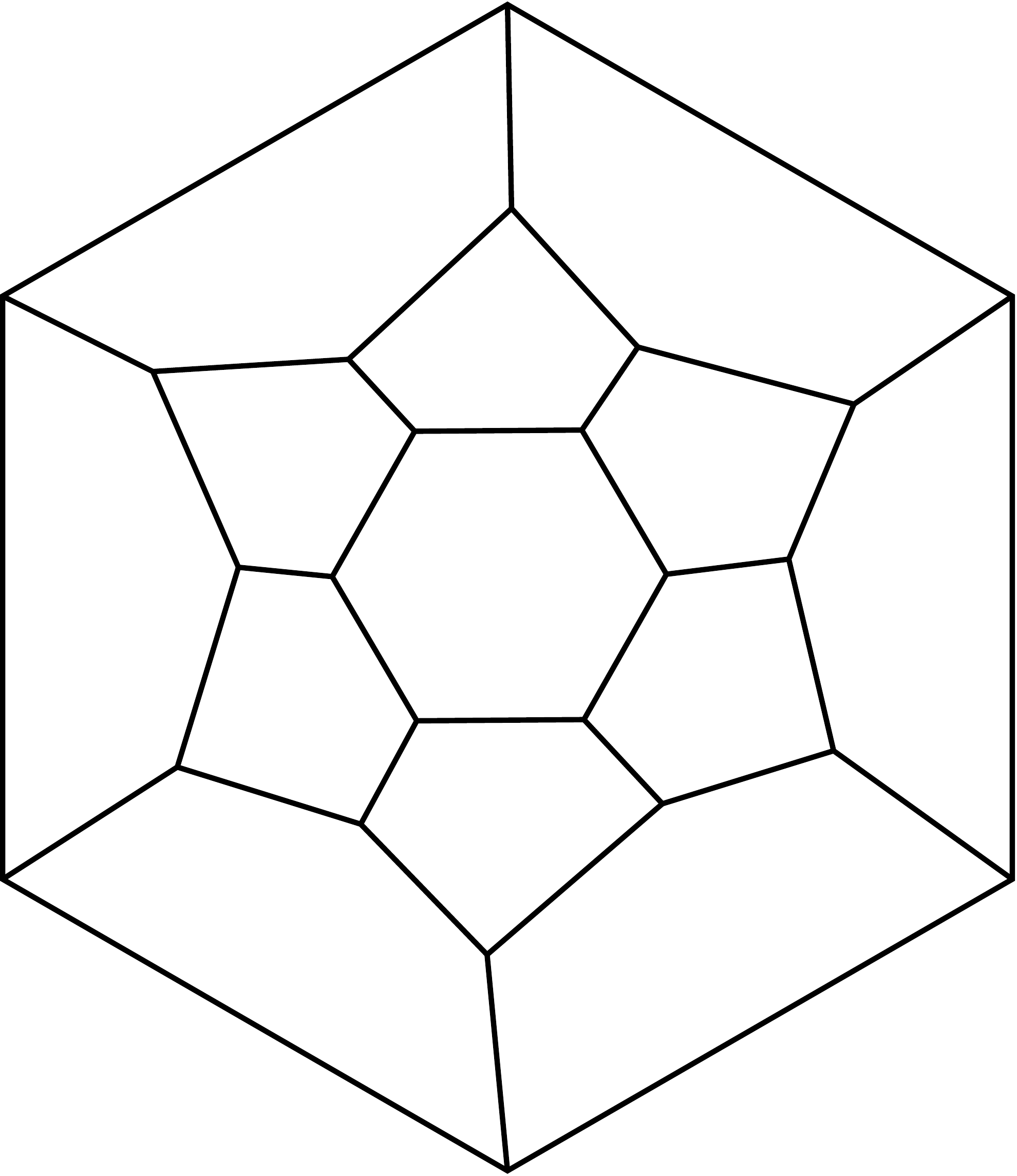}
\caption{Schlegel diagram of the barrel}\label{Barrelfig}\index{barrel}
\end{center}
\end{figure}

\newpage
\section{Lecture 9. Construction of fullerenes} \label{ConFul}
\subsection{Number of combinatorial types of fullerenes}
\begin{definition}
Two combinatorially nonequivalent fullerenes with the same number
 $p_6$ are called {\em combinatorial isomers}\index{fullerene!isomers}.
\end{definition}
Let $F(p_6)$ be the number of combinatorial isomers with given $p_6$.

From the results by W. Thurston \cite{T98} it follows that $F(p_6)$ grows like $p_6^9$.

There is an effective algorithm of combinatorial enumeration
of fullerenes using supercomputer (Brinkmann-Dress \cite{BD97}, 1997). It gives:
\begin{center}
\begin{tabular}{|c|c|c|c|c|c|c|c|c|c|c|c|}
\hline
$p_6$&0&1&2&3&4&5&6&7&8&$\dots$&75\\
\hline
$F(p_6)$&1&0&1&1&2&3&6&6&15&$\dots$&46.088.157\\
\hline
\end{tabular}
\end{center}
We see that for large value of $p_6$ the number of combinatorial isomers is very huge. Hence there is an important problem to 
study different structures on the set of fullerenes. 

\subsection{Growth operations}
The well-known problem \cite{BGJ09,HFM08} is to find a simple set of operations sufficient to construct arbitrary fullerene from the dodecahedron. 
\begin{definition}
A {\em patch}\index{patch} is a disk bounded by a simple edge-cycle on the boundary of a simple $3$-polytope. 
\end{definition}
\begin{definition}
A {\em growth operation} \index{growth operation} \index{fullerene!growth operation}is a combinatorial operation that gives a new $3$-polytope $Q$ from a simple $3$-polytope $P$ by substituting  a new patch with the same boundary and more facets for the patch on the boundary of $P$.   
\end{definition}

The Endo-Kroto operation  \cite{EK92} (Fig. \ref{EKfig})\index{Endo-Kroto operation} is the simplest example of a growth operation that changes a fullerene into a fullerene.
\begin{figure}
\begin{center}
\includegraphics[height=3cm]{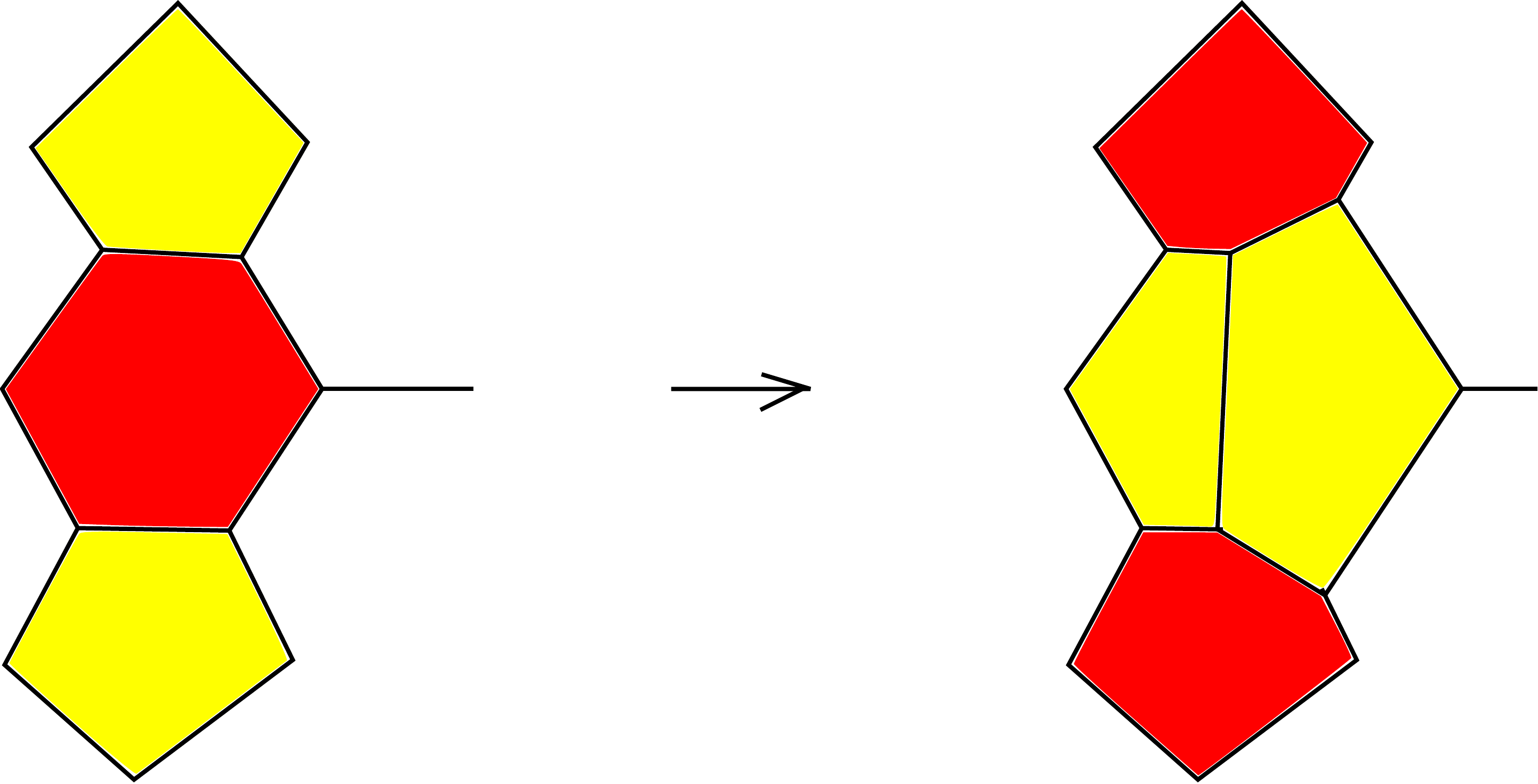}
\end{center}
\caption{Endo-Kroto operation}\label{EKfig}
\end{figure}
It was proved in \cite{BGJ09} that there is no finite set of growth operations transforming fullerenes into fullerenes sufficient to construct arbitrary fullerene from the dodecahedron. In \cite{HFM08} the example of an infinite set was found. Our main result is the following (see \cite{BE15b}): if we allow at intermediate steps polytopes with at most one singular face (a quadrangle or a heptagon), then only $9$ growth operations (induced by $7$ truncations) are sufficient.

\noindent {\bf Exercise:} Starting from the Barrel fullerene (see Fig. \ref{Barrelfig}) using the Endo-Kroto operation construct a fullerene with arbitrary $p_6\geqslant 2$.

\subsection{$(s,k)$-truncations}
First we mention a well-known result about construction of simple $3$-polytopes.
\begin{theorem} \label{3ptheorem}(Eberhard (1891), Br\"{u}ckner (1900))
A  $3$-polytope is simple if and only if it is combinatorially equivalent to a polytope 
obtained from the tetrahedron by a sequence of {\bf vertex, edge} and {\bf
$(2,k)$}-truncations.
\end{theorem}
\begin{figure}
\begin{center}
\includegraphics[height=2.7cm]{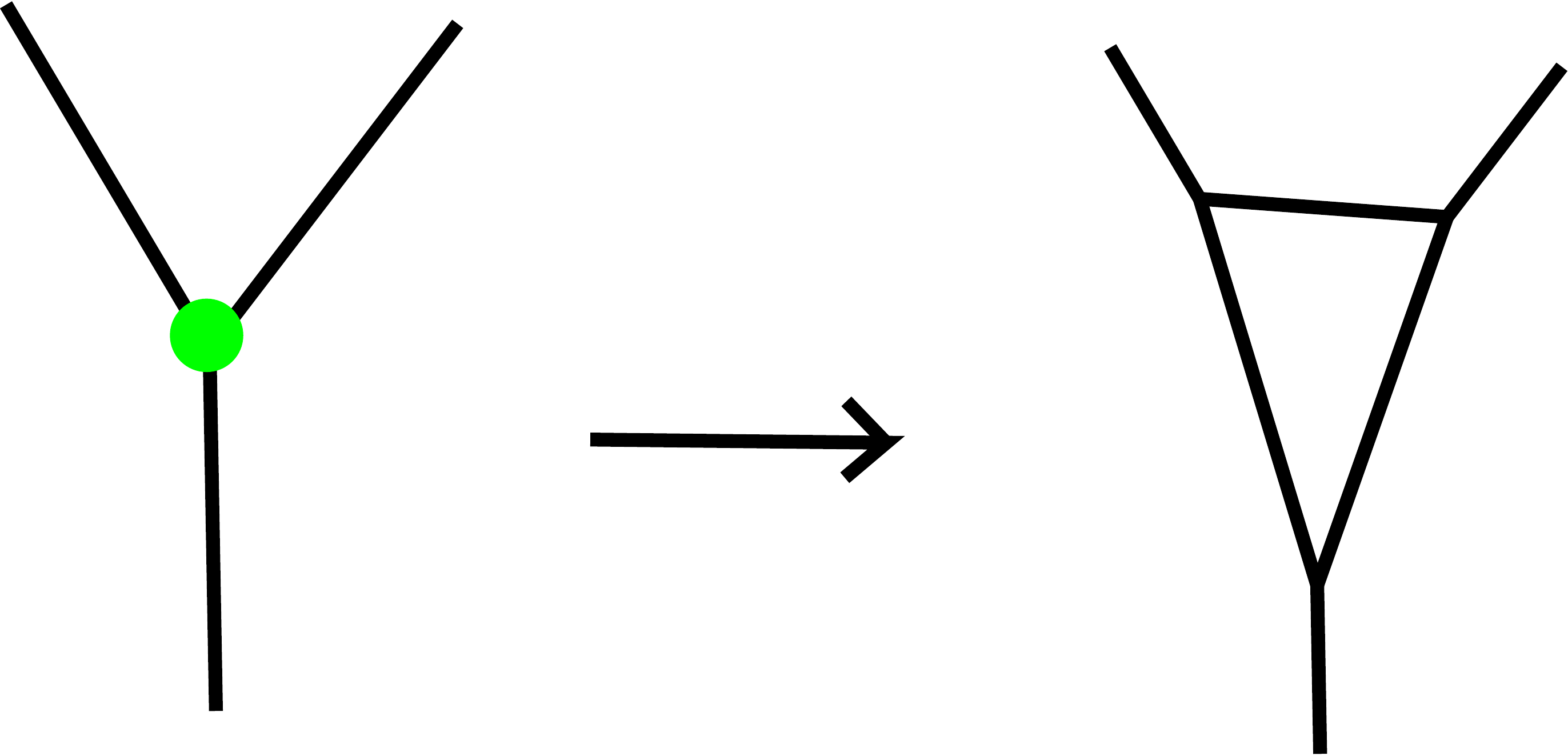}\quad\includegraphics[height=2.7cm]{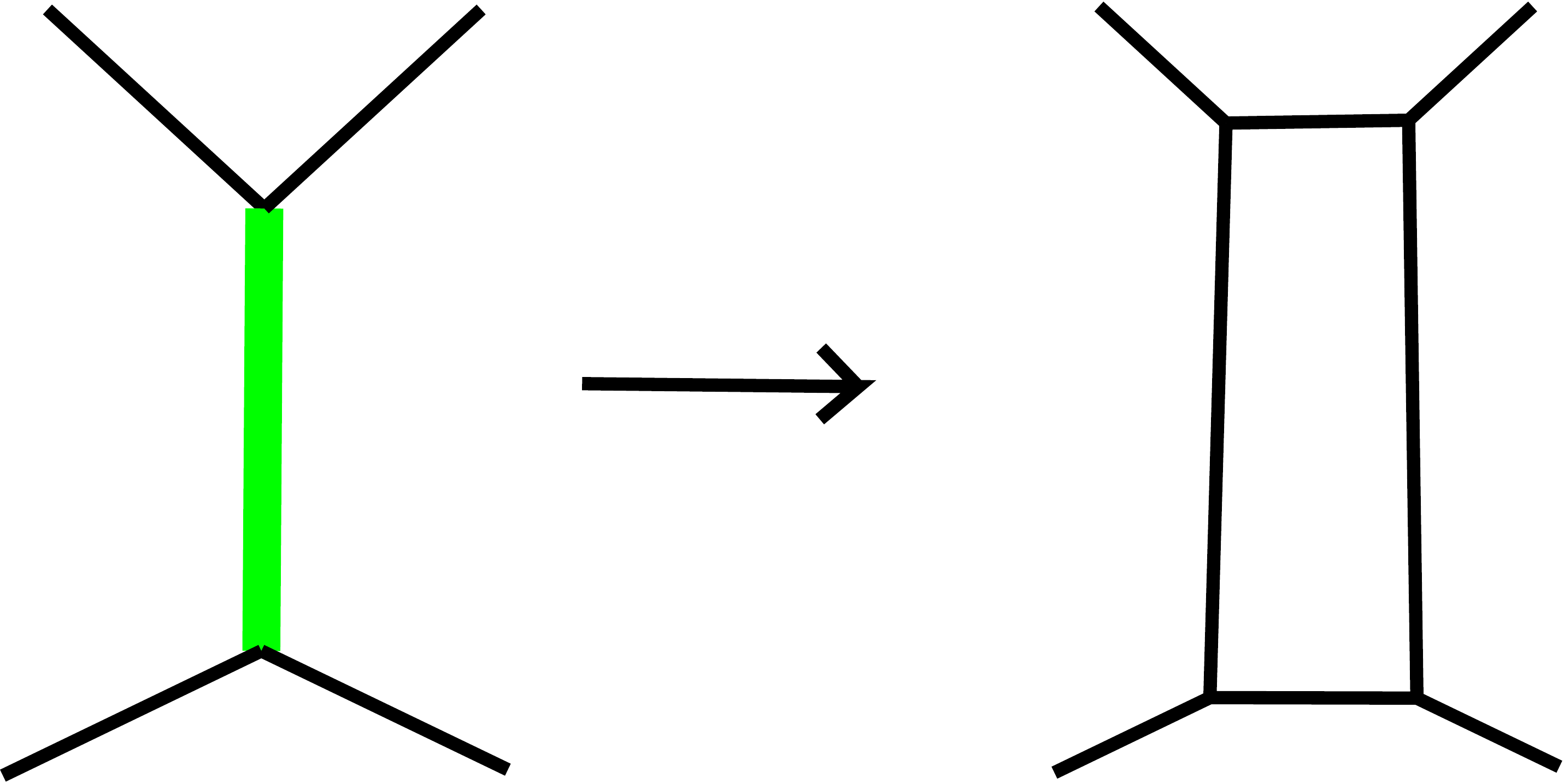}\\
\includegraphics[height=2.7cm]{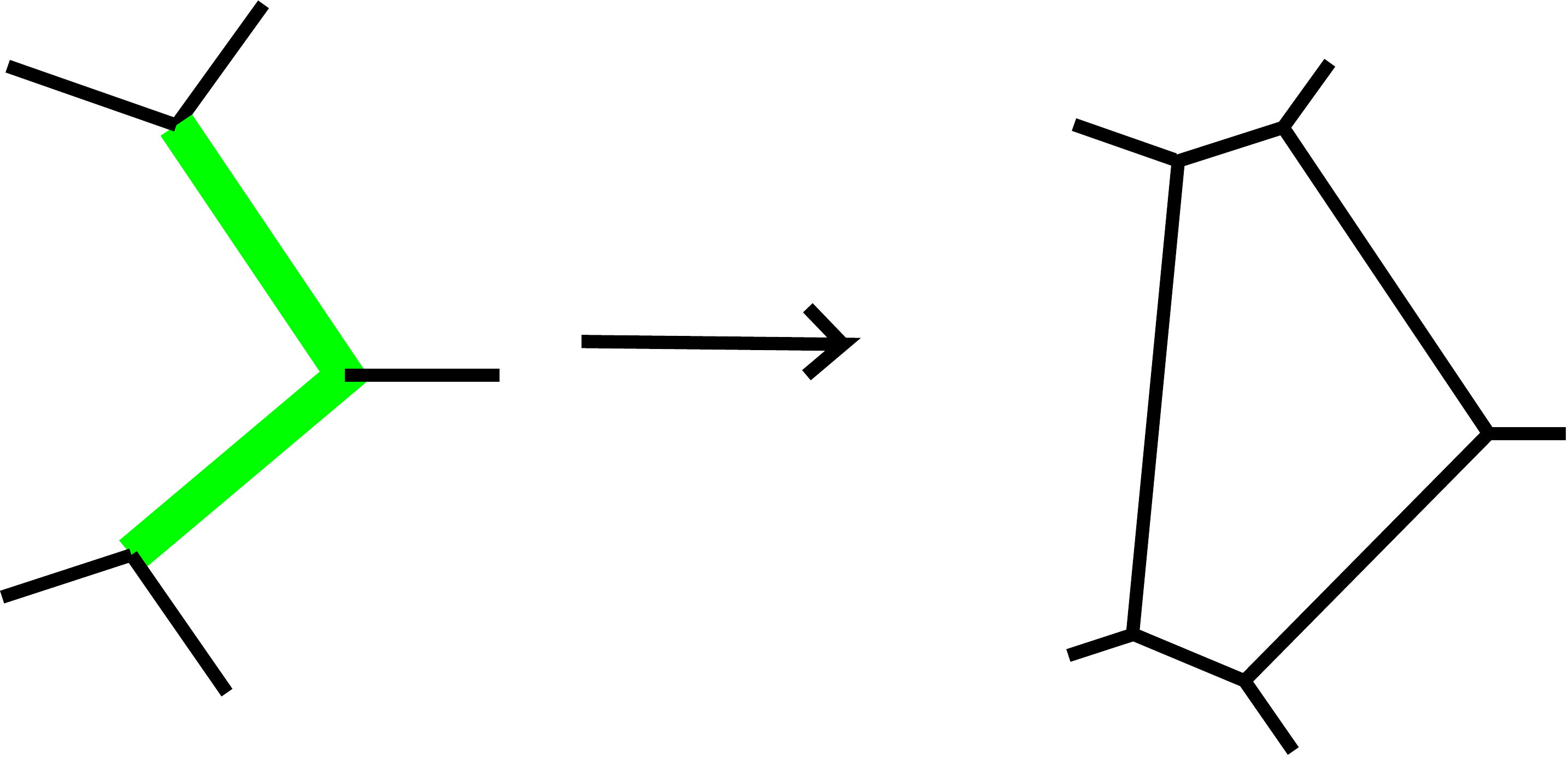}
\end{center}
\caption{Vertex-, edge- and $(2,k)$-truncations}\index{vertex-truncation}\index{edge-truncation}
\end{figure}
\begin{figure}
\begin{center}
\includegraphics[height=3cm]{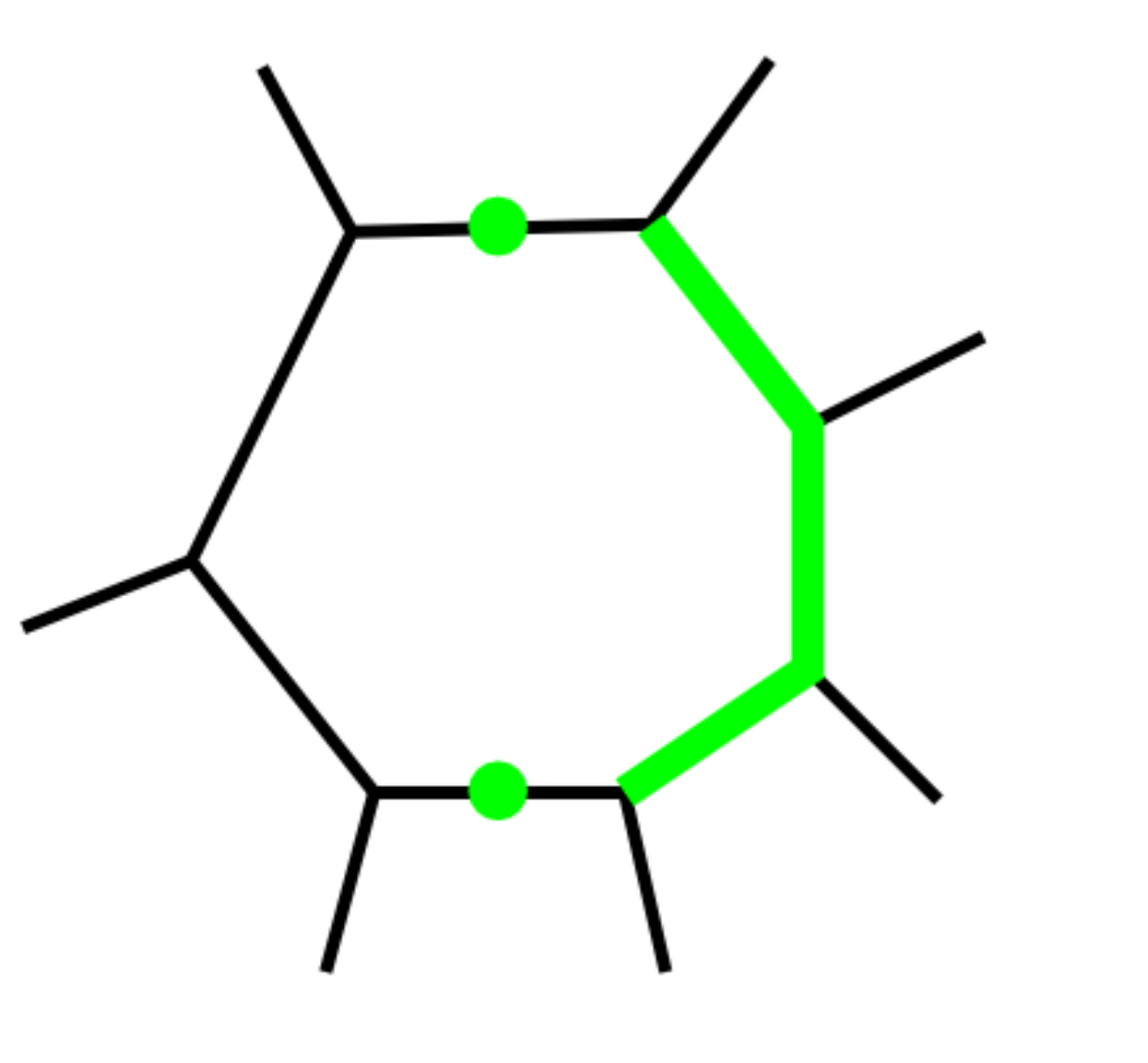}
\includegraphics[height=3cm]{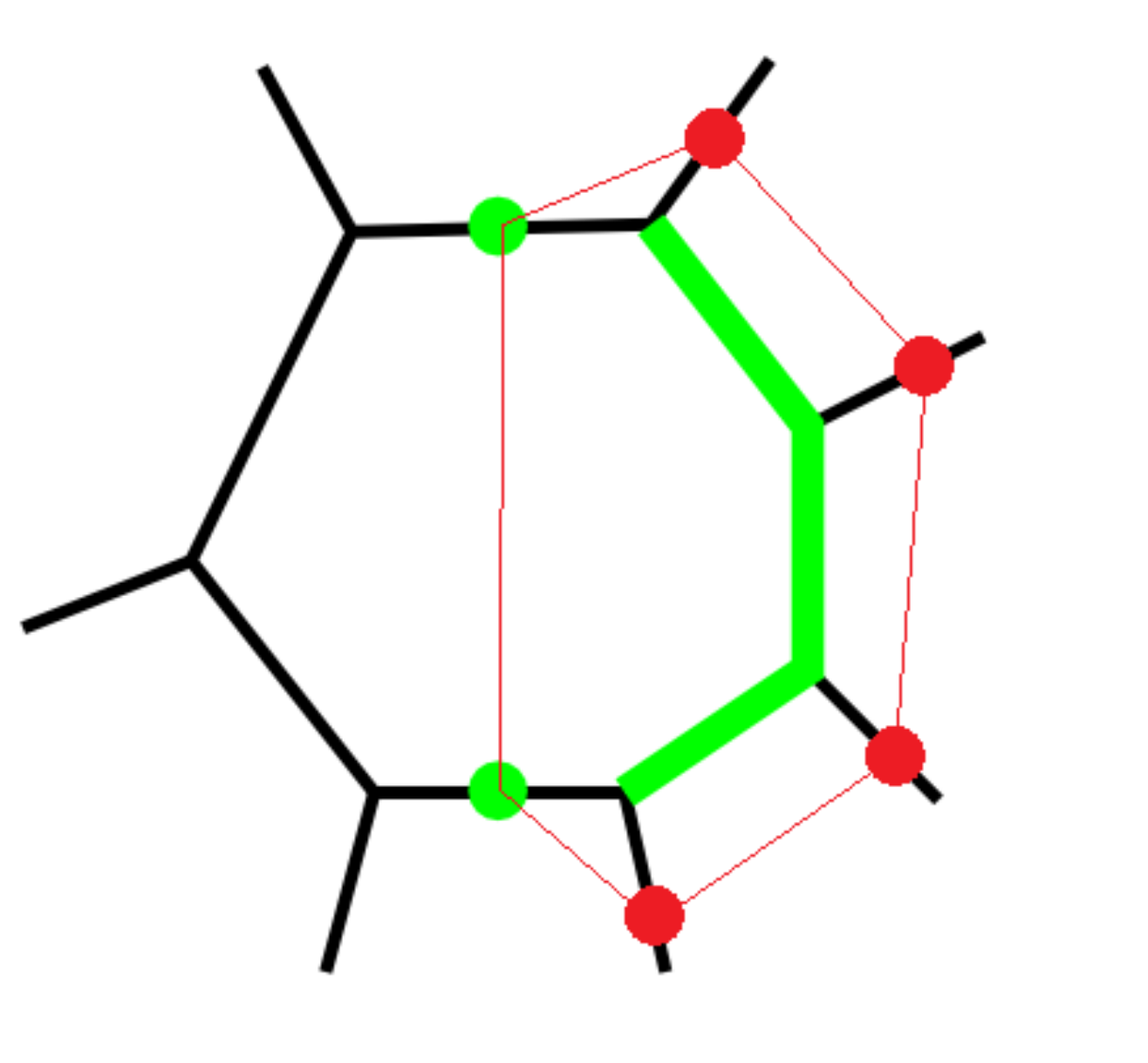}
\includegraphics[height=3cm]{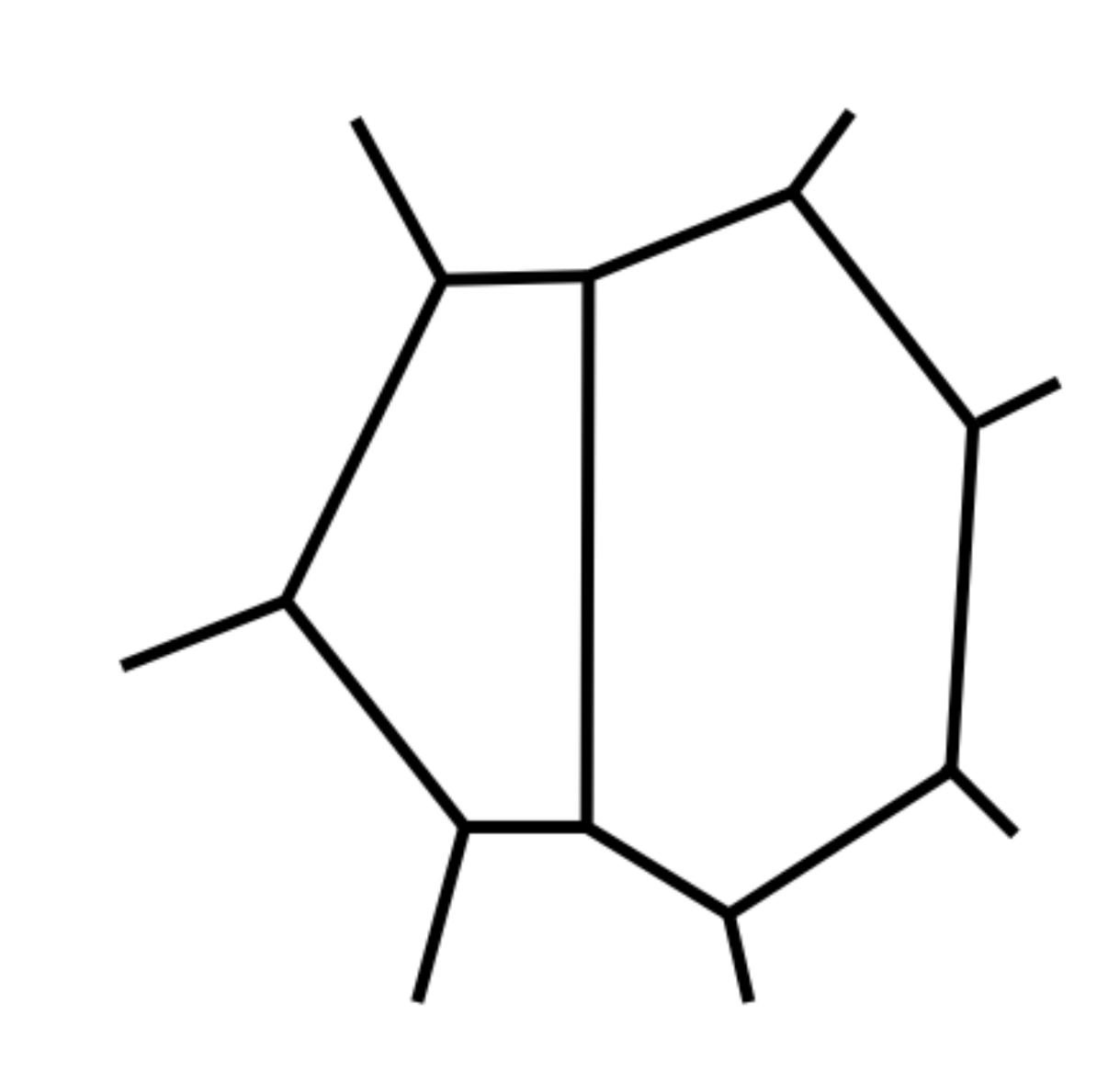}
\end{center}
\caption{$(3,7)$-truncation}
\end{figure}

\noindent {\bf Construction ($(s,k)$-truncation):}\index{polytope!$(s,k)$-truncation} Let $F_i$ be a {\em $k$-gonal face} of a simple $3$-polytope $P$.
\begin{itemlist}
\item choose {\em $s$ consequent edges} of $F_i$;
\item rotate the supporting hyperplane of $F_i$ around the axis passing
    through the midpoints of adjacent two edges (one on each side);
\item take the corresponding hyperplane truncation.
\end{itemlist}
We call it {\em $(s,k)$-truncation} \index{$(s,k)$-truncation}\index{polytope!$(s,k)$-truncation}.

\begin{example}
\begin{enumerate}
\item Vertex truncation is a $(0,k)$-truncation.
\item Edge truncation is a $(1, k)$-truncation.
\item The Endo-Kroto operation is a $(2, 6)$-truncation.
\end{enumerate}
\end{example}
\begin{figure}
\begin{center}
\includegraphics[scale=0.3]{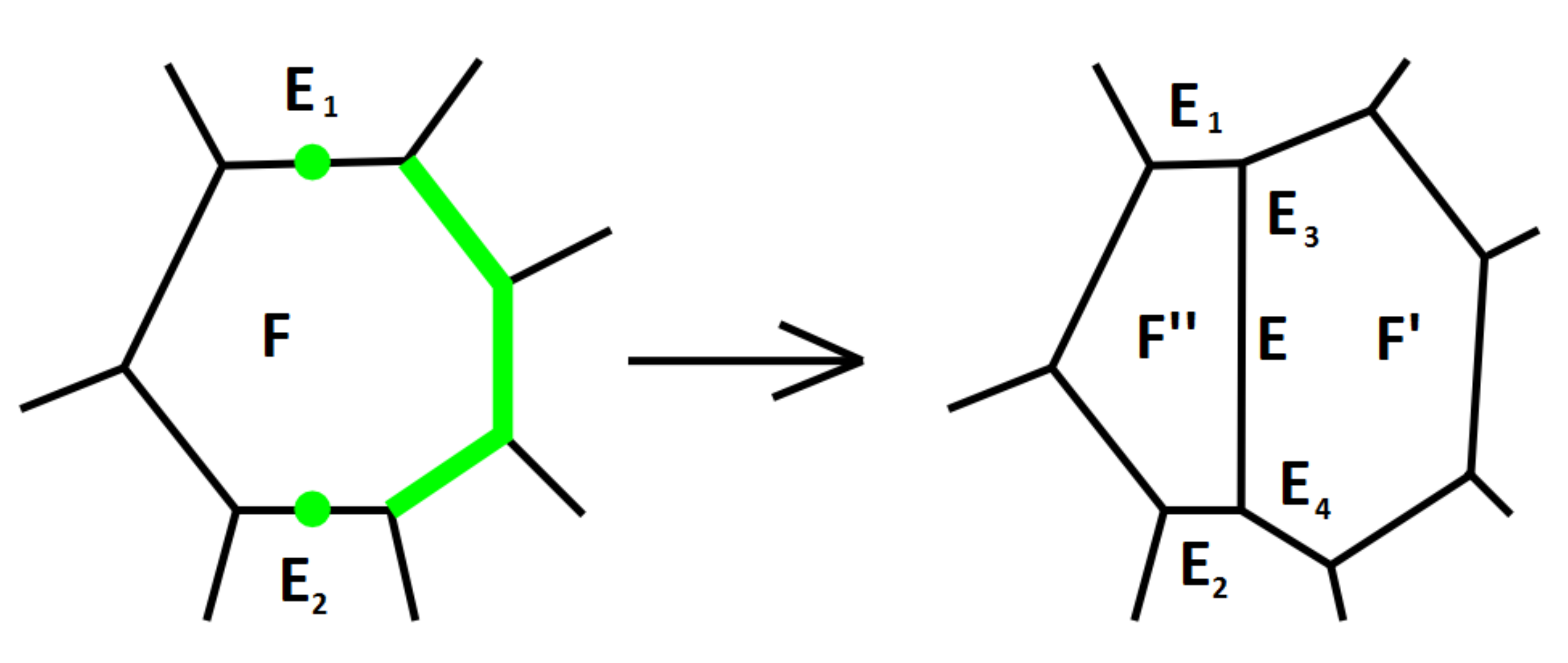}
\end{center}
\caption{$(s,k)$-truncation}\label{skfig}
\end{figure}
The next result follows from definitions.
\begin{proposition}
\begin{itemlist}
\item Under the $(s, k)$-truncation of  the polytope $P$ its facets that
do not contain the edges $E_1$ and $E_2$ (see Fig. \ref{skfig}) preserve the number of sides.
\item The facet $F$ is split into two facets: an $(s+3)$-gonal facet $F'$
    and a $(k-s+1)$-gonal facet $F''$, $F' \cap F'' = E$.
\item The number of sides of each of the two facets adjacent to $F$ along
    the edges $E_1$ and $E_2$ increases by one.
\end{itemlist}
\end{proposition}

\begin{remark}
We see that $(s,k)$-truncation is a combinatorial operation and is always defined. It is easy to show that the straightening along the edge $E$ on the right side is a combinatorially inverse operation. It is not always defined.\index{polytope!straightening along the edge}\index{straightening along an edge}
\end{remark}

\begin{definition}
If the facets intersecting $F$ by $E_1$ and $E_2$ (see Fig. \ref{skfig}) are $m_1$- and $m_2$-gons respectively, then we also call the corresponding operation an \emph{$(s,k;m_1,m_2)$-truncation}.\index{$(s,k;m_1,m_2)$-truncation}\index{polytope!$(s,k;m_1,m_2)$-truncation}

For $s=1$ combinatorially $(1,k;m_1,m_2)$-truncation is the same
operation as $(1,t;m_1,m_2)$-truncation of the same edge of the other facet
containing it. We call this operation simply a \emph{$(1;m_1,m_2)$-truncation}.\index{$(1;m_1,m_2)$-truncation}\index{polytope!$(1;m_1,m_2)$-truncation}
\end{definition}
\begin{remark}\label{trrem} \index{$(s,k)$-truncation!as growth operation}Let $P$ be a flag simple polytope. Then any $(s,k)$-truncation is a growth operation. Indeed, for $s=0$ and $s=k-2$ we have the vertex truncation, which can be considered as the substitution of the corresponding fragment for the three facets containing the vertex. For $0<s<k-2$, since $P$ is flag, the facets $F_{i_1}$ and $F_{i_{s+2}}$ intersecting $F$ by edges adjacent to truncated edges do not intersect; hence the union $F_{i_1}\cup F\cup F_{i_{s+2}}$ is bounded by a simple edge-cycle (see Fig. \ref{grtr} on the left). After the $(s,k)$-truncation the union of facets $F'\cup F''\cup F_{i_1}\cup F_{i_{s+2}}$ is bounded by combinatorially the same simple edge-cycle. We describe this operation by the scheme on Fig. \ref{grtr} on the right.\\
\begin{figure}
\begin{center}
\includegraphics[scale=0.3]{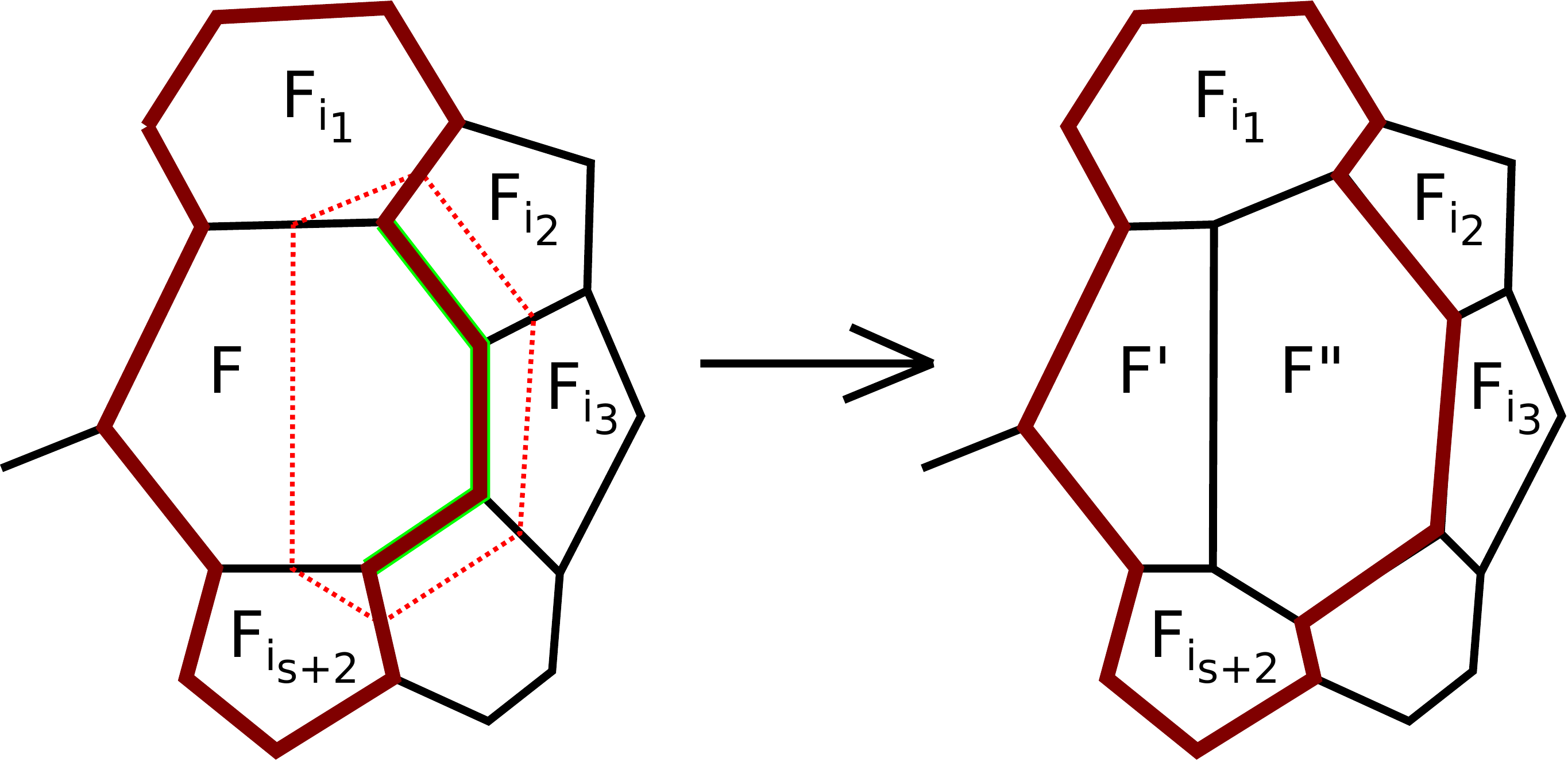}\qquad\includegraphics[scale=0.3]{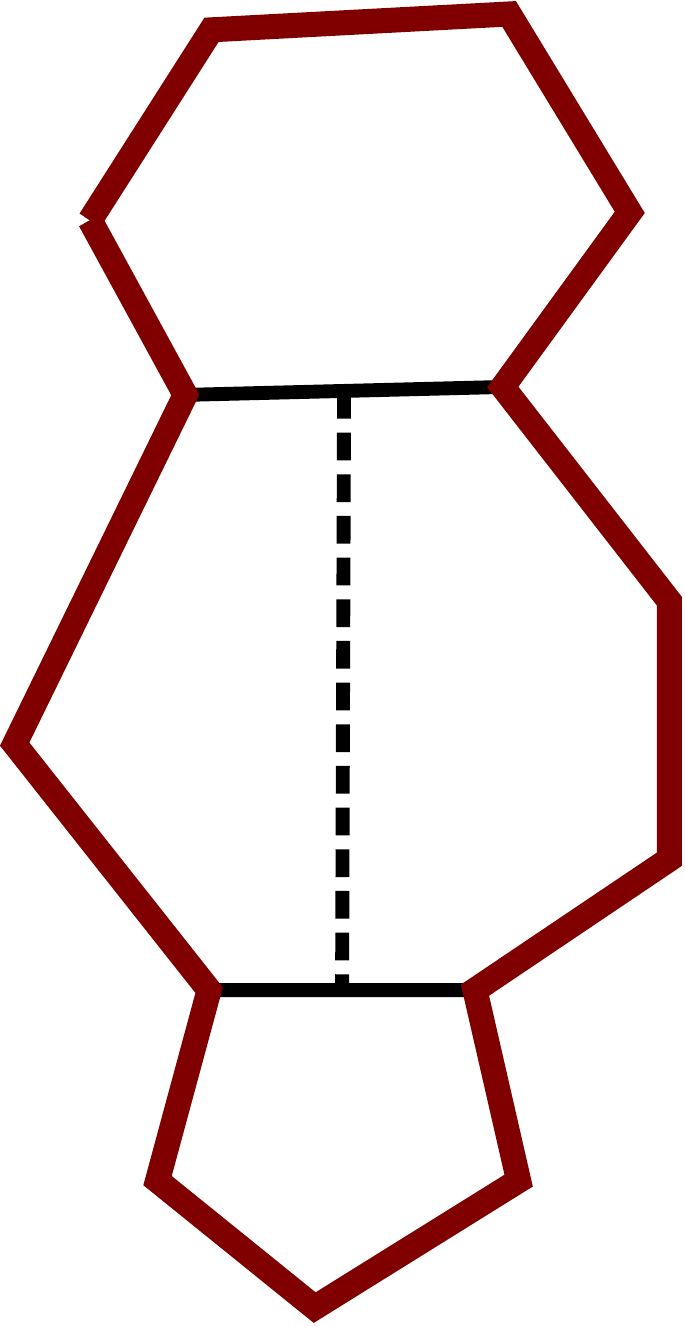}
\end{center}
\caption{$(s,k)$-truncation as a growth operation}\label{grtr}
\end{figure}
For $s=1$ as mentioned above the edge-truncation can be considered as a $(1,k)$-truncation and a $(1,t)$-truncation for two facets containing the truncated edge: an $s$-gon and a $t$-gon. This gives two different patches, which differ by one facet. 
\end{remark}

\noindent {\bf{Exercise:}} Consider the set of $k-s-2$ edges of the face $F$ that
are not adjacent to the $s$ edges defining the $(s, k)$-truncation. The
polytope $Q'$ obtained by the $(k-s-2, k)$-truncation along these edges is
combinatorially equivalent to the polytope $Q$. In particular 
\begin{itemlist}
\item The $(k-3, k)$-truncation is combinatorially equivalent
to the edge truncation;
\item The $(k-2, k)$-truncation is combinatorially equivalent
to the vertex truncation.
\end{itemlist}

\noindent{\bf Exercise:} Let $P$ be a flag $3$-polytope. Then the polytope obtained from $P$ by an $(s,k)$-truncation is flag if and only if $0<s<k-2$.

In \cite{Bu-Er15} the analog of Theorem \ref{3ptheorem} for flag polytopes was proved.
\begin{theorem}
A simple $3$-polytope is flag if and only if it is combinatorially
equivalent to a polytope obtained from the cube by a sequence of {\bf edge
truncations} and {\bf $(2,k)$-truncations, $k\geqslant 6$}.
\end{theorem}

\subsection{Construction of fullerenes by truncations}
\begin{definition}
Let $\mathfrak{F}_{-1}$ be the set of combinatorial simple polytopes with all facets pentagons and hexagons except for one singular facet quadrangle. 

Let $\mathfrak{F}$ be the set of all fullerenes.

Let $\mathfrak{F}_1$ be the set of simple polytopes with one singular facet heptagon adjacent to a pentagon such that either there are two pentagons with the common edge intersecting the heptagon and a hexagon (we will denote this fragment $F_{5567}$, see Fig. \ref{7556}), or for any two adjacent pentagons exactly one of them is adjacent to the heptagon. 
\begin{figure}
\begin{center}
\includegraphics[height=2cm]{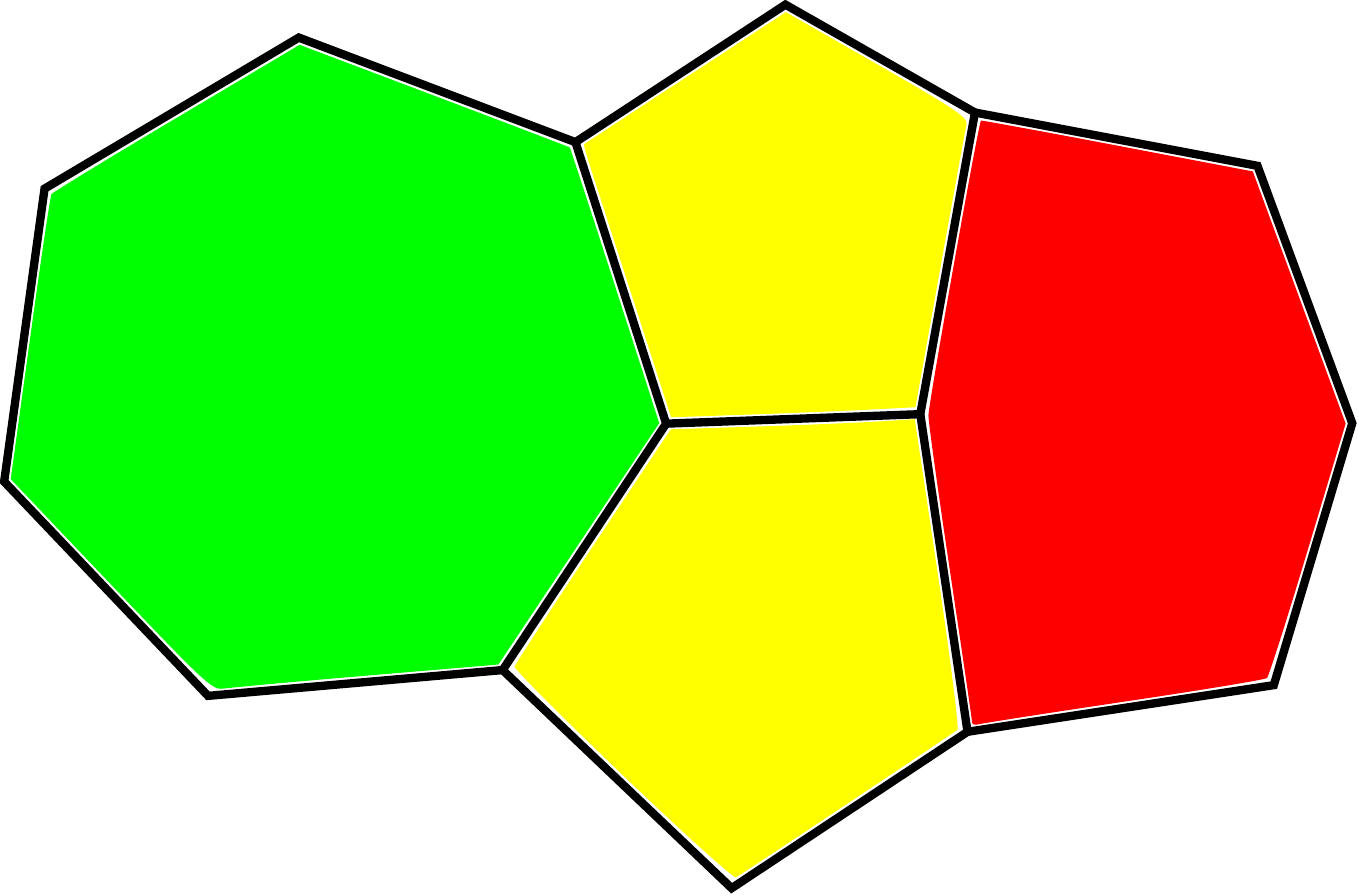}
\end{center}
\caption{Fragment $F_{5567}$}\label{7556}
\end{figure}
Set $\mathfrak{F}_s=\mathfrak{F}_{-1}\sqcup \mathfrak{F}\sqcup\mathfrak{F}_1$ to be se set of {\em singular fullerenes}\index{singular fullerene}\index{fullerene!singular fullerene}
\end{definition}

\begin{theorem} \label{trtheorem}
Any polytope in $\mathfrak{F}_s$ can be obtained from the dodecahedron by a sequence of $p_6$ truncations: $(1;4,5)$-, $(1;5,5)$-, $(2,6;4,5)$-, $(2,6;5,5)$-, $(2,6;5,6)$-, $(2,7;5,5)$-, and $(2,7;5,6)$-, in such a way that intermediate polytopes belong to $\mathfrak{F}_s$.

More precisely:
\begin{enumerate} 
\item any polytope in $\mathfrak{F}_{-1}$ can be obtained by a $(1;5,5)$- or $(1;4,5)$-truncation from a fullerene or a polytope in $\mathfrak{F}_{-1}$ respectively;
\item any polytope in $\mathfrak{F}_1$ can be obtained by a $(2,6;5,6)$- or  $(2,7;5,6)$-truncation from a fullerene or a polytope in $\mathfrak{F}_1$ respectively; 
\item any fullerene can be obtained by a  $(2,6;5,5)$-, $(2,6;4,5)$-, or $(2,7;5,5)$-truncation from a fullerene or a polytope from $\mathfrak{F}_{-1}$ or $\mathfrak{F}_1$ respectively.
\end{enumerate}
\end{theorem}
\begin{figure}
\begin{center}
\includegraphics[height=7cm]{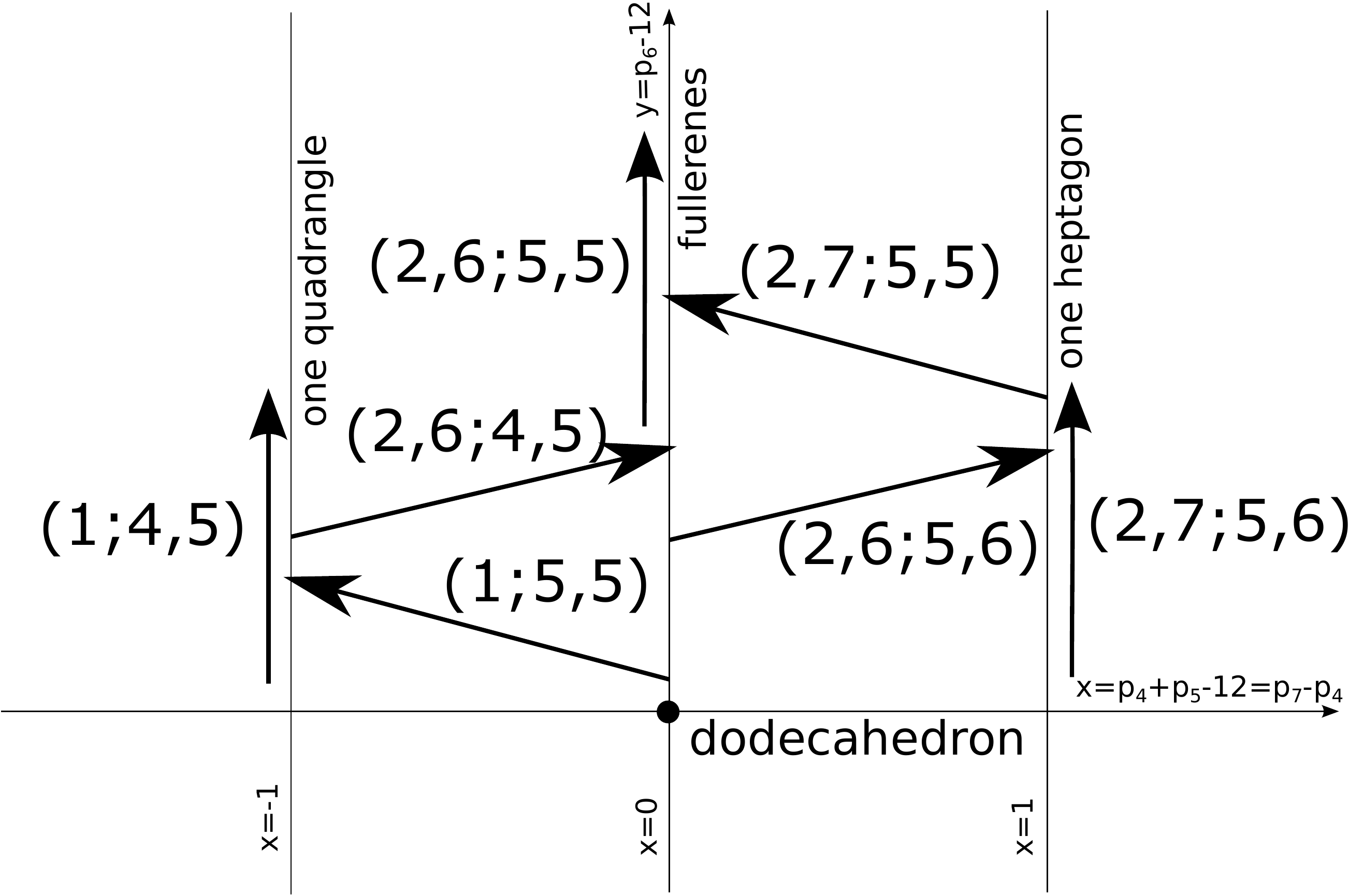}
\end{center}
\caption{Scheme of the truncation operations}\label{trfig}
\end{figure}
\begin{proof}
By Theorems \ref{3belts-theorem} and \ref{4belts-theorem} any polytope in $\mathfrak{F}_s$ has no $3$-belts and the only possible  $4$-belt surrounds a quadrangular facet. Hence for any edge the operation of straightening is well-defined. 

For (1) we need the following result.
\begin{lemma} 
There is no polytopes in $\mathfrak{F}_{-1}$ with the quadrangle surrounded by pentagons.
\end{lemma}
\begin{figure}
\begin{center}
\includegraphics[height=5cm]{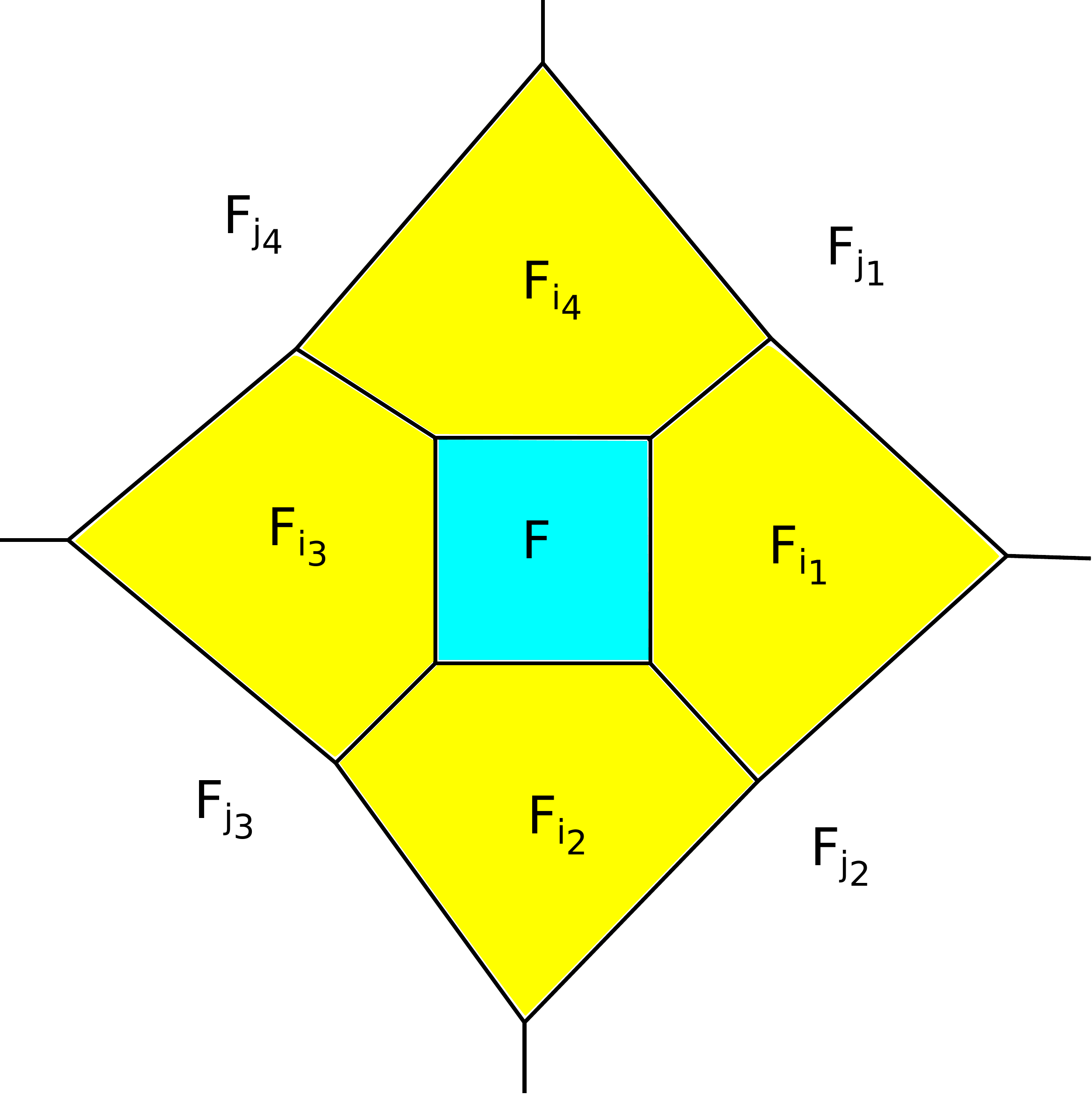}
\end{center}
\caption{Quadrangle surrounded by pentagons}\label{45555}
\end{figure}
\begin{proof}
Let the quadrangle $F$ be surrounded by pentagons $F_{i_1}$, $F_{i_2}$,$F_{i_3}$, and $F_{i_4}$ as drawn on Fig. \ref{45555}.  By Theorem \ref{4belts-theorem} we have the $4$-belt $\mathcal{B}=(F_{i_1},F_{i_2},F_{i_3},F_{i_4})$ surrounding $F$, and there are no other $4$-belts.  Let $\mathcal{L}=(F_{j_1},F_{j_2},F_{j_3},F_{j_4})$ be a $4$-loop that borders $\mathcal{B}$ along its boundary component different from $\partial F$. Its consequent facets are different.  If $F_{j_1}=F_{j_3}$, then we obtain a $4$-belt $(F,F_{i_1},F_{j_1},F_{i_3})$, which is a contradiction. Similarly $F_{j_2}\ne F_{j_4}$. Hence $\mathcal{L}$ is a simple $4$-loop. Since it is not a $4$-belt its two opposite facets intersect, say $F_{j_1}\cap F_{j_3}\ne\varnothing$. Then $F_{j_1}\cap F_{j_2}\cap F_{j_3}$ is a vertex and $F_{j_2}$ is a quadrangle. A contradiction. This proves the lemma. 
\end{proof}
Thus, for any polytope $P$ in $\mathfrak{F}_{-1}$ its quadrangle $F$ is adjacent to some hexagon $F_i$ by some edge $E$. Now straighten the polytope $P$ along the edge of $F$ adjacent to $E$ to obtain a new polytope $Q$ with a pentagon instead of $F_i$ and a pentagon or a quadrangle instead of the facet $F_j$ adjacent to $F$ by the edge of $F$ opposite to $E$. In the first case $Q$ is a fullerene and $P$ is obtained from $Q$ by a $(1;5,5)$-truncation. In the second case $Q\in\mathfrak{F}_{-1}$ and $P$ is obtained from $Q$ by a $(1;4,5)$-truncation. This proves (1).   

To prove (2) note that if $P\in\mathfrak{F}_1$ contains the fragment $F_{5567}$, then straightening along the common edge of pentagons gives a fullerene $Q$ such that $P$ is obtained from $Q$ by a $(2,6;5,6)$-truncation.

\begin{lemma}\label{555lemma}
If $P\in\mathfrak{F}_1$ does not contain the fragment $F_{5567}$, then 
\begin{enumerate}
\item $P$ does not contain fragments on Fig. \ref{555};
\item for any pair of adjacent pentagons any of them does not intersect any other pentagons.  
\end{enumerate}
\end{lemma} 
\begin{figure}
\begin{center}
\includegraphics[height=3cm]{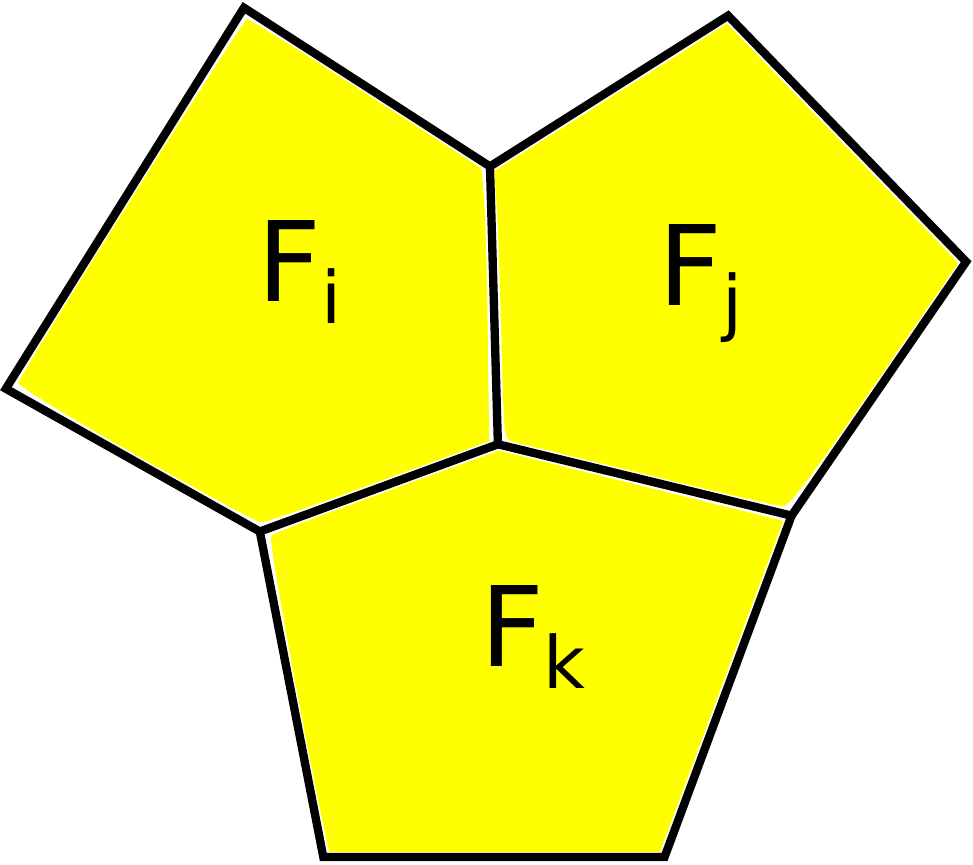}\;\includegraphics[height=3cm]{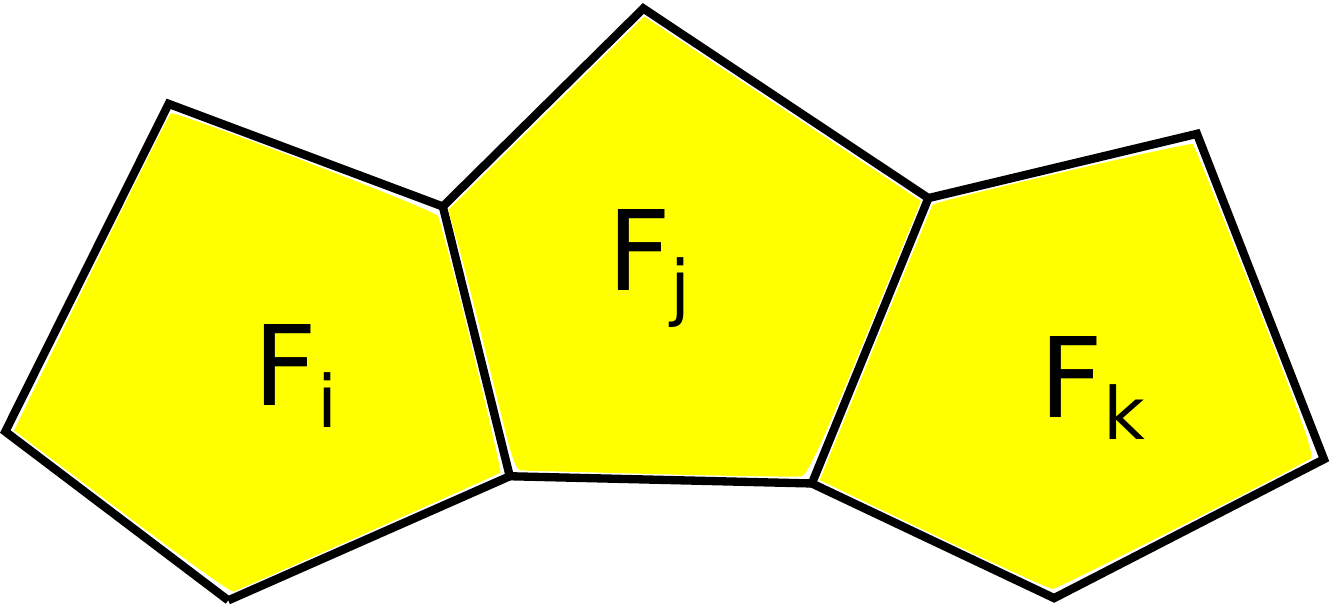}
\end{center}
\caption{Fragments that can not be present on the polytope in $\mathfrak{F}_1$ without the fragment $F_{5567}$}\label{555}
\end{figure}

\begin{proof}
Let $F_i$, $F_j$, $F_k$ be pentagons with a common vertex. Then for the pair $(F_i,F_j)$ exactly one pentagon intersects the heptagon $F$, say $F_i$. Also for the pair $(F_j,F_k)$ exactly one pentagon intersects $F$. This should be $F_k$. For the pair $(F_i,F_k)$ this is a contradiction.  

Let the pentagon $F_j$ intersects pentagons $F_i$ and $F_k$ by non-adjacent edges as shown in Fig. \ref{555} on the right. The heptagon $F$ should intersect exactly one pentagon of each pair $(F_i,F_j)$ and $(F_j,F_k)$. Then either it intersects $F_i$ and $F_k$ and does not intersect $F_j$, or it intersects $F_j$ and does not intersect $F_i$ and $F_k$. By Theorem \ref{3belts-theorem} $P$ has no $3$-belts; hence $F_i\cap F_k=\varnothing$. In the first case we obtain the $4$-belt $(F,F_i,F_j,F_k)$, which contradicts Theorem \ref{4belts-theorem}. In the second case $F$ intersects $F_j$ by one of the three edges different from $F_i\cap F_j$ and $F_j\cap F_k$. But any of these edges intersects either $F_i$, or $F_k$, which is a contradiction. 

Thus we have proved part (1) of the lemma. Let some pentagon of the pair of adjacent pentagons $(F_i,F_j)$, say $F_j$, intersects some other pentagon $F_k$. If the edges of intersection are adjacent in $F_j$, then we obtain the fragment on Fig. \ref{555} on the left. Else we obtain the fragment on Fig. \ref{555} on the right. A contradiction. This proves part (2) of the lemma. 
\end{proof}
Now assume that $P$ does not contain the fragment $F_{5567}$.

Let $(F_i,F_j)$ be a pair of two adjacent pentagons with $F_i$ intersecting the heptagon $F$. Then by Lemma \ref{555lemma} we obtain the fragment  on Fig. \ref{755} a). Since by Proposition \ref{34belts} the pair of adjacent facets is surrounded by a belt, the adjacent pentagons do not intersect other pentagons and exactly one of them intersects the heptagon. The straightening along the edge $F_i\cap F_p$ gives  a polytope $Q$ such that $P$ is obtained from $Q$ by a $(2,7;5,6)$-truncation. $Q$ has all facets pentagons and hexagons except for one heptagon adjacent to a pentagon. $Q$ contains the fragment $F_{5567}$; hence it belongs to $\mathfrak{F}_1$.
\begin{figure}
\begin{center}
\begin{tabular}{cc}
\includegraphics[height=3cm]{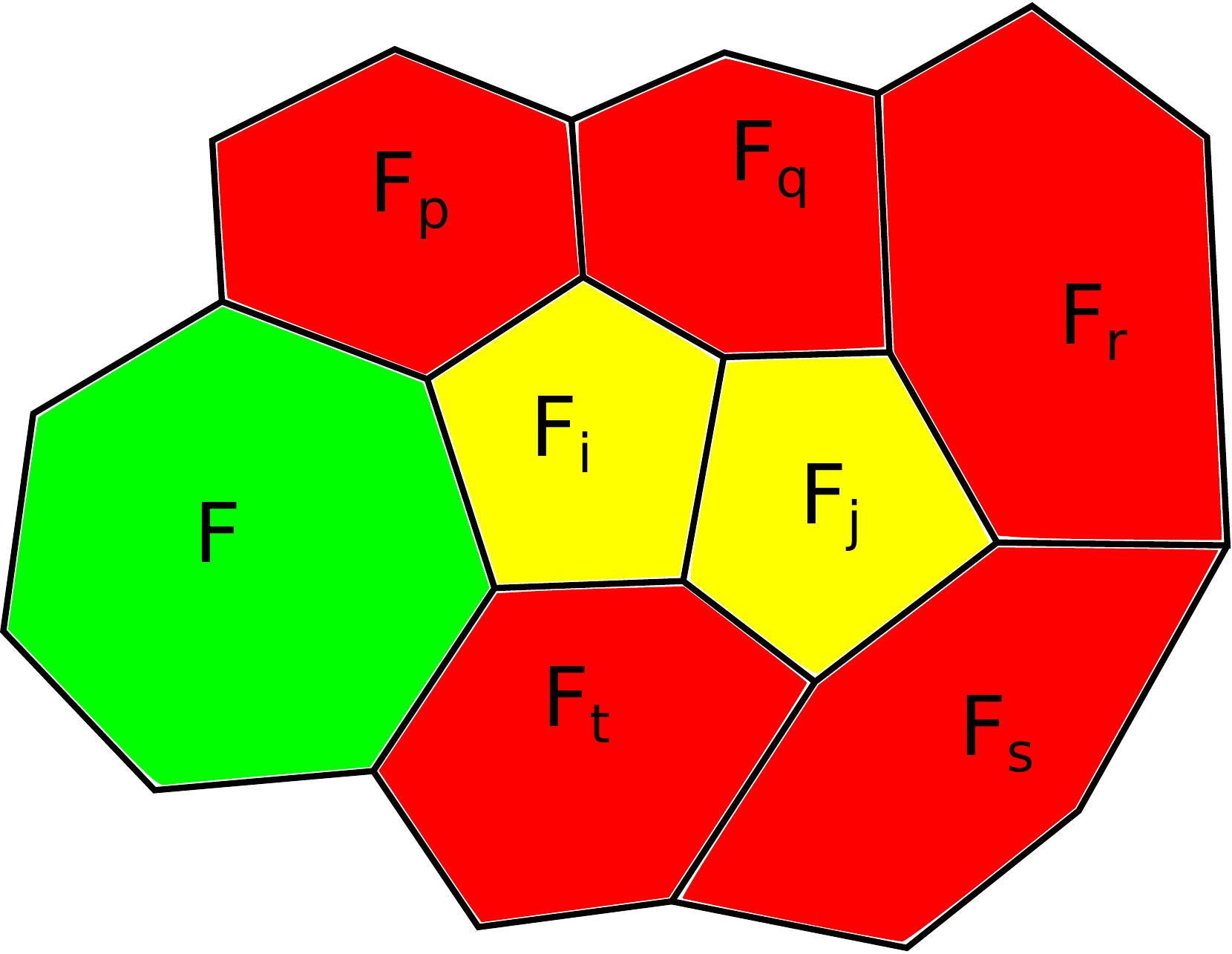}&\includegraphics[height=3cm]{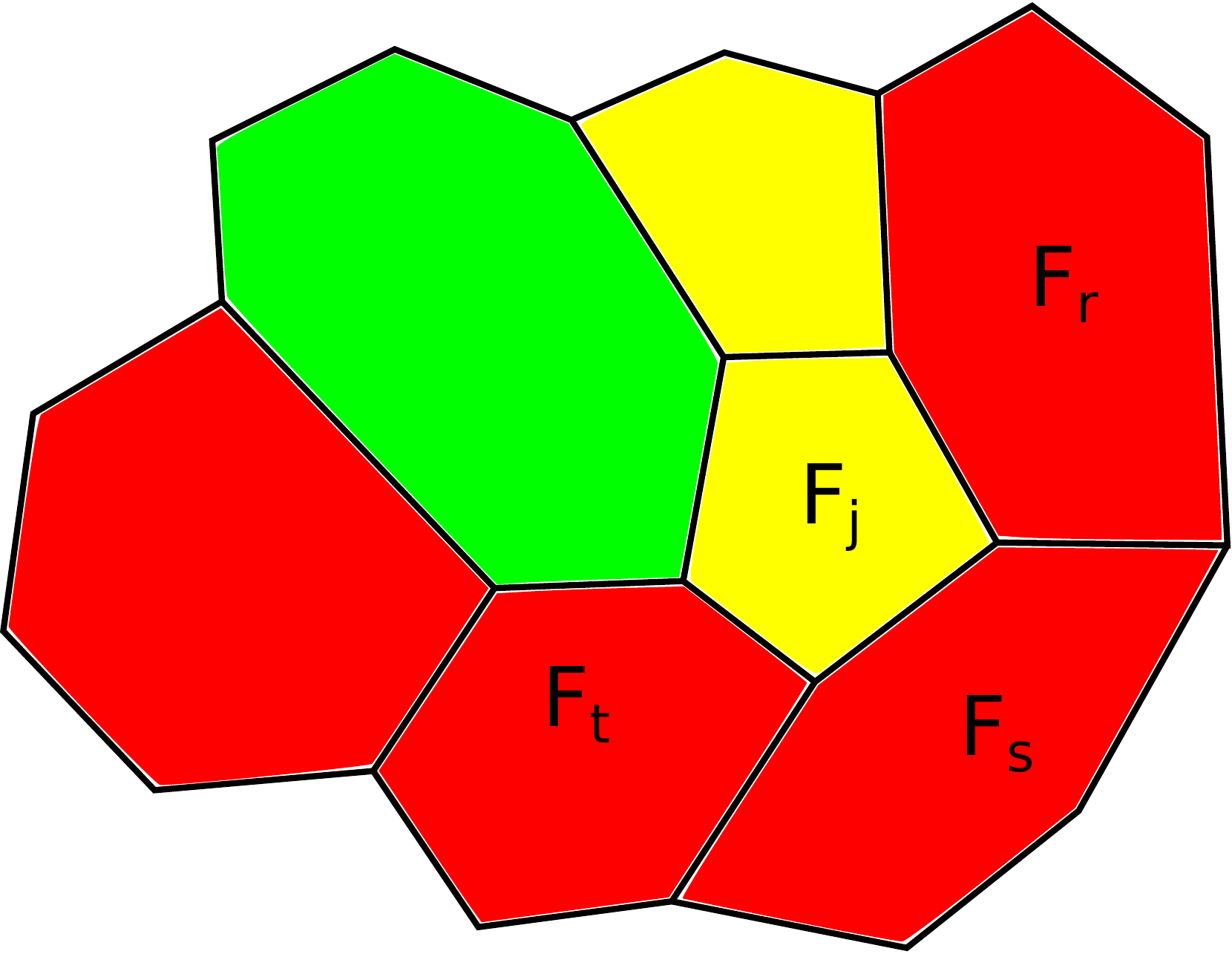}\\
a)&b)
\end{tabular}
\end{center}
\caption{a) facets surrounding the pair of adjacent pentagons; b) the same fragment after the straightening}\label{755}
\end{figure}
\begin{figure}
\begin{center}
\begin{tabular}{cc}
\includegraphics[height=4cm]{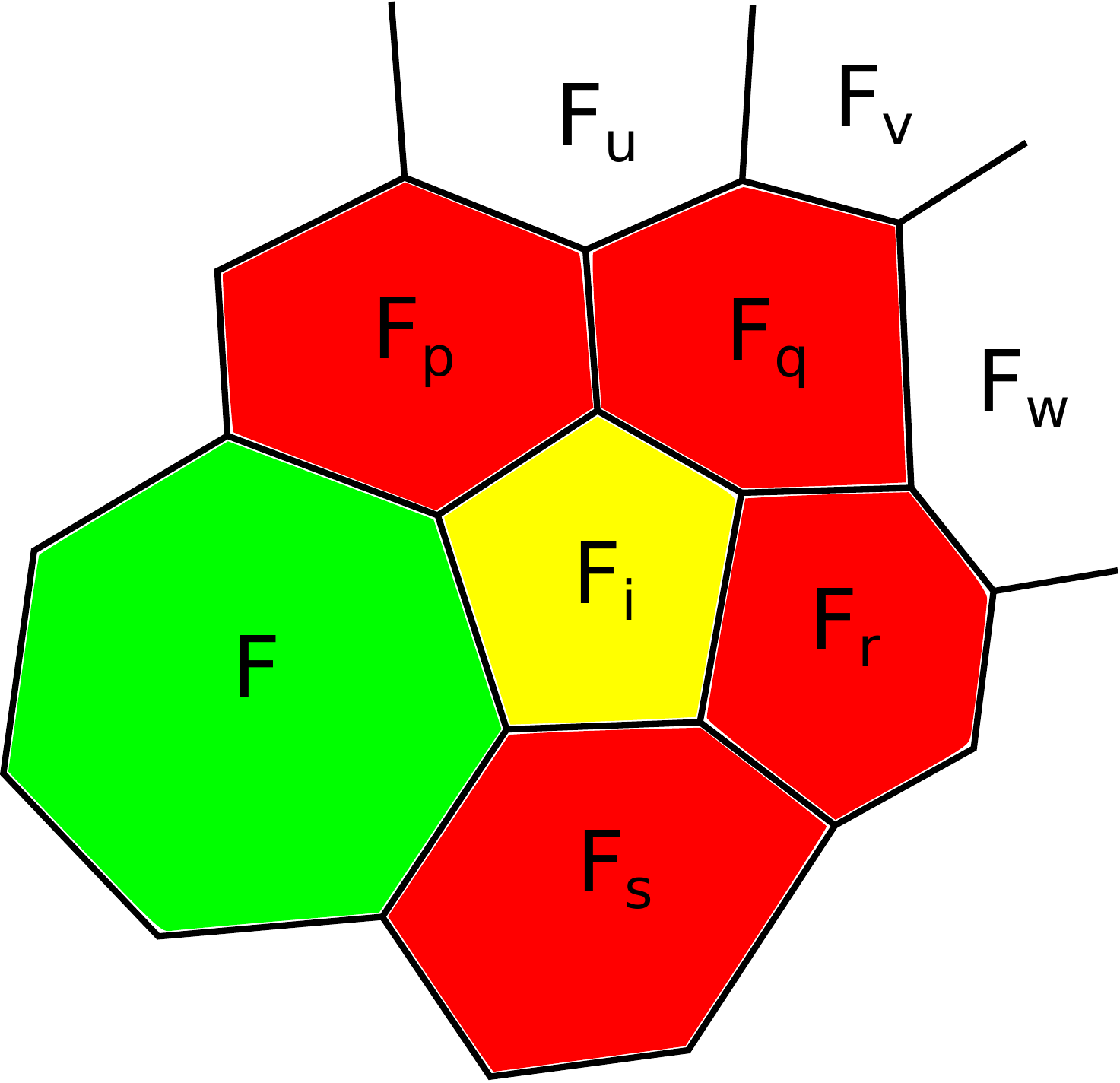}&\includegraphics[height=4cm]{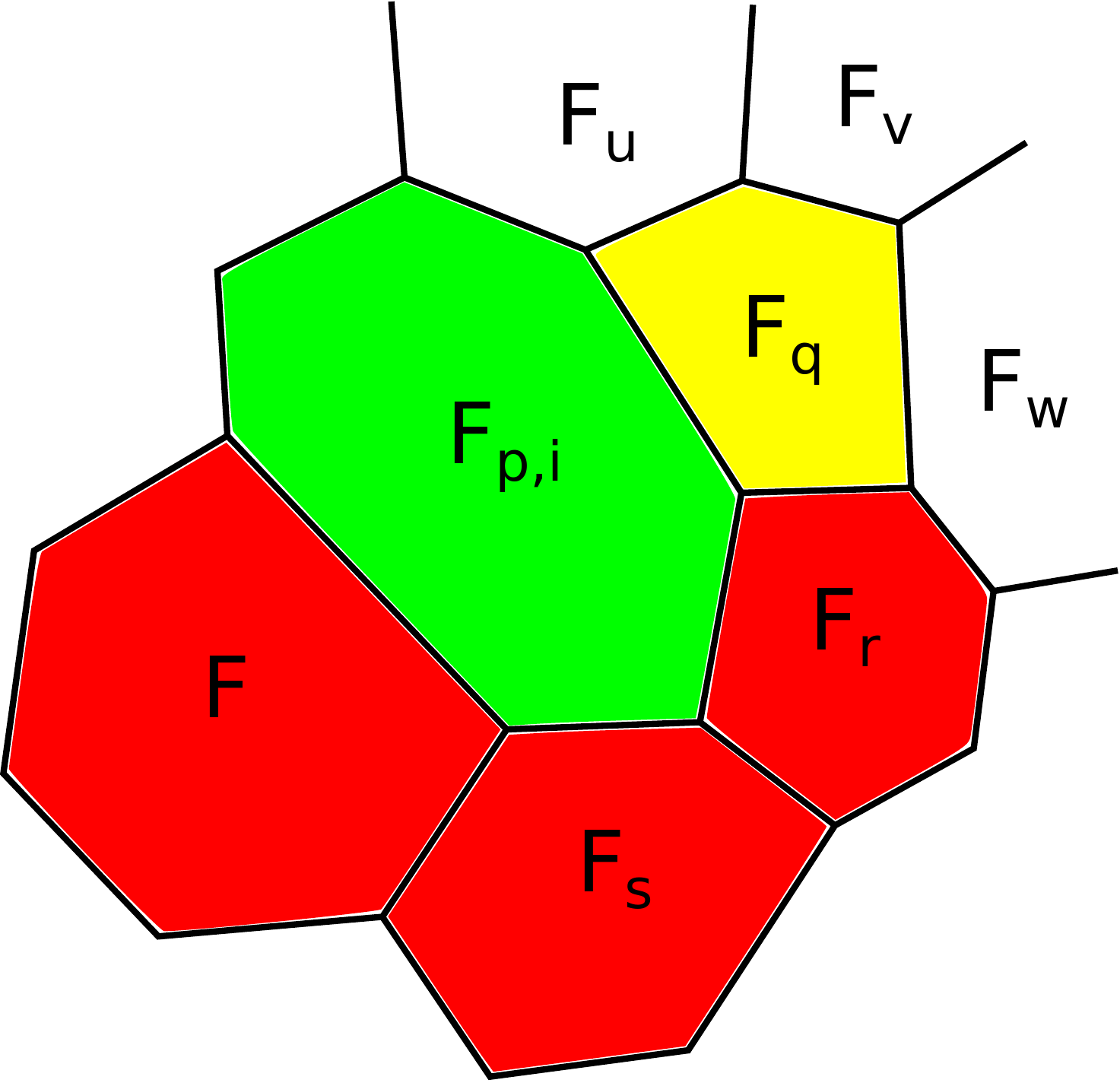}\\
a)&b)
\end{tabular}
\end{center}
\caption{a) facets surrounding a pentagon adjacent to the heptagon; b) the same fragment after the straightening}\label{75}
\end{figure}

Now let $P$ have no adjacent pentagons. Consider a pentagon adjacent to the heptagon $F$. Then it is surrounded by a $5$-belt $\mathcal{B}$ consisting of the heptagon and $4$ hexagons (Fig. \ref{75} a). The straightening along the edge $F_p\cap F_i$ gives a simple polytope $Q$ with the fragment on Fig. \ref{75} b) instead the fragment on Fig. \ref{75} a). The polytope $Q$ has all facets pentagons and hexagons except for one heptagon $F_{p,i}$ adjacent to the pentagon $F_q$.  Then $P$ is obtained from $Q$ by a $(2,7;5,6)$-truncation. We claim that $Q\in \mathfrak{F}_1$. Indeed, if $Q$ has the fragment $F_{5567}$, it is true.  If $Q$ has no such fragments consider two adjacent pentagons of $Q$. The polytopes $P$ and $Q$ have the same structure outside the fragments in consideration; hence  $Q$ has the same pentagons as $P$ except for $F_q$, which appeared instead of $F_i$. Also $P$ has all pentagons isolated; hence one of the adjacent pentagons is $F_q$. The second pentagon $F_t$ should be adjacent to the hexagon $F_q$ in $P$; hence it should be one of the facets $F_u$, $F_v$, or $F_w$ on Fig. \ref{75} a). Each of these facets is different from $F$, since they lie outside the $5$-belt $\mathcal{B}$ containing $F$. And in each case the pentagon $F_t$ is isolated in $P$ by assumption. 

If $F_t=F_u$, then $F_v$ is a hexagon, since $F_v\ne F$ and $F_v$ is not a pentagon. Then $Q$ contains the fragment $F_{5567}$, which is a contradiction. Thus $F_u$ is a hexagon.

If $F_t$ is one of the facets $F_v$ and $F_w$, then the other facet is a hexagon and there are no pairs of adjacent pentagons in $Q$ other than $(F_q,F_t)$. Each of the facets $F_v$, $F_w$ in $Q$  belongs to the $5$-belt surrounding $F_q$ together with $F_{p,i}$ and is not successive with it; hence $F_v$ and $F_w$ do not intersect $F_{p,i}$ in $Q$. Thus $F_t\cap F_{p,i}=\varnothing$ and $Q\in \mathfrak{F}_1$.
This proves (2).

To prove (3) consider a fullerene $P$. If it contains the fragment on Fig. \ref{5566} a) then the straightening along the edge $F_i\cap F_j$ gives a fullerene $Q$ such that $P$ is obtained from $Q$ by a $(2,6;5,5)$-truncation (the Endo-Kroto operation). Let $P$ contain no such fragments. 

\begin{figure}
\begin{center}
\begin{tabular}{cc}
\includegraphics[height=4cm]{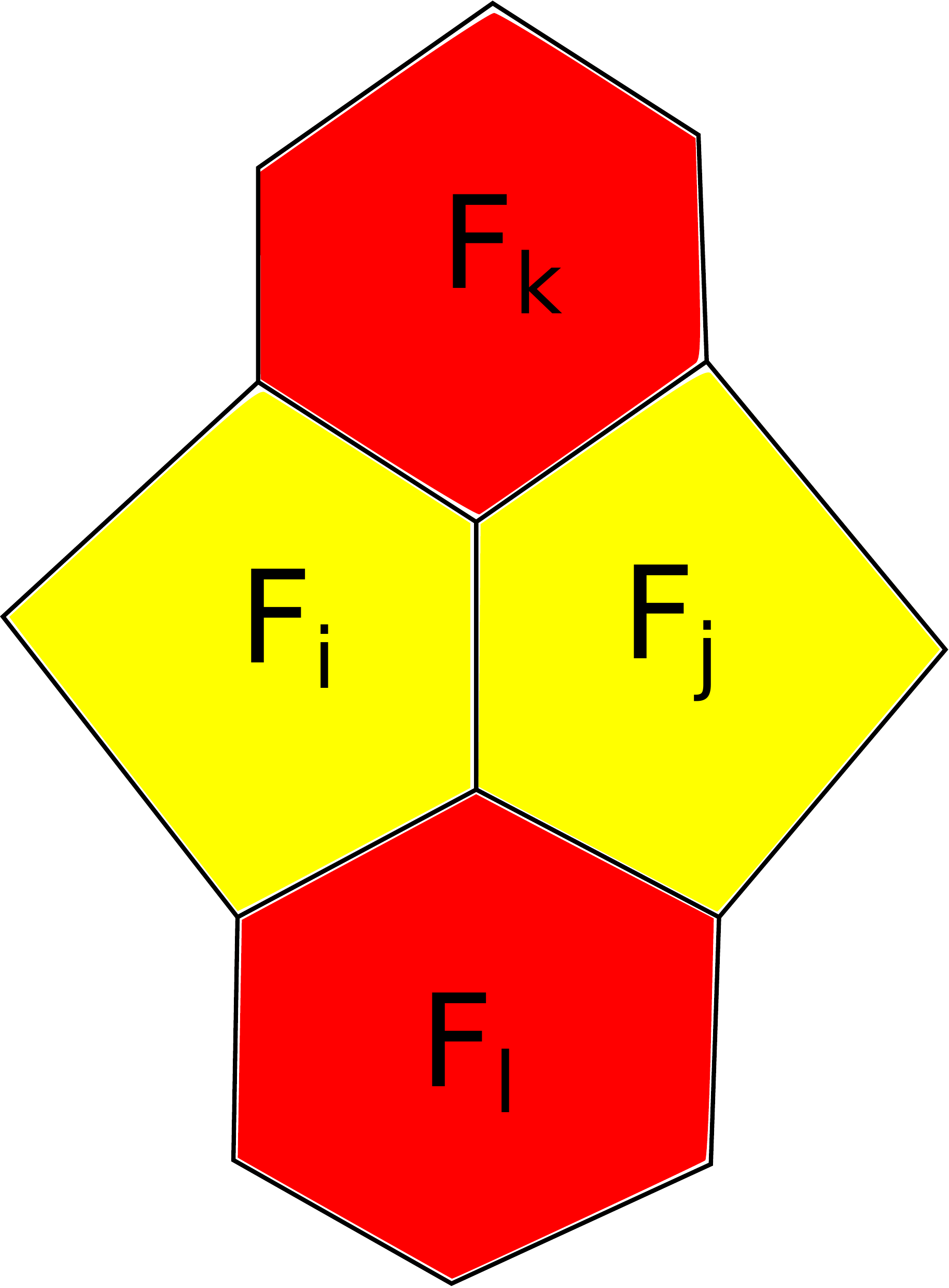}&\includegraphics[height=4cm]{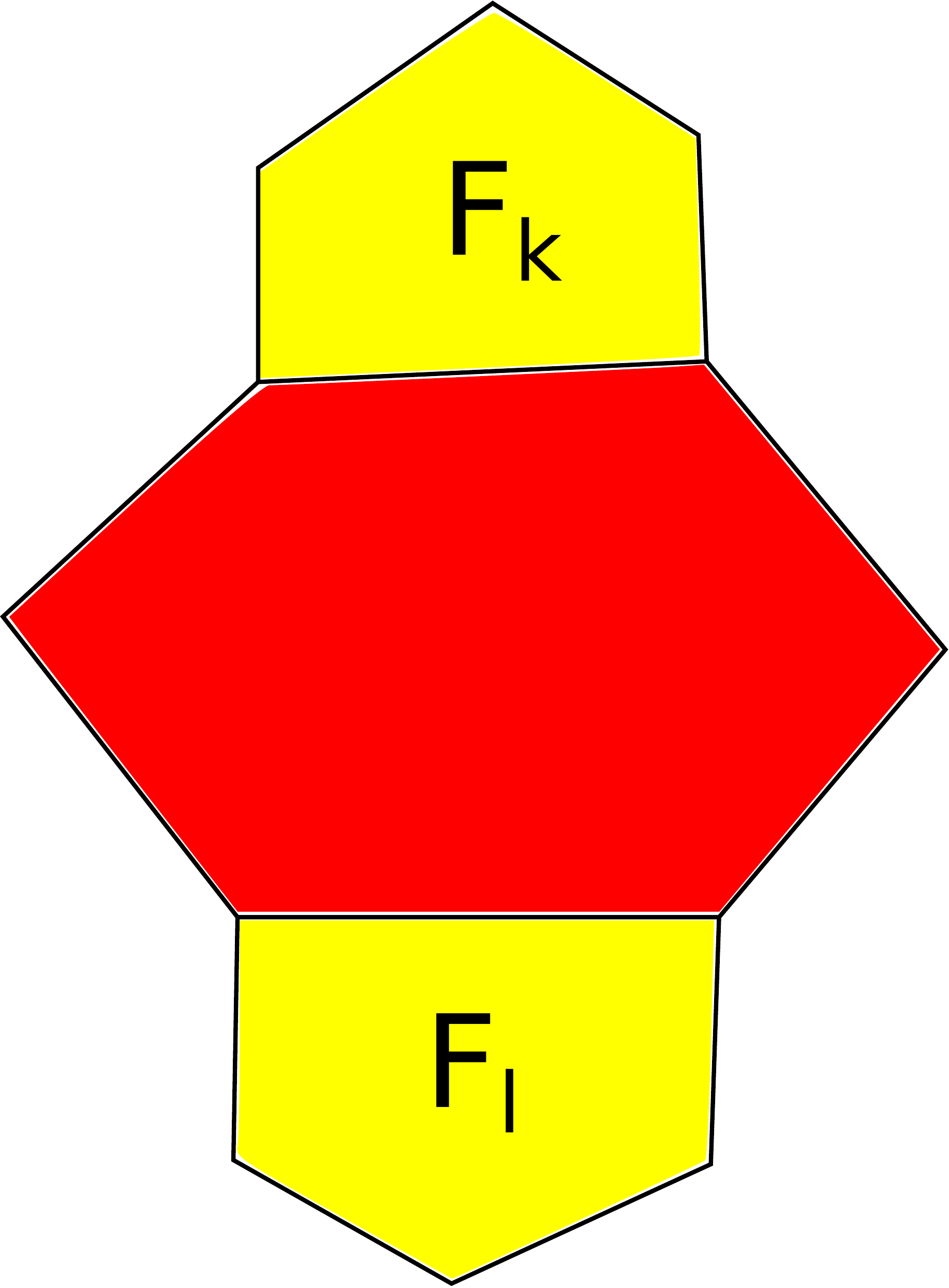}\\
a)&b)
\end{tabular}
\end{center}
\caption{a) Two adjacent pentagons with two hexagons; b) the same fragment after the straightening}\label{5566}
\end{figure}

If $P$ has two adjacent pentagons, then one of the connected components of unions of pentagons has more than two pentagons. If $P$ is not combinatorially equivalent to the dodecahedron, then each component is a sphere with holes. Consider the connected component with more than one pentagon and a vertex $v$ on its boundary lying in two pentagons $F_i$ and $F_j$. Then the third face containing  $v$ is a hexagon. Since $P$ contains no fragments on Fig. \ref{5566} a), the other facet intersecting the edge $F_i\cap F_j$ by the vertex is a pentagon and we obtain the fragment  on Fig. \ref{5556} a). Then the straightening along the edge $F_i\cap F_j$ gives the polytope $Q\in\mathfrak{F}_{-1}$ such that $P$ is obtained from $Q$ by a $(2,6;4,5)$-truncation.

\begin{figure}
\begin{center}
\begin{tabular}{cc}
\includegraphics[height=4cm]{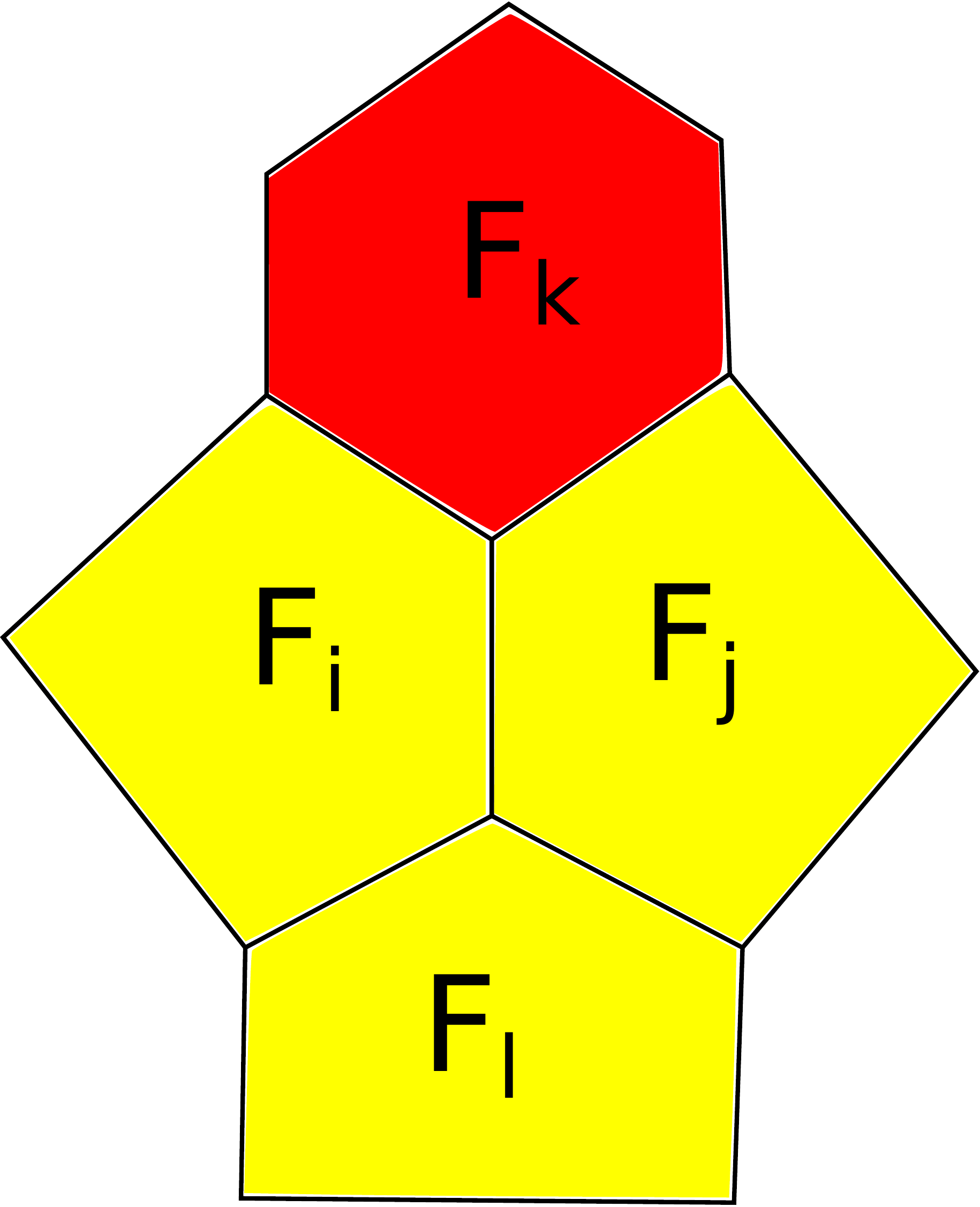}&\includegraphics[height=4cm]{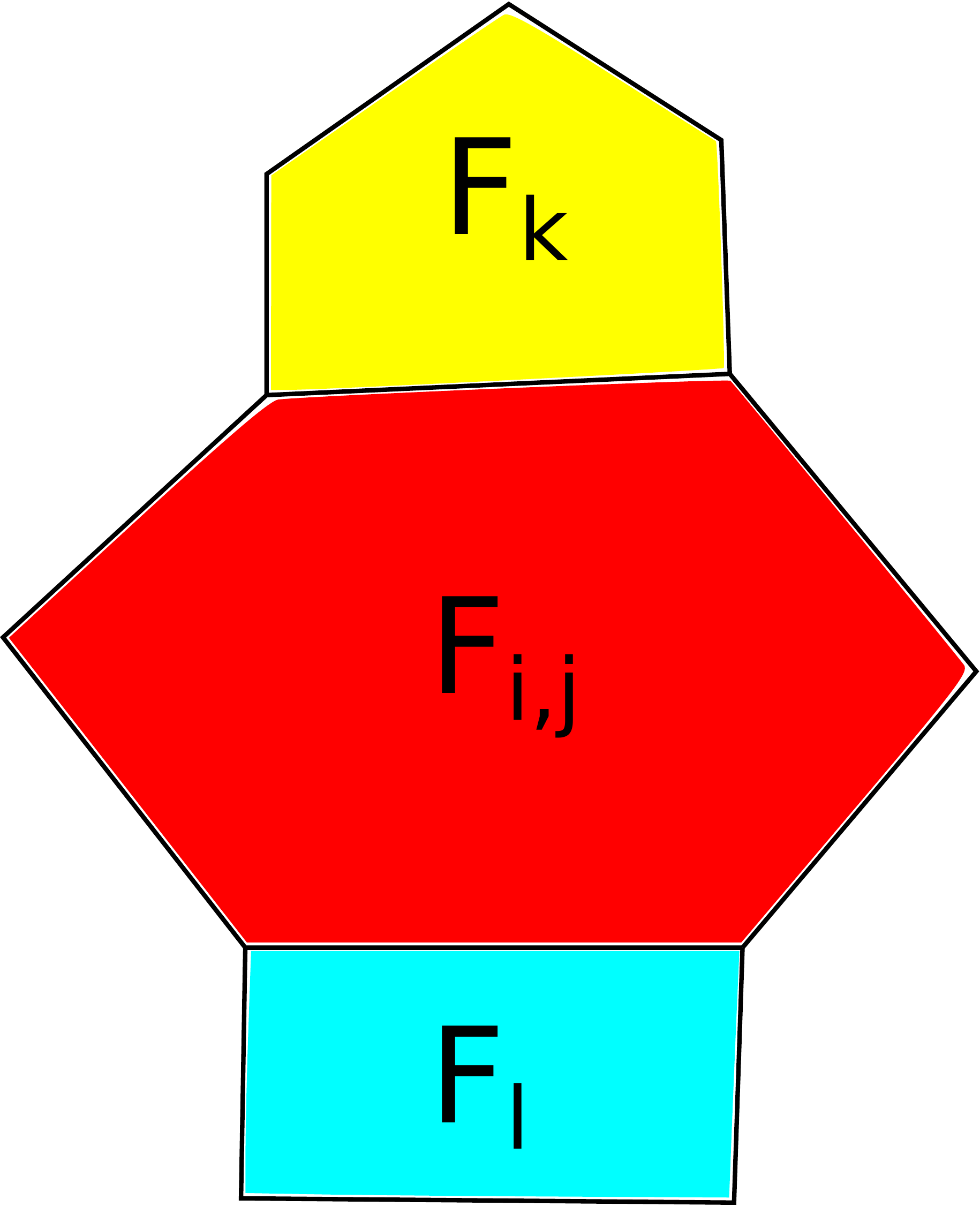}\\
a)&b)
\end{tabular}
\end{center}
\caption{a) Three adjacent pentagons and a hexagon; b) the same fragment after the straightening}\label{5556}
\end{figure}

If $P$ has no adjacent pentagons, then consider the pentagon $F_i$ adjacent to a hexagon $F_j$. The straightening along the edge $F_i\cap F_j$ gives the polytope $Q$ with all facets pentagons and hexagons except for one heptagon $F_{i,j}$ adjacent to a pentagon. $P$ is obtained from $Q$ by a $(2,7;5,5)$-truncation. We claim that $Q\in\mathfrak{F}_1$. Indeed, if $Q$ contains the fragment $F_{5567}$, then it is true. Else consider two adjacent pentagons in $Q$. The polytopes $P$ and $Q$ have the same structure outside the fragments on Fig \ref{5666}; hence  $Q$ has the same pentagons as $P$ except for pentagons $F_k$ and $F_l$, which appeared instead of $F_i$. Also $P$ has all pentagons isolated; hence one of the adjacent pentagons is $F_k$ or $F_l$.  We have $F_k\cap F_l=\varnothing$, else $(F_k,F_l,F_{i,j})$ is a $3$-belt. Hence the other adjacent pentagon $F_t$ does not belong to $\{F_k,F_l\}$. If $F_t$ is adjacent to the heptagon $F_{i,j}$, then in $P$ it is adjacent to $F_i$ or $F_j$. Since $F_i$ is an isolated pentagon, this is impossible. Hence $F_t$ should be adjacent to $F_j$. Then $F_t$ is one of the facets $F_u$, $F_v$, $F_w$ on Fig. \ref{5666}. Let $F_t=F_u$. Since $F_u$ is an isolated pentagon in $P$,  the facet $F_p$ is a hexagon on $P$ and on $Q$, since $F_p\ne F_l$ because $F_k\cap F_l=\varnothing$. Then we obtain the fragment $F_{5567}$, which is a contradiction. The same argument works for $F_w$ instead of $F_u$.  If $F_t=F_v$, then $F_v\cap F_k\ne\varnothing$, or $F_v\cap F_l\ne\varnothing$, which is impossible, since this gives the $3$-belts $(F_k,F_j,F_v)$, or $(F_l,F_j,F_v)$. Thus, $F_t$ does not intersect  the heptagon $F_{i,j}$, and $Q\in\mathfrak{F}_1$. This finishes the proof of (3) and of the theorem. 
\begin{figure}
\begin{center}
\begin{tabular}{cc}
\includegraphics[height=4cm]{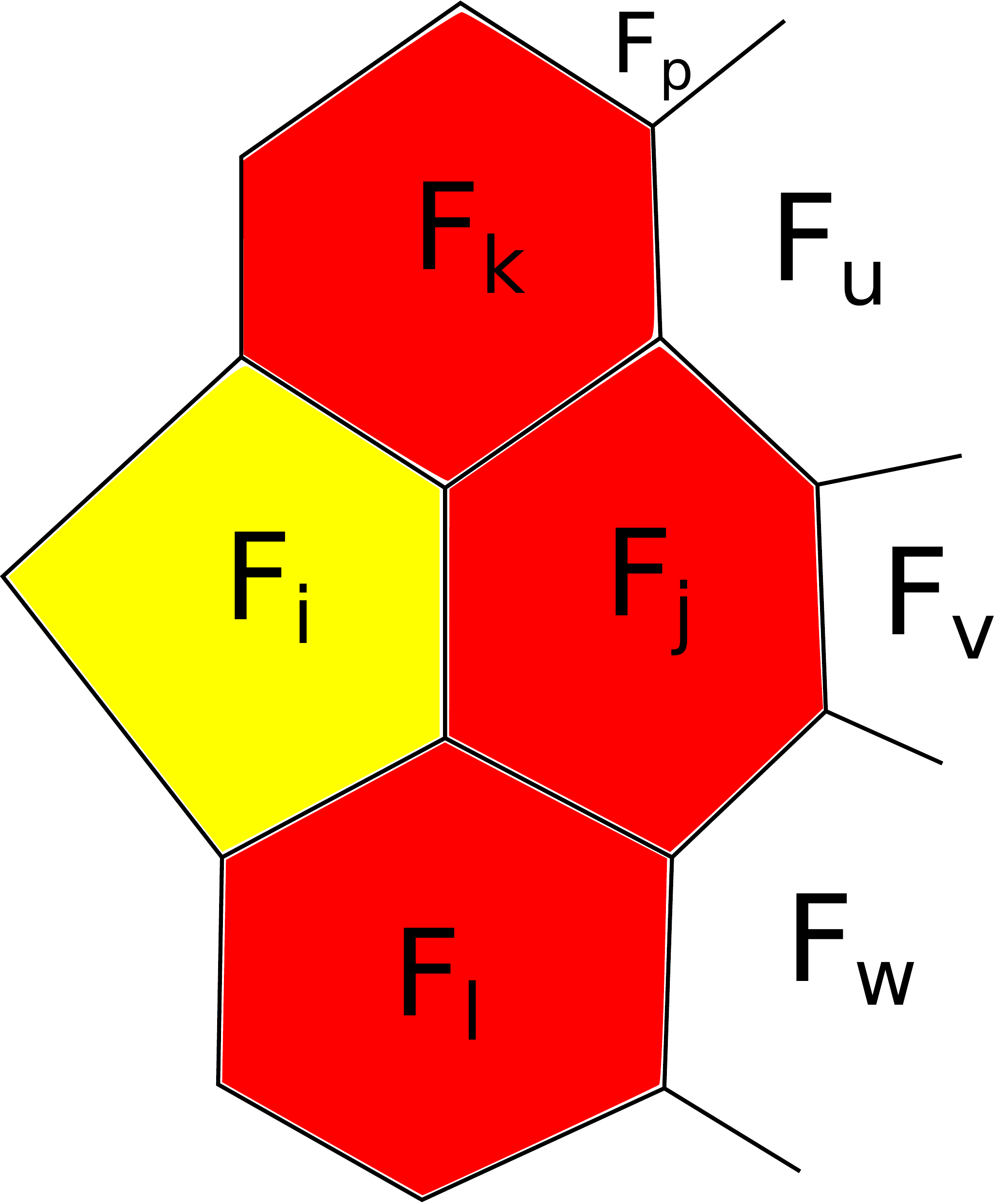}&\includegraphics[height=4cm]{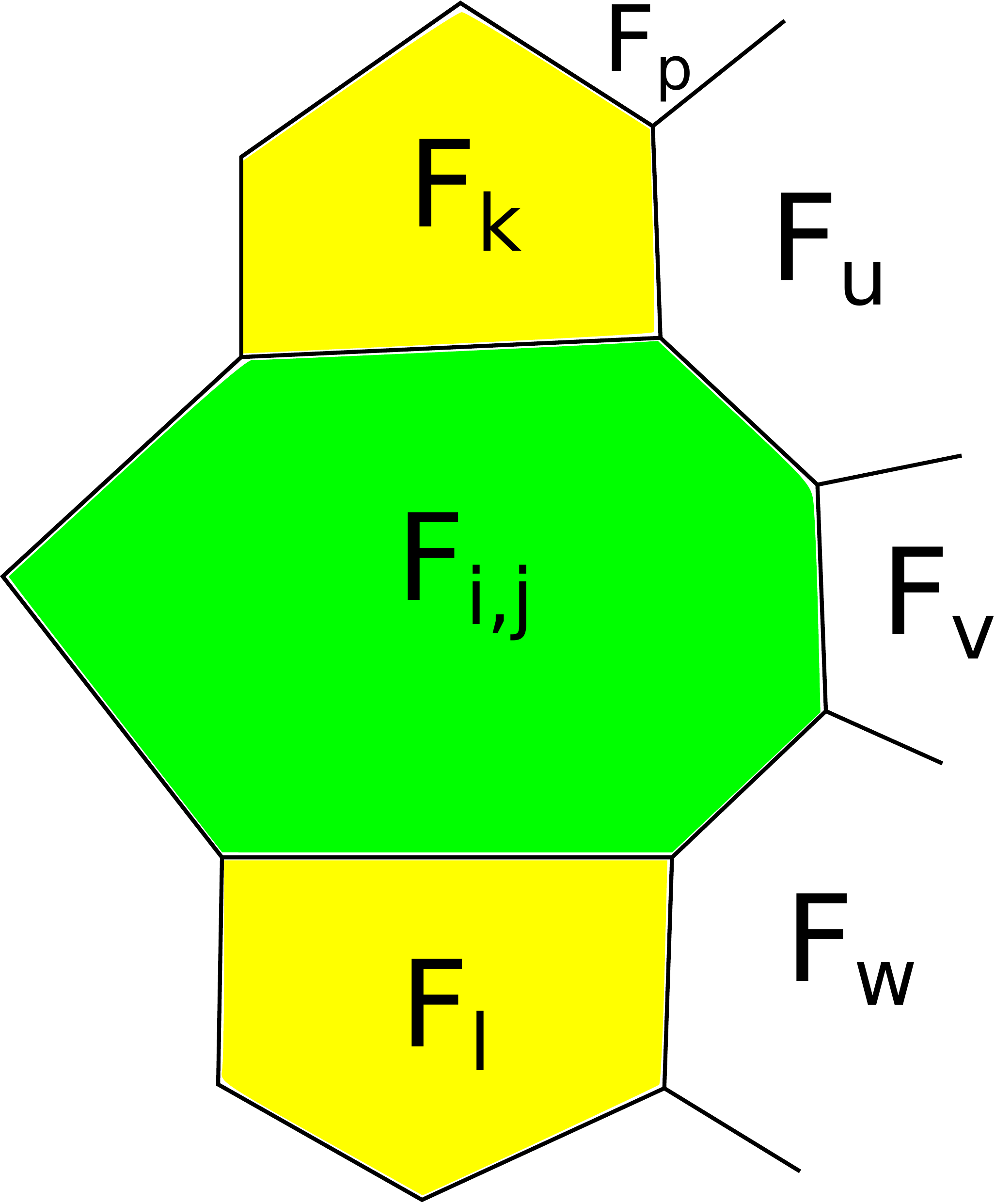}\\
a)&b)
\end{tabular}
\end{center}
\caption{a) Pentagon adjacent to three hexagons; b) the same fragment after the straightening}\label{5666}
\end{figure}
\end{proof}
\begin{remark}
According to Remark \ref{trrem} the $7$ truncations from Theorem \ref{trtheorem} give rise to $9$ different growth operations (see Fig. \ref{9op}):
\begin{itemlist}
\item Each $(1;m_1,m_2)$-truncation gives rise to $2$ growth operations:
\begin{itemlist}
\item if the truncated edge belongs to a pentagon, then we have the patch consisting of the pentagon adjacent to an $m_1$-gon and an $m_2$-gon  by non-adjacent edges;   
\item if the truncated edge belongs to two hexagons, then we have the patch consisting of the hexagon adjacent to an $m_1$-gon and an $m_2$-gon by two edges that are not adjacent and not opposite;
\end{itemlist} 
\item Each of the truncations $(2,6;4,5)$-, $(2,6;5,5)$-,  $(2,6;5,6)$-, $(2,7;5,5)$-, and $(2,7;5,6)$- gives rise to one growth operation.
\end{itemlist}
\begin{figure}
\begin{center}
\includegraphics[height=9cm]{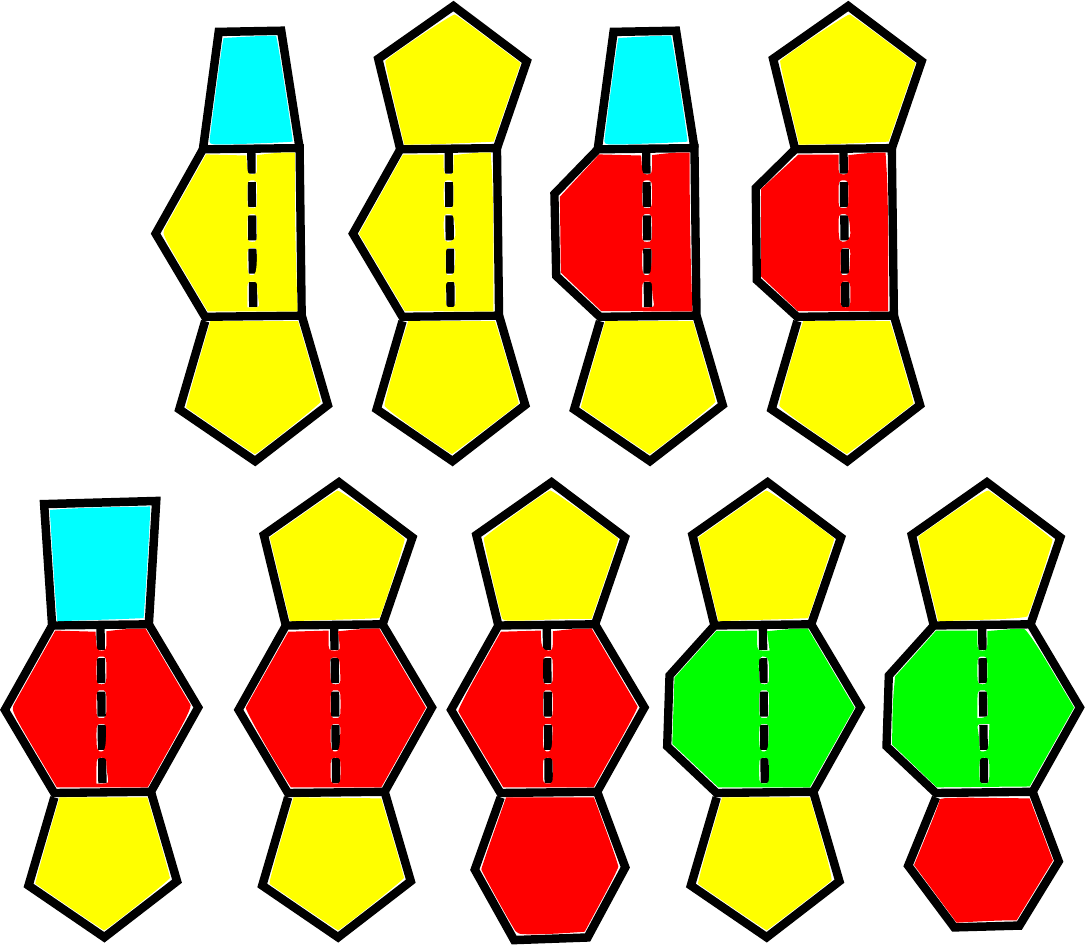}
\end{center}
\caption{$9$ growth operations induced by $7$ truncations}\label{9op}
\end{figure}
If we take care of the orientation, then  $3$ of the operations have left and right versions.
\end{remark}
\newpage
\section*{Acknowledgments}
The content of this lecture notes is based on lectures given by the first author at IMS of
National University of Singapore in August 2015 during the Program on Combinatorial and Toric Homotopy, and the work originated from the second authors participation in this Program. The authors thank Professor Jelena Grbic (University of Southampton), Professor Jie Wu (National University of Singapore), and IMS for organizing the Program and providing such a nice opportunity. 

This work was partially supported by the RFBR grants 14-01-00537 and 16-51-55017, and the Young Russian Mathematics award.

\input{Lectures-F.ind}

\end{document}